\documentclass[reqno,a4paper]{amsart}
\usepackage[all,cmtip]{xy}
\usepackage{amsmath}
\usepackage{amssymb}
\usepackage{amsthm}
\usepackage{amscd}
\usepackage{stmaryrd}
\usepackage{amsfonts}
\usepackage{enumitem}
\usepackage{latexsym}
\usepackage{mathabx}
\usepackage{amscd}
\usepackage{graphicx}
\usepackage{tikz, tikz-cd}
\usepackage{setspace}
\usepackage{todonotes}
\usepackage{mathtools}
\usepackage{stackengine}
\usepackage{etoolbox}
\usepackage{enumitem}
\usepackage[colorlinks=true, linkcolor=blue, citecolor=purple, urlcolor=blue, breaklinks=true]{hyperref}

\usetikzlibrary{decorations.pathmorphing}
\tikzset{anchorbase/.style={baseline={([yshift=-0.5ex]current bounding box.center)}}}
\tikzset{snake it/.style={decorate, decoration=zigzag}}

\makeatletter

\newcommand{\arxiv}[1]{\href{http://arxiv.org/abs/#1}{\tt
    arXiv:\nolinkurl{#1}}}

\def\ocirc#1{\stackrel{_{\,\circ}}{#1}}
\def\ind{\operatorname{ind}}
\def\res{\operatorname{res}}

\def\cotimes{\boxempty}
\def\mh{\mathfrak{h}}
\def\mg{\mathfrak{g}}
\def\la{\lambda}
\def\asc{{\operatorname{asc}}}
\def\desc{{\operatorname{dsc}}}
\def\uu{^{\shortuparrow\!\shortuparrow}}
\def\u{^\shortuparrow}
\def\d{^\shortdownarrow}
\def\star{\coasterisk}

\def\im{\operatorname{im}}
\def\pd{\operatorname{pd}}
\def\injd{\operatorname{id}}
\def\coim{\operatorname{coim}}
\def\soc{\operatorname{soc}\,}
\def\hd{\operatorname{hd}\,}
\def\F{F}
\def\G{G}
\def\omegadual{{\scriptscriptstyle\bigcirc\hspace{-1.9mm}\omega\hspace{.1mm}}}
\def\taudual{{\scriptscriptstyle\bigcirc\hspace{-1.9mm}\tau\hspace{.1mm}}}
\def\sigmadual{{\scriptscriptstyle\bigcirc\hspace{-1.9mm}\sigma\hspace{.1mm}}}
\def\op{{\operatorname{op}}}
\def\cop{{\operatorname{cop}}}

\newtheorem{theorem}{Theorem}[section]
\newtheorem{lemma}[theorem]{Lemma}

\newtheorem{corollary}[theorem]{Corollary} 
\theoremstyle{definition}  
\newtheorem{definition}[theorem]{Definition}
\newtheorem{example}[theorem]{Example}
\newtheorem{conjecture}[theorem]{Conjecture}  

\newtheorem{remark}[theorem]{Remark}

\@addtoreset{equation}{section}

\def\Rep{\mathcal{R}ep}
\def\Vec{\mathcal{V}ec}
\def\fVec{\mathcal{V}ec_{\operatorname{fd}}}
\def\R{\mathcal R}
\def\Tilt{\mathcal{T}ilt}
\def\Ocat{\mathcal O}
\def\Ccat{\mathcal C}
\def\Acat{\mathcal A}
\def\TL{\mathcal TL}
\def\OB{\mathcal OB}

\def\bI{\text{\boldmath$I$}}
\def\bi{\text{\boldmath$i$}}
\def\bj{\text{\boldmath$j$}}
\def\br{\text{\boldmath$r$}}
\def\bs{\text{\boldmath$s$}}

\def\rcomod{\operatorname{comod}\!\operatorname{-\!}}

\def\fdrcomod{\operatorname{comod}_{\operatorname{fd}}\!\operatorname{-}\!}
\def\fdlcomod{\!\operatorname{-}\!\operatorname{comod}_{\operatorname{fd}}}
\def\rmod{\operatorname{mod}\!\operatorname{-}\hspace{-.8mm}}
\def\lmod{\operatorname{\!-\!}\operatorname{mod}}
\def\lgmod{\operatorname{\!-\!}\operatorname{grmod}}
\def\lfdrmod{\operatorname{mod}\hspace{-.35mm}\operatorname{_{lfd}-\hspace{-.8mm}}}
\def\fgrmod{\operatorname{mod}\hspace{-.35mm}\operatorname{_{fg}-\hspace{-.8mm}}}

\def\fprmod{\operatorname{mod}\hspace{-.35mm}\operatorname{_{fp}-\hspace{-.8mm}}}
\def\fdrmod{\operatorname{mod}\hspace{-.35mm}\operatorname{_{fd}-\hspace{-.8mm}}}
\def\fglmod{\operatorname{\!-\!}\operatorname{mod}_{\operatorname{fg}}}

\def\fplmod{\operatorname{\!-\!}\operatorname{mod}_{\operatorname{fp}}}
\def\lfdlmod{\operatorname{\!-\!}\operatorname{mod}_{\operatorname{lfd}}}
\def\fdrmod{\operatorname{mod}\hspace{-.35mm}\operatorname{_{fd}-\hspace{-.8mm}}}
\def\pclmod{\operatorname{\!-\!}\operatorname{mod}_{\operatorname{pc}}}

\def\pcrmod{\operatorname{mod}\hspace{-.35mm}\operatorname{_{pc}-\hspace{-.8mm}}}
\def\dslmod{\operatorname{\!-\!}\operatorname{mod}_{\operatorname{ds}}}
\def\fdlmod{\operatorname{\!-\!}\operatorname{mod}_{\operatorname{fd}}}

\def\res{{\operatorname{res}}}
\newcommand{\der}{\mathbb}
\newcommand{\Coend}{\operatorname{Coend}}
\newcommand{\Cohom}{\operatorname{Cohom}}
\newcommand{\End}{\operatorname{End}}

\newcommand{\Hom}{\operatorname{Hom}} 
\newcommand{\Ext}{\operatorname{Ext}} 
\newcommand{\id}{\operatorname{id}}
\newcommand{\Id}{\operatorname{Id}}

\newcommand{\Z}{\mathbb{Z}}
\newcommand{\N}{\mathbb{N}}
\newcommand{\B}{\mathbf{B}}
\newcommand{\BS}{\mathbf{S}}
\newcommand{\Q}{\mathbb{Q}}
\newcommand{\K}{\mathbb{K}}
\renewcommand{\k}{\Bbbk}
\newcommand{\eps}{\varepsilon}

\newcommand{\unit}{\mathbf{1}}
\newcommand{\rad}{{\operatorname{rad}\,}}

\def\Loc{\operatorname{Env}}
\def\Fin{\operatorname{Fin}}

\def\Ind{\operatorname{Ind}}

\begin{document}

\title[Semi-infinite highest weight categories]{Semi-infinite highest
  weight categories}

\author[J. Brundan]{Jonathan Brundan}
\address{Department of Mathematics,
University of Oregon, Eugene, OR 97403, USA}
\email{brundan@uoregon.edu}

\author[C. Stroppel]{Catharina Stroppel}
\address{Department of Mathematics, University of Bonn, 53115 Bonn, Germany}
\email{stroppel@math.uni-bonn.de}

\dedicatory{Dedicated to Jens Carsten Jantzen on the occasion of his
  70th birthday.}
\thanks{
This article is based upon work done
 while the authors were in
residence at the Mathematical Sciences Research Institute in Berkeley,
California during the Spring 2018 semester.
It was supported by the National Science
Foundation grant DMS-1700905 and by the HCM in Bonn.
}

\begin{abstract}
We develop axiomatics of highest weight 
categories and quasi-hereditary algebras in order to
incorporate two semi-infinite situations which are in Ringel duality
with each other; the underlying posets 
are either
 {\em upper finite} or {\em lower finite}. 
We also consider various more general sorts of stratified categories.
In the upper finite cases, we give an alternative
characterization of these categories
in terms of based quasi-hereditary algebras and based stratified
algebras, which are certain locally unital algebras
possessing triangular bases. 
\end{abstract}

\maketitle
\vspace{-0.5cm}
\tableofcontents
\vspace{-1cm}

\section{Introduction}

Highest weight categories were introduced by Cline, Parshall and Scott
\cite{CPS} in order to provide an axiomatic framework
encompassing a number of important examples which had previously arisen in
representation theory.
In the first part of this article, we give a detailed exposition of 
two semi-infinite variants, 
which we call {\em lower finite}
and {\em upper finite highest weight categories}. 
Lower finite highest weight categories
were already included in
the original work of Cline, Parshall and Scott, although they did not use this language. Well-known
examples include the category $\Rep(G)$ of finite-dimensional
rational representations of a (connected) reductive algebraic group.
On the other hand, the upper finite highest weight categories studied here do not fit into the locally Artinian framework of \cite{CPS}. Nevertheless, there are many examples of upper finite highest weight categories
 already in the literature, often of a diagrammatic nature, and an appropriate axiomatic framework was sketched out by Elias and Losev in \cite[$\S$6.1.2]{EL}.
There are plenty of subtleties, so a full treatment seems desirable.

Then, in the next part, we
extend Ringel duality to the semi-infinite setting:
\begin{eqnarray*}
\left\{
\begin{array}{ccc}\text{lower finite} \\ 
\text{highest weight categories}
\end{array}
\right\}
&\xleftrightarrow{\text{Ringel duality}}&
\left\{
\begin{array}{ccc}
\text{upper finite}\\
\text{highest weight categories}
\end{array}
\right\}.
\end{eqnarray*}
Other approaches to ``semi-infinite Ringel duality'' exist
in the literature, 
but these typically require the existence of a $\Z$-grading; e.g., see
\cite{SoergelKac} (in a Lie algebra setting)
and also \cite{Maz2}.
We avoid this by working 
with finite-dimensional comodules over a coalgebra in the
lower finite case,
 and with
locally finite-dimensional modules over
a locally finite-dimensional locally unital algebra in the upper finite case.
Another approach to semi-infinite Ringel duality
based around pseudo-compact topological algebras 
was initiated by Marko and Zubkov \cite{MZ}. However, their theory requires some additional finiteness assumptions which are not satisfied
in important examples including all non-semisimple categories of the form $\Rep(G)$ for a reductive group $G$; see Corollary~\ref{finalcorollary}, Remark~\ref{finalremark} and 
 Remark~\ref{penultimateremark}.

Finally, as an application of semi-infinite Ringel duality, we give
an elementary algebraic characterization
of upper finite highest weight categories, showing that any such
category is equivalent to the category of locally finite-dimensional
modules over an {\em upper finite based quasi-hereditary algebra}.
This is an algebraic formulation of the notion of object-adapted
cellular category from \cite[Def.~2.1]{ELauda}, and a
generalization of the
based quasi-hereditary algebras of \cite[Def.~2.4]{KM}.
As well as Ringel duality, the proof of this characterization
uses a construction from \cite{AST} to construct bases for endomorphism algebras of tilting
objects. The observation that the bases
arising from \cite{AST} are
object-adapted cellular bases was made already by Elias and several others, and 
appears in recent work of
Andersen \cite{And}.

Throughout the article, we systematically develop the entire theory in the more
general setting of what we call
{\em $\eps$-stratified categories}. The 
idea of this definition is due to \'Agoston, Dlab and Luk\'acs:
in \cite[Def.~1.3]{ADL} one finds the notion
of a {stratified algebra of type $\eps$}; the category of
finite-dimensional left modules over such a finite-dimensional 
algebra is an example of a 
$\eps$-stratified category in our sense. 
The various other generalizations of highest weight category
that have been considered in existing literature 
fit naturally into our $\eps$-stratified framework.

\vspace{2mm}

To explain the contents of the paper in more detail, we start by
explaining our precise setup in the finite-dimensional case, since
even here it does not seem to have appeared explicitly
elsewhere in the literature.
 Consider a {\em finite Abelian category}, that is, a category $\R$
equivalent to the category $A\fdlmod$ of finite-dimensional left
$A$-modules for some finite-dimensional algebra $A$ over an
algebraically closed field $\k$.
A {\em stratification} of $\R$ is a quintuple
$(\B,L,\rho,\Lambda,\leq)$
consisting of a set $\B$,
a {\em labelling function} $L$ such that $\{L(b)\:|\:b \in \B\}$
is a full set of pairwise inequivalent irreducible objects of $\R$,
and
a {\em stratification function}
$\rho:\B \rightarrow \Lambda$
for a poset $(\Lambda, \leq)$.

Given a stratification,
let $P(b)$ (resp., $I(b)$) be a projective cover (resp., injective
hull) of $L(b)$.
For $\lambda \in \Lambda$, let 
$\R_{\leq\lambda}$ (resp., $\R_{<\lambda}$) be the Serre subcategory of $\R$ generated by the
irreducibles
$L(b)$ for $b \in \B$ with $\rho(b) \leq\lambda$ (resp., $\rho(b) < \lambda$).
Define the
{\em stratum} $\R_\lambda$ to be the Serre quotient $\R_{\leq
  \lambda} / \R_{<\lambda}$ with quotient functor $j^\lambda:\R_{\leq\lambda} \rightarrow
\R_\lambda$. For $b \in \B_\lambda := \rho^{-1}(\lambda)$, let 
$L_\lambda(b) := j^\lambda L(b)$. 
These give a full set of pairwise inequivalent irreducible objects in $\R_\lambda$.
Still for $b \in \B_\lambda$, let $P_\lambda(b)$ (resp., $I_\lambda(b)$) be a projective cover (resp., injective hull) of $L_\lambda(b)$ in $\R_\lambda$.

The functor $j^\lambda$
has a left adjoint $j^\lambda_!$ and a right adjoint $j^\lambda_*$.
We refer to these as the {\em standardization}
and {\em costandardization functors}, respectively, following the
language of \cite[$\S$2]{LW}.
Then we introduce the {\em standard, proper standard, costandard} and
{\em proper costandard objects} of $\R$ for $\lambda \in \Lambda$ and 
$b \in \B_\lambda$:
\begin{align}\label{rumble}
\Delta(b) &:= j^\lambda_! P_\lambda(b),
&
\bar\Delta(b) &:= j^\lambda_! L_\lambda(b),
&
\nabla(b) &:= j^\lambda_* I_\lambda(b),
&\bar\nabla(b) &:= j^\lambda_* L_\lambda(b).
\end{align}
Equivalently, $\Delta(b)$ (resp., $\nabla(b)$) is the largest quotient of
$P(b)$ (resp., the largest subobject of
$I(b)$)
that belongs to $\R_{\leq\lambda}$, and
$\bar\Delta(b)$ (resp., $\bar\nabla(b)$) is the largest quotient
of $\Delta(b)$
(resp., the largest subobject of $\nabla(b)$)
such that  
all composition factors apart from its irreducible
head (resp., its irreducible socle) belong to $\R_{<\lambda}$.

Fix a {\em sign function} $\eps:\Lambda
\rightarrow \{\pm\}$ and define
the {\em $\eps$-standard} and {\em $\eps$-costandard objects}
\begin{align}\label{device}
\Delta_\eps(b) &:= \left\{
\begin{array}{ll}
\Delta(b)&\text{if $\eps(\rho(b)) = +$}\\
\bar\Delta(b)&\text{if $\eps(\rho(b)) = -$}
\end{array}
\right.,&
\nabla_\eps(b) &:= \left\{
\begin{array}{ll}
\bar\nabla(b)&\text{if $\eps(\rho(b)) = +$}\\
\nabla(b)&\text{if $\eps(\rho(b)) = -$}
\end{array}
\right..
\end{align}
By a {\em $\Delta_\eps$-flag} (resp., a {\em $\nabla_\eps$-flag})
of an object of $\R$,
we mean a (necessarily finite)
filtration whose
  sections are of the form $\Delta_\eps(b)$ 
(resp., $\nabla_\eps(b)$)
for $b \in \B$.
Then
we call $\R$ an {\em $\eps$-stratified category} if one of the following
equivalent properties holds:
\begin{enumerate}
\item[($P\Delta_\eps$)]
For every $b \in \B$, the projective object $P(b)$ has a 
$\Delta_\eps$-flag with sections
$\Delta_\eps(c)$
for $c\in\B$ with $\rho(c) \geq \rho(b)$.
\item[($I\nabla_\eps$)]
For every $b \in \B$, the injective object $I(b)$ has a
$\nabla_\eps$-flag
with sections
$\nabla_\eps(c)$ for $c\in\B$ with $\rho(c) \geq \rho(b)$.
\end{enumerate}
The fact that these two properties are indeed equivalent 
was established in \cite[Th.~2.2]{ADL} (under slightly more
restrictive hypotheses than here),
extending the earlier work of Dlab \cite{Dlab}.
We give a self-contained proof in Theorem~\ref{fund};
see also $\S$\ref{eg1} for some elementary examples.
An equivalent statement is as follows.

\begin{theorem}[Dlab,\dots]\label{aone}
Let $\R$ be a finite Abelian category equipped with a stratification $(\B,L,\rho,\Lambda,\leq)$
and $\eps:\Lambda\rightarrow\{\pm\}$ be a sign function. 
Then
$\R$ is $\eps$-stratified if and only if $\R^\op$ is
$(-\eps)$-stratified.
\end{theorem}

If the stratification function $\rho:\B\rightarrow \Lambda$ is a bijection,
i.e., each stratum $\R_\lambda$ has a unique irreducible object (up to
isomorphism), then we can use $\rho$ to identify $\B$ with $\Lambda$, and
denote the various distinguished objects simply by $L(\lambda),
P(\lambda), \Delta_\eps(\lambda),$ \dots for 
$\lambda \in \Lambda$
instead of by $L(b), P(b), \Delta_\eps(b),$ \dots for $b \in \B$.
When ($P\Delta_\eps$)--($I\nabla_\eps$) hold in this situation, we
instead call
$\R$ an {\em $\eps$-highest weight category} with {\em weight poset}
$(\Lambda,\leq)$ and {\em labelling function} $L$.
The notion of $\eps$-highest weight category generalizes the original notion of highest weight category from \cite{CPS}:
a (finite) {\em highest weight category} in the sense of {\em loc. cit.} is an
$\eps$-stratified category in which each stratum $\R_\lambda$ is actually
simple, i.e., equivalent to $\fVec$.
This stronger assumption means not only that $\rho$ is a bijection but also that
$L_\lambda(\lambda) = P_\lambda(\lambda)=I_\lambda(\lambda)$, hence,
 $\Delta(\lambda) = \bar\Delta(\lambda)$ and
$\nabla(\lambda)=\bar\nabla(\lambda)$ for each $\lambda \in \Lambda$.
Consequently, in highest weight categories,
the sign function $\eps$ plays no role 
and may be omitted entirely,
and the above properties simplify to the following:
\begin{enumerate}
\item[\qquad($P\Delta$)] 
Each $P(\lambda)$ has a $\Delta$-flag with 
sections $\Delta(\mu)$ for $\mu \geq \lambda$.
\item[($I\nabla$)] Each $I(\lambda)$ has a $\nabla$-flag with 
sections $\nabla(\mu)$ for $\mu \geq \lambda$.
\end{enumerate}
In fact, in this context, the equivalence of $(P\Delta)$ and
$(I\nabla)$ was established already in \cite{CPS}. Moreover, in {\em
  loc. cit.}, it is shown that $A\fdlmod$ is a highest weight category
if and only if $A$ is a {quasi-hereditary algebra}.

The next important special cases arise when 
$\eps$ is the constant function $+$ or $-$.
The idea of a $+$-stratified
category originated in the work of Dlab \cite{Dlab} 
already mentioned, and in another work of Cline, Parshall and Scott \cite{CPS2}. 
In particular, the ``standardly stratified
categories'' of \cite[Def.~2.2.1]{CPS2}
are $+$-stratified
categories.

We say that a finite Abelian category $\R$ 
equipped with a stratification $(\B,L,\rho,\Lambda,\leq)$
is a {\em fully stratified category} if it is both
a $+$-stratified category and a  $-$-stratified category; in that case, it is
$\eps$-stratified for
all choices of the sign function $\eps:\Lambda\rightarrow\{\pm\}$.
Such categories arise as 
categories of modules over the
{fully stratified algebras} introduced in a 
remark after \cite[Def.~1.3]{ADL}.
In fact, these sorts of algebras and categories 
have appeared several times elsewhere in the literature
but under different names: they are called
``weakly properly stratified'' in \cite{F1}, ``exactly properly
  stratified'' in \cite{CZ}, and  ``standardly stratified'' in \cite{LW}.
The latter seems a particularly confusing choice since it clashes
with the established notion from \cite{CPS2} but we completely agree
with the sentiment of \cite[Rem.~2.2]{LW}: fully stratified
categories have a well-behaved structure theory.
One reason for this is that
all of the standardization and costandardization functors in a fully
stratified category are exact.
We note also that any $\eps$-stratified category with 
duality is
automatically fully stratified; see Corollary~\ref{sun} for a precise statement.

We use the language {\em fibered highest weight category} in place
of fully stratified category when the stratification function $\rho$ is
a bijection. Equivalently, a fibered highest weight category is a category which is $\eps$-highest weight for all choices of the sign function $\eps$.
Such categories arise as the categories of finite-dimensional modules over the {\em properly stratified 
algebras} introduced in \cite{Dlabproper}.
It is perhaps worth pointing out that any finite Abelian category can
be given the structure of a fully stratified category in a trivial way
taking the poset $\Lambda$ to be a singleton. Fibered
highest weight categories are at the other extreme with $\Lambda$
being as big as possible.

Table~\ref{tabzero} gives a dictionary between the various different types of finite Abelian category $\R$ discussed so far and the language we adopt for the underlying 
finite-dimensional algebras $A$ such that $\R$ is equivalent to $A\fdlmod$. 
Some of this language is non-standard; see Remark~\ref{czdiscussion} for further discussion.

\begin{table}
$$
\begin{array}{|l|l|}
\hline
\text{Finite-dimensional algebra $A$}&\text{Finite Abelian category $A\fdlmod$}\\
\hline
\text{Quasi-hereditary algebra}&\text{Highest weight category}\\
\text{$\eps$-Quasi-hereditary algebra}&\text{$\eps$-Highest weight category}\\
\text{Properly stratified algebra}&\text{Fibered highest weight category}\\
\text{$\eps$-Stratified algebra}&\text{$\eps$-Stratified category}\\
\text{Stratified algebra}&\text{Fully stratified category}\\
\hline
\end{array}
$$
\caption{Dictionary between algebras and categories}\label{tabzero}
\end{table}

There are many classical examples of highest weight categories,
including blocks of the BGG category $\Ocat$ for a semisimple Lie
algebra, the classical Schur algebra and Donkin's generalized Schur
algebras introduced in \cite{Dschur}, and many more
examples arising from categories of perverse sheaves with
stratifications of geometric origin \cite{BBD}.
Further examples of fully stratified categories and fibered highest
weight categories which are not highest weight 
arise in the context of categorification.  
This includes the pioneering examples of categorified
tensor products of finite dimensional 
irreducible representations for the quantum group attached to
$\mathfrak{sl}_k$ from \cite{FKS} (in particular Remark 2.5 therein),
and 
the categorified induced cell modules for Hecke algebras from
\cite[6.5]{MS}. 
Building on these examples and the subsequent work of Webster
\cite{W1}, \cite{W2}, 
Losev and Webster \cite{LW} formulated the important
axiomatic definition of a {\it tensor product categorification}.
These are
fully stratified categories
which have been used to give a categorical
interpretation of Lusztig's construction of
tensor product of based modules for a quantum group.

The device of incorporating the 
sign function $\eps$ into the definition of $\eps$-stratified or
$\eps$-highest weight category seems to be quite
convenient as it streamlines
many of the subsequent definitions and proofs.
It also leads to some interesting new possibilities when
it comes to the ``tilting theory'' which we discuss next.

\vspace{2mm}

Assume $\R$ is an $\eps$-stratified category as above.
An {\em $\eps$-tilting object} is an object of $\R$ which has both a
$\Delta_\eps$-flag and a $\nabla_\eps$-flag.
Isomorphism classes of 
{indecomposable} $\eps$-tilting objects are parametrized in a canonical way by
the set $\B$; see Theorem~\ref{gin}.
The construction of these objects is a non-trivial generalization of
Ringel's classical construction via
iterated extensions of standard objects: in general one takes a mixture
of extensions
of standard objects on the top for positive strata and extensions of
costandard objects on 
the bottom for negative strata.
We denote the indecomposable $\eps$-tilting objects by $\{T_\eps(b)\:|\:b \in \B\}$.

Now let $T$ be an {\em $\eps$-tilting generator},
i.e., an $\eps$-tilting object in which every $T_\eps(b)$ appears at least once as a summand.
If $\eps=+$ or $-$ (the constant functions) then $T$ is a {\em tilting} or {\em cotilting module}, respectively, for the underlying finite-dimensional
algebra in the general sense of tilting theory; 
for more general $\eps$, $T$ is an example of a {\em Wakamatsu tilting module} 
as defined in \cite[$\S$4.1]{Rei}.
The {\em Ringel dual} of $\R$ relative to $T$ is the category
$\R' := 
B\fdlmod$ where $B := \End_\R(T)^\op$ (so that $T$ is a right $B$-module).
The isomorphism classes of 
irreducible objects in 
$\R'$ are in natural bijection with the
isomorphism classes of indecomposable summands of $T$, hence, they may
be
indexed
by the same set $\B$ that labels the irreducibles in $\R$.
We denote them by $\{L'(b)\:|\:b \in \B\}$.
Let
\begin{align*}
\F := \Hom_{\R}(T,?)&:\R \rightarrow \R',\\
G := \Cohom_\R(T,?) = \Hom_\R(?,T)^*&:
\R \rightarrow \R'.
\end{align*}
These are the {\em Ringel duality functors}.
The following theorem is well known for highest weight categories
(where it is due
to Ringel \cite{R} and Happel \cite{H})  and for $+$- and $-$-stratified
categories (where it is developed in the framework of
standardly stratified algebras in \cite{AHLU}). We prove it for
general $\eps$-stratified categories in Theorem~\ref{Creek}.

\begin{theorem}[Ringel, Happel, \dots]\label{honey}
Let $\R'$ be the 
Ringel dual of $\R$ relative to an $\eps$-tilting generator $T$ as
above. Let $-\eps:\Lambda\rightarrow \{\pm\}$ be the negation of the
original sign function $\eps$.
\begin{enumerate}
\item
  The quintuple $(\B,L',\rho,\Lambda, \geq)$
  is a stratification
  of $\R'$ making it into a
  $(-\eps)$-stratified category with weight poset $(\Lambda, \geq)$,
  that is, the opposite of the poset used for $\R$.
Moreover, each stratum $\R'_\lambda := \R'_{\geq\lambda} / \R'_{>\lambda}$ of $\R'$ is
equivalent to the corresponding stratum $\R_\lambda : =\R_{\leq\lambda} / \R_{<\lambda}$ of $\R$.
\item
The functor
$\F$
defines an equivalence of categories between the category of
$\nabla_\eps$-filtered objects in $\R$ and
the category of $\Delta_{-\eps}$-filtered objects in $\R'$.
It sends $\eps$-tilting objects (resp., injective objects) in $\R$ to
projective objects
(resp., $(-\eps)$-tilting objects) in $\R'$.
\item
The functor
$G$
defines an equivalence of categories between the category of
$\Delta_\eps$-filtered objects in $\R$ and
the category of $\nabla_{-\eps}$-filtered objects in $\R'$.
It sends $\eps$-tilting objects (resp., projective objects) in $\R$ to
injective objects
(resp., $(-\eps)$-tilting objects) in $\R'$.
\item
Assume that $\R_\lambda$ is 
of finite global dimension for 
all strata $\lambda$ with $\eps(\lambda)=-$ (resp., $\eps(\lambda)=+$). Then the total 
derived functor $\der{R} \F:D^b(\R) \rightarrow
D^b(\R')$
(resp., $\der{L} \G:D^b(\R) \rightarrow
D^b(\R')$) is an equivalence
between the bounded derived categories.
\end{enumerate}
\end{theorem}

In the setup of the theorem, let $P$ be a projective generator
for $\R$. Then $T' := G P$ is a $(-\eps)$-tilting generator for
$\R'$ such that
$A := \End_\R(P)^\op \cong \End_{\R'}(T')^\op$.
Since $\R$ is equivalent to $A\fdlmod$, this shows that $\R$
is equivalent to the Ringel
dual $(\R')'$ of $\R'$ relative to $T'$. Thus,
the original category $\R$ can be recovered from its Ringel
dual $\R'$.
This statement 
can be interpreted as a {\em double centralizer property}:
starting from $\R = A\fdlmod$ so that $T$ is an $(A,B)$-bimodule,
and taking the projective generator $P$ to be 
the left regular $A$-module so that $A \cong \End_A(P)^\op$, the $(B,A)$-bimodule
$T' = GP$ is isomorphic to the dual $T^*$ of $T$. Now Theorem~\ref{honey}(3) implies that $A \cong \End_{B}(T^*)^\op$.

We do not consider here derived equivalences in the case of infinite
global dimension, but instead refer to \cite{PS}, where this and
involved $t$-structures are treated in detail by generalizing the
classical theory of co(resolving) subcategories. This requires the use
of certain coderived and contraderived categories
in place of ordinary derived categories.

\vspace{2mm}

Now we shift our attention to
{the semi-infinite case}, which is really
the main topic of the article.
Following \cite{EGNO},
a {\em locally finite Abelian category} is a category that is
equivalent to the category $\fdrcomod C$ of finite-dimensional right
comodules over some coalgebra $C$.
Let $\R$ be such a category.
A {\em lower finite stratification} of $\R$ is a quintuple
$(\B,L,\rho,\Lambda,\leq)$
consisting of a set $\B$, 
a function $L$ labelling
a full set $\{L(b)\:|\:b \in \B\}$
of pairwise inequivalent irreducible objects,
a stratification function $\rho:\B\rightarrow \Lambda$ required now to
have finite fibers $\B_\lambda
:= \rho^{-1}(\lambda)$, and 
a lower finite poset $(\Lambda, \leq)$
(i.e., the intervals
$(-\infty,\mu]$ are finite for all $\mu \in \Lambda$).
Fix also a sign function $\eps:\Lambda\rightarrow\{\pm\}$.
For any lower set (i.e., ideal of the poset) $\Lambda\d$ in $\Lambda$, we can consider the Serre
subcategory $\R\d$ of $\R$ generated by the objects
$\{L(b)\:|\:b \in \B\d\}$
where $\B\d:=\rho^{-1}(\Lambda\d)$.
The restriction of the stratification of $\R$
gives a stratification
$(\B\d,L,\rho,\Lambda,\leq)$ of $\R\d$.
We say that $\R$ is a {\em lower finite $\eps$-stratified category}
if $\R\d$ is a finite Abelian category that is $\eps$-stratified in
the earlier sense
for every finite lower set $\Lambda\d$ of $\Lambda$;
cf. Definition~\ref{newlfd}.
By the same procedure one also obtains definitions of lower finite
$\eps$-highest weight, lower finite fully stratified, 
lower finite fibered highest weight, and lower finite highest weight
categories.

In a lower finite $\eps$-stratified category $\R$, there are $\eps$-standard and
$\eps$-costandard objects $\Delta_\eps(b)$ and $\nabla_\eps(b)$;
they are the same as the $\eps$-standard and $\eps$-costandard objects
of the Serre subcategory 
$\R\d$ defined from any finite lower set $\Lambda\d$ containing
$\rho(b)$. 
As well as (finite) $\Delta_\eps$- and
$\nabla_\eps$-flags, one can consider certain infinite
$\nabla_\eps$-flags in objects of the ind-competion $\Ind(\R)$ (which
is the category $\rcomod C$ of all right $C$-comodules in the case
that $\R = \fdrcomod C$).
We refer to these as {\em ascending $\nabla_\eps$-flags}; see
Definition~\ref{good} 
for the precise formulation.
Theorem~\ref{thethmnew} establishes a homological criterion for an
object to possess an
ascending $\nabla_\eps$-flag similar to the
well-known criterion 
for good filtrations in rational representations
of reductive groups
\cite[Prop.~II.4.16]{J}.
From this, it follows that the injective hull $I(b)$ of $L(b)$ in
$\Ind(\R)$ has an ascending $\nabla_\eps$-flag. Moreover, the
multiplicity
of $\nabla_\eps(c)$ as a section of such a
flag satisfies
$$
(I(b):\nabla_\eps(c)) = [\Delta_\eps(c):L(b)],
$$
generalizing BGG reciprocity.
This leads to alternative ``global''  characterizations of lower
finite $\eps$-stratified and fully stratified categories;
see Theorems~\ref{globalchar} and \ref{anotherchar}.
The latter involves an $\Ext^2$-vanishing condition which
first appeared in work of Dlab and Ringel \cite{DlabRingel}.
\vspace{2mm}

In a lower finite $\eps$-stratified category, there are also
$\eps$-tilting objects. Isomorphism classes of the 
indecomposable ones are labelled by $\B$ just like in the finite
case. In fact, 
for $b \in \B$ the corresponding indecomposable $\eps$-tilting
object of $\R$ is the same as the object
$T_\eps(b)$ 
of the Serre subcategory
$\R\d$ defined from any finite lower set $\Lambda\d$
containing $\rho(b)$.
By an {\em $\eps$-tilting generator} for $\R$, we now mean an
object $T = \bigoplus_{i \in I} T_i \in \Ind(\R)$
with a given decomposition as a direct sum of
$\eps$-tilting objects $T_i \in \R$ such 
that each $T_\eps(b)$ appears at least once as a summand of $T$.
Then the {\em Ringel dual} $\R'$ of $\R$ relative to
$T$ is the category 
$A\lfdlmod$ of
{locally finite-dimensional left modules} over the 
locally finite-dimensional locally unital algebra 
$$
A = \Big(\bigoplus_{i,j \in I} \Hom_{\R}(T_i,T_j)\Big)^\op,
$$
where the op denotes that multiplication in $A$ is the opposite of composition in $\R$;
see Definition~\ref{rd1}.
Saying that $A$ is {\em locally unital} means that $A =
\bigoplus_{i,j \in I} e_i A e_j$ where $\{e_i\:|\:i \in I\}$ are the
mutually orthogonal idempotents defined by the identity endomorphisms of
each $T_i$, and {\em locally finite-dimensional} means that $\dim e_i A
e_j < \infty$ for all $i,j\in I$.
A {\em locally finite-dimensional module} is an $A$-module $V 
= \bigoplus_{i \in I} e_i V$ with $\dim e_i V < \infty$
for each $i$.
As $e_i A e_j = \Hom_{\R}(T_i, T_j)$ is finite-dimensional, 
each left ideal $A e_j$ is a locally finite-dimensional projective module.

This brings us to the notion of an {\em upper finite $\eps$-stratified category},
whose definition may be discovered by considering the nature of
the categories $\R'$ that can arise as Ringel duals of
lower finite $\eps$-stratified categories.
We refer to Definition~\ref{ufc} for the intrinsic formulation;
there are also upper finite counterparts of  $\eps$-highest
weight, fully stratified, fibered highest weight and highest weight categories.
Starting from $\R$
that is a lower finite $\eps$-stratified category as above, the Ringel dual
$\R'$ comes equipped with an {\em upper finite stratification}
$(\B,L',\rho,\Lambda,\geq)$
making it into an
upper finite $(-\eps)$-stratified category; see
Theorem~\ref{rt1} which extends parts (1) and (2) of
Theorem~\ref{honey}.

In general,
in an upper finite $\eps$-stratified category,
the underlying poset is required to be upper finite, i.e., 
all of the intervals $[\lambda,\infty)$ are finite. 
There are $\eps$-standard and
$\eps$-costandard objects, but now these can have infinite length
(although composition multiplicities in such objects are finite). On
the other hand, the indecomposable projectives and injectives do still have
finite $\Delta_\eps$-flags and $\nabla_\eps$-flags, exactly like in
$(P\Delta_\eps)$ and $(I\nabla_\eps)$. Perhaps the most interesting feature 
is that one can still make sense of $\eps$-tilting
objects. These are objects possessing certain infinite flags: 
both an {\em ascending $\Delta_\eps$-flag} and a {\em descending
  $\nabla_\eps$-flag};
see Definition~\ref{tpc}.
This allows us to define the {\it Ringel dual} of an upper finite
$\eps$-stratified category relative to an $\eps$-tilting generator $T$: it is the category $\fdrcomod C$ for the
coalgebra $C:=\Coend_\R(T)$
that is the continuous dual
of the opposite endomorphism algebra
$B:=\End_\R(T)^\op$; see Theorem~\ref{rt2} which extends parts (1) and (3) of Theorem~\ref{honey}.
This makes sense because $B$ is a pseudo-compact topological algebra;
see Lemma~\ref{prof}.

Again there are double centralizer properties.
For $\R'$ arising as the Ringel dual of a lower finite
$\eps$-stratified category $\R$ relative to 
$T=\bigoplus_{i \in I} T_i$, the
indecomposable $(-\eps)$-tilting
objects in $\R'$
are the images of the indecomposable injective objects of $\R$
under
$$
F := \bigoplus_{i \in I}
\Hom_{\R}(T_i,?) : \R \rightarrow \R'
$$ 
and, given a $(-\eps)$-tilting generator $T'$ for $\R'$, the
Ringel dual $(\R')'$ of $\R'$ relative to $T'$
is equivalent
to the original 
category $\R$; see Corollary~\ref{mustard1} and also 
$\S$\ref{eg1a} for an explicit example.
Similarly, for $\R'$ arising as the Ringel dual of an upper finite $\eps$-stratified category relative to $T$, the indecomposable $(-\eps)$-tilting objects of $\R'$ are the images of the indecomposable projective objects of $\R$ under 
$G := \Cohom_\R(T,?)$ and, given a $(-\eps)$-tilting generator 
$T' =\bigoplus_{i \in I} T_i'$ for $\R'$,
the Ringel dual $(\R')'$ of $\R'$ relative to
$T'$ is equivalent to $\R$; see Corollary~\ref{mustard2}.

\vspace{2mm}

In $\S$\ref{newsec1}, we apply semi-infinite Ringel
duality together with arguments from \cite{AST} to give
an elementary algebraic characterization of upper finite highest weight
categories in terms of upper finite {\em
based quasi-hereditary algebras}. In the finite-dimensional setting,
these are the {based quasi-hereditary algebras} 
defined by Kleshchev and Muth in \cite{KM}, who 
proved that their definition of based quasi-hereditary algebra 
is equivalent to the original definition of
quasi-hereditary algebra from \cite{CPS}; we have streamlined the definition a little further here.
Our more general algebras are locally finite-dimensional locally unital
algebras rather than unital algebras. Viewing them instead as
{finite-dimensional categories}, that is,
small $\k$-linear categories
with finite-dimensional morphism spaces, the definition translates
into something equivalent to the notion of an {\em object-adapted cellular category} which
was introduced already by Elias and Lauda \cite[Def.~2.1]{ELauda}. (In turn, the
Elias-Lauda definition evolved from
work of Westbury \cite{Westbury}, who extended the definition of cellular algebra
due to Graham and Lehrer \cite{GL} 
from finite-dimensional algebras to finite-dimensional categories.)


We say that a fully stratified category is {\em tilting-rigid} if
there is a bijection $\nu:\B\rightarrow \B$ such that
$T_+(b) \cong T_-(\nu(b))$ for all $b \in \B$; see Definition~\ref{tiltingrigiddef}.
In the finite case, $\R$ is tilting-rigid if and only if it is Gorenstein with strata that are 
quasi-Frobenius (then $\nu$ encodes their Nakayama permutations); see Theorem~\ref{goren} which generalizes 
\cite[Th.~2.2]{CM}.
The situation is even better if in addition all of the strata are
symmetric, since in that case the tilting objects $T_\eps(b)$ are  isomorphic for all choices of
the sign function $\eps$ so that they may all be denoted by
$T(b)$.
Most of the naturally-occurring examples of fully stratified
categories are tilting-rigid with symmetric strata, 
including the tensor product categorifications from
\cite{LW} mentioned earlier.
For us, the key point about the tilting-rigid hypothesis is that
the Ringel dual of a tilting-rigid fully stratified
category is again a tilting-rigid fully stratified category; see Theorem~\ref{tiltingrigiddual}.
This is important in $\S$\ref{newsec3},
when we introduce notions of
{\em based stratified algebras} 
and {\em based properly stratified algebras}; see Definitions~\ref{strawberries} and \ref{strawberries2}.
These
have a similar flavor to the fibered object-adapted cellular
categories of \cite[Def.~2.17]{ELauda}. We show that the category of locally
finite-dimensional modules over an upper finite based stratified algebra
(resp., upper finite based properly stratified algebra) is an upper finite 
fully stratified (resp., fibered highest weight) category, and
conversely any such category 
which is also tilting-rigid with symmetric strata 
can be realized in this way.

The definition of an upper finite based stratified algebra $A$
involves certain {basic} finite-dimensional algebras
$A_\lambda\:(\lambda \in \Lambda)$ which provide explicit realizations 
of the strata. Their direct sum
$\bigoplus_{\lambda \in \Lambda} A_\lambda$ is a locally unital algebra which plays the role of ``Cartan subalgebra'', although in general it is not a subalgebra of $A$.
The assumption that the algebras $A_\lambda$ are basic can in fact be dropped entirely.
On doing that one obtains a weaker notion which we 
call an algebra with a {\em triangular basis}; see Definition~\ref{raspberries}.
Our understanding of this definition was 
influenced by the recent preprint \cite{GRS} in which the authors introduce the closely-related notion of an
algebra with a {\em weak triangular decomposition}; up
to a choice of basis, this
is the same as an algebra with a
triangular basis in our sense 
in which all distinguished idempotents are special.
It is still the case that the category of locally finite-dimensional modules over
such an algebra is an upper finite fully stratified
category, just like
for based stratified algebras. This observation is due to Gao, Rui and
Song \cite[Th.~3.5]{GRS}; we give a slightly different proof in
Theorem~\ref{grs}.
Gao, Rui and Song also discuss some interesting examples arising from
cyclotomic quotients of the affine Brauer and 
oriented Brauer categories and their $q$-analogs.

For many of the naturally occurring algebras $A$ with a triangular
basis,
the upper and lower halves of the basis span a pair of {\em opposite Borel subalgebras} $A^\flat$ and $A^\sharp$;
this includes all of the level one cyclotomic quotients from
\cite{GRS} but not the ones of higher level.
In Definition~\ref{prankster}, we formalize this idea with the final notion of an algebra with a {\em triangular decomposition}.
The first author came upon essentially
this definition originally from considerations involving the oriented Brauer category 
and its $q$-analog; see \cite{Reynolds}, \cite{Bskein} and also \cite{RS}, 
which applies a similar approach in the context of the Brauer category.
A closely related notion of {\em triangular category} was developed
independently by Sam and Snowden \cite{SS} in order to study these and
other examples; see also \cite{CZ}.
In the presence of a triangular decomposition, the ``Cartan subalgebra"
$\bigoplus_{\lambda \in \Lambda} A_\lambda$
may be identified with $A^\circ := A^\flat \cap A^\sharp$, so that now it is actually a subalgebra of $A$, and the standardardization/costandardization functors can be realized as parabolic induction/coinduction functors. 
In Theorem~\ref{myconstruction}, 
we explain a general construction to make any algebra with a triangular decomposition into a based stratified algebra. If $A^\circ$ is semisimple, as is the case for the examples arising from the (oriented) Brauer category in characteristic zero but not in positive characteristic, this produces a based quasi-hereditary algebra.
There are other advantages to having a triangular decomposition rather than merely a triangular basis, e.g., see \cite{SS} where triangular decompositions are used to show that many of the motivating examples are Noetherian.

\begin{table}
  $$
  \begin{tikzpicture}
\node[align=center,draw,text width=3.9cm,fill=lime,rounded corners] at (-1.8,.8) (a) {Upper finite highest weight categories};
  \node[align=center,draw,fill=lime,text width=3.3cm,rounded corners] at (6.4,.8) (b) {Upper finite fully stratified categories};
\node[align=center,fill=yellow,draw,text width=3.3cm,rounded corners] at (6.4,-3.5) (c) {Upper finite based stratified algebras};
\node[align=center,fill=yellow,draw,text width=3.9cm,rounded corners] at (-1.8,-3.5) (d) {Upper finite based quasi-hereditary algebras};
\node[align=center,fill=yellow,draw, text width=3.8cm,rounded corners] at (2.5,-.9) (e) {Algebras with an upper finite triangular basis};
\node[align=center,draw,fill=yellow,text width=4.6cm,rounded corners] at (2.5,-5) (f) {Algebras with an upper finite triangular decomposition};
\draw ([yshift=-1.5mm]a.east) edge[->,bend right=15] ([yshift=-1.5mm]b.west);
\draw ([yshift=1.5mm]b.west) edge[->,dashed,bend right=15] node [above,midway] {\tiny$+$highest weight strata (\ref{eve})} ([yshift=1.5mm]a.east);
\draw [<-] (b) edge[bend right = 20] node [left] {\tiny (\ref{twoway})}(c);
\draw [->,dashed] (b) edge[bend left = 20] node [right,xshift=-1mm] {\tiny$+$tilting-rigid (\ref{TBA3})} (c);
\draw [->] (a) edge[bend right = 20] node [left] {\tiny(\ref{TBA})} (d);
\draw [<-] (a) edge[bend left = 20]  node [right] {\tiny(\ref{oneway})}(d);
\draw [->,dashed] (e) edge[bend left = 20] node[text width=1cm,right,xshift=-1mm,yshift=-6mm] {\tiny$+$opposite\\\vspace{-1.2mm} Borels} (f);
\draw [<-,dashed] (e) edge[bend right = 20] node[text width = 2.5cm,left,yshift=-6.3mm,xshift=9mm] {\tiny$+\B$-free (\ref{iflocallyfree})} (f);
\draw [dashed, <-] (d) edge[bend left = 20] node [right,text width=2.3cm,xshift=-2mm,yshift=-3mm] {\tiny$+$basic semisimple\\\vspace{-1.2mm}  Cartan}(e);
\draw[<-] (e) edge[bend left=20] (d);
\draw [<-,dashed] (c) edge[bend right = 20] node [left,yshift=-2mm,xshift=2mm] {\tiny$+$basic Cartan}(e);
\draw [->] (c) edge[bend left = 20](e);
\draw [->] (e) edge[bend right = 20] node [left,yshift=1mm,xshift=.5mm] {\tiny(\ref{grs})} (b);
\draw [dashed,->] (f) edge[bend left = 20]node [text width=3cm,left,xshift=1.4cm,yshift=-.5mm] {\tiny$+$semisimple\\\vspace{-1.2mm}Cartan (\ref{myconstructioncor})}(d);
\draw [->] (f) edge[bend right = 20] node[right] {\tiny(\ref{myconstruction})}(c);
\draw [dashed,<-] (a) edge[bend right = 20]node [right,xshift=-5mm,yshift=4mm] 
{\tiny$+$quasi-hereditary Cartan (\ref{laughter})}(e);
\end{tikzpicture}
$$
\caption{Upper finite algebras and categories}\label{yellowgreentable}
\end{table}

Table \ref{yellowgreentable} summarizes some of the connections established between these various types of algebras and their module categories.
 In the main body of the text, we also discuss a parallel situation
involving
{\em essentially finite} rather than upper finite algebras and
categories. 
For example, the finite-dimensional graded algebras with a triangular
decomposition studied in \cite{HN}, \cite{BT} fit naturally into our more
general framework of algebras with an essentially finite
triangular decomposition; see Remark~\ref{triangularhistory1}.

As we have already mentioned, the category
$\R := \Rep(G)$ for a reductive group $G$ is the archetypical example of a lower finite
highest weight category. Its Ringel dual $\R'$ is an upper finite
highest weight category. 
This case has been studied in particular
by Donkin
(e.g., see \cite{Dschur}, \cite{Dtilt}), but Donkin's approach 
involves truncating 
to a finite-dimensional algebra from the outset.
The double centralizer property allowing $\R$ to be reconstructed
from $\R'$ in this case can be interpreted as a shadow of the Tannakian formalism; see Theorem~\ref{rdt}.
Other important examples of semi-infinite Ringel duality come from blocks of 
category $\Ocat$ over an affine
Lie algebra: in negative levels one obtains lower finite highest weight
categories, while positive levels produce the upper finite ones which
are their Ringel duals.
These and several
other prominent examples are outlined in $\S\S$\ref{eg2}--\ref{eg5}.

We would finally like to remark that our semi-infinite versions of highest weight categories should not be confused with the affine highest weight categories of \cite{Klhw}, and
our based quasi-hereditary algebras are not 
affine quasi-hereditary algebras in the sense of \cite{Klhw}. 
The latter are special examples of affine cellular algebras introduced in \cite{Xi}, \cite{KX}. They are not covered by out setup since we require that strata can be realized by finite-dimensional algebras over an algebraically closed field. To incorporate them, one would need to develop the theory here over more general commutative ground rings as suggested in Remark~\ref{oo}.

\vspace{2mm}

\noindent
{\em Acknowledgements.}
The first author would like to thank Ben Elias, Alexander Kleshchev 
and Ivan Losev for many illuminating
discussions. In particular, the fact that Ringel duality could be
extended to upper finite highest weight categories was originally explained to
this author by Losev. 
The second author would like to thank Henning Andersen, 
Shrawan Kumar and Wolfgang Soergel for several useful
discussions on topics related to this paper. The authors also thank
Tomoyuki Arakawa, Peter Fiebig and Julian K{\"u}lshammer for helpful
comments, and Kevin Coulembier for pointing out a mistake in the
treatment of ind-completions in $\S$\ref{errorhere} of the first
version of this article.

\section{Some finiteness properties on Abelian categories}

We fix an algebraically closed field $\k$.
All algebras, categories, functors, etc. will be assumed to be linear
over $\k$.
We write $\otimes$ for $\otimes_\k$.
The naive terms {\em direct limit} and {\em inverse limit} will be used for
small 
filtered colimits and limits, respectively.
We begin by introducing some language for Abelian
categories with various finiteness properties; see Table \ref{pinktable}.

\subsection{Finite and locally finite Abelian categories}\label{pct}
According to \cite[Def.~1.8.5]{EGNO}, a {\em finite Abelian
  category} is a category that is 
equivalent to the category $A\fdlmod$ of finite-dimensional (left)
modules over some finite-dimensional algebra $A$.
We refer to a choice for the algebra $A$ here as an {\em algebra realization} of $\R$.
Note that the opposite category is also a finite Abelian category
as it is equivalent to the category $A^\op \fdlmod = \fdrmod A$ 
due to the existence of the contravariant equivalence
\begin{equation}
?^*:A\fdlmod \rightarrow \fdrmod A
\end{equation}
taking a finite-dimensional 
left $A$-module to the linear dual viewed as a right
$A$-module in the natural way.

\begin{table}
  $$
\begin{tikzpicture}
\node [align=center,draw,fill=pink,rounded corners,text width=4cm] (a) {Finite Abelian categories};
\node [align=center,draw,text width=3cm,fill=pink,rounded corners,below= of a] (b) {Essentially finite Abelian categories};
\node [align=center,draw,text width=1.8cm,fill=pink,below right = of b,rounded corners] (d) {Schurian categories};
\node [align=center,draw,text width=2.9cm,below left= of b,fill=pink,rounded corners] (c) {Locally finite Abelian categories};
\draw (a) edge[->] node [left]{\tiny(\ref{kitchen})} (b);
\draw (b) edge[->] [bend left = 20] node [right]{\tiny (\ref{vera})} (c);
\draw (b) edge[->,dashed] [bend right = 20] node [left]{\tiny (\ref{coffee})} (d);
\draw (a) edge[->] [bend right = 20]node [left]{\tiny (\ref{kitchen})}  (c);
\draw (a) edge[->] [bend left = 20] node [right]{\tiny (\ref{kitchen})}(d);
\end{tikzpicture}
$$
\caption{Finiteness properties}\label{pinktable}\end{table}

A finite Abelian category can also be characterized as
a category which is equivalent to the category $\fdrcomod
C$ of finite-dimensional (right) comodules over some
finite-dimensional coalgebra $C$.
To explain this in more detail,
recall that the dual $A:=C^*$ of a 
finite-dimensional 
coalgebra $C$ has a natural algebra structure with multiplication
$A \otimes A \rightarrow A$ that is the dual of the
comultiplication $C\rightarrow C \otimes C$;
for this, one needs to use the canonical isomorphism
\begin{equation}\label{basic}
C^* \otimes C^* \rightarrow (C \otimes C)^*,
\qquad
f \otimes g \mapsto (v \otimes w \mapsto f(v) g(w))
\end{equation}
to identify $C^* \otimes C^*$ with $(C \otimes
C)^*$.
Then any right $C$-comodule can be viewed as a left $A$-module with
action defined from
$a v := \sum_{i=1}^n a(c_i) v_i$
assuming here that the structure map
$\eta:V \rightarrow V \otimes C$ sends
$v\mapsto \sum_{i=1}^n v_i \otimes c_i$.
Conversely, the $C$-comodule structure on $V$ can be recovered uniquely from the
action of $A$.
Thus, the categories
$\fdrcomod C$ and $A \fdlmod$ are isomorphic.

A {\em locally finite Abelian category} 
is a category $\R$ that is equivalent
to $\fdrcomod C$ for a (not
necessarily finite-dimensional) coalgebra $C$.
We refer to a choice of $C$ as a {\em coalgebra realization} of $\R$.
The following result of Takeuchi
gives an intrinsic characterization of locally finite Abelian
categories; see \cite{T} and
\cite[Th.~1.9.15]{EGNO}. It is a version of \cite[Th.~IV.4]{Gab} adapted to our situation.
Note Takeuchi's original paper uses the language ``locally finite
Abelian'' slightly differently (following \cite{Gab})
but his formulation of the result is equivalent to
the one here (which follows \cite[Def.~1.8.1]{EGNO}).
In {\em loc. cit.} 
it is shown moreover that $C$ 
can be chosen so that it is {\em  pointed}, i.e., all of
its irreducible comodules are one-dimensional; in that case, $C$ is unique up to isomorphism.

\begin{lemma}\label{char}
An essentially small category $\R$ is a locally finite Abelian
category 
if and only if it is Abelian, all of its objects are of finite
length, and all of its morphism spaces are
finite-dimensional.
\end{lemma}

In view of Lemma~\ref{char}, one could also define a locally finite Abelian category to be a category that is equivalent to $A\fdlmod$ for a (not necessarily finite-dimensional) unital algebra $A$, but 
we prefer to work in terms of comodules since this language facilitates the passage to the ind-completion.
To explain this in more detail, consider the locally finite Abelian
category
$$
\R = \fdrcomod C.
$$
Fix a full set of pairwise inequivalent irreducible objects
$\{L(b)\:|\:b \in \B\}$ in $\R$.
By Schur's Lemma, we have that $\End_\R(L(b)) = \k$ for each $b \in \B$.
Note that the opposite
category $\R^\op$ is again a locally finite Abelian category, and
a coalgebra realization for it is given by 
the opposite coalgebra $C^\cop$.
This follows because there is a
contravariant equivalence
\begin{equation}\label{aduality}
?^*:\fdrcomod C \rightarrow C \fdlcomod
\end{equation}
sending a finite-dimensional right comodule to the dual
vector space 
viewed as a left comodule in the natural way:
if $v_1,\dots,v_n$ is a basis for $V$,  with dual basis $f_1,\dots,f_n$ for $V^*$,
and the structure map $V \rightarrow V \otimes C$ sends $v_j
\mapsto \sum_{i=1}^n v_i \otimes c_{i,j}$
then the dual's structure map
$V^* \rightarrow C \otimes V^*$
sends $f_i \mapsto \sum_{i=1}^n c_{i,j} \otimes f_j$.
Since we have that 
$C\fdlcomod \cong \fdrcomod C^\cop$, we deduce that
$\R^\op$ is equivalent to $\fdrcomod C^\cop$.

In general, $\R$ need not have
enough injectives or projectives. To get injectives, 
we pass to the
{\em ind-completion} $\Ind(\R)$;
see e.g. \cite[$\S$6.1]{KS}. 
For $V, W \in \Ind(\R)$, we write 
$\Ext^n_\R(V,W)$, or sometimes $\Ext^n_C(V,W)$, for 
$\Ext^n_{\Ind(\R)}(V,W)$; it may be computed
via an injective resolution of $W$ in the ind-completion.
This convention 
is unambiguous due to \cite[Th.~15.3.1]{KS}; see also
\cite[Th.~2.2.1]{C2}.
One can also consider the right derived functors
$\der{R}^n F$ of a left exact functor $F:\Ind(\R) \rightarrow \R'$ 
to an Abelian category $\R'$.

Let $\rcomod C$ be the category of all right $C$-comodules.
Every comodule is
the union (hence, the direct limit) of its finite-dimensional
subcomodules.
Moreover,
a comodule $V$ is compact, i.e., the functor $\Hom_C(V,?)$ commutes
with direct limits, if and only if it is finite-dimensional.
Using this, 
\cite[Cor.~6.3.5]{KS} implies that
the canonical functor $\Ind(\R)\rightarrow \rcomod C$
is an equivalence of categories.
This means that one can work with $\rcomod C$ in place of
$\Ind(\R)$, as we do in the next few paragraphs.

The category $\rcomod C$
is a Grothendieck
category: it is Abelian, it possesses all small coproducts,
direct colimits of monomorphisms are monomorphisms,
and there is a generator. A generating family may be obtained by
choosing representatives
for the isomorphism classes of finite-dimensional $C$-comodules.
By the general theory of Grothendieck categories, every $C$-comodule
has an injective hull. 
We use the notation $I(b)$ to denote an injective hull of $L(b)$.
The right regular comodule
decomposes as
\begin{equation}
C \cong \bigoplus_{b \in \B} I(b)^{\oplus \dim L(b)}.
\end{equation}
By Baer's criterion for Grothendieck categories
(e.g., see \cite[Prop.~8.4.7]{KS}), arbitrary direct sums of injectives
are injective. It follows that an injective hull of $V \in \rcomod
C$ 
comes from an injective hull of its socle: 
 if $\soc V \cong
\bigoplus_{s \in S} L(b_s)$ then
$\bigoplus_{s \in S} I(b_s)$ is an injective hull of $V$.

In any Abelian category, we write $[V:L]$ for the {\em composition
  multiplicity} of an irreducible object $L$ in an object $V$. By
definition, this is the supremum of sizes of the sets $\{i=1,\dots,n\:|\:V_i/V_{i-1}
\cong L\}$ over all finite filtrations $0 = V_0 < V_1 < \cdots < V_n =
V$; possibly, $[V:L] = \infty$.
Composition multiplicity is additive on short exact sequences.
For any right $C$-comodule $V$, we have by Schur's Lemma that
\begin{equation}
[V:L(b)] = \dim \Hom_{C}(V, I(b)).
\end{equation}

When $C$ is infinite-dimensional,
the map (\ref{basic}) is not an isomorphism, but one can still use it 
to make the dual vector space $B := C^*$ 
into a unital algebra.
Since $C$ is the union of its finite-dimensional subcoalgebras, the
algebra $B$ is the inverse limit of its finite-dimensional quotients,
i.e., the canonical homomorphism $B \rightarrow\varprojlim (B / J)$ is an
isomorphism where the 
limit is
over all two-sided ideals $J$ of $B$ of finite codimension.
These two-sided ideals $J$ 
form a base of neighborhoods of $0$ making
$B$ into a {\em pseudo-compact topological algebra}; see
\cite[Ch.~IV]{Gab} or \cite[Def.~2.4]{Simson}.
We refer to the topology on $B$ defined in this way as the {\em
  profinite topology}.
The coalgebra $C$ can be recovered from $B$ as
the {\em
  continuous dual}
\begin{multline}\label{cd}
B^\star := \left\{f \in B^*\:\big|\:\text{$f$ vanishes on some two-sided
  ideal $J$ of finite codimension}\right\}.
\end{multline}
It has a natural coalgebra structure dual to the algebra
structure on $B$. This is discussed further in \cite[$\S$3]{Simson}; see also
\cite[$\S$1.12]{EGNO} where $B^\star$ is called the {\em finite dual}.
We note that any left ideal $I$ of $B$ of finite
codimension contains a two-sided ideal $J$ of finite codimension,
namely,
$J:=\operatorname{Ann}_B(B/I)$. So, in the definition (\ref{cd}) of
continuous dual, ``two-sided ideal $J$ of finite
codimension'' can be replaced by ``left ideal $I$
of finite codimension''. Similarly for right ideals.

Any right $C$-comodule $V$ is naturally a left $B$-module by the same
construction as in the finite-dimensional case.
We deduce that the category
$\rcomod C$ of all right $C$-comodules 
is isomorphic to the full subcategory 
$B\dslmod$ of $B\lmod$ consisting
of all {\em discrete} left $B$-modules,
that is, all $B$-modules which are the 
unions of their finite-dimensional submodules.
In particular, $\fdrcomod C$ and $B\fdlmod$ are identified under this
construction.
This means that any locally finite Abelian category may be realized
as the category of 
finite-dimensional modules over an algebra 
which is pseudo-compact with
respect to the profinite topology; see also \cite[$\S$3]{Simson}.

The definition of the left $C$-comodule structure on the linear dual $V^*$ of
a right $C$-comodule $V$ in (\ref{aduality}) 
required $V$ to be finite-dimensional in
order for it to make sense.
If $V$ is an infinite-dimensional right $C$-comodule, it can be viewed
equivalently as a discrete left module
over the dual algebra $B := C^*$.
Then its dual $V^*$ is a 
{\em pseudo-compact right $B$-module}, 
that is, a $B$-module
isomorphic to the inverse limit of
its finite-dimensional quotients.
Viewing pseudo-compact modules as topological $B$-modules with respect to
the profinite topology (i.e., submodules of finite codimension form a basis
of neighborhoods of $0$), we obtain
the category $\pcrmod B$ of all pseudo-compact right
$B$-modules and continuous $B$-module homomorphisms. The
functor (\ref{aduality}) extends to
\begin{equation}\label{boring1}
?^*:B \dslmod \rightarrow \pcrmod B.
\end{equation}
This is a contravariant equivalence with quasi-inverse given by the functor
\begin{equation}\label{boring2}
?^\star: \pcrmod B \rightarrow B \dslmod
\end{equation}
taking $V \in \pcrmod B$ to its {\em continuous dual} 
\begin{equation*}
V^\star := \left\{f \in V^*\:\big|\:\text{$f$ vanishes on some
    submodule of $V$ of
  finite codimension}\right\}.
\end{equation*}
We are using subtlely different notation here ($?^*$ vs. $?^\star$), but confusion seldom
arises due to context.

We record one more basic lemma about comodules over a coalgebra.

\begin{lemma}\label{dumb}
Suppose that $C$ is a coalgebra and $B := C^*$ is its dual algebra.
For any right $C$-comodule $V$, composing with the counit $\epsilon:C\rightarrow\k$
defines an isomorphism of left $B$-modules $
\alpha_V:\Hom_C(V,C)
\stackrel{\sim}{\rightarrow} V^*$.
When $V = C$, the right regular comodule, this map
gives an algebra isomorphism
$\End_C(C)^\op \cong B$.
\end{lemma}

\begin{proof}
Let $\eta:V \rightarrow V \otimes C$ be the
comodule structure map.
To show that $\alpha_V$ is an isomorphism, one 
checks that the map
$
\beta_V:V^* \rightarrow \Hom_C(V,C),
f \mapsto (f\bar\otimes
\id)\circ \eta
$
is its two-sided inverse; cf. \cite[Lem.~4.9]{Simson}.
It remains to show 
that $\alpha_C: \End_C(C)^\op \stackrel{\sim}{\rightarrow} B$ is an
algebra homomorphism: for $f,g \in B$ we have that
\begin{align*}
\alpha_C(\beta_C(g)\circ\beta_C(f)) &= 
\epsilon \circ (g\bar\otimes \id) \circ
\eta
\circ (f \bar\otimes \id) \circ \eta
\\&=
(g \bar\otimes \id) \circ (\id \otimes \epsilon) \circ \eta \circ (f
\bar\otimes \id) \circ \eta
=
g \circ (f \bar\otimes \id)\circ \eta
= fg.\qedhere
\end{align*}
\end{proof}

\subsection{Locally unital algebras}
We are going to work with certain Abelian categories which are not locally
finite, but which nevertheless have some well-behaved finiteness
properties. We will define these in the next subsection.
First we must review some basic notions about
{locally unital algebras}. These ideas originate in the work
of Mitchell \cite{Mitchell}.

A {\em locally unital algebra} is an associative (but not necessarily unital)
algebra $A$ equipped with a distinguished system 
$\{e_i\:|\:i \in I\}$ 
of mutually orthogonal idempotents 
such that $$
A =
\bigoplus_{i,j \in I} e_i A e_j.
$$ 
We say $A$ is {\em locally finite-dimensional}
if each subspace $e_i A e_j$
is finite-dimensional.

A {\em locally unital homomorphism} (resp., {\em isomorphism}) 
between
two locally unital algebras $A$ and $B$ is an algebra homomorphism
(resp., isomorphism) which takes 
distinguished idempotents to 
distinguished idempotents.
We say that $A$ is an {\em idempotent contraction} of $B$, or $B$ is an {\em
  idempotent expansion} of $A$,
if there is an algebra isomorphism $A\stackrel{\sim}{\rightarrow} B$ 
sending each distinguished idempotent in $A$ to
a sum of distinguished idempotents in $B$.
Usually when we use this language it will be the case that $B = A$ and
the isomorphism $A \rightarrow B$ is the identity function; then $A= \bigoplus_{i,j \in \hat I} \hat e_i A \hat e_j$ is an
idempotent expansion of $A = \bigoplus_{i,j \in I} e_i A e_j$ if each
of the idempotents $e_i\:(i \in I)$ is a finite sum of the
idempotents $\hat e_j\:(j \in \hat I)$.

For a locally unital algebra $A$, an {\em $A$-module} means a left module $V$ as usual such that
$V = \bigoplus_{i \in I} e_i V$. 
A vector $v \in V$ is {\em homogeneous} if $v \in e_i V$ for some $i \in I$.
A module $V$ is
\begin{itemize}
\item
{\em locally finite-dimensional} if $\dim e_i V < \infty$ for all $i \in I$;
\item 
{\em finitely generated} if $V= A v_1+\cdots+A v_n$
for vectors $v_1,\dots,v_n \in V$ (which may be assumed to be
homogeneous)
or, equivalently, it is a quotient of the finitely generated
projective $A$-module $A e_{i_1}\oplus\cdots \oplus A e_{i_n}$
for $i_1,\dots,i_n \in I$ and $n \in \N$;
\item
{\em finitely presented}
if there is an exact sequence $$
A e_{j_1} \oplus\cdots\oplus A e_{j_m}
\longrightarrow
A e_{i_1}\oplus\cdots\oplus A e_{i_n}
\longrightarrow V
\longrightarrow 0
$$
for $i_1,\dots,i_n,j_1,\dots,j_m \in I$
and $m,n \in \N$.
\end{itemize}
Let $A\lmod$ (resp., $A\lfdlmod$, resp., $A\fglmod$, resp., $A\fplmod$) 
be the category of all $A$-modules (resp., the locally
finite-dimensional ones, resp., the finitely generated ones, 
resp., the finitely presented ones).
Similarly, we define the categories $\rmod A$, $\lfdrmod A$,
$\fgrmod A$ and $\fprmod A$ of right modules.

\begin{remark}\label{dataof}
Any locally unital algebra
$A = \bigoplus_{i,j \in I} e_i A
e_j$ can be viewed as a category with object set $I$ and
$\Hom_{\Acat}(j,i) = e_i A e_j$, with the
idempotent $e_i \in A$ corresponding to
the identity endomorphism $1_i \in \End_\Acat(i)$.
Conversely, any small category $\Acat$ ($\k$-linear, of course) gives rise to a
corresponding locally unital algebra $A$ which we call the {\em
 path algebra} of $\Acat$.
In these terms, locally finite-dimensional
locally unital algebras correspond to {\em finite-dimensional
  categories}, that is, small categories all of whose morphism spaces are
finite-dimensional. The notion of idempotent
expansion of the algebra $A$ becomes the notion of {\em thickening} of
the category $\Acat$, which is a sort of ``partial Karoubi envelope''.
Also, a left $A$-module (resp., a locally finite-dimensional
left $A$-module) is the same as a $\k$-linear functor from $\Acat$ to the
category $\Vec$ (resp., $\fVec$) of vector spaces
(resp., finite-dimensional vector spaces); right $A$-modules are functors to 
$\Vec^\op$.
\end{remark}

\begin{lemma}\label{mor}
An essentially small category $\R$ is equivalent to 
$A \lmod$ for some locally unital algebra $A$ if and
only if $\R$ is Abelian, it possesses all small coproducts, and it has a projective
generating family, i.e., there is a family $(P_i)_{i \in I}$ of
compact projective objects such that
$V \neq 0\Rightarrow \Hom_\R(P_i, V) \neq 0$ for some $i \in
I$.
\end{lemma}

\begin{proof}
This is similar to \cite[Ex.~5.F]{F}. One
shows 
that $\R$ is equivalent to $A\lmod$ for the
locally unital algebra $A = \bigoplus_{i,j \in I} e_i A e_j$ defined
by setting $e_i A e_j := \Hom_{\R}(P_i, P_j)$
with multiplication that is the opposite of composition in $\R$.
The canonical equivalence $\R \rightarrow A\lmod$ is given by the functor
$\bigoplus_{i \in I} \Hom_{\R}(P_i,?)$.
\end{proof}

\begin{lemma}\label{kevin1}
Let $A$ be a locally unital algebra. 
An $A$-module $V$ is compact if and only if it is finitely presented.
Also, for projective modules, the notions of finitely presented and finitely
generated coincide.
\end{lemma}

\begin{proof}
This is well known for modules over a ring, and the usual proof in
that setting
carries over almost unchanged to the locally unital case.
\end{proof}

\begin{lemma}\label{kevin2}
Let $A$ be a locally unital algebra. 
Any $A$-module is isomorphic to a direct limit of finitely presented $A$-modules.
\end{lemma}

\begin{proof}
As any $A$-module is the union of its finitely generated submodules,
it suffices to show that any finitely generated $A$-module $V$ is isomorphic to a direct limit of finitely presented modules.
But then $V$
is a quotient of $P = A e_{i_1}\oplus\cdots\oplus A e_{i_n}$ by a submodule.
This submodule is the union of its finitely generated
submodules $W$, so 
we have that $V \cong P / \varinjlim W \cong \varinjlim P / W$.
This is a direct limit of finitely presented modules.
\end{proof}

The following lemma is fundamental. It is the analog of ``adjointness
of tensor and hom'' in the locally unital setting; see e.g. 
\cite[$\S$2.1]{BD} for a fuller discussion.

\begin{lemma}\label{adjointness}
Let $A = \bigoplus_{i,j \in I} e_i A e_j$ and $B = \bigoplus_{i,j \in
  J} f_i B f_j$ be locally unital algebras, and let $M =
\bigoplus_{i \in I, j \in J} e_i M f_j$ be an
$(A,B)$-bimodule.
\begin{enumerate}
\item
The functor $M\otimes_B ?:B\lmod \rightarrow A\lmod$
is left adjoint to
$\bigoplus_{j \in J} \Hom_A(Mf_j,?)$.
\item
The functor $?\otimes_A M:\rmod A \rightarrow \rmod B$
is left adjoint to
$\bigoplus_{i \in I} \Hom_B(e_i M,?)$.
\end{enumerate}
\end{lemma}

For any locally unital algebra $A$, there is a contravariant equivalence
\begin{equation}\label{cdua}
?^\circledast:
A\lfdlmod \rightarrow \lfdrmod A
\end{equation}
sending a left module $V$ to 
$V^\circledast := \bigoplus_{i \in I} (e_i V)^*$, viewed as a right
module in the obvious way.
The analogous functor
$?^\circledast:\lfdrmod A \rightarrow A \lfdlmod$ gives a quasi-inverse.
The contravariant functor (\ref{cdua}) also makes sense on arbitrary left (or right)
$A$-modules. It is no longer an equivalence, but we still have that 
\begin{equation}\label{imp}
\Hom_{A}(V, W^{\circledast}) \cong \Hom_A(W, V^\circledast)
\end{equation}
for any $V \in A\lmod$ and $W \in \rmod A$.
To prove this, apply Lemma~\ref{adjointness}(1) to the
$(\k,A)$-bimodule $W$ to show that
$\Hom_A(V,W^\circledast) \cong (W \otimes_A V)^*$,
then apply Lemma~\ref{adjointness}(2) to the $(A, \k)$-bimodule $V$ to show that
$(W\otimes_A V)^* \cong \Hom_A(W, V^\circledast)$.

\begin{lemma}\label{trimp}
The dual $V^\circledast$ of a projective (left or right) $A$-module is an
injective (right or left) $A$-module.
\end{lemma}

\begin{proof}
Just like in 
the classic treatment of duality for vector spaces
from \cite[IV.2]{Maclane}, (\ref{imp}) shows that the covariant functor
$?^\circledast:A\lmod \rightarrow (\rmod A)^\op$ is left adjoint to
the exact covariant functor
$?^\circledast:(\rmod A)^\op
\rightarrow A\lmod$.
So it sends projective left $A$-modules to projectives in $(\rmod
A)^\op$, which are injective right $A$-modules.
\end{proof}

Now we assume that $A$ is a locally unital algebra and $T \in
A\lfdlmod$. We are going to give a self-contained account of the
construction of a coalgebra
$\Coend_A(T)$
which is the continuous dual of the endomorphism algebra $\End_A(T)^\op$.
This is the {\em coend construction} which is an 
essential ingredient in the proof of Lemma~\ref{char} 
as discussed for example in \cite[$\S$1.10]{EGNO}, although as usual
we are using the language of algebras and modules rather than the
language of categories and functors used there. 
To start with, let 
\begin{equation}
B:=\End_A(T)^\op,
\end{equation}
which is a unital algebra.
Then $T$ is an $(A,B)$-bimodule and the dual $T^\circledast$ is a
$(B,A)$-bimodule.
Let $T_i := e_i T$, so that $T = \bigoplus_{i \in I} T_i$ and 
$T^\circledast = \bigoplus_{i \in I} T_i^*$.

\begin{lemma}\label{cts}
Suppose that $T =\bigoplus_{i \in I} T_i \in A\lfdlmod$
and $B := \End_A(T)^\op$ are as above.
For any $V \in A \lmod$,
there is a natural
isomorphism of right $B$-modules
\begin{equation}\label{nasty}
\Hom_A(V, T) \stackrel{\sim}{\rightarrow} (T^\circledast \otimes_A
V)^*,\qquad
\theta \mapsto (f \otimes v \mapsto f(\theta(v))).
\end{equation}
In particular, taking $V = T$, we get 
that
$(T^\circledast \otimes_A T)^* \cong B$
as $(B,B)$-bimodules.
\end{lemma}

\begin{proof}
By Lemma~\ref{adjointness}
applied to the $(A,\k)$-bimodule $T^\circledast$,
the functor $T^\circledast \otimes_A ?$ is left adjoint to
$\bigoplus_{i \in I}\Hom_\k(T_i^*,?)$.
Hence,
$$
(T^\circledast \otimes_A V)^*
= \Hom_\k(T^\circledast \otimes_A V, \k)
\cong \Hom_A\Big(V, \bigoplus_{i \in I}
\Hom_\k(T_i^*, \k)\Big)
\cong \Hom_A(V, T).
$$
This is the natural isomorphism in the statement of the lemma.
We leave it to the reader to check that it is a $B$-module homomorphism.
\end{proof}

Continuing with this setup, let
\begin{equation}\label{C}
C := T^\circledast \otimes_A T.
\end{equation}
There is a unique way to make this into a coalgebra so that
the bimodule isomorphism 
$B\stackrel{\sim}{\rightarrow} C^*$
from Lemma~\ref{cts}
is actually an algebra isomorphism
(viewing the dual
$C^*$ of a coalgebra as an algebra as in the previous subsection).
Explicitly, let $u^{(i)}_1,\dots,u^{(i)}_{d(i)}$
be a basis for $T_i$ and $v^{(i)}_1,\dots,v^{(i)}_{d(i)}$ be the dual
basis for $T_i^*$.
Let $c_{r,s}^{(i)} := v^{(i)}_s \otimes u^{(i)}_r \in C$.
Then the comultiplication $\delta:C \rightarrow C \otimes C$ and
counit $\epsilon:C \rightarrow \k$ 
satisfy
\begin{align}\label{pens}
\delta\big(c_{r,s}^{(i)}\big) &= \sum_{t=1}^{d(i)} c_{r,t}^{(i)}
         \otimes c_{t,s}^{(i)},&
\epsilon\big(c_{r,s}^{(i)}\big)&= \delta_{r,s}
\end{align}
for each $i \in I$ and $1 \leq r,s \leq d(i)$.
For the next lemma, recall the definition of 
continuous dual of a pseudo-compact topological algebra from (\ref{cd}).

\begin{lemma}\label{prof}
The endomorphism algebra $B = \End_A(T)^\op$ of 
$T \in A\lfdlmod$ is a
pseudo-compact topological algebra with respect to the profinite
topology, i.e., $B$ is isomorphic to $\varprojlim B / J$ where the inverse
limit is over all two-sided ideals $J$ of finite codimension.
Moreover, the coalgebra $C$ from (\ref{C}) may be identified with
the continuous dual $B^\star$.
\end{lemma}

\begin{proof}
 This
follows because $B \cong C^*$ as algebras.
\end{proof}

Thus, the coalgebra $C$ defined by (\ref{C})
is identified with the continuous dual
\begin{equation}\label{coend}
\Coend_A(T) := \left(\End_A(T)^\op\right)^\star
\end{equation}
of $B$. Explicitly, 
using the formula (\ref{nasty}), the element
$c_{r,s}^{(i)} = v_s^{(i)} \otimes u_r^{(i)} \in C$
is identified with the function sending $\theta \in \End_A(T)$ to
$v_s(\theta(u_r))$.

Now consider the functor
$T^\circledast\otimes_A ?:A\lmod \rightarrow B\lmod$.
Since $T$ is locally finite-dimensional, 
it takes finitely generated $A$-modules to finite-dimensional
$B$-modules. Any $A$-module $V$ is the union of its finitely generated
submodules, and $T^\circledast\otimes_A ?$ commutes with direct limits, 
so we see that $T^\circledast \otimes_A V$ is actually a discrete
$B$-module.
Since $B \cong C^*$, the category $B\dslmod$ is isomorphic
to $\rcomod C$. So we have constructed a functor
\begin{equation}\label{tf}
T^\circledast\otimes_A ?:A\lmod \rightarrow 
\rcomod C.
\end{equation}
For $V \in A\lmod$, the comodule structure map on 
$T^\circledast\otimes_A V$
is given explicitly by the formula
\begin{equation}\label{puns}
  \eta:T^\circledast\otimes_A V \rightarrow T^\circledast \otimes_A V
\otimes C,
\qquad v_s^{(i)} \otimes v \mapsto \sum_{r=1}^{d(i)} v_r^{(i)} \otimes v
\otimes c_{r,s}^{(i)}.
\end{equation}
Recall the definition of the functor $?^\star$
from (\ref{boring2}).

\begin{lemma}\label{ha}
Suppose that 
$T =\bigoplus_{i \in I} T_i\in A\lfdlmod$, $B := \End_A(T)^\op$
and $C \cong B^\star$ are as above.
The functor $T^\circledast \otimes_A ?$ just constructed is 
isomorphic to
\begin{equation}\label{tf1}
G = \Cohom_A(T,?) := \Hom_A(?,T)^\star:A\lmod \rightarrow \rcomod C,
\end{equation}
and it is
left
adjoint to the functor
\begin{equation}\label{tf2}
G_*= \bigoplus_{i \in I} \Hom_C(T_i^*,?):\rcomod C
  \rightarrow A\lmod.
\end{equation}
Thus, $(G,G_*)$ is an adjoint pair.
\end{lemma}

\begin{proof}
The fact that (\ref{tf}) is left adjoint to (\ref{tf2})
follows by Lemma~\ref{adjointness}.
To see that it is isomorphic to (\ref{tf1}), take $V \in A\lmod$ and
consider the natural isomorphism
$\Hom_A(V,T) \cong (T^\circledast\otimes_A V)^*$ of right $B$-modules
from Lemma~\ref{cts}.
As $T^\circledast \otimes_A V$ is discrete, its dual is
a pseudo-compact left $B$-module, hence, $\Hom_A(V,T)$ is pseudo-compact
too.
Then we apply ${\star}$, using that it is quasi-inverse to $*$,
to get that $\Hom_A(V,T)^\star  \in B\dslmod$ is naturally
isomorphic to $T^\circledast\otimes_A V$.
\end{proof}

\subsection{Schurian categories}\label{errorhere}
By a {\em Schurian category}, we mean a category $\R$ that is
equivalent to $A\lfdlmod$ for a locally finite-dimensional locally
unital algebra $A$. 
This non-standard terminology is considerably more restrictive than other usage of the same term elsewhere in the literature, where ``Schurian category" is typically used to
indicate a $\k$-linear category in which the endomorphism algebras of the indecomposable objects are one-dimensional\footnote{Note also that the present usage is different
 from several recent papers of the first author:
in \cite{BD}, the phrase ``locally Schurian''
was used to describe the categories we now call ``Schurian";
more precisely, in \cite{BD}, a locally
Schurian category
referred to a category of the form $A\lmod$ (rather than
$A\lfdlmod$) for locally finite-dimensional 
locally unital algebras $A$.
We could not use the phrase ``Schurian'' in {\em loc. cit.} since that 
was reserved for a more restrictive 
notion defined in \cite[$\S$2.1]{BLW}; this more restrictive notion
will be discussed in the next subsection, again using different language.} (e.g., see work of Roiter).

By an {\em algebra realization} of a Schurian category $\R$, we mean
a locally finite-dimensional locally unital algebra $A$ (together with the set $I$
indexing 
its distinguished idempotents) such that $\R$ is equivalent to
$A\lfdlmod$.
Now we assume that
$$
\R
=A\lfdlmod
$$
and proceed to summarize some of the basic properties of such
categories, referring to \cite[$\S$2]{BD} for a more detailed treatment.
Let
$\{L(b)\:|\:b \in \B\}$ be a full set of pairwise inequivalent
irreducible objects of $\R$.
Schur's Lemma holds: we have that $\End_\R(L(b)) = \k$ for each $b \in \B$.
Note that the opposite 
category $\R^\op$ is also Schurian,
and $A^\op$ gives an algebra realization for 
it.
This follows because $\R^\op = (A\lfdlmod)^\op$
is equivalent to $\lfdrmod A \cong (A^\op)\lfdlmod$ using (\ref{cdua}).

Let $\R_c$ be the (not necessarily Abelian) full subcategory of $\R$ consisting of
all compact objects, and $\Ind(\R_c)$ be its ind-completion.
The canonical functor $\Ind(\R_c) \rightarrow A\lmod$
is an equivalence of categories. 
To see this, we note that all finitely generated $A$-modules are 
locally finite-dimensional as $A$ itself 
is locally finite-dimensional. Hence, finitely presented $A$-modules
are locally finite-dimensional too, i.e, $A\fglmod$ is a subcategory
of $A\lfdlmod$.
In view of Lemma~\ref{kevin1}, this is the category $\R_c$.
It just remains to apply \cite[Cor.~6.3.5]{KS}, using
Lemma~\ref{kevin2} when checking the required hypotheses.

The category $A\lmod$ is a Grothendieck category. In particular, this
means that
every $A$-module has an injective hull in $A\lmod$.
Since every
$A$-module is a quotient of a direct sum of projective $A$-modules
of the form $A e_i$,
the category $A\lmod$ also has enough projectives.
It is {\em not} true that an
arbitrary $A$-module has a projective cover, but we will see in
Lemma~\ref{enough} below
that finitely generated $A$-modules do.

Like we did in $\S\ref{pct}$, we write 
$\Ext^n_\R(V,W)$, or sometimes $\Ext^n_A(V,W)$, in place of $\Ext^n_{\Ind(\R_c)}(V,W)$
for any $V, W \in \Ind(\R_c)$. This can be computed
either from a projective resolution of $V$ or from an
injective resolution of $W$.
We can also consider both right derived functors $\mathbb{R}^n F$ 
of a left exact functor
$F:\Ind(\R_c)\rightarrow \R'$
and left
derived functors $\der{L}_n G$ of 
a right exact functor
$G:\Ind(\R_c)\rightarrow \R'$.
We provide an elementary proof of the following, but note it also follows
from \cite[Th.~15.3.1]{KS}.

\begin{lemma}\label{skyfall}
For $V, W \in \R$ and
$n \geq 0$,
there is a natural isomorphism
$$
\Ext^n_{\R}(V,W) \cong \Ext^n_{\R^\op}(W, V).
$$
\end{lemma}

\begin{proof}
Using 
(\ref{cdua}),
we must show that 
$\Ext^n_A(V, W) \cong \Ext^n_{A}(W^\circledast,V^\circledast)$ for
locally finite-dimensional
$A$-modules $V$ and $W$. To compute $\Ext^n_A(V,W)$, take a projective resolution 
$$\cdots\longrightarrow P_1\longrightarrow P_0 \longrightarrow V \longrightarrow 0
$$ 
of $V$ in $A\lmod$.
By Lemma~\ref{trimp},
on applying the exact functor $\circledast$, we get an injective
resolution $$
0 \longrightarrow
V^\circledast \longrightarrow P_0^\circledast \longrightarrow
P_1^\circledast\longrightarrow \cdots
$$ of
$V^\circledast$ in $\rmod A$.
Since $W$ is locally finite-dimensional, we can use (\ref{imp}) 
to see that 
$\Hom_{A}(P_i, W) \cong \Hom_A(W^\circledast, P_i^\circledast)$ for
each $i$.
So $\Ext^n_A(V,W) \cong \Ext^n_{A}(W^\circledast, V^\circledast)$.
\end{proof}

Let $I(b)$ be an injective hull of $L(b)$ in $A\lmod$.
The dual $(e_i A)^\circledast$ of the projective
right $A$-module $e_i A$ is injective in 
$A\lmod$.
Since $\End_A((e_i A)^\circledast)^\op \cong \End_A(e_i A) \cong e_i
A e_i$, which is finite-dimensional, the injective module $(e_i A)^\circledast$
can be written as a finite direct sum of indecomposable injectives.
To determine which ones, we compute its socle: we have that
$\Hom_A(L(b), (e_i
A)^\circledast)
\cong
\Hom_A(e_i A, L(b)^\circledast) \cong (L(b)^\circledast) e_i
= (e_i L(b))^*$,
hence,
\begin{equation}\label{iid}
(e_i A)^\circledast \cong \bigoplus_{b \in \B} I(b)^{\oplus \dim e_i
  L(b)},
\end{equation}
with all but finitely many summands on the right hand side being zero.
In particular, this shows for fixed $i$ that $\dim e_i L(b) = 0$ for all but finitely
many $b \in \B$. Conversely, for fixed $b \in \B$, we can always choose $i \in I$
so that $e_i L(b) \neq 0$, and deduce that $I(b)$ is a summand of
$(e_i A)^\circledast$. This means that
each indecomposable injective $I(b)$ is a locally finite-dimensional
left $A$-module. 

Let $P(b)$ be the dual of the injective hull of
the irreducible right $A$-module $L(b)^\circledast$.
By dualizing the right module analog of the decomposition (\ref{iid}),
we get also that
\begin{equation}\label{pid}
A e_i \cong \bigoplus_{b \in \B} P(b)^{\oplus \dim e_i L(b)},
\end{equation}
with all but finitely many summands being zero.
In particular, $P(b)$ is projective in $A\lmod$, hence, it is a
projective cover of $L(b)$ in $A\lmod$.
The composition multiplicities of any $A$-module satisfy
\begin{equation}\label{compmults}
[V:L(b)] = \dim \Hom_A(V, I(b)) = \dim \Hom_A(P(b), V).
\end{equation}

\begin{lemma}
For $A$ as above,
left $A$-module $V$ is locally finite-dimensional if and only if $[V:L(b)] <
\infty$ for all $b \in \B$. 
\end{lemma}

\begin{proof}
Note that $V$
is locally finite-dimensional if and only if $\dim \Hom_A(A e_i, V) <
\infty$ for each $i \in I$. Using the decompositon (\ref{pid}), this
is if and only if $\dim \Hom_A(P(b), V) < \infty$ for each $b \in \B$.
\end{proof}

There is a little more to be said about finitely generated modules.
Recall from the previous subsection that a module is
finitely generated if $V = A v_1+\cdots+A v_n$
for homogeneous vectors $v_1,\dots,v_n \in V$.
We say that $V$ is {\em finitely cogenerated} if its dual is finitely generated.
It is obvious from these definitions that
$\Hom_A(V,W)$ is finite-dimensional
either
if $V$ is finitely generated and $W$ is locally finite-dimensional, or if
$V$ is locally finite-dimensional and $W$ is finitely cogenerated.
The following checks that our naive definitions
are consistent with the usual notions
of finitely generated and cogenerated objects of
Grothendieck categories.

\begin{lemma}\label{enough}
For $V
\in A\lmod$,
the following properties are equivalent:
\begin{itemize}
\item[(i)] $V$ is finitely generated;
\item[(ii)]
the radical $\rad V$, i.e., the sum of its maximal proper submodules, is a superfluous
submodule and $\hd V := V / \rad V$ is of finite length;
\item[(iii)] $V$ is a quotient of a finite direct sum of the modules
  $P(b)$ for $b \in \B$.
\end{itemize}
Moreover, any finitely generated $V$ has a projective cover.
\end{lemma}

\begin{proof}
We have already observed that $P(b)$ is a projective cover of
$L(b)$. The lemma follows by standard arguments given this and the decomposition
(\ref{pid}).
\end{proof}

\begin{corollary}\label{fccor}
For $V \in A\lmod$, the following properties are equivalent:
\begin{itemize}
\item[(i)] $V$ is finitely cogenerated;
\item[(ii)]
$\soc V$ is an
essential submodule of finite length;
\item[(iii)] $V$ is isomorphic to a submodule of a finite direct sum
  of modules $I(b)$ for $b \in \B$.
\end{itemize}
\end{corollary}

We say that a locally finite-dimensional locally unital algebra $A =
\bigoplus_{i,j \in I} e_i A e_j$ is {\em pointed} if $A$ is a basic algebra, i.e., all of its irreducible modules are one-dimensional, and all of its distinguished idempotents 
$\{e_i\:|\:i \in I\}$ are primitive.

\begin{lemma}\label{willthisneverend}
  Let $A = \bigoplus_{i,j \in I} e_i A e_j$ be a locally
  finite-dimensional locally unital algebra.
  Pick an idempotent expansion $A = \bigoplus_{i,j \in \hat I} \hat
  e_i A \hat e_j$ such that for some subset $\B\subseteq\hat I$
  the set
  $\{\hat e_b\:|\:b \in \B\}$
is a complete set of pairwise non-conjugate primitive
  idempotents in $A$.
  Let $B := \bigoplus_{a,b \in \B} \hat e_a A \hat e_b$.
  Then $B$ is a pointed locally unital algebra that is Morita
  equivalent to $A$, and any such pointed locally unital algebra is isomorphic to $B$.
  \end{lemma}

  \begin{proof}
    It is clear that $B$ is pointed. To see that $A$ and $B$ are
    Morita equivalent, note that the functor $A\lmod \rightarrow
    B\lmod,
V \mapsto \bigoplus_{b \in \B} \hat e_b V$ is an equivalence of
categories with quasi-inverse given by the functor $\big(\bigoplus_{b \in
  \B} A \hat e_b\big) \otimes_B ?$.
Finally if $B'$ another pointed locally unital algebra that is Morita
equivalent to $A$, let $F:A\lmod \rightarrow B'\lmod$ be an equivalence
of categories.
Then we have that
$B' = \bigoplus_{b \in \B} B'_b$ for left ideals $B'_b \cong F (A \hat
e_b)$. So $$
B' \cong \left(\bigoplus_{a,b\in\B} \Hom_{B'}(B'_a,B'_b)\right)^{\op}
\cong \bigoplus_{a,b \in \B} \Hom_A(A \hat e_a, A \hat e_b) =
\bigoplus_{a,b \in \B} \hat e_a A \hat e_b = B.
$$
This proves the uniqueness.
\end{proof}

Finally, we introduce some terminology which will not be neeeded until
$\S$\ref{std}.

\begin{definition}\label{locfreedef}
Let $A = \bigoplus_{i,j \in I} e_i A e_j$ be a locally
finite-dimensional locally unital algebra.
Let $\BS \subseteq I$ be a subset.
We say that a left $A$-module $V$ is {\em $\BS$-free}
if there is a subset $X = \bigsqcup_{s \in \BS} X(s) \subset V$ such that
the following properties hold:
\begin{itemize}
\item[(LF1)] $V =\bigoplus_{x \in X} Ax$.
\item[(LF2)] The homomorphism
  $Ae_s \rightarrow Ax, a \mapsto ax$ is an isomorphism
  for $x \in X(s)$.
\end{itemize}
Equivalently, there is a $\K$-submodule $U$ of $eV := \bigoplus_{s \in
  \BS} e_s V$ such that the
multiplication map
$A e \otimes_\K U\rightarrow V$ is an isomorphism, where
$Ae := \bigoplus_{s \in \BS} A e_s$ and $\K := \bigoplus_{s \in \BS} \k e_s$.
\end{definition}

\begin{lemma}\label{locfreelem}
  Suppose that $A = \bigoplus_{i,j \in I} e_i A e_j$ is a locally
  finite-dimensional locally unital algebra and
$\{e_b\:|\:b \in \B\}$ is a full set of pairwise
  non-conjugate primitive idempotents in $A$ for some subset
  $\B\subseteq I$.
  Then every finitely generated projective left $A$-module is $\B$-free.
  \end{lemma}

\begin{proof}
Any finitely generated projective left $A$-module $V$  decomposes as
a finite direct direct sum of indecomposable projectives, and any indecomposable
projective is isomorphic to $A e_b$ for some $b \in \B$.
Hence, we can pick a finite subset $X = \bigsqcup_{b \in \B} X(b)$ so that $V
= \bigoplus_{x\in X} A x$ with $Ax \cong A e_b$ for $x \in X(b)$.
\end{proof}

There are obvious right module analogs of these notions. 

\subsection{Essentially finite Abelian categories}\label{pups}
We say that a locally unital algebra $A = \bigoplus_{i,j \in
  I} e_i A e_j$ is {\em essentially finite-dimensional} if each right
ideal $e_i A$ and each left ideal $A e_j$ is finite-dimensional.
By an {\em essentially finite Abelian category}, we mean
a category $\R$ 
that is equivalent to $A\fdlmod$ for such an $A$.
In that case, we refer to $A$ as an {\em algebra realization} of $\R$. Note 
that $\R$ is essentially finite Abelian if and only if $\R^\op$ is
essentially finite Abelian. Moreover, if $A$ is an algebra realization for $\R$ then $A^\op$ is
one for $\R^\op$ by the obvious
contravariant equivalence $?^*:A\fdlmod \rightarrow \fdrmod A$.

\begin{lemma}\label{veral}
An essentially small category $\R$ is equivalent to $A\fdlmod$ for a
locally unital algebra $A
= \bigoplus_{i,j \in I} e_i A e_j$ such that each left ideal
$A e_j$ (resp., each right ideal $e_i A$) is finite-dimensional if and
only if $\R$ 
is a locally finite
Abelian category with enough projectives (resp., enough injectives).
\end{lemma}

\begin{proof}
We just prove the result for left ideals and projectives; the
parenthesized statement for right ideals and injectives follows by
replacing $\R$ and $A$ with $\R^\op$ and $A^\op$.

Suppose first that $A = \bigoplus_{i,j \in I} e_i A e_j$ is a locally
unital algebra such that each left ideal $A e_j$ is
finite-dimensional. 
Then $A\fdlmod$ is a locally finite Abelian category.
It has enough projectives because the left ideals $A e_j$ are
finite-dimensional.

Conversely, suppose $\R$ is a locally finite Abelian category with
enough projectives.
Let $\{L(b)\:|\:b \in \B\}$ be a
full set of pairwise inequivalent irreducible objects, and $P(b) \in
\R$ a projective cover of $L(b)$. Define
$A$ to be the locally unital algebra
$A = \bigoplus_{a,b \in \B} e_a A e_b$
where $e_a A e_b :=\Hom_{\R}(P(a),P(b))$
with multiplication that is the opposite of composition in $\R$. This is a pointed locally finite-dimensional locally unital algebra.
As in the proof of 
Lemma~\ref{mor}, the functor $\bigoplus_{b \in \B} \Hom_\R(P(b),?)$ 
defines an equivalence
$\R \rightarrow A\fdlmod$.
It remains to note that
the ideals $A e_b$ are finite-dimensional since they are the 
images under this functor of the projectives $P(b)$, which are of
finite length.
\end{proof}

\begin{corollary}\label{vera}
An essentially small category $\R$ is an 
essentially finite Abelian category if and only if it is 
a locally finite Abelian category with enough injectives and projectives.
\end{corollary}

Essentially finite Abelian categories are almost as convenient to work with as finite
Abelian categories since one can perform all of the usual
constructions of homological algebra without needing to pass to the
ind-completion.

\begin{lemma}\label{kitchen}
For a category $\R$, the following are equivalent:
\begin{itemize}
\item[(i)] $\R$ is a finite Abelian category;
\item[(ii)] $\R$ is a Schurian category with only finitely many
  isomorphism classes of irreducible objects;
\item[(iii)] $\R$ is an essentially finite Abelian category with only finitely many
  isomorphism classes of irreducible objects;
\item[(iv)] $\R$ is a locally finite Abelian category with only finitely many
  isomorphism classes of irreducible objects and either enough
  projectives or enough injectives;
\item[(v)] $\R$ is both a locally finite Abelian category and a Schurian
  category.
\end{itemize}
\end{lemma}

\begin{proof}
Clearly, (i) implies (ii) and (iii).
The implication (ii)$\Rightarrow$(i) follows on considering a pointed
algebra realization of $\R$.
The implication (iii)$\Rightarrow$(iv) follows from Corollary~\ref{vera}.
The implication (iv)$\Rightarrow$(i) follows from Lemma~\ref{veral}.
Clearly (ii) and (iv) together imply (v).
Finally, to see that (v) implies (ii), it suffices to note that a
Schurian category with infinitely many isomorphism classes of
irreducible objects cannot be locally finite Abelian:
the direct sum of infinitely many non-isomorphic irreducibles is a
well-defined object of $\R$ but it is 
not of finite length.
\end{proof}

Essentially finite Abelian categories with infinitely many isomorphism classes of
irreducible objects are {\em not} Schurian
categories. However they are closely related as we explain next.
\begin{itemize}
\item If
$\R$ is essentially finite Abelian, we define its {\em Schurian envelope} 
$\Loc(\R)$ to be the full subcategory
of $\Ind(\R)$ consisting of all objects that have finite composition
multiplicities. 
\item
If $\R$ is Schurian, let $\Fin(\R)$
be the full subcategory of $\R$ consisting of all objects of finite length.
\end{itemize}
We say that a Schurian category $\R$ is {\em Cartan-bounded}
if its Cartan matrix $C$
has only finitely many non-zero entries in every row and column,
where by
Cartan matrix we mean the matrix 
\begin{equation}\label{CM}
\left(\dim \Hom_\R(P(a), P(b)\right)_{a,b \in \B}
= \left(\dim \Hom_\R(I(a), I(b)\right)_{a,b \in \B},
\end{equation}
where $\B$ is labelling indecomposable projectives and injectives in the usual way.

\begin{lemma}\label{coffee}
If $\R$ is an essentially finite Abelian category then $\Loc(\R)$
is a Cartan-finite Schurian category,
and conversely if $\R$ is a Cartan-finite Schurian category then $\Fin(\R)$ is an essentially finite Abelian category.
Morever, $\Loc$ and $\Fin$ are inverses in the sense that
$\Fin(\Loc(\R))$ is equivalent to $\R$ for any essentially finite Abelian $\R$,
and $\Loc(\Fin(\R))$ is equivalent to $\R$ for any Cartan-finite Schurian $\R$:
\begin{equation*}
\left(
\begin{array}{l}
\text{Essentially finite}\\\text{Abelian categories}
\end{array}
\right)
\substack{\Loc\\\longrightarrow\\\longleftarrow\\\Fin}
\left(\begin{array}{l}
\text{Cartan-finite}\\
\text{Schurian categories}\end{array}
\right).
\end{equation*}
\end{lemma}

\begin{proof}
This is easy to see in terms of an algebra realization:
if $\R=A\fdlmod$ for an essentially
finite-dimensional locally unital algebra $A$ then $\Loc(\R)=
A\lfdlmod$,
so it is Schurian. Since the indecomposable injectives and
projectives in $\Loc(\R)$ 
are the same as in $\R$, they have finite length. 
Conversely, using Lemma~\ref{willthisneverend}, 
we may assume that $\R =A\lfdlmod$ 
for a {\em pointed} locally finite-dimensional
locally unital algebra, such that all of the indecomposable injectives and
projectives are of finite length.
Since $A$ is pointed, this means equivalently that all of the left
ideals $A e_i$ and right ideals $e_i A$ are finite-dimensional.
Hence, $A$ is essentially
finite-dimensional,
and $\Fin(\R) = A\fdlmod$ is essentially finite Abelian.
\end{proof}

\subsection{Recollement}
We conclude the section with some reminders about ``recollement'' in
our algebraic setting; see \cite[$\S$1.4]{BBD} or \cite[$\S$2]{CPS} for further background.
We need this here only for Abelian categories $\R$ satisfying
finiteness properties as developed above.
The recollement formalism provides us with
an adjoint triple of functors $(i^*,i,i^!)$ associated to the inclusion
$i:\R\d\rightarrow \R$ of a Serre
subcategory,
and an adjoint triple of functors $(j_!,j,j_*)$ associated to 
the projection $j:\R\rightarrow \R\u$ onto a Serre quotient category,
with the image of $i$ being the kernel of $j$.
These functors will play an essential role in all subsequent arguments.

First suppose that $\R$ is any Abelian category.
Assume that we are given a full set 
$\{L(b)\:|\:b \in \B\}$ 
 of pairwise inequivalent
irreducible objects.
Let $\B\d$ be a subset of $\B$ and
$\R\d$ be the full subcategory of $\R$ consisting of all the
objects $V$ such that
$[V:L(b)] \neq 0 \Rightarrow b \in \B\d.$
This is a Serre subcategory of $\R$
with irreducible objects $\{L\d(b) \:|\:b \in \B\d\}$ defined by $L\d(b) := L(b)$.

\begin{lemma}\label{recollementi}
In the above setup,
the inclusion functor $i:\R\d\rightarrow \R$ 
has a left adjoint $i^*$ and a right adjoint $i^!$:
\vspace{-2mm}
\begin{equation*}
\begin{tikzpicture}
\node(a) at (0,0) {$\R\d$};
\node(b) at (2,0) {$\R.$};
\draw[-latex] (a) -- (b);
\draw[bend right=40,-latex] (b.north west) to (a.north east);
\draw[bend left=40,-latex] (b.south west) to (a.south east);
\node at (1.0,0.2) {$\scriptstyle i$};
\node at (1.05,.75) {$\scriptstyle i^!$};
\node at (1.05,-.7) {$\scriptstyle i^*$};
\end{tikzpicture}
\end{equation*}
The counit of one of these adjunctions and the unit of the other give isomorphisms:
$$
i^* \circ i \stackrel{\sim}{\rightarrow} \Id_{\R\d}
\stackrel{\sim}{\rightarrow}
i^! \circ i.
$$
In particular, $i$ is fully faithful.
\end{lemma}

\begin{proof}
This is straightforward. Explicitly, $i^*$ (resp., $i^!$) sends an object of $\R$ to the largest
quotient (resp., subobject) that belongs to $\R\d$.
\end{proof}

Now we are going to pass to the {\em Serre quotient} $\R\u := \R /
\R\d$. This is an Abelian category 
equipped with an exact
{\em quotient functor}
$j:\R \rightarrow \R\u$ 
satisfying the
following universal property:
if $h:\R \rightarrow \Ccat$ is any exact functor to an
Abelian category $\Ccat$ with $h L(b) = 0$ for all $b \in
\B\d$, then there is a unique 
functor $\bar h: \R\u \rightarrow \Ccat$
such that $h = \bar h \circ j$.
The irreducible objects in $\R\u$ are 
$\{L\u(b) \:|\:b \in \B\u\}$ where $\B\u := \B \setminus \B\d$
and
$L\u(b):= j L(b)$.
For a fuller discussion of these statements, see e.g. \cite{Gab}.

The quotient functor $j$ need not have a left or a right adjoint in
general, so we need to impose some additional hypotheses.
We first assume that $\R$ is finite Abelian, essentially
finite Abelian or Schurian. Then one can understand $j$ rather
explicitly as an idempotent truncation functor and it always has both a left and right
adjoint:

\begin{lemma}\label{recollementj}
  Suppose that $\R$ is finite Abelian, essentially finite Abelian or
  Schurian, $\B = \B\d\sqcup\B\u$, and $i:\R\d\rightarrow \R$ and $j:\R\rightarrow\R\u = \R /
  \R\d$ are as above. Then $\R\d$ and $\R\u$ are of the same type
  (finite Abelian, essentially finite Abelian or Schurian) as
  $\R$. Moreover,  the quotient functor $j:\R\rightarrow \R\u$ 
has a left
adjoint $j_!$ and a right adjoint $j_*$:
\vspace{-2mm}
\begin{equation*}
\begin{tikzpicture}
\node(a) at (0,0) {$\R$};
\node(b) at (2,0) {$\R\u.$};
\draw[-latex] (a) -- (b);
\draw[bend right=40,-latex] (b.north west) to (a.north east);
\draw[bend left=40,-latex] (b.south west) to (a.south east);
\node at (.95,0.2) {$\scriptstyle j$};
\node at (1.0,.75) {$\scriptstyle j_*$};
\node at (1.0,-.7) {$\scriptstyle j_!$};
\end{tikzpicture}
\end{equation*}
The counit of one of the adjunctions and the unit of the other give isomorphisms:
$$
j \circ j_* \stackrel{\sim}{\rightarrow} \Id_{\R\u}
\stackrel{\sim}{\rightarrow} j \circ j_!.
$$
In particular, $j_!$ and $j_*$ are fully faithful.
\end{lemma}

\begin{proof}
Fix a pointed algebra
realization
$$
A = \bigoplus_{a,b\in \B} e_a A e_b
$$
of $\R$, so $A$ is finite-dimensional, essentially finite-dimensional
or
locally finite-dimensional according to whether $\R$ is finite
Abelian, essentially finite Abelian or Schurian.
Let
\begin{align*}
  A\d = \bigoplus_{a,b \in \B\d} \bar e_a A\d \bar e_b &:= A \big/ (e_c\:|\:c \in
                                             \B\u),
&
A\u &:= \bigoplus_{a,b
  \in \B\u} e_a A e_b,
\end{align*} 
where $\bar x$ denotes the canonical image of $x \in A$
under the quotient map $A \twoheadrightarrow A\d$.
Then it is clear that 
$\R\d$ is equivalent to $A\d\fdlmod$ in the finite Abelian or essentially finite Abelian
cases,
and to 
$A\d\lfdlmod$ in the Schurian case. 
As $A\d$ satisfies the same finiteness properties as $A$,
we deduce that $\R\d$ is of the same type as $\R$.

The quotient category $\R\u$ is
realized by the algebra $A\u$, and the quotient
functor $j$ becomes the functor that sends an $A$-module $V$ to 
\begin{equation}\label{guns1}
jV := \bigoplus_{a \in \B\u} e_a V
\end{equation}
with $A\u$ acting by restricting the action of $A$.
We deduce that $\R\u$ is again of the same type as $\R$. 
Since $j$ is isomorphic to
$\bigoplus_{b \in \B\u}\Hom_A(A e_b,
-)$,
it has the left adjoint 
\begin{equation}\label{guns2}
j_! := \Big(\bigoplus_{b \in \B\u} A
  e_b\Big) \otimes_{A\u}?:A\u\lmod \rightarrow A\lmod
\end{equation}
thanks to Lemma~\ref{adjointness}(1).
From this, it is clear that the unit of adjunction
$\Id_{\R\u} \rightarrow j \circ j_!$ is an
isomorphism.
On the other hand, $j$ is also isomorphic to the tensor functor
$\left(\bigoplus_{b \in \B\u} e_b A\right) \otimes_A ?$,
so Lemma~\ref{adjointness}(1) also gives that $j$ has the right adjoint
\begin{equation}\label{guns3}
j_* := \bigoplus_{a \in \B}\Hom_{A\u}\Big(\bigoplus_{b \in \B\u} e_b
  A e_a,?\Big):A\u\lmod\rightarrow A\lmod.
\end{equation}
Again using this we see that the counit $j \circ j_* \rightarrow
\Id_{\R\u}$ 
is an isomorphism.
\end{proof}

The situation when $\R$ is locally finite Abelian is more complicated.
Continuing with the above notation,
it follows immediately from
Lemma~\ref{char} that the Serre subcategory $\R\d$ and the quotient
category $\R\u$ are locally Schurian
too.
The following lemma explains how to obtain an explicit coalgebra realization of $\R\d$ starting from one for $\R$.

\begin{lemma}\label{crack}
Suppose that
  $\R = \fdrcomod C$ for a coalgebra $C$.
Let $C\d$ be the largest right coideal of $C$ belonging to $\R\d$.
Then $C\d$ is a subcoalgebra of $C$.
Moreover, $\R\d$ consists of all $V \in \fdrcomod C$ such that the image of the
structure map $\eta:V \rightarrow V\otimes C$ is contained in $V
\otimes C\d$, i.e., we have that $\R\d =
\fdrcomod C\d$.
\end{lemma}

\begin{proof}
For a right comodule $V$ with structure map $\eta:V \rightarrow
V\otimes C$, we can consider $V \otimes C$ 
as a right comodule with structure map $\id \otimes \delta$.
The coassociative and counit axioms
imply that $\eta$ is an injective homomorphism of right comodules.
We deduce that all irreducible subquotients of $V$ belong to $\R\d$
if and only if $\eta(V) \subseteq V \otimes C\d$.
Applying this with $V = C\d$ shows that $C\d$ is a subcoalgebra.
Applying it to $V \in \R$
shows that
$V \in \R\d$ if and only if $\eta(V) \subseteq V \otimes C\d$.
\end{proof}

For locally finite Abelian $\R$, the quotient category $\R\u$ can also
be realized explicitly as a category of comodules: if $\R = \fdrcomod
C$ then $\R\u = \fdrcomod eCe$ for an idempotent $e \in C^*$ and the quotient functor
$j$ becomes the idempotent truncation functor defined by $e$. This
is reviewed in detail in \cite{Navarro}. It follows that the extension $j:\Ind(\R)\rightarrow
\Ind(\R\u)$ of $j$ to the
ind-completions always has a right adjoint $j_*$ with $j \circ j_*
\cong \Id_{\Ind(\R\u)}$. However, this adjoint does not necessarily take
objects of $\R\u$ to objects of $\R$, so that the original functor $j:\R\rightarrow\R\u$ need not have
a right adjoint itself. For left adjoints, the situation is even a bit
worse since one should really pass to the pro-completions.
For our purposes, though, it will always be
sufficient to impose the stronger condition 
from (i) of the following lemma; this ensures
that both adjoints exist without any need to pass to ind- or pro-completions.

\begin{lemma}\label{recollementk}
  Suppose that $\R$ is locally finite Abelian, and let
  $\B\u\subseteq\B$ and $j:\R\rightarrow\R\u$ be as above.
  Then the following are equivalent:
  \begin{itemize}
\item[(i)] $L(b)$ has an injective hull $I(b)$
  and a projective cover $P(b)$ in $\R$ for all $b \in \B\u$;
    \item[(ii)] $\R\u$ is essentially finite Abelian and
    the quotient functor $j:\R\rightarrow\R\u$ has a left adjoint 
$j_!$ and a right adjoint $j_*$:
\vspace{-2mm}
\begin{equation*}
\begin{tikzpicture}
\node(a) at (0,0) {$\R$};
\node(b) at (2,0) {$\R\u.$};
\draw[-latex] (a) -- (b);
\draw[bend right=40,-latex] (b.north west) to (a.north east);
\draw[bend left=40,-latex] (b.south west) to (a.south east);
\node at (.95,0.2) {$\scriptstyle j$};
\node at (1.0,.75) {$\scriptstyle j_*$};
\node at (1.0,-.7) {$\scriptstyle j_!$};
\end{tikzpicture}
\end{equation*}
\end{itemize}
When these properties hold, there are isomorphisms
$j \circ j_* \cong
\Id_{\R\u}
\cong j \circ j_!$ just like in Lemma~\ref{recollementj}.
 \end{lemma}

\begin{proof}
  (i)$\Rightarrow$(ii):
Let $j_*:\Ind(\R\u)\rightarrow \Ind(\R)$ be the right adjoint of
$j:\Ind(\R)\rightarrow\Ind(\R\u)$ as in \cite{Navarro}.
For $b \in \B\u$, let $I\u(b)$ be an injective hull of $L\u(b)$ in
$\Ind(\R\u)$.
By adjunction properties, $j_* I\u(b)$ is an injective hull of
$L(b)$ in $\Ind(\R)$, hence, $j_* I\u(b) \cong I(b)$ which has finite length by assumption. From $j \circ j_* \cong \Id_{\Ind(\R\u)}$,
we deduce that $I\u(b) \cong j I(b)$ is of
finite length too, so $I\u(b)\in\R\u$ and $\R\u$ has enough
injectives.
We have shown that $j_*$ takes $I\u(b)$ to $I(b) \in \R$, hence using left
exactness we deduce that it takes any object of finite length to an
object of finite length. This means that the restriction
$j_*:\R\u\rightarrow\R$ is well-defined and gives a right adjoint to
$j:\R\rightarrow\R\u$.
The dual argument shows that $\R\u$ has enough projectives
and that $j:\R\rightarrow\R\u$ has a left
adjoint
$j_!:\R\u\rightarrow\R$.
Finally we deduce that $\R\u$ is essentially finite Abelian due to Corollary~\ref{vera}.

\vspace{1.5mm}
\noindent
(ii)$\Rightarrow$(i):
We can take $I(b) := j_*I\u(b)$ and $P(b) := j_! P\u(b)$ where $I\u(b)$ is an
injective hull and $P\u(b)$ is a projective cover of $L\u(b)$ in $\R\u$.
\end{proof}

In the locally finite Abelian or Schurian cases, we may use the
same notations $i, i^*, i^!$ for the natural extensions of these
functors to the ind-completions $\Ind(\R), \Ind(\R\d)$ or
$\Ind(\R_c), \Ind(\R\d_c)$, respectively. Similarly, we will use the
notations $j, j_*, j_!$ for the extensions of these to the appropriate
ind-completions, assuming the equivalent conditions from
Lemma~\ref{recollementk} hold in the locally finite Abelian case.

\begin{lemma}\label{guns}
  Continuing with the above setup,
  assume either that $\R$ is finite Abelian, essentially finite
  Abelian, or Schurian, or that $\R$ is locally finite Abelian and the equivalent
  conditions from Lemma~\ref{recollementk} hold.
  For $b \in \B\u$, let $P(b)$ (resp., $I(b)$) and $P\u(b)$ (resp., $I\u(b)$)
be a projective cover (resp., injective hull) of $L(b)$ in $\R$
and a projective cover (resp., injective hull) of
$L\u(b)$ in $\R\u$. Then we have that
$$
j P(b) \cong P\u(b), 
\qquad
j I(b) \cong I\u(b),
\qquad
j_! P\u(b) \cong P(b), 
\qquad
j_* I\u(b) \cong I(b).
$$
Moreover, the adjunction gives isomorphisms
\begin{align}\label{thetrick}
\Hom_{\R}(P(b), j_* V) \cong \Hom_{\R\u}(P\u(b), V),
\quad
\Hom_{\R}(j_! V, I(b))
\cong
\Hom_{\R\u}(V, I\u(b))
\end{align} 
for $V \in \R\u$,
hence, $[V:L\u(b)] = [j_* V: L(b)] = [j_! V:L(b)]$ for all $b \in \B\u$.
\end{lemma}

\begin{proof}
  Take $b \in \B\u$.
  By adjunction properties, $j_! P\u(b)$ is a projective cover of
  $L(b)$ in  $\R$, so it is isomorphic to $P(b)$. Hence, $j (j_!
  P\u(b)) \cong P\u(b) \cong
  j P(b)$; similarly for injectives. The remaining assertions follow.
\end{proof}

\section{Generalizations of highest weight categories}

In this section, we define the various generalizations of highest weight
categories and derive some of their fundamental properties 
in the four settings of finite Abelian, essentially finite Abelian,
Schurian, and
locally finite Abelian categories.
The important definitions in the section are Definitions~\ref{vector},
\ref{ufc} and \ref{newlfd}.
 The reader new to these ideas 
may find it helpful to assume initially that all of the strata are simple 
in the sense of Lemma~\ref{allsimple}, when the definitions specialize
to the notions of
finite, essentially finite, upper finite and lower finite highest
weight categories, respectively.

\subsection{Stratifications and the associated standard and costandard objects}\label{fs}
Let $(\Lambda,\leq)$ be a poset.
It is {\em interval finite}
(resp., {\em upper finite}, resp., {\em lower finite}) if the interval
$[\lambda,\mu] := \{\nu \in \Lambda\:|\:\lambda \leq \nu\leq \mu\}$
(resp., $[\lambda,\infty) := \{\nu\in\Lambda\:|\:\lambda\leq\nu\}$,
resp., $(-\infty, \mu] := \{\nu\in\Lambda\:|\:\nu \leq \mu\}$)
is finite for all $\lambda,\mu
\in \Lambda$.
A {\em lower set} (resp., {\em upper set}) means a subset $\Lambda\d$
(resp., $\Lambda\u$) such that
$\mu \leq \lambda \in \Lambda\d \Rightarrow \mu \in \Lambda\d$
(resp., $\mu \geq \lambda \in \Lambda\u\Rightarrow \mu \in \Lambda\u$).

A {\em stratification function} $\rho:\B\rightarrow
\Lambda$
is a function
from a set $\B$ to a poset $(\Lambda,\leq)$
such that all of the fibers $\B_\lambda := \rho^{-1}(\lambda)$ are
finite. We often use other obvious notations like $\B_{\leq\lambda} :=
\bigcup_{\mu \leq \lambda} \B_\mu, \B_{<\lambda} := \bigcup_{\mu <
  \lambda} \B_\mu$, etc..

A {\em stratification} of an Abelian category $\R$ is
a quintuple $(\B,L,\rho,\Lambda,\leq)$
consisting of a set $\B$,
a function 
$L$ labelling a full set
$\{L(b)\:|\:b \in \B\}$ of pairwise
inequivalent irreducible objects in $\R$,
and a stratification function
$\rho:\B \rightarrow \Lambda$
for the poset $(\Lambda, \leq)$.
In the case that $\rho$ is a bijection, one can use it to identify $\B$ with $\Lambda$,
writing $L(\lambda)$ instead of $L(b)$; similarly for all of the other
families of objects indexed by the set $\B$ to be introduced shortly.

Given a stratification $(\B,L,\rho,\Lambda,\leq)$ of $\R$ and
$\lambda \in \Lambda$, let
$\R_{\leq\lambda}$ and $\R_{<\lambda}$ be the Serre subcategories of $\R$ 
associated to the subsets $\B_{\leq\lambda}$ and $\B_{< \lambda}$ of
$\B$, respectively.
We denote the inclusion functors by
\begin{equation}
  i_{\leq\lambda}:\R_{\leq\lambda} \rightarrow \R,
  \qquad
  i_{<\lambda}:\R_{<\lambda}\rightarrow \R,
\end{equation}
The left and right adjoints of $i_{\leq\lambda}$ are
$i_{\leq\lambda}^*$ and $i_{\leq\lambda}^!$ as in
Lemma~\ref{recollementi}. 
We say that the stratification is
\begin{itemize}
\item[(F)]
a {\em finite stratification} if $\R$ is a finite Abelian category
(so that $\B$ is a finite set);
\item[(EF)]
an  {\em essentially finite stratification} if $\R$ is an essentially
finite Abelian category and the poset $\Lambda$ is interval finite;
\item[(LF)]
 a {\em lower finite stratification} if $\R$ is a locally finite
 Abelian category and the poset $\Lambda$ is lower finite;
 \item[(UF)]
 an {\em upper finite stratification} if $\R$ is a Schurian category
  and the poset $\Lambda$ is upper finite.
\end{itemize}
In these four cases, the induced stratifications of the subcategories $\R_{<\lambda}$ and $\R_{\leq
  \lambda}$ are automatically of
the same type.

By an {\em admissible stratification}, we mean a stratification
of one of the above four types such that
the following axiom is satisfied when in type (LF) (it holds
automatically for the other types):
\begin{itemize}
\item[(A)]
  The irreducible object $L(b)$ has 
  both a projective cover and an injective hull in
  $\R_{\leq\rho(b)}$ for all $b \in \B$.
  \end{itemize}
  This is a significant restriction on the sorts of lower finite Abelian categories that can be considered; for example, the category
  $\Rep(\mathbb{G}_a)$ of rational representations of the additive group does not have this property.
  Using
Lemma~\ref{kitchen}
together with
Lemma~\ref{recollementk} in the lower finite case, we deduce for
$\lambda \in \Lambda$ that the quotient category
$\R_\lambda := \R_{\leq\lambda} / \R_{<
  \lambda}$ is finite Abelian in all cases. Let
\begin{equation}
j^\lambda:\R_{\leq\lambda} \rightarrow \R_\lambda
\end{equation}
be the
quotient functor.
The objects
\begin{equation}\label{sooften}
\left\{L_\lambda(b) := j^\lambda L(b) \:\big|\:b \in
  \B_\lambda\right\}\end{equation}
give a full set of
pairwise inequivalent irreducible objects in $\R_\lambda$.
Moreover, we are in a recollement situation as in
Lemmas~\ref{recollementi}, \ref{recollementj} and \ref{recollementk}:
\begin{equation}
\begin{tikzpicture}[anchorbase]
\node(a) at (0,0) {$\R_{<\lambda}$};
\node(b) at (2,0) {$\R_{\leq\lambda}$};
\draw[-latex] (a) -- (b);
\draw[bend right=40,-latex] (b.north west) to (a.north east);
\draw[bend left=40,-latex] (b.south west) to (a.south east);
\node at (1.0,0.2) {$\scriptstyle i_{<\lambda}$};
\node at (1.05,.75) {$\scriptstyle i_{<\lambda}^!$};
\node at (1.05,-.7) {$\scriptstyle i_{<\lambda}^*$};
\node(c) at (4,0) {$\R_\lambda$.};
\draw[-latex] (b) -- (c);
\draw[bend right=40,-latex] (c.north west) to (b.north east);
\draw[bend left=40,-latex] (c.south west) to (b.south east);
\node at (2.95,0.2) {$\scriptstyle j^\lambda$};
\node at (3,.75) {$\scriptstyle j^\lambda_*$};
\node at (3,-.7) {$\scriptstyle j^\lambda_!$};
\end{tikzpicture}
\label{rst}
\end{equation}
Let $P_\lambda(b)$ be a projective cover and $I_\lambda(b)$ be an
injective hull of $L_\lambda(b)$ in $\R_\lambda$. By Lemma~\ref{guns},
these are isomorphic to the images of the projective cover and injective
hull of $L(b)$ in $\R_{\leq\lambda}$, respectively.
Finally, define
{\em standard, costandard, proper standard} and {\em
  proper costandard objects}
$\Delta(b), \nabla(b), \bar\Delta(b)$ and $\bar\nabla(b)$
according to (\ref{rumble}).

\begin{lemma}\label{charac}
Suppose we are given an admissible stratification
$(\B,L,\rho,\Lambda,\leq)$ of $\R$.
  Take $b \in \B$ and set $\lambda := \rho(b)$.
  \begin{enumerate}
    \item The standard object $\Delta(b)$ is a projective cover of $L(b)$ in $\R_{\leq
        \lambda}$. The proper standard object
      $\bar\Delta(b)$ is the largest quotient
      of $\Delta(b)$ such that all composition factors of $\rad
      \bar\Delta(b)$
      are of the form $L(c)$ for $c \in \B_{<\lambda}$.
   \item The costandard object $\nabla(b)$
is an
injective hull of $L(b)$ in $\R_{\leq\lambda}$.
The proper costandard object $\bar\nabla(b)$ is the largest subobject
of $\nabla(b)$ such that all composition factors of $\bar\nabla(b) /
\soc \bar\nabla(b)$ are of the form $L(c)$ for $c \in \B_{<\lambda}$.
\end{enumerate}
\end{lemma}

\begin{proof}
  We just check (1) since (2) is similar.
  We have that $\Delta(b)$ is a projective cover of $L(b)$ in
  $\R_{\leq \lambda}$ by Lemma~\ref{guns}.
  It remains to prove the statement about $\bar\Delta(b)$.
Assume $[\bar\Delta(b):L(c)] \neq 0$. 
Since $\bar\Delta(b) \in
\R_{\leq\lambda}$, we have $\rho(c) \leq \rho(b)$. If $\rho(c) =
\rho(b)$ then $$[\bar\Delta(b):L(c)] =
[j^\lambda\bar\Delta(b):j^\lambda L(c)] = [L_\lambda(b):L_\lambda(c)] =
\delta_{b,c}.$$
Thus, $\bar\Delta(b)$ is such a quotient of $\Delta(b)$.
To show that it is the largest such quotient, it suffices to show that
the kernel $K$ of $\Delta(b) \twoheadrightarrow \bar\Delta(b)$
is finitely generated with head
that only involves irreducibles $L(c)$ with $\rho(c) =
\rho(b)$.
To see this, apply 
the right exact functor $j^\lambda_!$ to a short exact sequence
$0 \rightarrow \widehat{K} \rightarrow P_\lambda(b) \rightarrow L_\lambda(b)
\rightarrow 0$ to get an epimorphism
$j^\lambda_! \widehat{K} \twoheadrightarrow K$.
Then observe that $j^\lambda_! \widehat{K}$ is
finitely generated as $j^\lambda_!$ is a left adjoint, and its head
only involves irreducibles $L(c)$ with $\rho(c) =
\rho(b)$. The latter assertion follows because $\Hom_{\R}(j^\lambda_! \widehat{K}, L(c)) \cong
\Hom_{R_\lambda}(\widehat{K}, j^\lambda L(c))$ for $c \in \B_{\leq\lambda}$.
\end{proof}

\begin{corollary}\label{reallysilly}
We have that
$\dim \Hom_{\R}(\Delta(b), \bar\nabla(c)) = 
\dim \Hom_{\R}(\bar\Delta(b), \nabla(c)) = 
\delta_{b,c}$
for all $b, c \in \B$.
\end{corollary}

\begin{lemma}\label{dualitylemma}
Suppose that we are given an admissible stratification
$(\B,L,\rho,\Lambda,\leq)$ of $\R$, 
and in addition that
$\R$ possesses a contravariant autoequivalence $?^\vee$ which preserves isomorphism classes of irreducibles.
Then we have that
$P(b)^\vee \cong I(b)$,
$I(b)^\vee \cong P(b)$, $\Delta(b)^\vee \cong \nabla(b)$,
$\bar\Delta(b)^\vee \cong \bar\nabla(b)$, 
$\nabla(b)^\vee \cong \Delta(b)$ and
$\bar\nabla(b)^\vee \cong \bar\Delta(b)$ for all $b \in \B$.
\end{lemma}

\begin{proof}
Since $L(b)^\vee \cong L(b)$, 
we have that $I(b)^\vee \cong P(b)$ and $P(b)^\vee \cong I(b)$. Then the statements about $\Delta(b)^\vee, \bar\Delta(b)^\vee, \bar\nabla(b)^\vee$ and $\nabla(b)^\vee$ follow using Lemma~\ref{charac}.
\end{proof}

For $\lambda \in \Lambda$, we say that the stratum $\R_\lambda$ is
{\em simple} if it is equivalent to the category $\Vec_{\operatorname{fd}}$ of
finite-dimensional vector spaces. 

\begin{lemma}\label{allsimple}
The following are equivalent:
\begin{itemize}
\item[(i)]
all of the strata are simple;
\item[(ii)] $\rho$ is a bijection and $\Delta(\lambda) =
  \bar\Delta(\lambda)$ 
  for all $\lambda\in\Lambda$;
  \item[(iii)] $\rho$ is a bijection and
    $\Hom_\R(\Delta(\lambda),\nabla(\lambda))$ is one-dimensional;
\item[(iv)] $\rho$ is a bijection
and 
$\nabla(\lambda) = \bar\nabla(\lambda)$ for all $\lambda\in\Lambda$.
\end{itemize}
\end{lemma}

\begin{proof}
(i)$\Rightarrow$(ii): 
Take $\lambda \in \Lambda$.
As the stratum $\R_\lambda$ is simple, $\B_\lambda =
\{b_\lambda\}$ is a singleton and $P_\lambda(b_\lambda) =
L_\lambda(b_\lambda)$.
We deduce that 
$\rho$ is a bijection
and 
$\Delta(b_\lambda) = \bar\Delta(b_\lambda)$.

\vspace{1.5mm}\noindent
(ii)$\Rightarrow$(iii):
This follows because $\nabla(\lambda)$ is the injective hull of $L(\lambda)$ in
$\R_{\leq \lambda}$.

\vspace{1.5mm}\noindent
(iii)$\Rightarrow$(iv):
This follows because $\Delta(\lambda)$ is the projective cover of $L(\lambda)$ in
$\R_{\leq \lambda}$.

\vspace{1.5mm}\noindent
 (iv)$\Rightarrow$(i):
Take $\lambda \in \Lambda$.
Then $\R_\lambda$ has just one irreducible object (up to
isomorphism), namely, $j^\lambda \bar\nabla(\lambda)$.
Since this equals $j^\lambda
\nabla(\lambda)$, it is also projective.
Hence, 
$\R_\lambda$ is simple.
\end{proof}

Given a {\em sign function} $\eps:\Lambda \rightarrow \{\pm\}$,
we introduce 
the {\em $\eps$-standard} and {\em $\eps$-costandard objects}
$\Delta_\eps(b)$ and $\nabla_\eps(b)$ as in (\ref{device}).
Corollary~\ref{reallysilly} implies that
\begin{equation}\label{onedimensionality}
\dim \Hom_{\R}(\Delta_\eps(b), \nabla_\eps(c)) = 
\delta_{b,c}
\end{equation}
for $b, c \in \B$.
A {\em $\Delta_\eps$-flag} of $V \in \R$ means a finite filtration
$0 = V_0 < V_1 < \cdots < V_n = V$ 
with sections $V_m / V_{m-1}\cong
\Delta_\eps(b_m)$
for $b_m \in \B$. 
Similarly, we define $\nabla_\eps$-flags.
We denote the exact subcategories of $\R$
consisting of all objects with a $\Delta_\eps$-flag or a
$\nabla_\eps$-flag by $\Delta_\eps(\R)$ and $\nabla_\eps(\R)$,
respectively.

A {\em $\Delta$-flag} (resp., {\em $\bar\nabla$-flag}) is a $\Delta_\eps$-flag
(resp., $\nabla_\eps$-flag) in the special case that $\eps=+$.
A {\em $\bar\Delta$-flag} (resp., {\em $\nabla$-flag}) is a $\Delta_\eps$-flag
(resp., $\nabla_\eps$-flag) in the special case that $\eps=-$.
We denote the exact subcategories of $\R$ consisting of all objects
with a $\Delta$-flag, a $\bar\Delta$-flag, a $\nabla$-flag or a
$\bar\nabla$-flag
by $\Delta(\R)$, $\bar\Delta(\R)$, $\nabla(\R)$ and
$\bar\nabla(\R)$, respectively.

\subsection{\boldmath Finite and essentially finite $\eps$-stratified categories}\label{ssn}
Throughout this subsection, $\R$ is
a finite or essentially finite Abelian category
equipped with a finite or essentially
finite stratification $(\B,L,\rho,\Lambda,\leq)$.
Also $\eps:\Lambda \rightarrow \{\pm\}$ denotes a sign function.
Let $P(b)$ and $I(b)$ be a projective cover and an injective hull of
$L(b)$, respectively. We also need the 
objects from (\ref{rumble})--(\ref{device}).
Consider the following two properties:
\begin{enumerate}
\item[($\widehat{P\Delta}_\eps$)]
For each $b \in \B$, there exists a projective object
$P_b$ admitting a 
$\Delta_\eps$-flag with 
$\Delta_\eps(b)$ at the top and other sections 
$\Delta_\eps(c)$ for $c\in\B$ with $\rho(c) \geq \rho(b)$.
\item[($\widehat{I\nabla}_\eps$)]
For each $b \in \B$, there exists an injective object
$I_b$ admitting a 
$\nabla_\eps$-flag with $\nabla_\eps(b)$ at the bottom  and other sections
$\nabla_\eps(c)$ for $c\in\B$ with $\rho(c) \geq \rho(b)$.
\end{enumerate}
It is trivial to see that the property $(P\Delta_\eps)$ formulated in the
introduction implies 
$(\widehat{P\Delta}_\eps)$, and similarly $(I\nabla_\eps)$ implies
$(\widehat{I\nabla}_\eps)$. 
The seemingly weaker properties
$(\widehat{P\Delta}_\eps)$--$(\widehat{I\nabla}_\eps)$ are often easier to
check in concrete examples.
The essence of the following fundamental theorem appeared originally
in \cite{ADL}, extending earlier work of Dlab \cite{Dlab}.

\begin{theorem}\label{fund}
The four
properties 
($\widehat{P\Delta}_\eps$), $(\widehat{I\nabla}_\eps$),
($P\Delta_\eps$) and $(I\nabla_\eps$) are equivalent.
When these properties hold,
the standardization functor $j^\lambda_!$ is exact whenever
$\eps(\lambda)=-$, and
the costandardization functor $j^\lambda_*$ is exact whenever
$\eps(\lambda)=+$.
\end{theorem}

\begin{remark} When 
all strata are simple, 
the properties $(\widehat{P\Delta}_\eps)$--$(\widehat{I\nabla}_\eps)$ may be written more succinctly as the
following:
\begin{enumerate}
\item[($\widehat{P\Delta}$)]
For each $\lambda\in\Lambda$, there exists a projective object
$P_\lambda$ admitting a 
$\Delta$-flag with 
$\Delta(\lambda)$ at the top and other sections of the form
$\Delta(\mu)$ for $\mu\in\Lambda$ with $\mu\geq\lambda$.
\item[($\widehat{I\nabla}$)]
For each $\lambda \in \Lambda$, there exists an injective object
$I_\lambda$ 
admitting a 
$\nabla$-flag with $\nabla(\lambda)$ at the bottom and other sections
of the form
$\nabla(\mu)$ for $\mu\in\Lambda$ with $\mu\geq\lambda$.
\end{enumerate}
Theorem~\ref{fund} shows that these are equivalent to the properties
$(P\Delta)$--$(I\nabla)$ from the introduction, as
 was explained originally by Cline, Parshall and Scott in \cite{CPS}.
\end{remark}

We postpone the proof of Theorem~\ref{fund} until a little later in the
the subsection. It is important because it justifies the next key definition
($\eps$S) and its variations (FS), ($\eps$HW), (FHW) and (HW).

\begin{definition}\label{vector}
  Let $\R$ be an Abelian category
  equipped with a finite (resp., essentially finite) stratification
$(\B,L,\rho,\Lambda,\leq)$.
\begin{enumerate}
\item[($\eps$S)]
We say that $\R$ is a
{\em finite} (resp., {\em essentially finite}) {\em
  $\eps$-stratified category} if 
one of the equivalent properties 
$(\widehat{P\Delta}_\eps)$--$(\widehat{I\nabla}_\eps)$ holds
for a given choice of sign function
$\eps:\Lambda\rightarrow\{\pm\}$.
\item[(FS)]
We say $\R$ is a {\em finite} (resp., {\em essentially finite})
{\em fully stratified category} if 
one of these properties holds
for all choices of sign function
$\eps:\Lambda\rightarrow\{\pm\}$.
\item[($\eps$HW)]
We say $\R$ is a {\em finite} (resp., {\em essentially finite}) 
{\em $\eps$-highest weight category} if 
the stratification function $\rho$ is a bijection, i.e., 
each stratum has a unique irreducible object (up to isomorphism),
and
one of the equivalent properties
$(\widehat{P\Delta}_\eps)$--$(\widehat{I\nabla}_\eps)$ holds
for a given choice of sign function
$\eps:\Lambda\rightarrow\{\pm\}$.
\item[(FHW)]
We say $\R$ is a {\em finite} (resp.,
{\em essentially finite}) {\em fibered highest weight category}
if the stratification function $\rho$ is a bijection and one of these
properties holds for all choices of sign function.
\item[(HW)]
We say $\R$ is a {\em finite} (resp., {\em essentially finite})
{\em highest weight category}
if all of the strata are simple
(cf. Lemma~\ref{allsimple}) and one of the equivalent
properties $(\widehat{P\Delta})$--$(\widehat{I\nabla})$ holds.
\end{enumerate}
\end{definition}

\begin{remark}\label{czdiscussion}
The language ``fibered
highest weight'' in Definition~\ref{vector} is a
departure from the existing literature, where
such categories are usually referred to as
{\em properly stratified categories}; this terminology
goes back to the work of Dlab \cite{Dlabproper}.
A recent exposition which takes a more traditional viewpoint than here can be
found in \cite{CZ}. In particular, in \cite[Def.~2.7.4]{CZ}, one finds five types of
finite-dimensional algebra $A$ defined in terms of properties of the category
$A\fdlmod$, namely, standardly stratified algebras, exactly standardly stratified algebras, strongly
stratified algebras, properly stratified algebras, and
quasi-hereditary algebras.
In our preferred language, these are $+$-stratified algebras, stratified algebras, $+$-quasi-hereditary algebras, properly stratified algebras, and quasi-hereditary algebras, respectively, as in Table~\ref{tabzero} from the introduction.
For further reference to the original literature, \cite[$\S$A.2]{CZ}
is helpful.
\end{remark}

We can view 
$\{L(b)\:|\:b \in \B\}$ equivalently as a full set
of pairwise inequivalent irreducible objects in $\R^\op$.
The stratification of $\R$ is
also one of $\R^\op$. 
The indecomposable projectives and injectives in $\R^\op$ are $I(b)$ and
$P(b)$, while the
$(-\eps)$-standard and
$(-\eps)$-costandard objects in $\R^\op$ are $\nabla_\eps(b)$ and
$\Delta_\eps(b)$, respectively.
So we can reinterpret Theorem~\ref{fund} as the following.

\begin{theorem}\label{opstrat} $\R$ is $\eps$-stratified, fully stratified,
  $\eps$-highest weight, fibered highest weight or highest weight
if and only if $\R^\op$ is $(-\eps)$-stratified,
fully stratified, $(-\eps)$-highest weight, fibered highest weight or highest weight, respectively.
\end{theorem}

Now we must prepare for the proof Theorem~\ref{fund}.
The main step in the argument 
will be provided by the {\em homological criterion} for
$\nabla_\eps$-flags
from the next Theorem~\ref{gf1}. In turn, the proof of this criterion 
reduces to the following lemma which treats a key 
special case. The reader wanting to work
fully through the proofs should look also at this point at the lemmas
in $\S$\ref{ssp} below.

\begin{lemma}\label{thesublemma}
  Assume that $\R$ is an Abelian category
  equipped with a finite or essentially finite stratification $(\B,L,\rho,\Lambda,\leq)$ and sign
function
$\eps$, such that property
$(\widehat{P\Delta}_\eps)$ holds.
Let $\lambda$ be a maximal element of $\Lambda$ with respect to the ordering $\leq$, and
$V \in \R$ be an object satisfying the following properties:
\begin{itemize}
\item[(i)]
$\Ext^1_\R(\Delta_\eps(b), V) = 0$ for all $b \in \B$;
\item[(ii)]
$\soc V \cong L(b_1)\oplus\cdots\oplus L(b_n)$ for $b_1,\dots,b_n \in
\B_\lambda$.
\end{itemize}
Then $V$ belongs to
$\R_{\leq\lambda}$ (so that it makes sense to apply the functor $j^\lambda$ to
it), and
\begin{equation}
V \cong 
\begin{cases}
j^\lambda_* (j^\lambda V)&\text{if
  $\eps(\lambda)=+$,}\\
\nabla(b_1)\oplus\cdots\oplus\nabla(b_n)&
\text{if $\eps(\lambda)=-$.}
\end{cases}\label{preach}
\end{equation}
Moreover, in the case $\eps(\lambda)=+$,
the functor $j^\lambda_*$ is exact. Hence, 
in both cases, we have that $V \in \nabla_\eps(\R)$.
\end{lemma}

\begin{proof}[Proof (assuming lemmas from $\S$\ref{ssp} below)]
We first prove (\ref{preach}) in case $\eps(\lambda) = -$.
Let $W := \nabla(b_1) \oplus \cdots \oplus \nabla(b_n)$.
By the maximality of $\lambda$ and Lemma~\ref{maxcase}, this is an
injective hull of $\soc V$.
So there is a short exact sequence
$0 \rightarrow V \rightarrow W \rightarrow W/V \rightarrow 0$.
For any $a \in \B$, we apply $\Hom_{\R}(\Delta_\eps(a),?)$ and use
property (i) to get a short exact
sequence
\begin{equation}\label{salary}
0 \longrightarrow \Hom_{\R}(\Delta_\eps(a), V)
\stackrel{f}{\longrightarrow} \Hom_{\R}(\Delta_\eps(a), W)
\longrightarrow \Hom_{\R}(\Delta_\eps(a), W/V)
\longrightarrow 0.
\end{equation}
If $\rho(a) \neq \lambda$ then $\Hom_{\R}(\Delta_\eps(a), W) = 0$
as none of the
composition factors of $\Delta_\eps(a)$ are constituents of $\soc W$.
If $\rho(a) = \lambda$ then $\Delta_\eps(a) = \bar\Delta(a)$ and any
homomorphism $\bar\Delta(a) \rightarrow W$ 
must factor through the
unique irreducible quotient $L(a)$ of $\bar\Delta(a)$. So its image is
contained in $\soc W \subseteq V$, showing that $f$ is an isomorphism.
These arguments show
 that $\Hom_{\R}(\Delta_\eps(a), W/V) = 0$
for all $a \in \B$.
We deduce that $\soc (W / V) =0$, hence, $W/V = 0$, which is what we
needed.

Now consider (\ref{preach}) when $\eps(\lambda) = +$. 
By Lemma~\ref{maxcase} again, the injective hull of $V$ is 
$\nabla(b_1)\oplus\cdots\oplus \nabla(b_n)$, 
which is an object of $\R_{\leq\lambda}$.
Hence, $V \in \R_{\leq\lambda}$.
The unit of adjunction gives us a morphism $g:V \rightarrow
W := j^\lambda_* (j^\lambda V)$. Since $g$ becomes an isomorphism when we
apply $j^\lambda$, its kernel belongs to $\R_{<\lambda}$. In view of property
(2), we deduce that $\ker g = 0$, so $g$ is a monomorphism.
Hence, we can identify $V$ with a subobject of $W$.
To show that $g$ is an epimorphism as well, we apply
$\Hom_{\R}(\Delta_\eps(a),?)$ to $0 \rightarrow V \rightarrow W
\rightarrow W  / V \rightarrow 0$ to get the 
short exact sequence (\ref{salary}).
By adjunction, the middle morphism space is isomorphic to
$\Hom_{\R_\lambda}(j^\lambda \Delta_\eps(a), j^\lambda V)$, which is
zero if $\rho(a) < \lambda$. If $\rho(a) = \lambda$
then $\Delta_\eps(a) = \Delta(a)$ is the projective cover of $L(a)$ in
$\R$ by Lemma~\ref{maxcase}, and $j^\lambda \Delta_\eps(a)$ is the projective cover of
$L_\lambda(a)$ in $\R_\lambda$.
We deduce that both the first and second morphism spaces in
(\ref{salary}) are of the
same dimension $[V:L(a)] = [j^\lambda V:L_\lambda(a)]$, so $f$ must be
an isomorphism.
Therefore $\Hom_{\R}(\Delta_\eps(a), W / V) = 0$ for all $a \in \B$,
hence, $V = W$ and (\ref{preach}) is proved.

To complete the proof, we must show that $j^\lambda_*$ is exact when
$\eps(\lambda) = +$. 
For this, we use
induction on composition length to show that $j^\lambda_*$ is exact on any 
short exact sequence $0 \rightarrow K \rightarrow X \rightarrow Q
\rightarrow 0$ in $\R_\lambda$. For the induction step, suppose we are
given such an exact sequence with $K, Q \neq 0$.
By induction, $j^\lambda_* K$ and $j^\lambda_* Q$ both have
filtrations with sections $\bar\nabla(b)$ for $b \in
\B_\lambda$.
Hence, by Lemma~\ref{anotherpar}, we have that $\Ext^n_{\R}(\Delta_\eps(b),
j^\lambda_* K) = \Ext^n_{\R}(\Delta_\eps(b), j^\lambda_* Q) = 0$ for
all $n\geq 1$ and $b \in \B$.
As it is a right adjoint, $j^\lambda_*$ is left exact, so there is an
exact sequence
\begin{equation}\label{jones}
0 \longrightarrow j^\lambda_* K \longrightarrow j^\lambda_* X \longrightarrow
j^\lambda_* Q.
\end{equation}
Let $Y := j^\lambda_* X / j^\lambda_* K$, so that there is a short exact sequence
\begin{equation}\label{jones2}
0 \longrightarrow j^\lambda_* K \longrightarrow j^\lambda_* X \longrightarrow
Y \longrightarrow 0.
\end{equation}
To complete the argument, it suffices to show that $Y \cong
j^\lambda_* Q$. To establish this, we show that $Y$ satisfies
both of the properties (i) and (ii); then, by the previous paragraph 
and exactness of $j^\lambda$, we get that
$Y \cong j^\lambda_* (j^\lambda Y) 
\cong j_*^\lambda (X / K)\cong
j_*^\lambda Q$, and we are done.
To see that $Y$ satisfies (i), we apply $\Hom_{\R}(\Delta_\eps(b),?)$
to (\ref{jones2}) to get an exact sequence
$$
\Ext^1_{\R}(\Delta_\eps(b), j^\lambda_* X)
\longrightarrow
\Ext^1_{\R}(\Delta_\eps(b), Y)
\longrightarrow
\Ext^2_{\R}(\Delta_\eps(b), j^\lambda_* K).
$$
The first $\Ext^1$ is zero by Lemma~\ref{quot}. Since we already know 
that the $\Ext^2$ term is zero,
$\Ext^1_{\R}(\Delta_\eps(b), Y) =0$.
 To see that $Y$ satisfies (ii), note comparing
(\ref{jones})--(\ref{jones2}) that $Y \hookrightarrow j^\lambda_* Q$, and
$\soc j^\lambda_* Q$ 
is of the desired form by what we know about its
$\bar\nabla_\eps$-flag.
\end{proof}

 \begin{theorem}[Homological criterion for $\Delta_\eps$-flags]\label{gf1}
   Assume that $\R$ is an Abelian category
   equipped with a finite or essentially finite
   stratification $(\B,L,\rho,\Lambda,\leq)$ and sign
function
$\eps$, such that property
$(\widehat{P\Delta}_\eps)$ holds.
For $V \in \R$, the following properties are equivalent:
\begin{itemize}
\item[(i)] $V \in \nabla_\eps(\R)$;
\item[(ii)] $\Ext^1_\R(\Delta_\eps(b), V) = 0$ for all $b \in \B$;
\item[(iii)] $\Ext^n_\R(\Delta_\eps(b), V) = 0$ for all $b \in \B$ and $n
  \geq 1$.
\end{itemize}
If these properties hold, 
the multiplicity $(V:\nabla_\eps(b))$ of
$\nabla_\eps(b)$
as a section of a $\nabla_\eps$-flag of $V$ is well-defined  
independent of the choice of flag, as it equals 
$\dim \Hom_{\R}(\Delta_\eps(b), V)$.
\end{theorem}

\begin{proof}[Proof (assuming lemmas from $\S$\ref{ssp} below)]
 (iii)$\Rightarrow$(ii): Trivial.

\vspace{1.5mm}\noindent
(i)$\Rightarrow$(iii) and the
final assertion of the lemma: These follow from
Lemma~\ref{anotherpar}.

\vspace{1.5mm}\noindent
(ii)$\Rightarrow$(i):
Assume that $V$ satisfies (ii). We prove that it has a $\nabla_\eps$-flag by induction on 
$$
d(V) := \sum_{b \in \B} \dim
\Hom_\R(\Delta_\eps(b), V) \in \N.
$$
The base case when $d(V) = 0$ is trivial as we have then that $V=0$.
For the induction step, let $\lambda \in \Lambda$ be minimal 
such that $\Hom_\R(\Delta_\eps(b),V) \neq 0$ for some $b \in \B$.
The Serre subcategory $\R_{\leq\lambda}$ with the induced
(finite or essentially finite) stratification also satisfies
($\widehat{P\Delta}_\eps$) thanks to Lemma~\ref{sub}(2).
Let $W := i_{\leq\lambda}^! V$.
Because $W$ is a subobject of $V$, 
we have by the minimality of $\lambda$ 
that $\Hom_{\R}(\Delta_\eps(b),
W) \neq 0$ only if $b \in \B_\lambda$. 
Hence, $\soc W \cong L(b_1)\oplus\cdots\oplus L(b_n)$ for
$b_1,\dots,b_n \in \B_\lambda$.
Thus, $W$ satisfies the 
hypothesis (ii) from Lemma~\ref{thesublemma} (with $V$ and
$\R$ there replaced by $W$ and $\R_{\leq\lambda}$).

Now let $Q := V / W$. Take any $b \in \B$
and apply $\Hom_{\R}(\Delta_\eps(b),?)$ to the short exact sequence
$0 \rightarrow W \rightarrow V \rightarrow Q \rightarrow 0$ to get
the exact sequence
\begin{multline*}
0\longrightarrow\Hom_{\R}(\Delta_\eps(b), W)
\longrightarrow
\Hom_{\R}(\Delta_\eps(b), V) \longrightarrow
\Hom_{\R}(\Delta_\eps(b), Q)\\
\longrightarrow
\Ext^1_{\R}(\Delta_\eps(b), W)
\longrightarrow 
0 \longrightarrow \Ext^1_{\R}(\Delta_\eps(b), Q)
\longrightarrow \Ext^2_{\R}(\Delta_\eps(b), W).
\end{multline*}
By the definition of $W$, the socle of $Q$
has no constituent $L(b)$ for $b \in \B_{\leq\lambda}$.
So, for $b \in \B_{\leq\lambda}$ the space
$\Hom_{\R}(\Delta_\eps(b), Q)$ is zero, and we get that
$\Ext^1_{\R_{\leq\lambda}}(\Delta_\eps(b), W) 
\cong \Ext^1_\R(\Delta_\eps(b), W) = 0$ for all such $b$.
This verifies hypothesis (i) from Lemma~\ref{thesublemma}.
So now we can appeal to the lemma to deduce that $W \in
\nabla_\eps(\R_{\leq\lambda})$. Hence, $W \in \nabla_\eps(\R)$.

In view of Lemma~\ref{anotherpar}, we get that
$\Ext^n_{\R}(\Delta_\eps(b), W) = 0$ for all $n \geq 1$ and $b \in
\B$.
So, by the above exact sequence again, we get that
$\Ext^1_{\R}(\Delta_\eps(b), Q)=0$ for all $b \in \B$, and moreover 
$d(Q) = d(V) - d(W) < d(V)$.
Finally we appeal to the
induction hypothesis to deduce that $Q \in \Delta_\eps(\R)$.
Since we already know that $W \in \Delta_\eps(\R)$,
this shows that $V \in \Delta_\eps(\R)$.
\end{proof}

\begin{corollary}
In the setup of Theorem~\ref{gf1}, multiplicities in
a $\nabla_\eps$-flag of $I(b)$ satisfy
$(I(b):\nabla_\eps(c)) = [\Delta_\eps(c):L(b)]$.
\end{corollary}

\begin{corollary}\label{hanks}For $\R$ as in Theorem~\ref{gf1},
let $0 \rightarrow U \rightarrow V \rightarrow W \rightarrow 0$ be a
short exact sequence. If $U$ and $V$ have
$\nabla_\eps$-flags
then so does $W$.
\end{corollary}

\begin{proof}[Proof of Theorem~\ref{fund}]
Suppose 
that $\R$ satisfies ($\widehat{P\Delta}_\eps$). 
Since $V=I(b)$ is injective, it 
satisfies the hypothesis of Theorem~\ref{gf1}(ii).
Hence, by that theorem,
$I(b)$ has a $\nabla_\eps$-flag
and the multiplicity $(I(b):\nabla_\eps(c))$ of $\nabla_\eps(c)$ as a
section of any such flag is given by
$$
(I(b):\nabla_\eps(c))= \dim \Hom_\R(\Delta_\eps(c), I(b)) =
[\Delta_\eps(c):L(b)].
$$
This is zero unless $\rho(b) \leq \rho(c)$.
Also the bottom section must be $\nabla_\eps(b)$ since $I(b)$ has
socle $L(b)$.
Thus, we have verified
that $\R$ satisfies ($I\nabla_\eps$).
Moreover, Lemma~\ref{thesublemma} shows 
that $j^\lambda_*$ is exact whenever $\eps(\lambda)=+$, giving half of 
final assertion made in the statement of the theorem we are trying to prove.

Repeating the arguments in the previous paragraph 
but with $\R$ replaced by $\R^\op$ and $\eps$ replaced
with $-\eps$ show that $(\widehat{I\nabla}_\eps)$ implies $(P\Delta_\eps)$
and that $j^\lambda_!$ is exact whenever $\eps(\lambda)=-$.
Since
$(P\Delta_\eps)\Rightarrow
(\widehat{P\Delta}_\eps)$ and
$(I\nabla_\eps)\Rightarrow (\widehat{I\nabla}_\eps)$,  
this completes the proof.
\end{proof}

So now Theorem~\ref{fund} is proved and Definition~\ref{vector} is in place.
In the remainder of the subsection, we are going to develop some further
fundamental properties of these sorts of category. We start off in the
most general setup with $\R$ being a finite or essentially finite
 $\eps$-stratified category. Again some of the proofs that follow invoke parts of the lemmas from
$\S$\ref{ssp}. From 
Lemma~\ref{against} and the dual statement, deduce that
\begin{equation}\label{repeatedly}
\Ext^1_{\R}(\Delta_\eps(b), \Delta_\eps(c)) = 
\Ext^1_{\R}(\nabla_\eps(c), \nabla_\eps(b)) = 0
\end{equation}
for $b,c \in \B$ with
$\rho(b)\not\leq\rho(c)$.
By ``dual statement'' here, we mean that one takes
Lemma~\ref{against} with
$\R$ replaced by $\R^\op$ and $\eps$ by $-\eps$,
which we may do due to Theorem~\ref{opstrat}
 and Lemma~\ref{skyfall},
then applies the contravariant isomorphism between $\R$ and $\R^\op$.
In a similar way, the following theorem follows immediately as it is the
dual statement to Theorem~\ref{gf1}.

\begin{theorem}[Homological criterion for $\nabla_\eps$-flags]\label{gf2}
Assume that $\R$ is a finite or essentially finite $\eps$-stratified category.
For $V \in \R$, the following properties are equivalent:
\begin{itemize}
\item[(i)] $V \in \Delta_\eps(\R)$;
\item[(ii)] $\Ext^1_\R(V,\nabla_\eps(b)) = 0$ for all $b \in \B$;
\item[(iii)] $\Ext^n_\R(V, \nabla_\eps(b)) = 0$ for all $b \in \B$ and $n
  \geq 1$.
\end{itemize}
Assuming that these properties hold, 
the multiplicity $(V:\Delta_\eps(b))$ of
$\Delta_\eps(b)$
as a section of a $\Delta_\eps$-flag of $V$ is well-defined
independent of the choice of flag, as it equals
$\dim \Hom_{\R}(V, \nabla_\eps(b))$.
\end{theorem}

\begin{corollary}\label{nattering}
$(P(b):\Delta_\eps(c)) = [\nabla_\eps(c):L(b)]$.
\end{corollary}

\begin{corollary}\label{shanks}
Let $0 \rightarrow U \rightarrow V \rightarrow W \rightarrow 0$ be a
short exact sequence in a finite or essentially finite
$\eps$-stratified 
category. If $V$ and $W$ have
$\Delta_\eps$-flags
then so does $U$.
\end{corollary}

The following results about truncation to lower and upper sets are
extremely useful; some aspects of them 
were already used in the proof of Theorem~\ref{gf1}.

\begin{theorem}[Truncation to lower sets]\label{gf}
Assume that $\R$ is a finite or essentially finite $\eps$-stratified category.
Suppose that $\Lambda\d$ is a lower set in $\Lambda$.
Let
$\B\d := \rho^{-1}(\Lambda\d)$
and
 $i:\R\d\rightarrow \R$
be the corresponding Serre
subcategory of $\R$ with the induced stratification.
Then $\R\d$ is itself a finite or essentially finite $\eps$-stratified
category according to whether $\Lambda\d$ is finite or infinite.
Moreover:
\begin{enumerate}
\item
The distinguished
objects in $\R\d$ satisfy
$L\d(b) \cong L(b)$,
$P\d(b) \cong i^* P(b)$,
$I\d(b) \cong i^! I(b)$,
$\Delta\d(b) \cong \Delta(b)$,$\bar\Delta\d(b)\cong\bar\Delta(b)$,
$\nabla\d(b)\cong\nabla(b)$ and $\bar\nabla\d(b)\cong\bar
\nabla(b)$
for $b \in \B\d$.
\item 
$i^*$ sends short exact
sequences of objects in $\Delta_\eps(\R)$ to short exact
sequences of objects in $\Delta_\eps(\R\d)$, with
$i^* \Delta(b) \cong \Delta\d(b)$ and $i^* \bar\Delta(b) \cong \bar\Delta\d(b)$
for $b \in \B\d$
and $i^*\Delta(b) = i^*\bar\Delta(b) = 0$ for $b \notin \B\d$.
\item
$\Ext^n_{\R}(V, i W) \cong \Ext^n_{\R\d}(i^* V, W)$
for $V \in \Delta_\eps(\R)$, $W \in \R\d$
and all $n \geq 0$.
\item 
$i^!$ sends short exact sequences of objects in $\nabla_\eps(\R)$ to short exact sequences
of objects in $\nabla_\eps(\R\d)$, with 
$i^! \nabla(b) \cong \nabla\d(b)$ and $i^! \bar\nabla(b) \cong \bar\nabla\d(b)$
for $b \in \B\d$
and $i^!\nabla(b) = i^!\bar\nabla(b) = 0$ for $b \notin \B\d$.
\item
$\Ext^n_{\R}(i V, W) \cong \Ext^n_{\R\d}(V, i^! W)$
for $V \in \R\d, W \in \nabla_\eps(\R)$
and all $n \geq 0$.
\item
$\Ext^n_{\R}(i V, iW) \cong \Ext^n_{\R\d}(V, W)$
for $V, W \in \R\d$
and $n \geq 0$.
\end{enumerate}
\end{theorem}

\begin{proof}
Apart from (6), this follows by Lemma~\ref{sub} and its dual. 
To prove (6), by the same argument as used to prove Lemma~\ref{sub}(4),
it suffices to show that $(\der{L}_n i^*) V
= 0$ for $V \in \R\d$ and $n \geq 1$.
Since any such $V$ has finite length it suffices to consider an
irreducible object in $\R\d$, i.e., we must show that $(\der{L}_n i^*)
L(b) = 0$ for $b \in \B\d$
and $n \geq 1$.
Take 
a short exact sequence $0 \rightarrow K \rightarrow 
\Delta_\eps(b) \rightarrow L(b)
\rightarrow 0$ and apply $i^*$ and Lemma~\ref{sub}(3) to get
$$
0 \longrightarrow (\mathbb{L}_1 i^*) L(b) \longrightarrow i^* K \longrightarrow i^*
\Delta_\eps(b) \longrightarrow i^* L(b) \longrightarrow 0.
$$
But $K, \Delta_\eps(b)$ and $L(b)$ all lie in $\R\d$ so $i^*$ is the
identity on them. We deduce that $(\mathbb{L}_1 i^*) L(b) = 0$.
Degree shifting easily gives the result for $n > 1$.
\end{proof}

\begin{theorem}[Truncation to upper sets]\label{gf3}
Assume that $\R$ is a finite or essentially finite $\eps$-stratified category.
Suppose that $\Lambda\u$ is an upper set in $\Lambda$.
Let $\B\u := \rho^{-1}(\Lambda\u)$ and $j:\R\rightarrow \R\u$
be the corresponding Serre quotient category
of $\R$ with the induced stratification.
Then $\R\u$ is itself a finite or essentially finite $\eps$-stratified
category 
according to whether $\Lambda\u$ is finite or infinite.
Moreover:
\begin{enumerate}
\item
For $b \in \B\u$,
the distinguished objects
$L\u(b)$,
$P\u(b)$,
$I\u(b)$,
$\Delta\u(b)$, $\bar\Delta\u(b)$,
$\nabla\u(b)$ and $\bar\nabla\u(b)$
 in $\R\u$ are isomorphic to
the images under $j$ of the corresponding objects of
 $\R$.
\item We have that
$
j L(b) = j \Delta(b) = j \bar\Delta(b) = j \nabla(b) = j
\bar\nabla(b) = 0$ if $b \notin \B\u$.
\item
  $\Ext^n_\R(V, j_* W) \cong \Ext^n_{\R\u}(j V, W)$ for $V \in \R,
 W \in \nabla_\eps(\R\u)$ and all $n\geq 0$.
\item 
$j_*$ sends short exact sequences of objects in $\nabla_\eps(\R\u)$ to short exact sequences
of objects in $\nabla_\eps(\R)$, with
$j_* \nabla\u(b) \cong \nabla(b)$, $j_* \bar\nabla\u(b) \cong
\bar\nabla(b)$
and $j_* I\u(b) \cong I(b)$
for $b \in \B\u$.
\item
  $\Ext^n_\R(j_! V, W) \cong \Ext^n_{\R\u}(V, jW)$ for $V \in
  \Delta_\eps(\R\u)$, $W \in \R$ and all $n\geq 0$.
\item 
$j_!$ sends short exact sequences of objects in $\Delta_\eps(\R\u)$ to short exact sequences of objects in $\Delta_\eps(\R)$, with
$j_! \Delta\u(b) \cong \Delta(b)$, $j_! \bar\Delta\u(b) \cong
\bar\Delta(b)$
and $j_! P\u(b) = P(b)$
for $b \in \B\u$.
\end{enumerate}
\end{theorem}

\begin{proof}
Apart from (4) and (6), 
this follows from Lemma~\ref{sup} and its dual.
For (4) and (6), it suffices to prove (4), since (6) is the
equivalent dual statement.
The descriptions of 
$j_* \nabla\u(b)$, $j_*\bar\nabla\u(b)$ and $j_* I\u(b)$,
follow from Lemma~\ref{sup}(1).
It remains to prove the exactness. We can actually show slightly more,
namely, that $(\mathbb{R}^n j_*) V = 0$ for $V \in \nabla_\eps(\R\u)$ and $n
\geq 1$.
Take $V \in \nabla_\eps(\R\u)$. 
Consider a short exact sequence $0 \rightarrow V \rightarrow I
\rightarrow Q \rightarrow 0$ in $\R\u$ with $I$ injective.
Apply the left exact functor $j_*$ and consider the resulting long
exact sequence:
$$
0 \longrightarrow j_* V \longrightarrow j_* I \longrightarrow j_* Q \longrightarrow
 (\mathbb{R}^1 j_*) V \longrightarrow 0.
$$
As $V$ has a $\nabla_\eps$-flag, we can use (3) to see that
$\Hom_{\R}(\Delta_\eps(b), j_* V) \cong
\Hom_{\R\u}(j \Delta_\eps(b), V)$ and 
$\Ext^1_{\R}(\Delta_\eps(b), j_* V) \cong
\Ext^1_{\R\u}(j \Delta_\eps(b), V)$ for every $b \in \B$.
Hence, Theorem \ref{gf1}, $j_* V$ has a $\nabla_\eps$-flag with 
$$
(j_* V:\nabla_\eps(b)) = \dim \Hom_{\R}(j \Delta_\eps(b), V) = 
\left\{
\begin{array}{ll}
(V:\nabla_\eps\u(b))&\text{if $b \in \B\u$},\\
0&\text{otherwise.}
\end{array}\right.
$$
Both $I$ and $Q$ have
$\nabla_\eps$-flags too, so we get similar statements for $j_*I$ and
$j_* Q$.
Since 
$(I:\nabla_\eps\u(b)) =
(V:\nabla_\eps\u(b))+(Q:\nabla_\eps\u(b))$ by the exactness of the
original sequence, we deduce that
$0 \rightarrow j_* V \rightarrow j_* I \rightarrow j_* Q \rightarrow
0$ is exact.
Hence, $(\mathbb{R}^1 j_*) V = 0$.
This proves the result for $n=1$. The result for $n > 1$ follows by a
degree shifting argument.
\end{proof}

\begin{corollary}\label{iuseitoften}
Let notation be as in Theorem~\ref{gf3}
and set $\B\d := \B \setminus \B\u$.
\begin{enumerate}
\item
For $V \in \nabla_\eps(\R)$, there is a short exact sequence
$
0 \rightarrow K \rightarrow V \stackrel{\gamma}{\rightarrow} j_* (j V)
\rightarrow 0
$
where $\gamma$ comes from the unit of adjunction,
$j_* (j V)$ has a $\nabla_\eps$-flag with sections
$\nabla_\eps(b)$ for $b \in \B\u$, and $K$ has a
$\nabla_\eps$-flag with sections $\nabla_\eps(c)$ for $c
\in \B\d$.
\item
For $V \in \Delta_\eps(\R)$, there is a short exact sequence
$
0 \rightarrow j_! (j V) \stackrel{\delta}{\rightarrow} V \rightarrow Q
\rightarrow 0
$
where $\delta$ comes from the counit of adjunction,
$j_! (j V)$ has a $\Delta_\eps$-flag with sections
$\Delta_\eps(b)$ for $b \in \B\u$ and $Q$ has a
$\Delta_\eps$-flag with sections $\Delta_\eps(c)$ for $c
\in \B\d$.
\end{enumerate}
\end{corollary}

\begin{proof}
We prove only (1), since (2) is just the dual statement. 
Using (\ref{repeatedly}), 
we can order the $\nabla_\eps$-flag of $V$ to get a short exact
sequence
$0 \rightarrow K \rightarrow V \rightarrow Q \rightarrow 0$ such that
$K$ has a $\nabla_\eps$-flag with sections $\nabla_\eps(b)$ for $b \in
\B\d$ and $Q$ has a $\nabla_\eps$-flag with sections
$\nabla_\eps(c)$ for $c \in \B\u$.
For each $b \in \B\u$, 
the unit of adjunction $\nabla_\eps(b) \rightarrow j_* 
(j \nabla_\eps(b))$ 
is an isomorphism; this follows from
Theorem~\ref{gf3}(4) using the observation that it becomes an isomorphism on
applying $j$.
Since $j_*$ sends short exact sequences of objects in $\nabla_\eps(\R\u)$ to short exact sequences, we deduce that the
the unit of adjunction 
$Q \rightarrow j_* (j Q)$ is an isomorphism too.
It remains to note that $jV \cong jQ$, hence, $j_* (j V) \cong j_* (j Q)$.
\end{proof}

We proceed to discuss some of the additional features which show up
when in one of the more refined settings
(FS), ($\eps$HW), (FHW) and (HW).
By Theorem~\ref{opstrat},
$\R$ is a fully stratified category 
(resp., fibered highest weight category) if and only if
so is $\R^\op$.
The following lemma shows that fully stratified categories in our terminology are
the same as the ``standardly stratified categories'' defined by Losev
and Webster in \cite[$\S$2]{LW}.

\begin{lemma}\label{rain}
  Given a stratification $(\B,L,\rho,\Lambda,\leq)$ of $\R$,
  the following are equivalent:
\begin{itemize}
\item[(i)] $\R$ is a fully stratified category;
\item[(ii)] $\R$ is $\eps$-stratified for every choice of sign
  function $\eps:\Lambda\rightarrow \{\pm\}$;
\item[(iii)] $\R$ is $\eps$-stratified and
  $(-\eps)$-stratified for some choice of sign function
  $\eps:\Lambda\rightarrow \{\pm\}$;
\item[(iv)] $\R$ is $\eps$-stratified
for some 
$\eps:\Lambda \rightarrow \{\pm\}$ and all of its standardization and
costandardization functors are exact;
\item[(v)] $\R$ is a $+$-stratified category and
each $\Delta(b)$ has a $\bar\Delta$-flag
with sections $\bar\Delta(c)$ for $c$ with $\rho(c) = \rho(b)$;
\item[(vi)] $\R$ is a $-$-stratified category
and each $\nabla(b)$ has a $\bar\nabla$-flag
with sections $\bar\nabla(c)$ for $c$ with $\rho(c) = \rho(b)$.
\end{itemize}
\end{lemma}

\begin{proof}
(i)$\Rightarrow$(ii)$\Rightarrow$(iii):
Obvious.

\vspace{1.5mm}\noindent
(iii)$\Rightarrow$(iv):
Take $\eps$ as in (iii) so that $\R$ is $\eps$-stratified.
The standardization functor $j^\lambda_!$ is exact when
$\eps(\lambda) = -$ by the last part of Theorem~\ref{fund}.
Also $\R$ is $(-\eps)$-stratified, so the same result gives that
$j^\lambda_!$ is exact when $\eps(\lambda)=+$.
Similarly,
all of the costandardization functors are exact too.

\vspace{1.5mm}\noindent
(iv)$\Rightarrow$(v):
Applying the exact standardization functor $j^\lambda_!$ 
to a composition series of $P_\lambda(b)$, we
deduce that $\Delta(b)$ has a $\bar\Delta$-flag with sections
$\bar\Delta(c)$ for $c$ with $\rho(c) = \rho(b)$.
Similarly, applying $j^\lambda_*$, we get that $\nabla(b)$ has a
$\bar\nabla$-flag with sections $\bar\nabla(c)$ for $c$ with $\rho(c)
= \rho(b)$.

To show that $\R$ is $+$-stratified, we check that
each $I(b)$ has a 
$\bar\nabla$-flag with sections
$\bar\nabla(c)$ for $c$ with $\rho(c) \geq \rho(b)$.
This is immediate if $\eps(b) = +$ since we are assuming that $\R$ is
$\eps$-stratified.
If $\eps(b) = -$ then $I(b)$ has a $\nabla$-flag with sections
$\nabla(c)$ for $c$  with $\rho(c) \geq \rho(b)$.
Hence, 
by the previous paragraph, it also has the required sort
of $\bar\nabla$-flag.

\vspace{1.5mm}\noindent
(v)$\Rightarrow$(i):
We just need to show that $\R$ is $-$-stratified.
We know that each $P(b)$ has a $\Delta$-flag with sections $\Delta(c)$
for $c$ with $\rho(c) \geq \rho(b)$.
Now use the given $\bar\Delta$-flags of each $\Delta(c)$
to see that each $P(b)$ also has the appropriate sort of
$\bar\Delta$-flag.

\vspace{1.5mm}\noindent
(v)$\Leftrightarrow$(vi):
This follows from the above using the observation made earlier that 
$\R$ is fully stratified if and only
if $\R^\op$ is fully stratified.
\end{proof}

\begin{corollary}\label{sun}
If $\R$ is an $\eps$-stratified
category with a contravariant autoequivalence which preserves isomorphism classes of irreducible objects, 
then $\R$ is fully stratified.
Moreover, if $\R$ is an $\eps$-highest weight
category with a contravariant autoequivalence preserving isomorphism classes of irreducible objects, then $\R$ is fibered highest weight.
\end{corollary}

\begin{proof}
Since $\R$ is $\eps$-stratified, $\R^\op$ is $(-\eps)$-stratified.
Using Lemma~\ref{dualitylemma}, 
we deduce that $\R$ is $(-\eps)$-stratified.
This verifies Lemma~\ref{rain}(iii) and the first claim follows. The second is then obvious.
\end{proof} 

\begin{lemma}\label{exts}
Suppose that $\R$ is a finite or essentially finite fully 
stratified category.
For $b, c \in \B$ and $n \geq 0$, we have that
$$
\Ext^n_{\R}(\bar\Delta(b), \bar\nabla(c))
\cong 
\left\{
\begin{array}{ll}
\Ext^n_{\R_\lambda}(L(b), L(c))
&\text{if $\lambda = \mu$}\\
0
&\text{otherwise,}
\end{array}
\right.
$$
where $\lambda := \rho(b)$ and $\mu := \rho(c)$.
\end{lemma}

\begin{proof}
Choose $\eps$ so that $\eps(\lambda) = -$,
hence,
$\bar\Delta(b) = \Delta_\eps(b)$.
By Lemma~\ref{rain}, $\R$ is $\eps$-stratified, so we can apply
Theorem~\ref{gf}(4)
with $\R\d = \R_{\leq\mu}$ to deduce that
$$
\Ext^n_{\R}(\bar\Delta(b), \bar\nabla(c)) \cong \Ext^n_{\R_{\leq
    \mu}}(i_{\leq\mu}^* \bar\Delta(b), \bar\nabla(c)).
$$
This is zero unless $\lambda \leq \mu$.
If $\lambda \leq \mu$ it is 
$\Ext^n_{\R_{\leq
    \mu}}(\bar\Delta(b), \bar\nabla(c))$.
Now we change $\eps$ so that $\eps(\mu) = +$, hence, $\bar\nabla(c) =
\nabla_\eps(c)$.
Then by Theorem~\ref{gf3}(3) with $\R = \R_{\leq\mu}$ and $\R\u =
\R_\mu$ we get that
$
\Ext^n_{\R_{\leq
    \mu}}(\bar\Delta(b), \bar\nabla(c))
\cong
\Ext^n_{\R_\mu}(j^\mu \bar\Delta(b), L(c)).
$
This is zero unless $\lambda = \mu$, when $j^\mu \bar\Delta(b) = L(b)$
and we are done.
\end{proof}

The next results are concerned with global dimension.

\begin{lemma}\label{strata}
Let $\R$ be a finite $\eps$-stratified category.
\begin{enumerate}
\item
All $V \in \Delta_\eps(\R)$ are of finite projective 
dimension
if and only if all negative strata\footnote{We mean the strata $\R_\lambda$
for $\lambda \in \Lambda$ such that $\eps(\lambda) = -$.} have finite global dimension.
\item
All $V \in \nabla_\eps(\R)$ are of finite injective dimension
if and only if
all positive strata have finite global dimension.
\end{enumerate}
\end{lemma}

\begin{proof}
As the two parts are dual statements, it suffices to prove (1).
Replacing $\Lambda$ by the finite set $\rho(\B)$ if necessary, we may assume 
that $|\Lambda|<\infty$.

First assume that all negative strata have finite global dimension.
By \cite[Ex.~4.1.2]{Wei}, it suffices to show that $\pd
\Delta_\eps(b) < \infty$ for each $b \in
\B$. 
We proceed by downwards induction on the partial order on the finite poset
$\Lambda$.
Take any $\lambda \in \Lambda$ and
consider $\Delta_\eps(b)$ for $b \in
\B_\lambda$, 
assuming that $\pd \Delta_\eps(c) < \infty$ for each $c \in \B_{>\lambda}$.
We first observe that 
there is a short exact sequence $0 \rightarrow Q \rightarrow P(b)
\rightarrow \Delta(b) \rightarrow 0$ such that $Q$ has a
$\Delta_\eps$-flag with sections $\Delta_\eps(c)$ for $c \in \B_{>
  \lambda}$. If $\eps(\lambda) = +$ this follows immediately from ($P\Delta_\eps$); if $\eps(\lambda)=-$ one also needs to use (\ref{repeatedly}) 
  to see that a $\Delta_\eps$-flag in $P(b)$ can be ordered so that the sections $\bar\Delta(c)$ with $c \in \B_\lambda$ appear above the sections
  with $c \in \B_{> \lambda}$.
By the induction hypothesis, $Q$ has finite
projective dimension, hence, so does $\Delta(b)$.
This verifies the induction step in the case that
$\eps(\lambda)=+$.
Instead, suppose that $\eps(\lambda) = -$, i.e., $\Delta_\eps(b) =
\bar\Delta(b)$. 
Let
$0 \rightarrow P_n \rightarrow \cdots \rightarrow P_0 \rightarrow
L_\lambda(b)\rightarrow 0$ be a finite projective resolution of
$L_\lambda(b)$ in the
stratum $\R_\lambda$.
Applying $j^\lambda_!$, which is exact thanks to Theorem~\ref{fund},
we obtain an exact sequence $0 \rightarrow V_n \rightarrow \cdots
\rightarrow V_0 \rightarrow \bar\Delta(b) \rightarrow 0$
such that each $V_m$ is a direct sum of standard objects $\Delta(c)$
for $c \in \B_\lambda$. The result already established plus
\cite[Ex.~4.1.3]{Wei} implies that $\pd V_m < \infty$ for each $m$.
Arguing like in the
proof of \cite[Th.~4.3.1]{Wei}, we deduce that $\pd \bar\Delta(b)
< \infty$ too.

Conversely,
suppose that $\pd \Delta_\eps(b) < \infty$ for all $b \in\B$. 
Take $\lambda \in \Lambda$ with $\eps(\lambda) = -$.
To show that $\R_\lambda$ has finite global dimension, it suffices to show that there is some $d(\lambda) \geq 0$
such that $\Ext^n_{\R_\lambda}(L_\lambda(b), W) = 0$
for all $n > d(\lambda), b \in \B_\lambda$ and $W \in \R_\lambda$.
By Theorems~\ref{gf3}(3) and \ref{gf}(3), we have that 
$$
\Ext^n_{\R_\lambda}(L_\lambda(b), W)
\cong 
\Ext^n_{\R_{\leq \lambda}}(\Delta_\eps(b), j^\lambda_* W)
\cong \Ext^n_\R(\Delta_\eps(b), i_{\leq \lambda} (j^\lambda_* W)).
$$
So we can take $d(\lambda) = \max\{\pd \Delta_\eps(b)\:|\:b \in \B_\lambda\}$.
\end{proof}
The case when all strata are positive (respectively negative) will be of great importance.  

\begin{corollary}\label{inparticulartiltings}
If $\R$ is a finite $+$-stratified (resp., $-$-stratified) category then all $V \in \Delta(\R)$ (resp., $V \in \nabla(\R)$)
are of finite projective (resp., injective) dimension.
\end{corollary}

\begin{corollary}
Suppose that $\R$ is a finite $\eps$-stratified category.
If $\R$ is 
of finite global dimension then
all of its strata are of finite global dimension too.
\end{corollary}

\begin{proof}
Lemma~\ref{strata}(1) implies that all negative strata have finite global dimension,
and Lemma~\ref{strata}(2) implies that all positive strata have finite global dimension.
\end{proof}

\begin{corollary}\label{fgd}
Suppose that $\R$ is either a finite $+$-stratified category or
a finite $-$-stratified category.
If all of the strata are of finite global dimension then $\R$ is of
finite global dimension.
\end{corollary}

\begin{proof}
We just explain this in the case that $\R$ is $-$-stratified; the
argument in the 
$+$-stratified case is similar.
Lemma~\ref{strata}(1) implies that $\bar\Delta(b)$ is of finite
projective dimension for each $b \in \B$.
Moreover, there is a short exact sequence $0 \rightarrow K \rightarrow
\bar\Delta(b) \rightarrow L(b) \rightarrow 0$ where all
composition factors of $K$ are of the form $L(c)$ for $c$ with
$\rho(c) < \rho(b)$.
Ascending induction on the partial order on the finite set $\rho(\B)
\subseteq \Lambda$ implies that each $L(b)$ has finite projective dimension.
\end{proof}

A special case of Corollary~\ref{fgd} recovers the following well-known
result, see e.g. \cite{CPS}.
For further detailed remarks about the history of this, and the general 
notion of highest weight category,  
we refer to \cite[$\S$A5]{Donkin} and \cite{DlabRingel}.

\begin{corollary} 
Finite highest weight categories are of finite
global dimension.
\end{corollary}

\begin{remark}
In the fully stratified case,
Lemma~\ref{exts} can be used
to give a precise bound on the global dimension of $\R$ in 
Corollary~\ref{fgd}.
Assuming $\Lambda$ is finite, let $$
|\lambda| :=
\sup\!\left\{\frac{\max\left(\operatorname{gl.dim}\R_{\lambda_1},\dots,\operatorname{gl.dim}\R_{\lambda_{n}}\right)}{2}+n-1\:\bigg|\:
\!\!\begin{array}{l}
n \geq 1\text{ and }\lambda_1,\dots,\lambda_{n} \in \Lambda
\\
\text{with }
\lambda_1 <
\dots < \lambda_{n} = \lambda
\end{array}
\!\!\right\}.
$$
By mimicking the
proof of \cite[Prop.~A2.3]{Donkin},
one shows
that $\Ext^i_{\R}(L(b), L(c)) = 0$ for $b,c \in \B$ and any $i >
|\rho(b)|+|\rho(c)|$.
Hence, $
\operatorname{gl.dim} \R \leq 2 \max\{|\lambda|\:|\:\lambda \in \Lambda\}.
$
For finite 
highest weight categories, this shows that
$\operatorname{gl.dim} \R \leq 2 (n-1)$ where $n$ is length of 
the longest chain of weights in the weight poset $\Lambda$.
\end{remark}

\begin{remark}
Outside of the highest weight case, {\em finitistic dimension}
is used as a replacement for global dimension.
In particular, finite fibered highest weight categories have 
finitistic dimension $\leq 2 (n-1)$ where $n$ is length of 
the longest chain of weights in the weight poset $\Lambda$; 
this can be proved following the argument of
\cite[Cor.~2.7]{AHLU}. For finite 
fully stratified categories, it should be possible to bound the finitistic dimension of $\R$ 
in terms of the finitistic dimensions of the strata and chains in the poset like in the previous remark.
\end{remark}

\begin{remark}\label{moremark}
Another remarkable result about global dimension of finite 
highest weight categories was obtained in \cite{MO}, \cite{MP} proving conjectures formulated in 
\cite{CaeZ}, \cite{EP}: if $\R$ is a finite highest weight category with duality, i.e., possessing a contravariant autoequivalence preserving isomorphism classes of irreducible objects, 
then the global dimension of $\R$ is equal to twice the projective dimension of a tilting generator (see Definition~\ref{thesetup} below).
More generally, Mazorchuk and Ovsienko show that
the finitisic dimension is equal to twice the projective dimension 
of a tilting generator in any finite fibered highest weight category with duality
which is also tilting-rigid in the sense of Definition~\ref{tiltingrigiddef} below.
Recently, 
Cruz and Marczinik \cite[Th.~2.2]{CM} (see also Corollary~\ref{gorencor} below) have shown 
that a finite fibered highest weight category $\R$ is tilting-rigid if and only if it is Gorenstein,
in which case the finitistic dimension of $\R$ coincides with its Gorenstein dimension (e.g., see \cite[Lem.~2.3.2]{Chen}).
\end{remark}

\subsection{\boldmath Upper finite $\eps$-stratified categories}\label{sst}
In this subsection we assume that $\R$ is a Schurian category equipped
with an upper finite stratification $(\B,L,\rho,\Lambda,\leq)$. 
Also $\eps:\Lambda\rightarrow\{\pm\}$ denotes a sign function.
Let $I(b)$ and $P(b)$ be an injective hull and a projective
cover of $L(b)$ in $\R$. Recall
(\ref{rumble})--(\ref{device}), the properties
$(P\Delta_\eps)$--$(I\nabla_\eps)$ 
and $(P\Delta)$--$(I\nabla)$ 
from the introduction, and the seemingly stronger 
properties $(\widehat{P\Delta}_\eps)$--$(\widehat{I\nabla}_\eps)$ 
and $(\widehat{P\Delta})$--$(\widehat{I\nabla})$ from the
previous subsection.

Before formulating the main definitions in the upper finite setting, 
we prove an analog of the homological 
criterion for $\nabla_\eps$-flags 
from Theorem~\ref{gf1}.
The proof depends on the lemmas proved in $\S$\ref{ssp} below,
which we used already in the previous subsection, together with the
following two technical lemmas, which we prove by truncating to
finite Abelian quotients.

\begin{lemma}\label{phase1}
Suppose that $\R$ is 
Schurian with upper finite stratification 
$(\B,L,\rho,\Lambda,\leq)$ and sign function $\eps$,
and assume that the
property $(\widehat{P\Delta}_\eps)$ holds in $\R$.
Let $\Lambda\u$ be a finite upper set in $\Lambda$, $\B\u :=
\rho^{-1}(\Lambda\u)$, and $j:\R\rightarrow \R\u$ be the corresponding
Serre quotient category with the induced stratification.
The functor $j_*$ sends short exact sequences of objects in $\nabla_\eps(\R\u)$ to short exact sequences of objects in $\nabla_\eps(\R)$.
\end{lemma}

\begin{proof}[Proof (assuming lemmas in
  $\S$\ref{ssp} below)]
  Take 
   a short exact sequence $0 \rightarrow K \rightarrow X \rightarrow Q \rightarrow
0$
in $\R\u$ such that $K, X$ and $Q$ have
$\nabla_\eps$-flags.
We must 
show that $0 \rightarrow j_* K \rightarrow j_* X \rightarrow j_* Q
\rightarrow 0$ is exact with all objects belonging to
$\nabla_\eps(\R)$. We proceed by induction on the length of the $\nabla_\eps$-flag of $j_*(X)$, with the
base case (length one) 
following from Lemma~\ref{sup}(1).
For the induction step, we may assume that $K, Q \neq 0$ and
know by induction that $j_* K$ and $j_* Q$ have 
$\nabla_\eps$-flags. We must show that 
$0 \rightarrow j_* K \rightarrow j_* X \rightarrow j_* Q \rightarrow 0$
is exact. Since $j_*$ is left exact, this follows if we can show that
$$
[j_* X:L(b)] = [j_* K:L(b)] + [j_* Q:L(b)]
$$ 
for all $b \in \B$.
To see this, let $\Lambda\uu$ be the finite
upper set generated by $\Lambda\u$ and $b$. Let $\B\uu :=
\rho^{-1}(\Lambda\uu)$
and $k:\R\rightarrow \R\uu$ be the corresponding Serre quotient.
By Lemma~\ref{guns}, we have that $[j_* X:L(b)] = [k (j_* X):k L(b)] =
[k(j_* X):L\uu(b)]$, and similarly for $K$ and $Q$.
Since $\Lambda\u$ is an upper set in $\Lambda\uu$,
we can also view $\R\u$ as a quotient of $\R\uu$, and the quotient
functor $j$ factors as $j = \bar\jmath \circ k$ for another quotient
functor $\bar\jmath:\R\uu\rightarrow\R\u$.
We have that $k_* \circ \bar\jmath_* \cong j_*$, hence, applying $k$,
we get that $\bar\jmath_* \cong k \circ j_*$.
It follows that $[k(j_* X):L\uu(b)] = [\bar \jmath_* X: L\uu(b)]$, and
similarly for $K$ and $Q$. 
We have now reduced the proof to showing that
$$
[\bar\jmath_* X:L\uu(b)] = [\bar\jmath_* K:L\uu(b)] + [\bar\jmath_* Q:L\uu(b)].
$$ 
To see this, we note that $\R\uu$ and $\R\u$ are finite $\eps$-highest
weight categories due to Lemma~\ref{sup}(2) and Theorem~\ref{fund}.
So we can apply Theorem~\ref{gf3}(4)  to see 
that the sequence $0 \rightarrow \bar\jmath_* K \rightarrow
\bar\jmath_* X \rightarrow \bar\jmath_* Q \rightarrow 0$ is exact.
\end{proof}

\begin{lemma}\label{phase2}
Suppose that $\R$ is 
Schurian with upper finite stratification 
$(\B,L,\rho,\Lambda,\leq)$ and sign function $\eps$,
and assume that the
property $(\widehat{P\Delta}_\eps)$ holds in $\R$.
Let $V \in \R$ be a finitely cogenerated object such that
$\Ext^1_\R(\Delta_\eps(b), V) = 0$ for all $b \in \B$.
Then we have that $V \in \nabla_\eps(b)$, and the multiplicity
$(V:\nabla_\eps(b))$ of $\nabla_\eps(b)$ in any $\nabla_\eps$-flag
is equal to the dimension of $\Hom_{\R}(\Delta_\eps(b), V)$.
\end{lemma}

\begin{proof}[Proof (assuming lemmas from $\S$\ref{ssp} below)]
Since $V$ is finitely cogenerated, its injective hull 
is a finite direct sum of the indecomposable injective
objects $I(b)$.
This means that we can find 
a finite upper set $\Lambda\u$ and $\B\u := \rho^{-1}(\Lambda\u)$ so
that there is a short exact sequence
$$
0 \longrightarrow V \longrightarrow \bigoplus_{b \in \B\u} I(b)^{\oplus n_b}
\longrightarrow Q \longrightarrow 0
$$
for some $n_b \geq 0$.
Let $j:\R\rightarrow\R\u$ be the corresponding Serre
quotient. This is a finite $\eps$-stratified category by
Lemma~\ref{sup}(2) and Theorem~\ref{fund}.

Applying $j$ to the above short exact sequence gives us a short
exact sequence in $\R\u$. Then we take $b \in \B\u$ and apply the functor
$\Hom_{\R\u}(\Delta\u_\eps(b),?)$ to this using also Lemma~\ref{sup}(1)
to obtain the long exact sequence
\begin{multline*}
0 \longrightarrow \Hom_{\R\u}(\Delta\u_\eps(b),  jV) \longrightarrow \Hom_{\R\u}\Big(\Delta\u_\eps(b),\textstyle\bigoplus_{b \in \B\u} I\u(b)^{\oplus n_b}\Big)
\\\longrightarrow \Hom_{\R\u}(\Delta\u_\eps(b), jQ) \longrightarrow
\Ext^1_{\R\u}(\Delta\u_\eps(b), jV)\longrightarrow 0.
\end{multline*}
From adjunction and Lemma~\ref{sup}(1) again, we get a commuting diagram
$$
\begin{CD}
0 \rightarrow&\Hom_{\R\u}(\Delta\u_\eps(b),
jV)&\rightarrow&\Hom_{\R\u}\Big(\Delta\u_\eps(b), \bigoplus_{b \in \B\u}
I\u(b)^{\oplus n_b}\Big)&\rightarrow&\Hom_{\R\u}(\Delta\u_\eps(b),
j Q)&\rightarrow 0\\
&@VVV@VVV@VVV\\
0 \rightarrow&\Hom_\R(\Delta_\eps(b), V)&\rightarrow&\Hom_\R\Big(\Delta_\eps(b),
\bigoplus_{b \in \B\u} I(b)^{\oplus
  n_b}\Big)&\rightarrow&\Hom_\R(\Delta_\eps(b), Q)&\rightarrow 0.
\end{CD}
$$
The vertical maps are isomorphisms and the bottom row is exact
since $\Ext^1_\R(\Delta_\eps(b), V) = 0$.
Hence the top row is exact. 
Comparing with the previously displayed long exact sequence, 
it follows that
$\Ext^1_{\R\u}(\Delta\u_\eps(b), jV) = 0$.
Now we can apply Theorem~\ref{gf1} in the finite $\eps$-stratified
category $\R\u$ to deduce that $j V$ has a
$\nabla_\eps$-flag.

From Lemma~\ref{phase1}, we deduce that $j_* j V$ has a $\nabla_\eps$-flag.
Moreover the multiplicity of $\nabla_\eps(b)$ in any $\nabla_\eps$-flag
in $j_* j V$ 
is $\dim \Hom_{\R}(\Delta_\eps(b), j_* j V)$
thanks to Lemma~\ref{anotherpar}.
To complete the proof, we show that
the unit of adjunction $f:V \rightarrow j_* j V$ is an isomorphism.
We know from Lemma~\ref{sup}(1) that 
the unit of adjunction is an isomorphism $I(b) \rightarrow j_* j
I(b)$  for each $b \in \B\u$. Since $V$ embeds into a direct sum of
such $I(b)$, it follows that $f$ is injective.
To show that it is surjective, it suffices to show that
$$
[j_*j V:L(b)] = [V:L(b)]
$$
for all $b \in \B$. To prove this, we fix a choice of $b \in \B$ then 
define $\Lambda\uu$, $\B\uu$,
$k:\R\rightarrow \R\uu$ and $\bar\jmath: \R\uu\rightarrow \R\u$ as
in the proof of Lemma~\ref{phase1}.
Since $b \in \B\uu$, we have that $[V:L(b)] = [k V: L\uu(b)]$
and $[j_* j V:L(b)] = [k(j_* j V):L\uu(b)]$.
As in the proof of Lemma~\ref{phase1},
$k(j_* j V) \cong \bar\jmath_* (jV)
\cong \bar\jmath_* \bar \jmath (k V)$.
Thus, we are reduced to showing that
$$
[\bar \jmath_* \bar\jmath (kV):L\uu(b)] = [kV:L\uu(b)].
$$
This follows because $kV \cong \bar\jmath_* \bar\jmath (kV)$.
To see this,  we repeat 
the arguments in the previous paragraph to 
show that $kV\in \R\uu$ has a $\nabla_\eps$-flag.
Since the unit of adjunction is an isomorphism
$\nabla_\eps\uu(b) \stackrel{\sim}{\rightarrow}
\bar\jmath_* \bar\jmath \nabla_\eps\uu(b)$ for each $b \in \B\uu$,
we deduce using the exactness from
Theorem~\ref{gf3}(4) 
that it gives an isomorphism $k V \stackrel{\sim}{\rightarrow} \bar\jmath_*
\bar\jmath(kV)$ too.
\end{proof}

\begin{theorem}
Theorem~\ref{fund} holds in the upper finite setup too.
\end{theorem}

\begin{proof}
This is almost the same as the proof of Theorem~\ref{fund} given in
the previous subsection. One needs to
use Lemma~\ref{phase2} in place of Theorem~\ref{gf1} to see that
$I(b)$ has a $\nabla_\eps$-flag with the appropriate multiplicities.
The exactness
of $j^\lambda_*$ when $\eps(\lambda)=+$ follows from 
Lemma~\ref{phase1} applied to the quotient functor $j^\lambda:\R_{\leq\lambda} \rightarrow \R_\lambda$.
Note for this that $\R_{\leq\lambda}$ satisfies
$(\widehat{P\Delta}_\eps)$ due to Lemma~\ref{sub}(2), and we have that
$\nabla_\eps(\R_\lambda) = \R_\lambda$ as $\eps(\lambda)=+$.
\end{proof}

We are ready to proceed to the main
definition.

\begin{definition}\label{ufc}
  Let $(\B,L,\rho,\Lambda,\leq)$
  be an upper finite stratification on $\R$.
\begin{enumerate}
\item[($\eps$S)]
We say that $\R$ is an {\em upper finite $\eps$-stratified category}
if one of the equivalent properties 
$(\widehat{P\Delta}_\eps)$--$(\widehat{I\nabla}_\eps)$ holds for a given
choice of sign function $\eps:\Lambda\rightarrow\{\pm\}$.
\item[(FS)]
We say that $\R$ is an {\em upper finite fully stratified category}
if one of these properties
holds for all choices of sign function $\eps:\Lambda\rightarrow\{\pm\}$.
\item[($\eps$HW)]
We say that $\R$ is an {\em upper finite $\eps$-highest weight
  category}
if the stratification function
$\rho$ is a bijection, and one of the equivalent properties
$(\widehat{P\Delta}_\eps)$--$(\widehat{I\nabla}_\eps)$ holds for a given choice of sign function $\eps:\Lambda\rightarrow\{\pm\}$.
\item[(FHW)]
We say that $\R$ is an {\em upper finite fibered highest weight
  category}
if the stratification function is
a bijection and one of these properties holds for all choices of sign function.
\item[(HW)]
We say that $\R$ is an {\em upper finite highest weight category}
if all of the stata are simple
(cf. Lemma~\ref{allsimple}) and one of the equivalent properties
$(\widehat{P\Delta})$--$(\widehat{I\nabla})$ holds.
\end{enumerate}
\end{definition}

The $\Ext^1$-vanishing 
(\ref{repeatedly}) and
Theorem~\ref{opstrat}
both
still hold in the same way as before. 

Next we are going to consider two (in fact dual) notions of ascending
$\Delta_\eps$- and descending $\nabla_\eps$-flags, generalizing the
finite flags discussed already. One might be tempted to say that an
ascending $\Delta_\eps$-flag in $V$ is an ascending chain
$0 = V_0 < V_1 < V_2 < \cdots$ of subobjects of $V$ with $V = \sum_{n \in\N} V_n$ 
such that $V_m / V_{m-1} \cong \Delta_\eps(b_m)$,
and a descending $\nabla_\eps$-flag 
is a
descending chain
$V = V_0 > V_{1} > V_{2} > \cdots$ of subobjects of $V$
such that $\bigcap_{n \in \N} V_n = 0$ and 
$V_{m-1} / V_{m} \cong \Delta_\eps(b_m)$, for $b_m  \in \B$.
These would be 
serviceable definitions when $\Lambda$ is
countable. In order
to avoid this unnecessary restriction, we will work instead with the
following more general formulations.

\begin{definition}\label{tpc}
Suppose that $\R$ is an upper finite $\eps$-stratified category and $V \in \R$.
\begin{enumerate}
\item[$(A\Delta)$]
An {\em ascending $\Delta_\eps$-flag} in $V$ is the data of
a directed set $\Omega$ with smallest element $0$ and
a direct system $(V_\omega)_{\omega \in \Omega}$ 
of subobjects of $V$ such that
$V_0 = 0$,
$\sum_{\omega \in \Omega} V_\omega = V$, and
$V_\upsilon/V_\omega \in \Delta_\eps(\R)$ for each $\omega < \upsilon$.
Let $\Delta_\eps^\asc(\R)$ be the full subcategory of $\R$ consisting
of all objects $V$ possessing such a flag.
\item[$(D\nabla)$]
A {\em descending $\nabla_\eps$-flag} in $V$ is the data of
a directed set $\Omega$ with smallest element $0$
and an inverse system $(V / V_\omega)_{\omega \in \Omega}$ 
of quotients of $V$ such that
$V_0 = V$, $\bigcap_{\omega \in \Omega} V_\omega = 0$,
and $V_\omega / V_\upsilon\in\nabla_\eps(\R)$ for each $\omega < \upsilon$.
Let $\nabla_\eps^\desc(\R)$ be the full subcategory of $\R$ consisting
of all objects $V$ possessing such a flag.
\end{enumerate}
We stress that $\Delta_\eps^\asc(\R)$ 
and $\nabla_\eps^\desc(\R)$ are subcategories of $\R$: we have {\em not}
 passed to the completion $\Ind(\R_c)$.
\end{definition}

\begin{lemma}\label{dry}
Suppose that $\R$ is an upper finite $\eps$-stratified category.
\begin{enumerate}
\item
For  $V \in \Delta^\asc_\eps(\R)$,
$W \in \nabla^\desc_\eps(\R)$ and $n \geq 1$,
we have that 
$\Ext^n_{\R}(V,W) = 0$.
\item
For $V \in \Delta_\eps^\asc(\R)$
the multiplicity of $\Delta_\eps(b)$ in a $\Delta_\eps$-flag may be
defined from
$$
(V:\Delta_\eps(b)) := \dim \Hom_{\R}(V, \nabla_\eps(b)) =
\sup\left\{ (V_\omega:\Delta_\eps(b))\:\big|\:\omega \in
  \Omega\right\} < \infty,
$$
where $(V_\omega)_{\omega \in \Omega}$ is
any choice of ascending
$\Delta_\eps$-flag.
\item
For $V \in \nabla_\eps^\desc(\R)$,
the multiplicity of $\nabla_\eps(b)$ in a $\nabla_\eps$-flag may be
defined from
$$
(V:\nabla_\eps(b)) := \dim \Hom_{\R}(\Delta_\eps(b), V) =
\sup\left\{ (V / V_\omega:\nabla_\eps(b))\:\big|\:\omega \in
  \Omega\right\} < \infty,
$$
where $(V / V_\omega)_{\omega \in \Omega}$ is any choice of descending
$\nabla_\eps$-flag.
\end{enumerate}
\end{lemma}

\begin{proof}
 (1)
We first prove this in the special case that $W = \nabla_\eps(b)$. 
Let $(V_\omega)_{\omega\in\Omega}$ be an ascending $\Delta_\eps$-flag
in $V$, so that $V \cong \varinjlim V_\omega$.
Since $\Ext^n_{\R}(V_\omega,W) = 0$ by
Lemma~\ref{anotherpar}, 
it suffices to show that
$$
\Ext^n_{\R}(V, W) \cong \varprojlim \Ext^n_{\R}(V_\omega,
W).
$$
To see this, 
like in \cite[3.5.10]{Wei},
we need to check a Mittag-Leffler
condition. We show that the natural map
$\Ext^{n-1}_{\R}(V_\upsilon, W) \rightarrow
\Ext^{n-1}_{\R}(V_\omega,W)$ is surjective for each
$\omega < \upsilon$ in $\Omega$. 
Applying $\Hom_{\R}(?, W)$ to the short exact sequence
$0 \rightarrow V_\omega \rightarrow V_\upsilon \rightarrow
V_\upsilon/V_\omega\rightarrow 0$ gives an exact sequence
\begin{align*}
\Ext^{n-1}_{\R}(V_\upsilon, W) \longrightarrow
\Ext^{n-1}_{\R}(V_\omega, W)\longrightarrow 
\Ext^n_{\R}(V_\upsilon/V_\omega,W).
\end{align*}
It remains to observe that
$\Ext^n_{\R}(V_\upsilon/ V_\omega,
W) = 0$ by Lemma~\ref{anotherpar} again, since
we know from the definition of ascending $\Delta_\eps$-flag that $V_\upsilon/V_\omega \in
\Delta_\eps(\R)$.

The dual of the previous paragraph plus
Lemma~\ref{skyfall}
gives that
$\Ext^n_{\R}(V, W)=0$ for $n \geq 1$, $V = \Delta_\eps(b)$
and $W \in \nabla^\desc_\eps(\R)$.
Then we can repeat the argument of the previous paragraph yet again,
using this assertion in place of Lemma~\ref{anotherpar}, to obtain the
result we are after for 
general $V \in \Delta^\asc_\eps(\R)$ and $W \in \nabla^\desc_\eps(\R)$.

\vspace{1.5mm}
\noindent 
(2)
This follows 
from (1) and 
(\ref{onedimensionality}) 
 because
$$
\Hom_{\R}(V, \nabla_\eps(b)) \cong \Hom_{\R}(\varinjlim V_\omega, \nabla_\eps(b))
\cong
\varprojlim \Hom_{\R}(V_\omega, \nabla_\eps(b)),
$$
which is finite-dimensional as $\nabla_\eps(b)$, hence, each $V_\omega$, is finitely cogenerated.

\vspace{1.5mm}
\noindent
(3) Similarly to (2), we have that
$$
\Hom_{\R}(\Delta_\eps(b), V) \cong \Hom_{\R}(\Delta_\eps(b), \varprojlim (V/V_\omega))
\cong
\varprojlim \Hom_{\R}(\Delta_\eps(b), V / V_\omega),
$$
which is finite-dimensional as $\Delta_\eps(b)$ is finitely generated.
Then we can apply (1) and  (\ref{onedimensionality}) once again.
\end{proof}

\begin{theorem}[Homological criterion for ascending $\Delta_\eps$-flags]\label{gf8}
Assume that $\R$ is an upper finite $\eps$-stratified category.
For $V \in \R$, the following are equivalent:
\begin{itemize}
\item[(i)]
$V \in \Delta_\eps^\asc(\R)$;
\item[(ii)]
$\Ext^1_\R(V,\nabla_\eps(b)) = 0$ for all $b \in \B$;
\item[(iii)]
$\Ext^n_\R(V,\nabla_\eps(b)) = 0$ for all $b \in \B$ and  $n \geq 1$.
\end{itemize}
Assuming these properties, we have that
$V \in \Delta_\eps(\R)$ if and only if it is finitely generated.
\end{theorem}

\begin{proof}
(iii)$\Rightarrow$(ii). Trivial.

\vspace{1.5mm}
\noindent
(i)$\Rightarrow$(iii). This follows from Lemma~\ref{dry}(1).

\vspace{1.5mm}
\noindent
(ii)$\Rightarrow$(i).
Let $\Omega$ be the directed set of finite upper sets in $\Lambda$.
Take $\omega \in \Omega$; it is some finite upper set $\Lambda\u$.
Let $\B\u := \rho^{-1}(\Lambda\u)$ and $j:\R\rightarrow\R\u$ be the
corresponding Serre quotient.
By Lemma~\ref{sup}(3), $\Ext^1_{\R\u}(j V, \nabla_\eps(b)) = 0$ for
all $b \in \B\u$. Hence, $V_\omega := j_!( j V) \in \Delta_\eps(\R)$
thanks to the dual of Lemma~\ref{phase1}.
Let $f_\omega:V_\omega \rightarrow V$ be the morphism induced by the
counit of adjunction.
We claim for any $b \in \B\u$
that the map 
$$
f_\omega(b):\Hom_{\R}(P(b),
V_\omega)
\rightarrow \Hom_{\R}(P(b), V), \:\theta \mapsto f_\omega \circ \theta
$$ 
is an isomorphism. To see this, we assume that $\R = A\lfdlmod$ for a 
pointed locally finite-dimensional locally unital algebra
$A = \bigoplus_{a,b\in\B} e_a A e_b$. Then $\R\u = eAe\lfdlmod$
where
$e = \sum_{a \in \B\u} e_a$,
and
$V_\omega = Ae \otimes_{eAe} e V$.
In these terms, the map $f_\omega$ is the natural multiplication map.
For $b \in \B\u$, this multiplication map
gives an isomorphism $e_b V_\omega
\stackrel{\sim}{\rightarrow} e_b V$
with inverse $e_b v \mapsto e_b \otimes e_b v$.
This proves the claim.

Now take $\upsilon > \omega$, i.e., another finite upper set
$\Lambda\uu \supset\Lambda\u$, and
let $k:\R\rightarrow \R\uu$ be the associated quotient.
The quotient functor $j:\R\rightarrow\R\u$ factors as $j = \bar\jmath
\circ k$
for another quotient functor $\bar\jmath:\R\uu \rightarrow \R\u$, 
and we have that
$$
V_\omega = (\bar\jmath \circ k)_!
((\bar\jmath \circ k) V) \cong k_! (\bar\jmath_! (\bar\jmath (k V))),\qquad
V_\upsilon = k_! (k V).
$$
By Corollary~\ref{iuseitoften}(2),
there is a short exact sequence
$0 \rightarrow \bar\jmath_! (\bar\jmath (k V)) \rightarrow k V \rightarrow 
Q \rightarrow 0$ such that both
$\bar\jmath_! (\bar\jmath (k V))$ and $Q$ belong to $\Delta_\eps(\R\uu)$.
Applying $k_!$ and using the exactness 
from the dual of Lemma~\ref{phase1},  we
get an embedding $f_{\omega}^\upsilon:V_\omega
\hookrightarrow V_\upsilon$ such that 
$V_\upsilon /
V_\omega \cong k_! Q \in \Delta_\eps(\R)$.
Since the morphisms all came from counits of adjunction, we have that
$f_\upsilon \circ f_{\omega}^\upsilon =
f_\omega$.

Now we can show that each $f_\omega$ is a monomorphism.
It suffices to show that $f_\omega(b):\Hom_{\R}(P(b), V_\omega)
\rightarrow \Hom_{\R}(P(b), V)$ is injective for all $b \in \B$.
Choose $\upsilon$ in the previous paragraph to be sufficiently large
so as to ensure that $b \in \B\uu$.
We explained already that $f_\upsilon(b)$ is an isomorphism.
Since $f_\omega = f_\upsilon \circ f_{\omega}^\upsilon$ and
$f_{\omega}^\upsilon$ is a monomorphism, it follows that $f_\omega(b)$
is injective too.
Thus, identifying $V_\omega$ with its image under $f_\omega$, we have defined a direct
system
$(V_\omega)_{\omega \in \Omega}$ of subobjects of $V$ such that
$V_\upsilon / V_\omega \in \Delta_\eps(\R)$ for each $\omega <
\upsilon$.
It remains to observe that $V_\varnothing = 0$ for a trivial reason, and $\sum_{\omega \in
  \Omega} V_\omega = V$ because
we know for each $b \in \B$ that 
$f_\omega(b)$ is surjective for sufficiently large $\omega$.
 
\vspace{1.5mm}\noindent
{Final part:}
If $V \in \Delta_\eps(\R)$, it is obvious that it is finitely
generated since each $\Delta_\eps(b)$ is finitely generated.
Conversely, suppose that $V$ is finitely generated and has an
ascending $\Delta_\eps$-flag.
To see that it is actually a finite flag, it suffices to show that
$\Hom_{\R}(V, \nabla_\eps(b)) = 0$ for all but finitely many $b \in
\B$.
Say $\hd V \cong L(b_1)\oplus\cdots\oplus L(b_n)$.
If $V \rightarrow \nabla_\eps(b)$ is a non-zero homomorphism, we must
have that $\rho(b_i) \leq \rho(b)$ for some $i=1,\dots,n$. Hence,
there are only finitely many choices for $b$ as the poset is upper finite. 
\end{proof}

\begin{corollary}
Let $0 \rightarrow U \rightarrow V \rightarrow W \rightarrow 0$ be a
short exact sequence in $\R$. 
\begin{enumerate}
\item If $U$ and $W$ belong to $\Delta^\asc_\eps(\R)$ (resp., $\Delta_\eps(\R)$) so does $V$.
\item
If
$V$ and $W$ belong to $\Delta^\asc_\eps(\R)$ (resp.,
$\Delta_\eps(\R)$) 
so does $U$.
\end{enumerate}
\end{corollary}

\begin{theorem}[Homological criterion for descending $\nabla_\eps$-flags]\label{gf7}
Assume that $\R$ is an upper finite $\eps$-stratified category.
For $V \in \R$, the following are equivalent:
\begin{itemize}
\item[(i)]
$V \in \nabla_\eps^\desc(\R)$;
\item[(ii)]
$\Ext^1_\R(\Delta_\eps(b), V) = 0$ for all $b \in \B$;
\item[(iii)]
$\Ext^n_\R(\Delta_\eps(b), V) = 0$ for all $b \in \B$ and  $n \geq 1$.
\end{itemize}
Assuming these properties,
$V \in \nabla_\eps(\R)$ if and only if it is finitely cogenerated.
\end{theorem}

\begin{proof}
This is the equivalent dual statement
to Theorem~\ref{gf8}.
\end{proof}

\begin{corollary}
Let $0 \rightarrow U \rightarrow V \rightarrow W \rightarrow 0$ be a
short exact sequence in $\R$. 
\begin{enumerate}
\item
If
$U$ and $W$ belong to $\nabla^\desc_\eps(\R)$ (resp.,
$\nabla_\eps(\R)$) 
so does $V$.
\item
If
$U$ and $V$ belong to $\nabla^\desc_\eps(\R)$ (resp.,
$\nabla_\eps(\R)$) 
so does $W$.
\end{enumerate}
\end{corollary}

The following is the upper finite analog of
Theorem~\ref{gf}; we have dropped part (6) since the proof of that
required objects of $\R\d$ to have finite length.

\begin{theorem}[Truncation to lower sets]\label{gf9}
Assume that $\R$ is an upper finite $\eps$-stratified category.
Suppose that $\Lambda\d$ is a lower set in $\Lambda$.
Let
$\B\d := \rho^{-1}(\Lambda\d)$
and
 $i:\R\d\rightarrow \R$
be the corresponding Serre
subcategory of $\R$ with the induced stratification.
Then $\R\d$ is an upper finite $\eps$-stratified
category.
Moreover:
\begin{enumerate}
\item
The distinguished
objects in $\R\d$ satisfy
$L\d(b) \cong L(b)$,
$P\d(b) \cong i^* P(b)$,
$I\d(b) \cong i^! I(b)$,
$\Delta\d(b) \cong \Delta(b)$,$\bar\Delta\d(b)\cong\bar\Delta(b)$,
$\nabla\d(b)\cong\nabla(b)$ and $\bar\nabla\d(b)\cong\bar
\nabla(b)$
for $b \in \B\d$.
\item 
$i^*$ sends short exact sequences of objects in $\Delta_\eps(\R)$ to short exact sequences,
$i^* \Delta(b) \cong \Delta\d(b)$ and $i^* \bar\Delta(b) \cong \bar\Delta\d(b)$
for $b \in \B\d$,
and $i^*\Delta(b) = i^*\bar\Delta(b) = 0$ for $b \notin \B\d$.
\item
$\Ext^n_{\R}(V, i W) \cong \Ext^n_{\R\d}(i^* V, W)$
for $V \in \Delta_\eps(\R)$, $W \in \R\d$
and all $n \geq 0$.
\item 
$i^!$ sends short exact sequences of objects in $\nabla_\eps(\R)$ to short exact sequences, 
$i^! \nabla(b) \cong \nabla\d(b)$ and $i^! \bar\nabla(b) \cong \bar\nabla\d(b)$
for $b \in \B\d$,
and $i^!\nabla(b) = i^!\bar\nabla(b) = 0$ for $b \notin \B\d$.
\item
$\Ext^n_{\R}(i V, W) \cong \Ext^n_{\R\d}(V, i^! W)$
for $V \in \R\d, W \in \nabla_\eps(\R)$
and all $n \geq 0$.
\end{enumerate}
\end{theorem}

\begin{proof}
This follows from Lemma~\ref{sub} and the dual statement.
\end{proof}

Next is the upper finite analog of Theorem~\ref{gf3}.

\begin{theorem}[Truncation to upper sets]\label{gf6}
Assume that $\R$ is an upper finite $\eps$-stratified category.
Suppose that $\Lambda\u$ is an upper set in $\Lambda$.
Let $\B\u := \rho^{-1}(\Lambda\u)$ and $j:\R\rightarrow \R\u$
be the corresponding Serre quotient category
of $\R$ with the induced stratification.
Then $\R\u$ is itself a finite or upper finite $\eps$-stratified
category according to whether $\Lambda\u$ is finite or infinite.
Moreover:
\begin{enumerate}
\item
For $b \in \B\u$, the distinguished objects
$L\u(b)$,
$P\u(b)$,
$I\u(b)$,
$\Delta\u(b)$, $\bar\Delta\u(b)$,
$\nabla\u(b)$ and $\bar\nabla\u(b)$
 in $\R\u$ are isomorphic to the images under $j$ of the corresponding objects of
 $\R$. 
\item We have that
$
j L(b) = j \Delta(b) = j \bar\Delta(b) = j \nabla(b) = j
\bar\nabla(b) = 0$ if $b \notin \B\u$.
\item
  $\Ext^n_\R(V, j_* W) \cong \Ext^n_{\R\u}(j V, W)$ for $V \in \R,
 W \in \nabla^\desc_\eps(\R\u)$ and all $n\geq 0$.
\item 
$j_*$ sends short exact sequences of objects in $\nabla_\eps(\R\u)$ to short exact sequences, 
$j_* \nabla\u(b) \cong \nabla(b)$, $j_* \bar\nabla\u(b) \cong
\bar\nabla(b)$
and $j_* I\u(b) \cong I(b)$
for $b \in \B\u$.
\item
  $\Ext^n_\R(j_! V, W) \cong \Ext^n_{\R\u}(V, jW)$ for $V \in
  \Delta^\asc_\eps(\R\u)$, $W \in \R$ and all $n\geq 0$.
\item 
$j_!$ sends short exact sequences of objects in $\Delta_\eps(\R\u)$ to short exact sequences,
$j_! \Delta\u(b) \cong \Delta(b)$, $j_! \bar\Delta\u(b) \cong
\bar\Delta(b)$
and $j_! P\u(b) = P(b)$
for $b \in \B\u$.
\end{enumerate}
\end{theorem}

\begin{proof}
If $\Lambda\u$ is finite, this is
proved in just the same way as Theorem~\ref{gf3}.
Assume instead that $\Lambda\u$ is infinite.
Then the same arguments prove (1) and (2), but the proofs of the remaining parts
need 
some slight modifications. It suffices to prove (3) and (4), since (5) and
(6) are the same results for $\R^\op$.

For (3), the argument from the proof of Lemma~\ref{sup}(3) reduces to
checking that $j$ sends projectives to objects that are acyclic for
$\Hom_{\R\u}(?, W)$. To see this, it suffices to show that
$\Ext^n_{\R\u}(j P(b), W) = 0$ for $n \geq 1$ and $b \in \B$, which
follows from Lemma~\ref{dry}(1).

Finally, for (4), the argument from the proof of Theorem~\ref{gf3}(4)
cannot be used since it
depends on $\R\u$ being essentially finite Abelian. So we provide an alternate argument.
Take a short exact sequence 
$0
\rightarrow U \rightarrow V \rightarrow W \rightarrow 0$ in
$\nabla_\eps(\R\u)$.
Applying $j_*$, we get 
$0 \rightarrow j_* U \rightarrow j_* V \rightarrow j_* W$,
and just need to show that the final morphism here is an epimorphism.
This follows because,
by (3) and Theorem~\ref{gf7}, $j_* U$, $j_* V$ and $j_* W$ all have
$\nabla_\eps$-flags such that $(j_* V:\nabla_\eps(b)) = (j_*
U:\nabla_\eps(b)) + (j_* W : \nabla_\eps(b))$ for all $b \in
\B$.
\end{proof}

The reader should have no difficulty in transporting Lemma~\ref{rain}
and Corollary~\ref{sun} to
the upper finite setting. Also, Lemma~\ref{strata} remains valid when ``finite
$\eps$-stratified category"
is replaced by ``upper finite $\eps$-stratified category". To see this, we just note that
the argument by downwards induction on the partial order explained in the proof 
works just as well when $\Lambda$ is upper finite rather than finite.
The following is the upper finite analog of Corollary~\ref{inparticulartiltings}.

\begin{lemma}\label{inparticulartiltingsupper}
If $\R$ is an upper finite $+$-stratified (resp., $-$-stratified) category then all $V \in \Delta(\R)$ (resp., $V \in \nabla(\R)$)
are of finite projective (resp., injective) dimension.
\end{lemma}

\begin{proof}
This follows from the upper finite analog of Lemma~\ref{strata}.
\end{proof}

\subsection{\boldmath Shared lemmas for $\S\S$\ref{ssn}--\ref{sst}}\label{ssp}
In this subsection, we prove a series of lemmas needed in both
$\S$\ref{ssn} and in $\S$\ref{sst}. 
Let $\R$ be an Abelian category equipped with a
stratification $(\B,L,\rho,\Lambda,\leq)$ which is {\em either}
essentially finite ($\S$\ref{ssn}) {\em or} upper finite ($\S$\ref{sst}).
Also let
$\eps:\Lambda\rightarrow \{\pm\}$ be a sign function.
We assume throughout the subsection that the property
($\widehat{P\Delta}_\eps$) from $\S$\ref{ssn} holds.

\begin{lemma}\label{against}
We have that
$\Ext^1_{\R}(\Delta_\eps(b), \Delta_\eps(c)) = 0$
for $b,c \in \B$ such that $\rho(b)\not\leq\rho(c)$.
\end{lemma}

\begin{proof}
Using the projective objects $P_b$ given 
by the assumed property ($\widehat{P\Delta}_\eps$), 
we can construct 
the first terms of a projective resolution of $\Delta_\eps(b)$
in the form 
\begin{equation}\label{billy}
Q \longrightarrow \bigoplus_{\substack{a \in \B\\\rho(a) \geq
  \rho(b)}} P_a^{\oplus n_a} \longrightarrow P_b \longrightarrow
\Delta_\eps(b)\longrightarrow 0
\end{equation} 
for some $n_a \geq 0$.
Now apply $\Hom_\R(?,\Delta_\eps(c))$ to get that
$\Ext^1_{\R}(\Delta_\eps(b), \Delta_\eps(c))$
is the homology of the complex
$$
\Hom_\R(
P_b,\Delta_\eps(c)) \longrightarrow \Hom_\R\Big(\bigoplus_{\substack{a \in \B\\\rho(a) \geq
  \rho(b)}} P_a^{\oplus n_a}, \Delta_\eps(c)\Big)
\longrightarrow 
\Hom_\R(Q, \Delta_\eps(c)).
$$
The middle term of this already vanishes as 
$[\Delta_\eps(c):L(a)] \neq 0\Rightarrow \rho(a) \leq \rho(c)$.
\end{proof}

\begin{lemma}\label{sub}
Let $\Lambda\d$ be a lower set in $\Lambda$
and $\B\d := \rho^{-1}(\Lambda\d)$.
Let $i:\R\d\rightarrow \R$
be the corresponding Serre
subcategory of $\R$ equipped with the induced stratification.
\begin{enumerate}
\item
The standard, proper
standard
and indecomposable projective
objects of $\R\d$ are the objects
$\Delta(b)$, $\bar\Delta(b)$ and 
$i^* P(b)$ for $b \in \B\d$.
\item The object $i^* P_b$ is zero unless $b \in \B\d$, in which case
  it is a projective object admitting a
  $\Delta_\eps$-flag with top section $\Delta_\eps(b)$ and other
  sections of the form
$\Delta_\eps(c)$ for $c \in \B\d$ with $\rho(c) \geq \rho(b)$.
In particular, this shows that $(\widehat{P\Delta}_\eps)$ holds in $\R\d$.
\item
$(\der{L}_n i^*) V = 0$ for $V \in \Delta_\eps(\R)$ and $n \geq 1$.
\item
$
\Ext^n_{\R}(V, i W) \cong \Ext^n_{\R\d}(i^* V, W)
$
 for $V \in \Delta_\eps(\R)$, $W \in \R\d$
and $n \geq 0$.
\end{enumerate}
\end{lemma}

\begin{proof}
(1) 
For projectives, this follows from the usual adjunction properties.
This also shows that $i^* P_b$ is projective, as needed for (2).
For standard and proper standard objects, just note that the standardization functors for $\R\d$ are some of the
ones for $\R$. 

\vspace{1.5mm}\noindent
(2) 
Consider a $\Delta_\eps$-flag of $P_b$.
Using Lemma~\ref{against}, we can rearrange this filtration if necessary
so that all of the sections 
$\Delta_\eps(c)$ with $c \in \B\d$ appear above
the sections $\Delta_\eps(d)$ with $d \in
\B\setminus\B\d$.
So there exists a short exact sequence
$0 \rightarrow K \rightarrow P_b \rightarrow Q \rightarrow 0$
in which $Q$ has a finite filtration with 
sections $\Delta_\eps(c)$ for $c \in \B\d$ with $\rho(c)
\geq \rho(b)$,
and $K$ has a finite filtration with sections $\Delta_\eps(c)$ for
$c \in \B \setminus \B\d$.
It follows easily that $i^* P_b$ is isomorphic to $Q$, so it has the appropriate
filtration.

\vspace{1.5mm}\noindent
(3)
It suffices to show that
$(\der{L}_n i^*) \Delta_\eps(b) = 0$
for all $b \in \B$ and $n > 0$.
Take a short exact sequence $0 \rightarrow K \rightarrow P_b \rightarrow
\Delta_\eps(b) \rightarrow 0$ such that
$K$ has a $\Delta_\eps$-flag with sections
$\Delta_\eps(c)$ for $c$ with $\rho(c) \geq \rho(b)$.
Applying $i^*$, we obtain the long exact sequence
$$
0 \longrightarrow (\mathbb{L}_1 i^*) \Delta_\eps(b) \longrightarrow i^* K \longrightarrow
i^* P_b \longrightarrow i^*
\Delta_\eps(b)\longrightarrow 0
$$
and isomorphisms
$(\der{L}_{n+1} i^*) \Delta_\eps(b) \cong (\der{L}_n i^*) K$ for $n > 0$.
We claim that $(\der{L}_1 i^*) \Delta_\eps(b) = 0$.
We use Lemma~\ref{against}
to order the $\Delta_\eps$-flag of $K$ so that it yields
a short exact sequence
$0 \rightarrow L \rightarrow K \rightarrow Q \rightarrow 0$
in which $Q$ has a $\Delta_\eps$-flag with sections
$\Delta_\eps(c)$ for $c \in \B\d$, 
and $L$ has a $\Delta_\eps$-flag with sections $\Delta_\eps(c)$ for $c \in \B\setminus\B\d$.
It follows that $i^* K = Q$ and there is a short exact
sequence
$0 \rightarrow i^* K \rightarrow i^* P_b \rightarrow
\Delta_\eps(b) \rightarrow 0$.
Comparing with the long exact sequence, we deduce that $(\der{L}_1 i^*)\Delta_\eps(b) = 0$.
Finally some degree
shifting using the isomorphisms
$(\der{L}_{n+1} i^*) \Delta_\eps(b) \cong (\der{L}_n i^*) K$ 
gives that $(\der{L}_n i^*) \Delta_\eps(b) = 0$ for $n > 1$ too.

\vspace{1.5mm}\noindent
(4)
By the adjunction, we have that $
\Hom_{\R}(?, i W) \cong \Hom_{\R\d}(?, W) \circ i^*,$
i.e., the result holds when $n=0$.
Also $i^*$ sends projectives to projectives as it is left adjoint to
an exact functor.
Now the result for $n > 0$ follows by a standard Grothendieck spectral
sequence argument; the spectral sequence degenerates due to (3).
\end{proof} 

\begin{lemma}\label{maxcase}
Suppose that $\lambda \in\Lambda$ is maximal and $b \in \B_\lambda$. Then $P(b) \cong \Delta(b)$
and $I(b) \cong \nabla(b)$.
\end{lemma}

\begin{proof}
Lemma~\ref{charac} shows that 
$\Delta(b) \cong i^*_{\leq\lambda} P(b)$
and
$\nabla(b) \cong i^!_{\leq\lambda} I(b)$.

To complete the proof for $P(b)$, it remains to observe that
$P(b)$ belongs to $\R_{\leq\lambda}$, so 
$i^*_{\leq\lambda} P(b)=P(b)$.
This follows from $\widehat{P\Delta}_\eps$:
the object $P_b$ belongs to $\R_{\leq\lambda}$
due to the maximality of $\lambda$ and $P(b)$ is a summand of it.

The proof for $I(b)$ needs a different approach.
From $\nabla(b) \cong i^!_{\leq\lambda} I(b)$, we deduce that there is
a short exact sequence
$0 \rightarrow
\nabla(b) \rightarrow I(b) \rightarrow Q \rightarrow 0$
with $i^!_{\leq\lambda} Q = 0$, and
we must show that $Q = 0$.
Take $a \in \B$ and apply $\Hom_\R(\Delta_\eps(a),?)$ to this short exact sequence to get
an exact sequence
\begin{equation}\label{Ta1}
\Hom_\R(\Delta_\eps(a), I(b))\longrightarrow \Hom_\R(\Delta_\eps(a), Q)
\longrightarrow 0
\end{equation}
and isomorphisms
\begin{equation}\label{Ta2}
\Ext^{n+1}_\R(\Delta_\eps(a), \nabla(b)) \cong \Ext^n_\R(\Delta_\eps(a), Q)
\end{equation}
for $n \geq 1$.
If $\rho(a) = \lambda$ then
$\Hom_\R(\Delta_\eps(a), Q) = 0$
because $i^!_{\leq\lambda} Q = 0$.
If $\rho(a) \neq \lambda$, then in fact 
we have that $\rho(a) \not\geq \lambda$ by the assumed maximality of $\lambda$, so
$[\Delta_\eps(a):L(b)] = 0$. Hence, 
$\Hom_\R(\Delta_\eps(a), I(b)) = 0$, implying in view of (\ref{Ta1}) 
that $\Hom_\R(\Delta_\eps(a), Q) = 0$ again.
Thus, we have shown that $\Hom_\R(\Delta_\eps(a), Q) = 0$ for all $a \in \B$.
This implies that $\soc Q = 0$.
In the essentially finite Abelian case, this is all that is needed to deduce that $Q = 0$, completing the proof.
In the Schurian case,
we need to argue a little further
because $Q$ need not be finitely cogenerated, so can have zero socle
even when it is itself non-zero.
We have for any $a \in \B$ that
$\Ext^n_\R(\Delta_\eps(a), \nabla(b)) = 0$
for $n > 0$. This follows using Lemma~\ref{sub}(4): it shows
that
$\Ext^n_\R(\Delta_\eps(a), \nabla(b)) \cong \Ext^n_{\R_{\leq
    \lambda}}(i_{\leq\lambda}^* \Delta_\eps(a), \nabla(b))$ which is
zero as $\nabla(b)$ is injective in $\R_{\leq\lambda}$.
Combining this with (\ref{Ta2}), we get that 
$\Ext^1_\R(\Delta_\eps(a), Q) = 0$.
Now we observe that the properties
$\Hom_{\R}(\Delta_\eps(a), Q) = 0 = \Ext^1_{\R}(\Delta_\eps(a),Q)$
for all $a \in \B$ do imply that $Q$ is zero.
Indeed, we have that 
$\Hom_{\R}(P, Q) = \Ext^1_{\R}(P, Q) = 0$ for any $P \in
\R$ with a $\Delta_\eps$-flag. This follows using induction on the
length of the flag plus the long exact sequence.
Since $P_b$ has a $\Delta_\eps$-flag by the hypothesis $(\widehat{P\Delta}_\eps)$
and $P(b)$ is a summand of it,
we deduce that $\Hom_\R(P(b), Q) = 0$ for all $b \in \B$, which
certainly implies that $Q = 0$.
\end{proof}

\begin{lemma}\label{quot}
Assume that $\lambda \in \Lambda$ is maximal
and $\eps(\lambda) = +$.
For any $V \in \R_\lambda$ and $b \in \B$, we have that
$\Ext^1_{\R}(\Delta_\eps(b), j^\lambda_* V) = 0.$
\end{lemma}

\begin{proof}
If $b \in \B_\lambda$ then $\Delta_\eps(b)$ 
is projective in $\R_{\leq\lambda}$ by Lemma~\ref{maxcase}, 
so we get the $\Ext^1$-vanishing in this case.
For the remainder of the proof, suppose that $b \notin \B_\lambda$.
Let $I$ be an injective hull of $V$ in $\R_\lambda$.
Applying $j^\lambda_*$ to a short exact sequence $0 \rightarrow
V\rightarrow I \rightarrow Q \rightarrow 0$, we get an exact sequence
$0 \rightarrow j^\lambda_* V \rightarrow j^\lambda_* I \rightarrow
j^\lambda_* Q$.
By properties of adjunctions, $j^\lambda_* Q$ is finitely cogenerated
and all constituents of its socle are of the form $L(c)$ for $c \in \B_\lambda$.
The same is true for $j^\lambda_* I / j^\lambda_* V$ since it embeds
into $j^\lambda_* Q$.
We deduce that $\Hom_{\R}(\Delta_\eps(b), 
j^\lambda_* I / j^\lambda_* V) = 0$.

Now take an extension
$0 \rightarrow j^\lambda_* V  \rightarrow E  \rightarrow \Delta_\eps(b)
\rightarrow 0$.
Since $j^\lambda_* I$ is injective, we can find morphisms $f$ and $g$ making the following
diagram with exact rows commute:
$$
\begin{CD}
0 &@>>>&j^\lambda_* V&@>s>>&E&@>>>&\Delta_\eps(b)&@>>>0\\
&&&&@|&&@Vf VV&&@VV gV\\
0 &@>>>&j^\lambda_* V&@>t>>&j^\lambda_* I&@>>>&j^\lambda_* I /
j^\lambda_* V&@>>>0.
\end{CD}
$$
The previous paragraph implies that $g = 0$.
Hence, $\im f \subseteq \im t$. Thus, 
$f = t \circ \bar f$ 
for some $\bar f:E \rightarrow j^\lambda_* V$.
We deduce that $\bar f \circ s = \operatorname{id}$, i.e., the top short exact sequence splits,
proving that
$\Ext^1_{\R}(\Delta_\eps(b), j^\lambda_* V)  = 0$.
\end{proof}

\begin{lemma}\label{anotherpar}
For $b,c \in \B$ and $n \geq 0$, we have that
$\dim \Ext^n_\R(\Delta_\eps(b), \nabla_\eps(c))= \delta_{b,c}\delta_{n,0}.$
\end{lemma}

\begin{proof}
The case $n=0$ follows from (\ref{onedimensionality}), so assume that $n > 0$.
Suppose that $b \in \B_\lambda$ and $c \in \B_\mu$.
By Lemma~\ref{sub}(4), we have that
$$
\Ext^n_\R(\Delta_\eps(b), \nabla_\eps(c)) \cong \Ext^n_{\R_{\leq
    \mu}}(i_{\leq\mu}^* \Delta_\eps(b), \nabla_\eps(c)).
$$
If $\lambda\not\leq\mu$ then $i_{\leq\mu}^* \Delta_\eps(b) = 0$ and we get
the desired vanishing.
Now assume that $\lambda \leq \mu$,
when we may identify $i_{\leq\mu}^* \Delta_\eps(b) = \Delta_\eps(b)$.
If $\eps(\mu) = -$ then 
$\nabla_\eps(c) = \nabla(c)$, and the result follows since
$\nabla(c)$ is injective in $\R_{\leq\mu}$ by Lemma~\ref{charac}(2).
So we may assume also that $\eps(\mu) = +$.
If $\lambda = \mu$ then $\Delta(b)$ is
projective in $\R_{\leq\mu}$ by the same lemma, so again we are done.
Finally, we are reduced to $\lambda < \mu$ and $\eps(\mu) = +$, and
need to show that
$
\Ext^n_{\R_{\leq\mu}}(\Delta_\eps(b), \bar\nabla(c))  = 0
$
for $n > 0$.
If $n=1$, we get the desired conclusion from Lemma~\ref{quot} applied
in the subcategory $\R_{\leq\mu}$ (allowed due to Lemma~\ref{sub}(2)).
Then for $n \geq 2$ we use a degree shifting argument:
let $P := i_{\leq\mu}^* P_b$. 
By Lemma~\ref{sub}(2), $P$ is projective in $\R_{\leq\mu}$, and
there is a short exact sequence $0 \rightarrow K \rightarrow P
\rightarrow \Delta_\eps(b) \rightarrow 0$ such that
$K$ has a $\Delta_\eps$-flag with sections $\Delta_\eps(a)$ for
$a \in \B_{\leq\mu}$.
Applying $\Hom_{\R_{\leq\mu}}(?, \bar\nabla(c))$ we obtain
$\Ext^n_{\R_{\leq\mu}}(\Delta_\eps(b), \bar\nabla(c)) \cong \Ext^{n-1}_{\R_{\leq\mu}}(K,
\bar\nabla(c))$, which is zero by induction.
\end{proof}

\begin{lemma}\label{sup}
Let $\Lambda\u$ be an upper set in $\Lambda$
and $\B\u := \rho^{-1}(\Lambda\u)$.
Let $j:\R\rightarrow \R\u$
be the corresponding Serre
quotient category of $\R$ equipped with the induced stratification.
\begin{enumerate}
\item
For $b \in \B\u$, 
the objects
$P\u(b)$,
$I\u(b)$,
$\Delta\u(b)$, $\bar\Delta\u(b)$,
$\nabla\u(b)$ and $\bar\nabla\u(b)$
 in $\R\u$ are
the images under $j$ of the corresponding objects of
 $\R$.
Moreover, we have that $j_! \Delta\u(b) \cong \Delta(b)$,
$j_!\bar\Delta\u(b) \cong \bar\Delta(b)$,
$j_! P\u(b) \cong P(b)$
and
$j_* \nabla\u(b) \cong \nabla(b)$,
$j_*\bar\nabla\u(b) \cong \bar\nabla(b)$,
$j_* I\u(b) \cong I(b)$.
\item
For any $b \in \B$, the object $j P_b$ has a
$\Delta_\eps$-flag with top section $\Delta\u_\eps(b)$ and other
sections of the form $\Delta\u_\eps(c)$ 
for $c \in \B\u$ with $\rho(c)
\geq \rho(b)$.
In particular, this show that
($\widehat{P\Delta}_\eps$) holds in $\R\u$.
\item
$\Ext^n_{\R}(V, j_* W) \cong \Ext^n_{\R\u}(jV, W)$
for $V \in \R$, $W \in \nabla_\eps(\R\u)$
and $n \geq 0$.
\end{enumerate}
\end{lemma}

\begin{proof}
(1)
By Lemma~\ref{guns},
$P\u(b) = j P(b)$ for each $b
\in \B\u$.
Now take $b \in \B_\lambda$ for $\lambda \in \Lambda\u$.
Let $j^\lambda:\R_{\leq\lambda} \rightarrow \R_\lambda$ be the
quotient functor as usual, and denote the analogous functor for $\R\u$
by $k^\lambda:\R\u_{\leq\lambda} \rightarrow \R\u_\lambda$.
The universal property of quotient category gives us an exact functor
$\bar\jmath:\R_\lambda \rightarrow \R\u_\lambda$ making the diagram
$$
\begin{CD}
\R_{\leq\lambda}&@>j>>&\R\u_{\leq\lambda}\\
@Vj^\lambda VV&&@VVk^\lambda V\\
\R_\lambda&@>>\bar\jmath>&\R\u_\lambda
\end{CD}
$$
commute. In fact, $\bar\jmath$ is an equivalence of categories
because it sends the indecomposable projective $j^\lambda P(b)$ in $\R_\lambda$ to the
indecomposable projective $k^\lambda P\u(b)$ in $\R\u_\lambda$ for
each $b \in \B_\lambda$.
We deduce that there is an isomorphism of functors
$j_! \circ k^\lambda_! \circ \bar\jmath \cong j^\lambda_!$.
Applying this to $P_\lambda(b)$ and to $L_\lambda(b)$ gives that
$j_! \Delta\u(b) \cong \Delta(b)$ and $j_! \bar\Delta\u(b) \cong
\bar\Delta(b)$. Also by adjunction properties we have that $j_! P\u(b)
\cong P(b)$.
Similarly, applying the isomorphism
$j_* \circ k^\lambda_* \circ \bar\jmath \cong j^\lambda_*$ to $I_\lambda(b)$ and $L_\lambda(b)$
gives that
$j_* \nabla\u(b) \cong \nabla(b)$ and $j_* \bar\nabla\u(b) \cong
\bar\nabla(b)$. Also by adjunction properties we have that $j_* I\u(b)
\cong I(b)$.
It just remains to apply $j$ to the isomorphisms constructed thus far
and use 
$j\circ j_* \cong \Id_{\R\u}  \cong  j \circ j_!$.

\vspace{1.5mm}\noindent
(2) 
This follows from (1) and the exactness of $j$, using also that 
$j \Delta_\eps(b) = 0$ if $b \notin \B\u$.

\vspace{1.5mm}\noindent
(3)  The adjunction gives an isomorphism
$\Hom_{\R}(?, j_* W) \cong \Hom_{\R\u}(?, W) \circ j.$
This proves the result when $n=0$.
For $n > 0$, the functor $j$ is exact. In order to invoke the usual 
degenerate Grothendieck spectral sequence argument,
all that
remains is to check that $j$ sends projectives to objects that are
acyclic for $\Hom_{\R\u}(?, W)$.
By (2),
the functor $j$ sends projectives in $\R$ to objects with a
$\Delta_\eps$-flag.
It remains to note
that $\Ext^1_{\R\u}(X,W) = 0$ for $X \in \Delta_\eps(\R\u), W \in
\nabla_\eps(\R\u)$. This follows from
the analog of Lemma~\ref{anotherpar} for $\R\u$, which is valid due to (2).
\end{proof}

\subsection{\boldmath Lower finite $\eps$-stratified categories}\label{lfsc}
In this subsection, $\R$ is a locally finite Abelian category equipped a
lower finite stratification $(\B,L,\rho,\Lambda,\leq)$ and
$\eps:\Lambda\rightarrow \{\pm\}$ denotes a sign function.
For $b\in \B$, we use the notation $I(b)$ to denote
an injective hull of $L(b)$ in $\Ind(\R)$.
 
\begin{definition}\label{newlfd}
  Let $(\B,L,\rho,\Lambda,\leq)$ be a lower finite stratification of
  the locally finite Abelian category
  $\R$.
  For a finite lower set $\Lambda\d$ in $\Lambda$, let $\B\d :=
  \rho^{-1}(\Lambda\d)$
  and
  $\R\d$ be corresponding
  Serre subcategory of $\R$.
  We say that $\R$ is a {\em lower finite $\eps$-stratified category} (resp.,
  {\em lower finite fully stratified category}, resp., {\em lower finite
    $\eps$-highest weight category}, resp., {\em lower
    finite fibered highest
    weight category}, resp., {\em lower finite highest weight category})
  if $\R\d$ with its naturally induced stratification
  is a
  finite $\eps$-stratified category (resp., finite fully stratified category, resp., finite
  $\eps$-highest
  weight category, resp., finite fibered highest weight category, resp., finite highest weight category)
    for every finite lower set $\Lambda\d \subseteq \Lambda$.
  \end{definition}
  
 \begin{remark}
 For a simple example, let $Q$ be any quiver. 
 The category $\R$ of finite length nilpotent representations of $Q$ can be realized naturally as the category of finite-dimensional comodules over the path coalgebra of $Q$ as in \cite[(8.3)]{Simson}. 
In order for this to be a lower finite highest weight category, one must assume that there are only finitely many different paths between any two vertices. In that case, the path algebra $\k Q$ is locally finite-dimensional, and we have that
 $\R\cong \k Q \fdlmod$
 with irreducible objects labelled by the set $\Lambda$ of vertices of $Q$ in the usual way.  
We claim now that $\R$ is a lower finite highest weight category with weight poset $(\Lambda,\leq)$ 
 for any lower finite partial ordering $\leq$ on $\Lambda$.
 To see this, the Serre subcategory $\R\d$ corresponding to a finite lower set $\Lambda\d \subset \Lambda$ is $\k \Q\d\fdlmod$ where $\Q\d$ is the full subquiver $Q\d$ of $Q$ generated by $\Lambda\d$. It is well known that this is a hereditary category, hence, it is a finite highest weight category (e.g., see \cite[Th.~4.1]{Madsen}). 
 \end{remark}
 
  Let $\R$ be a lower finite $\eps$-stratified category. Since
  $\R_{\leq \lambda}$ is a finite Abelian category,
 the admissibility axiom (A) from $\S$\ref{fs} holds, so we
  can introduce the objects $\Delta(b), \bar\Delta(b), \bar\nabla(b)$
  and $\nabla(b)$ as explained there, also adopting the shorthands $\Delta_\eps(b)$ and $\nabla_\eps(b)$.
These objects are of finite length.
Note also that
Theorem~\ref{opstrat}, Lemma~\ref{rain}
and Corollary~\ref{sun}
carry over immediately
to the lower finite setting.

Now we are going to consider another sort of infinite good filtration in objects
of $\Ind(\R)$.
Usually (e.g., if $\Lambda$ is countable), it is sufficient to
restrict attention to filtrations
 given by an ascending 
 chain of subobjects $0 = V_0 < V_1 < V_2 < \cdots$ such
that $V = \sum_{n\in \N} V_n$ and $V_m / V_{m-1} \cong
\nabla_\eps(b_m)$ for some $b_m \in \B$.
Here is the general definition which avoids this restriction.

\begin{definition}\label{good}
An {\em ascending $\nabla_\eps$-flag} in an object $V \in \Ind(\R)$ is the data of
a direct system $(V_\omega)_{\omega \in \Omega}$ 
of subobjects of $V$ such that the following properties hold:
\begin{itemize}
 \item[($A\nabla 1$)]
    $V = \sum_{\omega \in \Omega} V_\omega$;
    \item[($A\nabla 2$)] each
      $V_\omega$ has a $\nabla_\eps$-flag with $\nabla_\eps(b)$
      appearing with
      multiplicity $(V_\omega:\nabla_\eps(b)) \in \N$;
      \item[($A\nabla 3$)] $(V:\nabla_\eps(b)) :=
        \sup((V_\omega,\nabla_\eps(b))\:|\:\omega \in \Omega\} < \infty$ for each $b \in \B$.
      \end{itemize}
Let $\nabla_\eps^\asc(\R)$ be the full subcategory of $\R$
consisting of all objects $V$ that possess an ascending
$\nabla_\eps$-flag.
In the special case $\eps=+$ (resp., $\eps=-$),
we call it an ascending $\bar\nabla$-flag (resp.,
$\nabla$-flag), denoting the category $\nabla_\eps(\R)$ by $\bar\nabla(\R)$ (resp., $\nabla(\R)$).
\end{definition}

The multiplicities $(V_\omega:\nabla_\eps(b))$ and
$(V:\nabla_\eps(b))$ appearing in this definition depend {\em a
  priori} on the choice of flag. In fact, they do not, so that the
notation is unambiguous:

\begin{lemma}\label{ridgemultiplicity}
  Assume $\R$ is a lower finite $\eps$-stratified category.
  For $V \in \nabla_\eps^\asc(\R)$, the multiplicity
  $(V:\nabla_\eps(b))$ of $\nabla_\eps(b)$ 
  in the ascending
  $\nabla_\eps$-flag appearing in
  Definition~\ref{good}
  is equal to $\dim \Hom_\R(\Delta_\eps(b), V)$. Hence, it is
  well-defined independent of the particular choice for this flag.
 \end{lemma}

\begin{proof}
By Theorem~\ref{gf2} applied in
  the Serre subcategory $\R\d$ associated to a finite lower set
  $\Lambda\d$ of $\Lambda$ chosen so that $V_\omega \in \R\d$,
  we have that $(V_\omega:\nabla_\eps(b)) = \dim
  \Hom_{\R}(\Delta_\eps(b), V_\omega)$.
Also
$\Hom_\R(\Delta_\eps(b), V) = \Hom_\R(\Delta_\eps(b),  
\varinjlim V_\omega) \cong
\varinjlim \Hom_{\R}(\Delta_\eps(b), V_\omega).$
We deduce that
$$
\dim \Hom_\R(\Delta_\eps(b), V) =
\max\{(V_\omega:\Delta_\eps(b))\:|\:\omega\in\Omega\},
$$
which is the definition of the multiplicity $(V:\nabla_\eps(b))$ from Definition~\ref{good}.
\end{proof}

\begin{lemma}\label{goodcnew}
Assume that $\R$ is a lower finite $\eps$-stratified category.
For $V \in \nabla^\asc_\eps(\R)$ and $b \in \B$, we have that
$\Ext^1_{\R}(\Delta_\eps(b),  V) = 0$.
\end{lemma}

\begin{proof}
If $V$ is of finite length then it belongs to the finite Abelian
category $\R\d$ associated to some finite
lower set $\Lambda\d$ of $\Lambda$, and the lemma follows from
Theorem~\ref{gf1}.
Now suppose that $V$ is not of finite length.
Let $(V_\omega)_{\omega \in \Omega}$ be an ascending $\nabla_\eps$-flag in $V$.
Take an extension $V \hookrightarrow E \twoheadrightarrow 
\Delta_\eps(b)$. We can find a subobject $E_1$ of $E$ of
finite length such that $V+E_1 = V+E$; this follows easily by
induction on the length of $\Delta_\eps(b)$ as explained at the start
of the proof of \cite[Lem.~3.8(a)]{CPS}.
Since $V \cap E_1$ is of finite length, there exists $\omega \in \Omega$ with
$V\cap E_1 \subseteq V_\omega$.
Then we have that $V \cap E_1 = V_\omega \cap E_1$ and 
$$
(V_\omega+E_1) / V_\omega \cong E_1 / V_\omega \cap E_1 = E_1 / V \cap E_1 \cong
(V+E_1) / V = (V+E) / V \cong \Delta_\eps(b).
$$
Thus, there is a short exact sequence
$
0 \rightarrow V_\omega \rightarrow V_\omega+E_1 \rightarrow \Delta_\eps(b)
\rightarrow 0.
$
The first sentence of the proof implies that $\Ext^1_\R(\Delta_\eps(b),
V_\omega) = 0$, hence, this splits. Thus, we can find a subobject $E_2
\cong \Delta_\eps(b)$ of
$V_\omega+E_1$
such that $V_\omega+E_1 = V_\omega \oplus E_2$.
Then $V+E = V+E_1 = V+V_\omega+E_1 = V+V_\omega+E_2 = V+E_2 = V\oplus
E_2,
$
and our original short exact sequence splits too.
\end{proof}

\begin{corollary}\label{ridgeinc}
  Let $i:\R\d\rightarrow\R$ be the inclusion of the Serre subcategory
  of $\R$ associated to a finite lower set $\Lambda\d$ of $\Lambda$
  and $i^!$ be its right adjoint.
  For $V \in \nabla_\eps^\asc(\R)$, we have that $i^! V \in \nabla_\eps(\R\d)$.
\end{corollary}

\begin{proof}
Take a short exact sequence
$0 \rightarrow i^! V \rightarrow V \rightarrow Q \rightarrow 0$.
Note that
$$
\Hom_{\R\d}(\Delta_\eps(b), i^! V) \cong
\Hom_{\R}(\Delta_\eps(b), V)
$$
is finite-dimensional
for each $b \in \B\d$.
Since
$\R\d$ is finite Abelian, 
it follows that $i^! V \in
\R\d$ (rather than $\Ind(\R\d)$). Moreover, 
$\Hom_{\R}(\Delta_\eps(b), Q) = 0$
for $b \in \B\d$.
So, on applying $\Hom_{\R}(\Delta_\eps(b),?)$ and 
considering the long exact sequence using Lemma~\ref{goodcnew},
we get that
$
\Ext^1_{\R\d}(\Delta_\eps(b), i^! V) =
\Ext^1_\R(\Delta_\eps(b), i^! V) = 0
$ 
for all $b \in \B\d$.
Thus, by Theorem~\ref{gf1}, we have that $i^! V \in
\nabla_\eps(\R\d)$.
\end{proof}

The following homological criterion for ascending 
$\nabla_\eps$-flags
generalizes Theorem~\ref{gf1}.

\begin{theorem}[Homological criterion for ascending $\nabla_\eps$-flags]\label{thethmnew}
Assume that $\R$ is a lower finite $\eps$-stratified category.
For $V \in\Ind(\R)$,
the following are equivalent:
\begin{itemize}
\item[(i)] $V\in \nabla_\eps^\asc(\R)$;
 \item[(ii)] $\Ext^1_\R(\Delta_\eps(b), V) = 0$
and $\dim  \Hom_\R(\Delta_\eps(b), V) < \infty$ for all $b \in \B$;
\item[(iii)]
$\Ext^n_\R(\Delta_\eps(b), V) = 0$ and
$\dim  \Hom_\R(\Delta_\eps(b), V) < \infty$ for all $b \in \B$ and $n \geq
  1$.
\end{itemize}
Assuming these properties, we have that $V \in \nabla_\eps(\R)$ if
and only if $V \in \R$.
\end{theorem}

\begin{proof}
(ii)$\Rightarrow$(i):
Let $\Omega$ be the directed set consisting of all finite lower sets
in $\Lambda$.
Take $\omega \in \Omega$. It is a finite lower set $\Lambda\d \subseteq
\Lambda$, so we have associated the corresponding finite $\eps$-stratified
subcategory $\R\d$.
Letting $i:\R\d \rightarrow \R$ be the
inclusion, we set $V_\omega := i^! V$.
By Corollary~\ref{ridgeinc}, we have that $V_\omega \in
\nabla_\eps(\R)$.
So the subobject $V' := \sum_{\omega \in \Omega} V_\omega$ of $V$ has
an ascending $\nabla_\eps$-flag.

Now we complete the proof by showing
that $V = V'$.
Applying $\Hom_{\R}(\Delta_\eps(b),?)$ to the short exact
sequence $0 \rightarrow V' \rightarrow V \rightarrow V/V' \rightarrow
0$ using Lemma~\ref{goodcnew}, we get a short exact sequence
$$
0 \longrightarrow \Hom_{\R}(\Delta_\eps(b), V') \longrightarrow
\Hom_{\R}(\Delta_\eps(b), V) 
\longrightarrow \Hom_{\R}(\Delta_\eps(b), V / V') \longrightarrow 0
$$
for every $b \in \B$. But any homomorphism $\Delta_\eps(b) \rightarrow
V$ has image contained in $V_\omega$ for sufficiently large $\omega$,
hence,  also in $V'$. Thus, the first morphism in this short exact sequence
is an isomorphism, and $\Hom_{\R}(\Delta(b), V / V') = 0$ for all $b
\in \B$. This implies that $V / V' = 0$ as required.

\vspace{1.5mm}\noindent
 (i)$\Rightarrow$(ii):
This follows by Lemmas~\ref{ridgemultiplicity} and \ref{goodcnew}.

\vspace{1.5mm}\noindent
(iii)$\Rightarrow$(ii):
Trivial.

\vspace{1.5mm}\noindent
 (i)$\Rightarrow$(iii):
 This follows from Lemma~\ref{ridgemultiplicity} and Theorem~\ref{salsanewer}(4). Since this is a forward
 reference, we should note that the proof of Theorem~\ref{salsanewer}
 only depends on (i)$\Leftrightarrow$(ii) from the present
 theorem.
 \end{proof}

\begin{corollary}\label{flowy}
In a lower finite $\eps$-stratified category, each indecomposable injective object $I(b)$
belongs to $\nabla_\eps^\asc(\R)$ and 
$(I(b):\nabla_\eps(c)) = [\Delta_\eps(c):L(b)]$ for each $b,c \in \B$.
\end{corollary}

\begin{proof}
The first part follows from the implication (ii)$\Rightarrow$(i) in the theorem.
For the second part, 
we get from Lemma~\ref{ridgemultiplicity}
that $(I(b):\nabla_\eps(c)) = \dim \Hom_\R(\Delta_\eps(c):L(b))$.
\end{proof}

\begin{corollary}\label{fli}
Let $0 \rightarrow U \rightarrow V \rightarrow W \rightarrow 0$ be a
short exact sequence in a lower finite $\eps$-stratified category. If $U, V \in
\nabla_\eps^\asc(\R)$
then $W \in \nabla_\eps^\asc(\R)$ too. Moreover
$$
(V:\nabla_\eps(b))  = (U:\nabla_\eps(b))+(W: \nabla_\eps(b)).
$$
\end{corollary}

The following is the lower finite counterpart of Theorem~\ref{gf}.

\begin{theorem}[Truncation to lower sets]\label{salsanewer}
Suppose $\R$ is a lower finite $\eps$-stratified category.
Let $\Lambda\d$ be a lower set, $\B\d := \rho^{-1}(\Lambda\d)$,
and $i:\R\d\rightarrow \R$ be the corresponding Serre subcategory of
$\R$ with the induced stratification.
Then $\R\d$ is a finite or lower finite $\eps$-stratified category
according to whether $\Lambda\d$ is finite or infinite.
Moreover:
\begin{enumerate}
\item
 The distinguished
objects of $\R\d$ are
$L\d(b) \cong L(b)$,
$I\d(b) \cong i^! I(b)$,
$\Delta\d(b) \cong \Delta(b)$,
$\bar\Delta\d(b)\cong\bar\Delta(b)$,
$\nabla\d(b)\cong\nabla(b)$ and $\bar\nabla\u(b)\cong\bar \nabla(b)$
for $b\in \B\d$.
\item
$(\der{R}^n i^!) V = 0$ for $n \geq 1$ assuming either that $V \in
\nabla^\asc_\eps(\R)$ or that $V \in \R\d$.
\item $i^!$ takes short exact sequences of objects in
$\nabla^\asc_\eps(\R)$
to short exact sequences of objects in $\nabla^\asc_\eps(\R\d)$,
with 
$i^! \nabla(b) \cong \nabla\d(b)$ and $i^! \bar\nabla(b) \cong \bar\nabla\d(b)$
for $b \in \B\d$
and $i^!\nabla(b) = i^!\bar\nabla(b) = 0$ for $b \notin \B\d$.
\item
$\Ext^n_{\R}(i V, W) \cong \Ext^n_{\R\d}(V, i^! W)$
for $V \in \R\d, W \in \nabla^\asc_\eps(\R)$
and all $n \geq 0$.
\item
$\Ext^n_{\R}(i V, i W) \cong \Ext^n_{\R\d}(V, W)$
for $V, W \in \R\d$
and all $n \geq 0$.
\end{enumerate}
\end{theorem}

\begin{proof}
  The fact that $\R\d$ is itself a lower finite $\eps$-stratified follows immediately
  from Definition~\ref{newlfd}.  It is finite if and only if $\B\d$ is
  finite. The identification of objects as in (1) is straightforward. In particular, the objects $\nabla_\eps(b)$ in $\R\d$
are just the same as the ones in $\R$ indexed by $b \in \B\d$, while
the indecomposable injectives in $\Ind(\R\d)$ are the objects $i^! I(b)$ for
$b \in \B\d$.

To prove (2), assume first that $V \in \nabla_\eps^\asc(\R)$.
Let $I$ be an injective hull of $\soc V$ in $\Ind(\R)$.
Note that
$I$ is of the form $\bigoplus_{a\in \B} I(a)^{\oplus
  n_a}$ for $$
0 \leq n_a \leq
\dim \Hom_\R(\Delta_\eps(a), V) = (V:\nabla_\eps(a)) < \infty.
$$
Hence, for $b \in \B_{\leq \lambda}$, we have that
$$
\dim \Hom_\R(\Delta_\eps(b), I) = \sum_{a \in \B_{\leq \lambda}} n_a
[\Delta_\eps(b):L(a)] < \infty$$
too. We deduce that $I \in \nabla^\asc_\eps(\R)$
using the implication
(ii)$\Rightarrow$(i) of Theorem~\ref{thethmnew}.
Now consider the short exact sequence $0 \rightarrow V \rightarrow
I \rightarrow Q \rightarrow 0$.
By Corollary~\ref{fli}, we have that $Q \in \nabla^\asc_\eps(\R)$ too.
Applying the left exact functor $i^!$ and considering the long
exact sequence, we see that to prove that $(\der{R}^1 i^!) V = 0$ it suffices to
show that the canonical map $i^! I \rightarrow i^!
Q$ is an epimorphism.
Once that has been proved we can use degree shifting
to establish the desired vanishing for all higher $n$; it is
important for the induction step that we have already established that
$Q \in \nabla_\eps^\asc(\R)$ just like $V$.

To prove the surjectivity, 
look at 
$0 \rightarrow i^! I / i^! V \rightarrow i^! Q
\rightarrow C \rightarrow 0$. Both $i^! I$ and $i^! V$ have
$\nabla_\eps$-flags by Lemma~\ref{ridgeinc}.
Hence, so does $i^! I / i^! V$, and on applying
$\Hom_{\R\d}(\Delta_\eps(b),?)$ for $b \in \B\d$, we get a short exact sequence
$$
0 \longrightarrow \Hom_{\R\d}(\Delta_\eps(b), i^! I / i^! V)
\longrightarrow \Hom_{\R\d}(\Delta_\eps(b), i^! Q)
\longrightarrow 
\Hom_{\R\d}(\Delta_\eps(b), C)
\longrightarrow 0.
$$
The first space here has dimension $$
(i^!I:\nabla_\eps(b)) - (i^! V:
\nabla_\eps(b))
= (I:\nabla_\eps(b))-(V:\nabla_\eps(b)) = 
(Q:\nabla_\eps(b)) = (i^! Q:\nabla_\eps(b)),
$$ 
which is the dimension
of the second space. This shows that the
first morphism is an isomorphism. Hence, 
$\Hom_{\R\d}(\Delta_\eps(b), C) = 0$. This implies that $C = 0$ as required.

Finally let $V \in \R\d$. Then $V$ is of
finite length, so it suffices just to consider the case that $V =
L(b)$ for $b \in \B\d$.
Then we consider the short exact sequence $0 \rightarrow L(b) \rightarrow \nabla_\eps(b)
\rightarrow Q \rightarrow 0$.
Applying $i^!$ and using the vanishing established so far 
gives $0 \rightarrow i^! L(b) \rightarrow i^!
\nabla_\eps(b)
\rightarrow i^! Q \rightarrow (\der{R}^1 i^!) L(b) \rightarrow 0$
and isomorphisms $(\der{R}^n i^!) Q \cong (\mathbb{R}^{n+1} i^!) L(b)$ for $n \geq
1$.
But $i^!$ is the identity on $L(b), \nabla_\eps(b)$ and $Q$, so this
immediately yields
$(\der{R}^1 i^!) L(b) = 0$, and then $(\der{R}^{n} i^!) L(b) = 0$ for higher $n$ by
degree shifting.

Having proved (2), property (3) follows easily.
Finally (4)--(5) follow by the usual Grothendieck spectral sequence argument
starting from the adjunction isomorphism
$\Hom_{\R\d}(iV,?) \cong \Hom_{\R}(V,?) \circ i^!$. One just needs
(2) and the observation that
$i^!$ sends injectives to injectives.
\end{proof}

Our next result gives an
alternative characterization of lower finite
$\eps$-stratified categories.
Note for this that if $\R$ is a lower finite $\eps$-stratified category then the hypotheses of the theorem
are automatically satisfied taking $I^\asc_b := I(b)$; cf. Corollary~\ref{flowy}.

\begin{theorem}[Global characterization of lower finite $\eps$-stratified categories]\label{globalchar}
Let $\R$ be a locally finite Abelian category equipped with a lower
finite stratification $(\B,L,\rho,\Lambda,\leq)$ and
$\eps:\Lambda\rightarrow\{\pm\}$ be a sign function.
Assume for each $b \in \B$ that $L(b)$ has an injective hull in
$\R_{\leq\rho(b)}$ so that we can introduce the objects
$\nabla_\eps(b)$ in the usual way\footnote{We do not insist that $L(b)$ has 
a projective cover in $\R_{\leq\rho(b)}$ and do not need the
objects $\Delta_\eps(b)$.}.
Suppose that the following property holds:
{\rm
\begin{itemize}
\item[($\widehat{I\nabla}_\eps^\asc$)]
For every $b \in \B$, 
there exists an injective object $I_b \in
\Ind(\R)$ 
with an ascending $\nabla_\eps$-flag $(V_\omega)_{\omega
  \in \Omega}$ in the sense of Definition~\ref{good}
such that for each $\omega \in \Omega$
the given $\nabla_\eps$-flag of
$V_\omega$ has $\nabla_\eps(b)$ at the bottom and all other sections
are of the form
$\nabla_\eps(c)$ for $c\in \B$ with $\rho(c) \geq \rho(b)$.
\end{itemize}
}
\noindent
Then $\R$ is a lower finite $\eps$-stratified category.
 \end{theorem}

 \begin{proof} We must verify the condition from Definition~\ref{newlfd}.
   Let $\Lambda\d$ be a finite lower set, $\B\d :=
   \rho^{-1}(\Lambda\d)$, and $\R\d$ be the corresponding Serre 
subcategory of $\R$. This is a locally finite Abelian category
 with irreducible objects labelled by the finite
   set $\B\d$. We need to show it is a finite $\eps$-stratified
   category with respect to the induced stratification.
 
\vspace{1.5mm}
\noindent
{Step 1:} {\em
$\Ext^1_{\R}(\nabla_\eps(a), \nabla_\eps(b)) = 0$ for
$\rho(a)\not\geq\rho(b)$.}
Let $(V_\omega)_{\omega \in \Omega}$ be the given ascending
$\nabla_\eps$-flag of $I_b$.
We have that $\nabla_\eps(b) \hookrightarrow I_b$ and
$I_b / \nabla_\eps(b) = \sum_{\omega \in \Omega} (V_\omega /
\nabla_\eps(b))$.
The socle of the latter object only involves constituents $L(c)$ with $\rho(c)
\geq \rho(b)$. 
We deduce that there is an injective resolution
$0 \rightarrow \nabla_\eps(b)
\rightarrow I_b
\rightarrow J\rightarrow \cdots$ in $\Ind(\R)$ in which
$J$ is a direct sum of $I_c$ with
$\rho(c) \geq \rho(b)$. The $\Ext^1$-vanishing now follows on applying
$\Hom_{\R}(\nabla_\eps(a),?)$ to this resolution and taking homology.

\vspace{1.5mm}
\noindent
{Step 2:} {\em
 For $b \in \B\d$, the object $I\d_b := i^! I_b \in \Ind(\R\d)$ has a $\nabla_\eps$-flag 
  with $\nabla_\eps(b)$ at the bottom and other sections of the form
  $\nabla_\eps(c)$ for $c \in \B\d$ with $\rho(c) \geq
  \rho(b)$. In particular, $I\d_b$ is of finite length.}
Take $b \in \B\d$ and let $(V_\omega)_{\omega \in \Omega}$ be the given ascending
   $\nabla_\eps$-flag in  $I_b$.
     Since $\B\d$ is finite, we can choose some
   sufficiently large $\omega\in \Omega$ so that $(V:\nabla_\eps(c)) =
   (V_\omega:\nabla_\eps(c))$ for all $c \in \B\d$; these
   multiplicities are the given ones from Definition~\ref{good}.
Then we see that $i^! V_\upsilon = i^! V_\omega$ for all larger
$\upsilon$, hence, $i^! V =i^! V_\omega$.
   In view of Step 1, we can rearrange the $\nabla_\eps$-flag of
   $V_\omega$ so that the sections $\nabla_\eps(c)$ with $c \in
   \B\d$ appear below the other sections, with bottom section
   $\nabla_\eps(b)$.
   So there is a short exact sequence
   $0 \rightarrow U_\omega \rightarrow V_\omega \rightarrow
   W_\omega\rightarrow 0$ such that $U_\omega \in \nabla_\eps(\R\d)$
   and $i^! W_\omega = 0$. Then we get that $i^! V = i^! V_\omega =
   U_\omega$, which has the desired $\nabla_\eps$-flag.
 
   \vspace{1.5mm}
   \noindent
{Step 3:} {\em $\R\d$ is a finite $\eps$-stratified category with
  respect to the induced stratification.}
By adjunction properties, the object $I_b\d \in \R\d$ from Step~2 is
injective and it
has $L(b)$ in its socle. This shows that the locally finite Abelian
category $\R\d$ has enough injectives,
hence, it is a finite Abelian category by Lemma~\ref{kitchen}.
Moreover, the objects $I_b\d\:(b \in \B\d)$
satisfy the condition $(\widehat{I\nabla}_\eps)$ from $\S$\ref{ssn}, so $\R\d$ is a
finite $\eps$-stratified category according to Definition~\ref{vector}.
 \end{proof}

\begin{corollary}\label{globalcharcor}
  Let $\R$ be a locally finite Abelian category,
  $(\Lambda,\leq)$ be a lower finite poset, and $L:\Lambda\rightarrow \R$ be a function labelling a complete set of pairwise inequivalent
  irreducible objects. Assume for all $\lambda \in \Lambda$ that
  $L(\lambda)$ has an injective hull $\nabla(\lambda) \in \R_{\leq\lambda}$ such that
  $[\nabla(\lambda):L(\lambda)] = 1$.
  Suppose that the following property holds:
{\rm
  \begin{itemize}
\item[($\widehat{I\nabla}^\asc$)]
For every $\lambda \in \Lambda$
there exists an injective object $I_\lambda \in
\Ind(\R)$ 
with an 
ascending $\nabla$-flag $(V_\omega)_{\omega
  \in \Omega}$
such that 
for each $\omega \in \Omega$
the given $\nabla$-flag of
$V_\omega$ has $\nabla(\lambda)$ at the bottom and all other sections
are of the form
$\nabla(\mu)$ for $\mu\in \Lambda$ with $\mu \geq \lambda$
\end{itemize}
}
\noindent
Then $\R$ is a lower finite highest weight category.
 \end{corollary}

 \begin{proof}
   Apply the theorem taking $\B=\Lambda$ and $\rho$ to be the
   identity function, using also Lemma~\ref{allsimple}.
   \end{proof}
 
\begin{remark}\label{cpsref}
Using Corollary~\ref{globalcharcor}, it follows that $\R$ is a lower finite highest
weight category with all intervals $(\lambda,\infty]$ in the weight
poset being countable
if and only
if $\Ind(\R)$ is a highest weight category 
in the original sense of
\cite[Def.~3.1]{CPS}
with a weight poset that is lower finite. This is also mentioned in \cite{C2}.
\end{remark}

The following theorem gives a related characterization for lower
finite fully
stratified categories.
The proof is based on the well-known proof of the homological criterion for good
filtrations in the context of reductive algebra groups
from \cite[Prop.~II.4.16]{J}.
The $\Ext^2$-vanishing property
needed for this is used as one of the defining properties in
\cite[Def.~2.1]{RW}; see also \cite[Def.~3.1.2($\dagger$)]{C2}.
We know already that lower finite fully stratified categories
automatically satisfy the conditions of this theorem.

\begin{theorem}[Homological characterization of lower finite fully stratfied categories]\label{anotherchar}
  Suppose that $\R$ is a locally finite Abelian category equipped with
  a lower finite stratification $(\B,L,\rho,\Lambda,\leq)$. Suppose
that every $L(b)$ has a projective cover and an injective hull in
$\R_{\leq\rho(b)}$ so that we can introduce standard and costandard objects.
Consider the following properties:
  \begin{itemize}
  \item[(1)] $\Ext^1_\R(\bar\Delta(b), \nabla(c)) =
    \Ext^2_\R(\bar\Delta(b), \nabla(c)) = 0$ for all $b,c \in \B$.
  \item[(2)] $\Ext^1_\R(\Delta(b), \bar\nabla(c)) =
    \Ext^2_\R(\Delta(b), \bar\nabla(c)) = 0$ for al $b,c\in\B$.
  \end{itemize}
  If (1) holds then $\R$ is a lower finite
  $-$-stratified category, and 
  if (2) holds then $\R$ is a lower finite
  $+$-stratified category. Hence, if both (1) and (2) hold then $\R$ is a lower finite fully stratified
category.
\end{theorem}

\begin{proof}
  We will prove that (1) implies that $\R$ is a lower finite $-$-stratified
  category.
  The fact that (2) implies that $\R$ is $+$-stratified then follows
  from this assertion with $\R$ replaced by $\R^\op$.
Hence, if both hold then $\R$ is fully stratified thanks to
Lemma~\ref{rain}(iii).

 So now we just assume (1). Define ascending $\nabla$-flags and
  the corresponding full subcategory $\nabla^\asc(\R)$ by repeating
  the $\eps = -$ case of
  Definition~\ref{good}. We first establish two claims.
  
\vspace{1.5mm}
\noindent
Claim 1: {\em For $V \in \nabla^\asc(\R)$, we have that
  $\Ext^1_\R(\bar\Delta(b), V)= 0$ for all $b \in \B$. Moreover, the
  multiplicity $(V:\nabla(b))$ defined  from a specific choice of
  ascending $\nabla$-flag in $V$ is equal to
$\dim \Hom_\R(\bar\Delta(b), V)$.}
For any $c \in \B$, we have as always that $\dim \Hom_\R(\bar\Delta(b), \nabla(c)) = \delta_{b,c}$, and
moreover $\Ext^1_\R(\bar\Delta(b), \nabla(c)) = 0$ by property (1).
Hence, Claim 1 holds when the $\nabla$-flag is of finite length.
Then it follows for arbitrary $V \in \nabla^\asc(\R)$ by
the same arguments as used to prove Lemmas~\ref{ridgemultiplicity} and \ref{goodcnew} above,
using the special case just established in place of the references to
Theorems~\ref{gf1} and \ref{gf2} made in those proofs.

\vspace{1.5mm}
\noindent
{Claim 2:}
{\em If $V \in \Ind(\R)$ satisfies
  $\dim \Hom_\R(\bar\Delta(b), V) < \infty$ and $\Ext^1_\R(\bar\Delta(b), V) = 0$ for
  all $b \in \B$ then $V$ has an ascending $\nabla$-flag
  $(V_\omega)_{\omega \in \Omega}$.}
Let $\Omega$ be the poset of finite lower sets in $\Lambda$
ordered by containment. For $\omega = \Lambda\d\in \Omega$, define $V_\omega$ to
be the subobject $i^! V$ where $i:\R\d\rightarrow \R$ is the inclusion
of the Serre subcategory of $\R$ associated to $\B\d :=
\rho^{-1}(\Lambda\d)$.
This defines a direct system $(V_\omega)_{\omega
  \in \Omega}$ of subobjects of $V$.
We prove the claim by establishing the following:
\begin{itemize}
  \item[(a)] Each $V_\omega\:(\omega \in \Omega)$ has a finite $\nabla$-flag.
  \item[(b)] $V = \sum_{\omega \in \Omega} V_\omega$.
    \end{itemize}
To check (a), take $\omega = \Lambda\d \in \Omega$ setting $\B\d :=
\rho^{-1}(\Lambda\d)$ once again.
We show that $V_\omega$ has a finite $\nabla$-flag by induction on
$n(V) := \sum_{b \in \B\d} \dim
\Hom_\R(\bar\Delta(b), V)$. If $n(V) = 0$ then $V_\omega = 0$ and there
is nothing to do. If $n(V) > 0$, let $\lambda$ be minimal such
that
$\dim \Hom_\R(\bar\Delta(b), V) \neq 0$ for some $b \in \B_\lambda$.
Then $\Hom_\R(L(c), V) = 0$ for $c \in \B_{< \lambda}$ and
$\Hom_\R(L(b), V) \neq 0$.
By applying $\Hom_\R(?,V)$ to the short exact sequence
$0 \rightarrow K \rightarrow \bar\Delta(c) \rightarrow L(c) \rightarrow 0$, it
follows that $\Ext^1_\R(L(c), V) = 0$ for all $c \in \B_{\leq
  \lambda}$.
Then by applying $\Hom_\R(?,V)$ to the short exact sequence
$0 \rightarrow L(b) \rightarrow \nabla(b) \rightarrow Q \rightarrow
0$,
it follows that the natural map $\Hom_\R(\nabla(b), V) \rightarrow
\Hom_\R(L(b), V)$ is surjective. Since the right hand space is
non-zero and $\soc \nabla(b) = L(b)$, it follows that there is an
injective homomorphism $f: \nabla(b)
\rightarrow V$. Let $U := \im f$ and $W := V / U$.
Thus, $U \cong \nabla(b)$ and there is a short exact sequence
$0 \rightarrow U \rightarrow V \rightarrow W \rightarrow 0$.
Applying $\Hom_\R(\bar\Delta(a),?)$ and using the hypotheses
$\Ext^1_\R(\bar\Delta(a), U) = \Ext^1_\R(\bar\Delta(a), V) = \Ext^2_\R(\bar\Delta(a), U) = 0$, we
deduce that $n(W) < n(V)$ and $\Ext^1_\R(\bar\Delta(a), W) = 0$
for all $a \in \B$. Thus we can apply induction to prove that
$W_\omega$ has a finite $\nabla$-flag. Since $V_\omega = W_\omega / U$
it follows that $V_\omega$ does too, and (a) is proved.
To check (b), we let
$V' := \sum_{\omega \in \Omega} V_\omega$ and show that $V = V'$ by repeating the argument
from the proof of (ii)$\Rightarrow$(i) in Theorem~\ref{thethmnew} with
$\Delta_\eps(b)$ replaced by $\bar\Delta(b)$, using Claim 1 to get that
$\Ext^1_\R(\bar\Delta(b), V') = 0$.
Thus, we have proved Claim 2.

\vspace{1.5mm}
Now we complete the proof of the theorem.
For $b \in \B$, let $I_b := I(b)$.
Like in the proof of Corollary~\ref{flowy},
Claims 1 and 2 imply that $I_b$ has an ascending
$\nabla$-flag
  $(V_\omega)_{\omega \in \Omega}$ with $(I_b:\nabla(c)) =
  [\bar\Delta(c):L(b)]$. By passing to a subset of $\Omega$ if necessary, we
may assume 
that all $V_\omega$ are non-zero.
It follows that the condition ($\widehat{I\nabla}_-^\asc$) from
Theorem~\ref{globalchar} is satisfied, and $\R$ is a lower finite
  $-$-stratified category.
\end{proof}

\begin{corollary}\label{prey}
  Suppose that $\R$ is a locally finite Abelian category,
  $(\Lambda,\leq)$ is a lower finite poset,
  and $L:\Lambda\rightarrow \R$ is a function labelling a complete set of pairwise inequivalent
  irreducible objects.
  Assume $L(\lambda)$ has both an injective hull
  $\nabla(\lambda)$ and a projective cover $\Delta(\lambda)$ in $\R_{\leq \lambda}$.
    Suppose that the following properties hold for all $\lambda,\mu\in\Lambda$:
    \begin{itemize}
      \item[(i)] $\Hom_\R(\Delta(\lambda),\nabla(\lambda))$ is one-dimensional;
  \item[(ii)] $\Ext^1_\R(\Delta(\lambda), \nabla(\mu)) =
    \Ext^2_\R(\Delta(\lambda), \nabla(\mu)) = 0$.
\end{itemize}
Then $\R$ is a lower finite highest weight
category.
\end{corollary}

\begin{proof}
  Property (i) implies that all strata are simple;
  cf. Lemma~\ref{allsimple}. Now apply the theorem.
  \end{proof}

Corollary~\ref{prey} applies in particular to the category $\R = \Rep(G)$
for a reductive algebraic group $G$; see $\S$\ref{eg3}.
The $\Ext$-vanishing properties in the corollary are consequences of
Kempf's vanishing theorem; see \cite[Prop.~II.4.13]{J}.

\subsection{Refining stratifications in fully stratified categories}
We end the section by formulating a basic lemma 
about refinement of stratifications
in fully stratified categories in any of the settings (finite,
essentially finite, upper finite or lower finite).

\begin{definition}\label{baby}
  Let $(\B,L,\rho,\Lambda,\leq)$ be a
  stratification of an Abelian category $\R$.
 A {\em refinement} of it means a stratification
 $(\B,L,\sigma,\Gamma,\preceq)$ of $\R$ with the same
 underlying labelling function 
together with a surjective function $q:\Gamma\twoheadrightarrow\Lambda$ such that the following
properties hold:
\begin{enumerate}
\item[(R1)] $\Gamma \cap \Lambda = \varnothing$.
  \item[(R2)] $\rho = q \circ \sigma$.
\item[(R3)] For $\beta,\gamma \in \Gamma$, we have that
$\beta \preceq\gamma\Rightarrow q(\beta) \leq q(\gamma)$ and
  $q(\beta) < q(\gamma)
  \Rightarrow \beta \prec \gamma$.
  \end{enumerate}
\end{definition}

In the setup of Definition~\ref{baby}, if one of the stratifications is admissible of one of our four types
then the other one is automatically admissible of the same type.
Assuming this is the case,
take $\gamma \in \Gamma$ and set $\lambda := q(\gamma)$.
We have the stratum $\R_\lambda := \R_{\leq\lambda} /
\R_{<\lambda}$ with quotient functor $j^\lambda$ coming from the
original stratification,
and
the stratum $\R_\gamma := \R_{\preceq\gamma} / \R_{\prec\gamma}$
with quotient functor $j^\gamma$
coming from the refined stratification\footnote{The axiom (R1) is needed so that this notation is unambiguous.}.
There is also an induced finite stratification
$(\rho_\lambda,\B_\lambda,\Gamma,\preceq,L_\lambda)$ on $\R_\lambda$
defined by setting $\rho_\lambda := \rho|_{\B_\lambda}$ and
$L_\lambda(b) := j^\lambda L(b)$ for each $b \in \B_\lambda$.
We denote the stratum of this labelled by $\gamma$ by $\R_{\lambda,\gamma}$
with quotient functor $(j_\lambda)^\gamma:\R_{\lambda, \preceq\gamma} \rightarrow
\R_{\lambda, \gamma}$.
In fact, $R_{\lambda,\gamma}$ may naturally be identified with $\R_\gamma$ so that
$j^\gamma = (j_\lambda)^\gamma \circ j^\lambda|_{\R_{\preceq\gamma}}$.
Now one can denote the standard and proper objects of $\R$ for the original
stratification by 
$$
\big\{\rho\Delta(b) := j^\lambda_! P_\lambda(b)\:\big|\:\lambda \in
\Lambda, b \in \B_\lambda\big\},\quad
\big\{\rho\bar\Delta(b) := j^\lambda_! L_\lambda(b)\:\big|\:\lambda \in
\Lambda, b \in \B_\lambda\big\},
$$
and the standard and proper standard objects of $\R$ for the refined stratification by $$
\big\{\sigma\Delta(b) := j^\gamma_! P_\gamma(b)\:\big|\:\gamma \in
\Gamma, b \in \B_\gamma\big\},\quad
\big\{\sigma\bar\Delta(b) := j^\gamma_! L_\gamma(b)\:\big|\:\gamma \in
\Gamma, b \in \B_\gamma\big\}.
$$
The standard and proper standard 
objects of $\R_\lambda$ for its induced stratification
are 
\begin{align*}
\big\{\Delta_\lambda(b) &:= (j_\lambda)^\gamma_!
P_\gamma(b)\:\big|\: \textstyle b \in \bigcup_{\gamma \in q^{-1}(\lambda)}\B_\gamma\big\},&
\big\{\bar\Delta_\lambda(b) &:= (j_\lambda)^\gamma_!
L_\gamma(b)\:\big|\:\textstyle b \in \bigcup_{\gamma \in q^{-1}(\lambda)}\B_\gamma\big\},
\end{align*}
and for such $b$ we have
that 
$\sigma\Delta(b) = j^\lambda_! \Delta_\lambda(b)$,
$\sigma\bar\Delta(b) = j^\lambda_! \bar\Delta_\lambda(b)$
since $j^\gamma_! = j^\lambda_! \circ (j_\lambda)^\gamma_!$.
We deduce for all $b \in \B$
that
\begin{align}
\rho\Delta(b) &\twoheadrightarrow \sigma\Delta(b),&
\sigma\bar\Delta(b)&\twoheadrightarrow\rho\bar\Delta(b),\\\intertext{Similar notation can be introduced for the costandard objects,
 and one sees that} 
 \rho\bar\nabla(b) &\hookrightarrow \sigma\bar\nabla(b),
 &\sigma\nabla(b) 
 &\hookrightarrow \rho\nabla(b)
 \end{align}
since $j^\gamma_* = j^\lambda_* \circ (j_\lambda)^\gamma_*$.

\begin{lemma}\label{tired}
  Let $\R$ be an Abelian category
  equipped with an admissible stratification
$(\B,L,\rho,\Lambda,\leq)$. Let $(\B,L,\sigma,\Gamma,\preceq)$ be a
refinement of it in the sense of Definition~\ref{baby}.
\begin{enumerate}
\item If $\R$ is fully stratified with respect  to the original
stratification, and the strata $\R_\lambda$ are fully stratified
with respect to their induced stratifications
for all $\lambda \in \Lambda$, then
$\R$ is fully stratified with respect to the refined stratification.
\item
If $\R$ is fully stratified with respect to the refined stratification,
and the functors $j^\lambda_!, j^\lambda_*:\R_\lambda \rightarrow \R_{\leq \lambda}$ are exact for all $\lambda \in \Lambda$,
then $\R$ is fully stratified with respect to the original stratification.
\end{enumerate}
\end{lemma}

\begin{proof}
Due to the local nature of the definition of ``fully stratified" in the lower finite case, the proof reduces just to 
the finite, essentially finite and upper finite cases.
We assume we are in one of these three situations 
for the remainder of the argument.

\vspace{1.5mm}
\noindent
(1)
Note that the functors $j^\gamma_!$ and $j^\gamma_*$ are exact since they are compositions of exact functors.
In view 
of Lemma~\ref{rain}(iv), it remains to show that 
$P(b)$ has a $\sigma\Delta$-flag with $\sigma\Delta(b)$ at the top and other sections of the form $\sigma\Delta(c)$ for $c \in \B$ with $\sigma(c) \succ \sigma(b)$. To see this, let $\lambda := \rho(b)$.
As $\R$ is fully stratified
with respect to the original stratification,
$P(b)$ has a $\rho\Delta$-flag with $\rho\Delta(b)$ at the top and other sections of the form $\rho\Delta(c)$ for $c \in \B$ with $\rho(c) > \rho(b)$.
Moreover each $\rho\Delta(b)$ has a $\sigma\Delta$-flag with $\sigma\Delta(b)$ at the top and other sections of the form $\sigma\Delta(c)$ for $c \in \B_{\lambda}$ with $\sigma(c) \succ \sigma(b)$; this follows by applying the exact functor $j^\lambda_!$ to a
$\Delta_\lambda$-flag in $P_\lambda(b)$.

\vspace{1.5mm}
\noindent
(2)
To show that $\R$ is fully stratified with respect to the original stratification,
both $j^\lambda_!$ and $j^\lambda_*$ are exact by assumption,  so 
it suffices to show that each $P(b)$ has a $\rho\bar\Delta$-flag.
This follows because $P(b)$ has a $\sigma\bar\Delta$-flag and each
$\sigma\bar\Delta(b)$ has a $\rho\bar\Delta(b)$-flag; to see the latter assertion apply the exact functor
$j^\lambda_!$ to a composition series of $\bar\Delta_\lambda(b)$.
\end{proof}

\begin{corollary}\label{eve}
 Let $\R$ be fully stratified category with
  stratification  $(\B,L,\rho,\Lambda,\leq)$.
  Assume that each stratum $\R_\lambda\:(\lambda \in \Lambda)$
  is a
  highest weight category with weight poset $(\Gamma_\lambda,\preceq_\lambda)$ and labelling function $L_\lambda$.
Let $\Gamma:=\bigsqcup_{\lambda \in \Lambda} \Gamma_\lambda$,
 $\sigma:\B\stackrel{\sim}\rightarrow\Gamma$  be the bijection such
 that $j^\lambda L(b) \cong L_\lambda(\sigma(b))$ for $b \in \B_\lambda$, and
$\preceq$ be the partial order on $\Gamma$ defined by
 $\sigma(b) \preceq \sigma(c)$ if and only if 
 either $\rho(b) < \rho(c)$, or
 $\lambda:=\rho(b)=\rho(c)$ and $\sigma(b) \preceq_\lambda \sigma(c)$.
 Then $(\B,L,\sigma,\Gamma,\preceq)$ is a refinement of the original 
 stratification which makes $\R$ into a highest
 weight category.
\end{corollary}

\begin{remark}
It is also interesting to consider changing the underlying partial order on the
set $\Lambda$.
For a fully stratified category $\R$  with stratification $(\B,L,\rho,\Lambda,\leq)$,
one can always replace the given order
$\leq$ by the {\em minimal order} $\preceq$, that is, the partial
order generated
by the relation that $\lambda \prec \mu$ if
$[\nabla(b):L(c)] +[\bar\Delta[b]:L(c)] \neq 0$
for some $b
\in \B_\lambda, c\in\B_\mu$.
Then $\R$ is also fully stratified with respect to
$(\B,L,\rho,\Lambda,\preceq)$ with all the same strata, standard
objects, etc..
For highest weight categories, Coulembier \cite{C}, \cite{C2} has
made the following elegant observation: if $\R$ is a finite Abelian,
locally finite Abelian or Schurian category,
$\{L(\lambda)\:|\:\lambda \in
\Lambda\}$ is a full set of pairwise inequivalent irreducible
objects, and $\R$ possesses a contravariant autoequivalence preserving isomorphism classes of irreducible objects, then
all partial orders on
$\Lambda$ making $\R$ into a highest weight category give rise to the
same minimal order.
There are examples showing that this statement is false for essentially finite
highest weight categories.
\end{remark}

\section{Tilting modules and semi-infinite Ringel duality}

We now develop the theory of tilting objects and Ringel duality. Even
in the finite case, we are not aware of a complete exposition of these
results in the existing literature in the general $\eps$-stratified setting. 

\subsection{Tilting objects in the finite and lower finite cases}\label{sbugs}
In this subsection, $\R$ is a finite or locally finite Abelian category with a finite or
lower finite stratification $(\B,L,\rho,\Lambda,\leq)$,
and $\eps:\Lambda\rightarrow\{\pm\}$ is a fixed
sign function with respect to which $\R$ is a finite or lower finite
$\eps$-stratified category, respectively; see Definitions~\ref{vector} and \ref{newlfd}.
By an {\em $\eps$-tilting object}, we mean an object of the following
full subcategory of $\R$:
\begin{equation}\label{to}
\Tilt_\eps(\R) := \Delta_\eps(\R)
\cap \nabla_\eps(\R).
\end{equation}
The following shows that the additive subcategory $\Tilt_\eps(\R)$ of
$\R$ is Karoubian.

\begin{lemma}\label{karo}
Direct summands of $\eps$-tilting objects are $\eps$-tilting objects.
\end{lemma}

\begin{proof}
This follows easily from the homological criteria from Theorems~\ref{gf1}
and \ref{gf2}.
In the lower finite case, one needs to pass first to a finite
$\eps$-stratified subcategory containing the object in question using Theorem~\ref{salsanewer}.
\end{proof}

The next goal is to construct and classify
$\eps$-tilting objects. Our exposition of this is based roughly on 
\cite[Appendix]{Donkin}, which in turn goes back
to the work of 
Ringel \cite{R}. There are some additional complications in the
$\eps$-stratified setting.

\begin{theorem}[Classification of $\eps$-tilting objects]\label{gin}
Assume that $\R$ is a finite or lower finite $\eps$-stratified category.
For $b \in \B_\lambda$
there is an indecomposable object $T_\eps(b) \in
\Tilt_\eps(\R)$ satisfying the following properties:
\begin{itemize}
\item[(i)] $T_\eps(b)$ has a $\Delta_\eps$-flag with bottom section
  isomorphic to $\Delta_\eps(b)$;
\item[(ii)] $T_\eps(b)$ has a $\nabla_\eps$-flag with top section
  isomorphic to $\nabla_\eps(b)$;
\item[(iii)] $T_\eps(b) \in \R_{\leq\lambda}$ and $j^\lambda T_\eps(b) \cong
\left\{
\begin{array}{ll}P_\lambda(b)&\text{if $\eps(\lambda)=+$}\\
I_\lambda(b)&\text{if $\eps(\lambda) = -$}
\end{array}
\right.$.
\end{itemize}
These properties determine $T_\eps(b)$ uniquely
up to isomorphism: if $U$ is any indecomposable object of
$\Tilt_\eps(\R)$
satisfying any one of the properties (i)--(iii) then $U \cong T_\eps(b)$;
hence, it satisfies the other two properties as well.
\end{theorem}

\begin{proof}
By replacing $\R$ by the Serre subcategory associated to a
sufficiently large but finite lower set $\Lambda\d$ in $\Lambda$,
chosen so as to contain $\lambda$ and (for the uniqueness statement) all $\rho(b)$ for $b$ such that
$[T:L(b)] \neq 0$, one reduces to the case that $\R$ is a
finite $\eps$-stratified category. This reduction depends only on
Theorem~\ref{salsanewer}.
Thus, we may assume henceforth that $\Lambda$ is finite.

\vspace{1.5mm}
\noindent
{Existence:}
The main step
is to construct an indecomposable object $T_\eps(b) \in \Tilt_\eps(\R)$
such that (iii) holds. The argument for this proceeds by induction on $|\Lambda|$.
If $\lambda \in \Lambda$ is minimal, 
we set $T_\eps(b) := \Delta(b)$ if $\eps(\lambda) = +$ or
$\nabla(b)$ if $\eps(\lambda) = -$. Since $\bar\Delta(b) = L(b) =
\bar\nabla(b)$ by the minimality of $\lambda$, this has both a
$\Delta_\eps$- and a $\nabla_\eps$-flag.  It is indecomposable, and
we get (iii) from Lemma~\ref{guns}.

For the induction step, suppose that $\lambda$ is not minimal
and pick
$\mu < \lambda$ that is minimal.
Let $\Lambda\u := \Lambda\setminus\{\mu\}, \B\u :=
\rho^{-1}(\Lambda\u)$,
and $j:\R \rightarrow \R\u$ be the corresponding Serre quotient.
By induction, there is
an indecomposable object $T\u_\eps(b) \in \Tilt_\eps(\R\u)$ satisfying
the analog of (iii).
Now there are two cases according to whether $\eps(\mu) = +$ or $-$.

\vspace{1.5mm}
\noindent
{Case $\eps(\mu)=+$:} 
For any $V \in \R$, let $d_+(V) := \sum_{c \in \B_\mu} \dim
\Ext^1_{\R}(\Delta(c), V)$.
We recursively
construct $n \geq 0$ and $T_0,T_1,\dots,T_n$ so that 
$d_+(T_0) > d_+(T_1) > \cdots > d_+(T_n) = 0$ and the following
properties hold for all $m$:
\begin{enumerate}
\item[(1)]
$T_m \in \Delta_\eps(\R)$.
\item[(2)]
$j^\lambda T_m \cong 
P_\lambda(b)$ if $\eps(\lambda) = +$ or
$I_\lambda(b)$ if $\eps(\lambda) = -$.
\item[(3)]
$\Ext^1_{\R}(\Delta_\eps(a), T_m) = 0$ for all $a \in \B \setminus
\B_\mu$.
\end{enumerate}
To start with, set $T_0 := j_! T_\eps\u(b)$.
This satisfies all of the above properties:
(1) follows from Theorem~\ref{gf3}(6);
(2) follows because
$j^\lambda$ factors through $j$ and we know that $T\u_\eps(b)$ satisfies
the analogous property; (3) follows by 
Theorem~\ref{gf3}(5).
For the recursive step, assume that we are given $T_m$ satisfying
(1), (2) and (3) and  $d_+(T_m) > 0$.
We can find $c \in \B_\mu$ and a non-split extension
\begin{equation}\label{ses}
0 \longrightarrow T_m \longrightarrow T_{m+1} \longrightarrow \Delta(c)
\longrightarrow 0.
\end{equation}
This constructs $T_{m+1}$. We claim that $d_+(T_{m+1}) < d_+(T_m)$ and that
$T_{m+1}$ 
satisfies
(1), (2) and (3) too.
Part (1) is clear from the definition.
For (2), we just apply the exact functor $j^\lambda$ to the exact sequence \eqref{ses},
noting that
$j^\lambda \Delta(c) = 0$. For (3), take $a \in \B\setminus\B_\mu$ and
apply the functor
$\Hom_{\R}(\Delta_\eps(a),?)$ to the short exact sequence \eqref{ses}
to get 
$$
\Ext^1_{\R}(\Delta_\eps(a), T_m)
\longrightarrow
\Ext^1_{\R}(\Delta_\eps(a), T_{m+1})
\longrightarrow
\Ext^1_{\R}(\Delta_\eps(a), \Delta(c)).
$$
The first and last term are zero by hypothesis and 
(\ref{repeatedly}), hence, so is the middle term.
It remains to show $d_+(T_{m+1}) < d_+(T_m)$.
For $a \in \B_\mu$, 
we have $\Ext^1_{\R}(\Delta(a), \Delta(c)) = 0$ by
(\ref{repeatedly}),
so again we have an exact sequence
$$
\Hom_{\R}(\Delta(a), \Delta(c))
\stackrel{f}{\longrightarrow}
\Ext^1_{\R}(\Delta(a), T_m)
\longrightarrow
\Ext^1_{\R}(\Delta(a), T_{m+1})
\longrightarrow 0.
$$
This shows that $\dim \Ext^1_{\R}(\Delta(a), T_{m+1}) \leq \dim
\Ext^1_{\R}(\Delta(a), T_m)$, and we just need to observe that the
inequality is actually a strict one in the case $a = c$.
To see this, note that the first morphism $f$ is non-zero in the case $a=c$
as $f(\id_{\Delta(c)}) \neq 0$ 
due to the fact that the original short
exact sequence was not split.
This completes the proof of the claim. 
We have now defined an object
$T_n \in \Delta_\eps(\R)$ 
such that $j^\lambda T_n \cong P_\lambda(b)$ if
$\eps(\lambda)=+$ or
$I_\lambda(b)$ if $\eps(\lambda)=-$, 
and moreover $\Ext^1_{\R}(\Delta_\eps(a), T_n) = 0$ for all $a \in
\B$.
By Theorem~\ref{gf1}, we deduce that $T_n \in \nabla_\eps(\R_{\leq
  \lambda})$ too, hence, it is an $\eps$-tilting object.
Decompose $T_n$ into indecomposables $T_n = T_{n,1}\oplus\cdots\oplus
T_{n,r}$.
Then each $T_{n,i}$ is also an $\eps$-tilting object by Lemma~\ref{karo}.
Since $j^\lambda T_n$ is indecomposable, we must have that $j^\lambda
T_n = j^\lambda T_{n,i}$ for some unique $i$.
Then we set $T_\eps(b) := T_{n,i}$ for this $i$.
This gives us the desired indecomposable $\eps$-tilting object.

\vspace{1.5mm}
\noindent
{Case $\eps(\mu)=-$:} 
Let $d_-(V) := \sum_{c \in
  \B_\mu} \dim \Ext^1_{\R}(V, \nabla(c))$.
This time, one recursively constructs $T_0 := j_* T_\eps\u(b), T_1,\dots,T_n$ so that 
$d_-(T_0) > \cdots > d_-(T_n) = 0$ and
\begin{enumerate}
\item[(1$'$)]
$T_m \in \nabla_\eps(\R)$.
\item[(2$'$)]
$j^\lambda T_m \cong 
P_\lambda(b)$ if $\eps(\lambda) = +$ or
$I_\lambda(b)$ if $\eps(\lambda) = -$.
\item[(3$'$)]
$\Ext^1_{\R}(T_m, \nabla_\eps(a)) = 0$ for all $a \in \B \setminus
\B_\mu$.
\end{enumerate}
Since this is this is just the dual construction to the case $\eps(\mu)=+$
already treated, i.e., it is the same construction in the opposite
category, we omit the details.
Then, at the end, one decomposes
$T_n$ into indecomposables $T_n = T_{n,1}\oplus\cdots\oplus
T_{n,r}$. By Theorem~\ref{gf2} each $T_{n,i}$ is an $\eps$-tilting object.
Since $j^\lambda T_n$ is indecomposable, we must have that $j^\lambda
T_n = j^\lambda T_{n,i}$ for some unique $i$, and
set $T_\eps(b) := T_{n,i}$ for this $i$.

This completes the construction of $T_\eps(b)$ in general. We have
shown it satisfies (iii). Let us show that it also satisfies
(i) and (ii). For (i), we know by (iii) that $T_\eps(b)$ belongs to
$\R_{\leq\lambda}$, and it has a $\Delta_\eps$-flag.
By (\ref{repeatedly}), we may order this flag so that 
the sections $\Delta_\eps(c)$ for $c \in \B_\lambda$ appear at the bottom.
Thus, there is a short exact sequence $0 \rightarrow K \rightarrow
T_\eps(b) \rightarrow Q\rightarrow 0$ such that $K$ has a $\Delta_\eps$-flag with sections
$\Delta_\eps(c)$ for $c \in \B_\lambda$ and $j^\lambda Q  =0$.
Then $j^\lambda K \cong j^\lambda T_\eps(b)$. 
If $\eps(\lambda) = +$, this is $P_\lambda(b)$. Since $j^\lambda$ is exact and $j^\lambda
\Delta(c) = P_\lambda(c)$ for each $c \in \B_\lambda$, we must
have that $K \cong \Delta(b)$, and (1) follows.
Instead, if $\eps(\lambda)=-$, the bottom section of the
$\bar\nabla$-flag of $K$ must be $\bar\nabla(b)$ since $j^\lambda K 
\cong I_\lambda(b)$ has irreducible socle $L_\lambda(b)$,
giving (i) in this case too.
The proof of (ii) is similar.

\vspace{1.5mm}

\noindent
{Uniqueness:}
Let $T := T_\eps(b)$ and $U$ be some other indecomposable object of $\Tilt_\eps(\R)$
satisfying one of the properties (i)--(iii).
We must prove that $T \cong U$.
By the argument from the previous paragaph, we 
may assume actually that $U$ satisfies either (i) or (ii).
We just explain how to see this in the case that $U$ satisfies (i);
the dual argument treats the case that $U$ satisfies
(ii). So
there are short exact sequences
$0 \rightarrow \Delta_\eps(b) \stackrel{f}{\rightarrow} U \rightarrow Q_1 \rightarrow
0$ and
$0 \rightarrow \Delta_\eps(b) \stackrel{g}{\rightarrow} T
\rightarrow Q_2 \rightarrow 0$
such that $Q_1$, $Q_2$ have $\Delta_\eps$-flags.
Applying $\Hom_{\R}(?, T)$ to the first and using
$\Ext^1_{\R}(Q_1, T) =0$, we get that
$\Hom_{\R}(U, T) \twoheadrightarrow \Hom_{\R}(\Delta_\eps(b),
T)$. Hence, $g$ extends to a homomorphism
$\bar g:U \rightarrow T$. Similarly, $f$ extends to $\bar
f:T  \rightarrow U$.
We have constructed morphisms making the triangles in the following diagram
commute:
$$
\begin{tikzcd}[row sep=tiny]
& U \arrow[dd,shift left=-.8mm,"\bar g" left]\\
\Delta_\eps(b) \arrow[ur,"f"] \arrow[dr,"g" below] & \\ &T \arrow[uu, shift
right=.8mm,"\bar f" right]
\end{tikzcd}
$$
Since $\bar f \circ \bar g \circ  f = f$, we deduce that $\bar f
\circ \bar g$ is not nilpotent. Since $U$ is indecomposable, Fitting's Lemma implies $\bar f \circ \bar g$ is an
isomorphism. Similarly, so is $\bar g \circ \bar f$. Hence, $U \cong T$.
\end{proof}

\begin{remark}\label{genie}
Let $b \in \B_\lambda$.
When $\eps(\lambda) = +$, Theorem~\ref{gin} 
implies that $(T_\eps(b):\Delta_\eps(b))
= 1$ and $(T_\eps(b):\Delta_\eps(c)) = 0$ for all other $c \in
\B_\lambda$.
Similarly, when $\eps(\lambda)=-$,
we have that
$(T_\eps(b):\nabla_\eps(b))
= 1$ and $(T_\eps(b):\nabla_\eps(c)) = 0$ for all other $c \in
\B_\lambda$.
\end{remark}

The following corollaries show that 
$\eps$-tilting objects behave well with respect to 
passage to lower and upper sets, extending
Theorems~\ref{gf}, \ref{salsanewer} and \ref{gf3}.
This follows easily from those theorems plus the
characterization of tilting objects in Theorem~\ref{gin};
the situation is just
like \cite[Lem.~A4.5]{Donkin}.

\begin{corollary}\label{it}
Let $\R$ be a finite or lower finite $\eps$-stratified category and
$\R\d$ be the
finite $\eps$-stratified subcategory associated to a finite lower set
$\Lambda\d$ of $\Lambda$. For $b \in \B\d := \rho^{-1}(\Lambda\d)$,
the corresponding indecomposable $\eps$-tilting object of $\R\d$ 
is $T_\eps(b)$ (the same as
in $\R$).
\end{corollary}

\begin{corollary}\label{jt}
Assume $\R$ is a finite $\eps$-stratified category and let $\Lambda\u$
be an upper set in $\Lambda$ with associated quotient
$j:\R\rightarrow \R\u$. Let $b \in \B\u:=\rho^{-1}(\Lambda\u)$. 
The corresponding indecomposable
$\eps$-tilting object $T\u_\eps(b)$ of $\R\u$ 
 satisfies
$T\u_\eps(b) \cong j T_\eps(b)$.
Also $j T_\eps(b) = 0$ if $b \notin \B\u$.
\end{corollary}

The next result is concerned with tilting resolutions.

\begin{definition}\label{tres}
Assume that $\R$ is a finite or lower finite $\eps$-stratified
category.
An {\em $\eps$-tilting resolution} $d:T_\bullet \rightarrow V$ 
of $V \in \R$ is 
the data of an exact sequence
$$
\cdots 
\stackrel{d_2}{\longrightarrow} T_1
\stackrel{d_1}{\longrightarrow} T_0
\stackrel{d_0}{\longrightarrow} V \longrightarrow 0
$$
such that 
\begin{itemize}
\item[(TR1)]
$T_m \in \Tilt_\eps(\R)$ for each $m=0,1,\dots$;
\item[(TR2)] $\im d_m \in \nabla_\eps(\R)$
for $m\gg 0$.
\end{itemize}
Similarly, an {\em $\eps$-tilting coresolution}
$d:V \rightarrow T^\bullet$ of $V \in \R$ 
is the data of an exact sequence
$$0 \longrightarrow V \stackrel{d^0}{\longrightarrow} T^0
\stackrel{d^1}{\longrightarrow} 
T^1 \stackrel{d^2}{\longrightarrow} \cdots
$$
such that 
\begin{itemize}
\item[(TC1)]
$T^m \in \Tilt_\eps(\R)$ 
for $m=0,1,\dots$;
\item[(TC2)] $\coim d^m \in \Delta_\eps(\R)$
for $m \gg 0$.
\end{itemize}
We say it is a {\em finite} resolution (resp., coresolution) if there is some $n$
such that
$T_m = 0$ (resp., $T^m = 0$) for $m >
n$. Note in the finite case that axioms (TR2) and (TC2) are redundant since the zero object belongs to both $\nabla_\eps(\R)$ and $\Delta_\eps(\R)$.
\end{definition}

\begin{lemma}\label{ben0}
Assume that $\R$ is a finite or lower finite $\eps$-stratified
category.
\begin{enumerate}
\item
If $d:T_\bullet \rightarrow V$ is an $\eps$-tilting resolution of $V
\in \R$ then $\im d_m \in \nabla_\eps(\R)$ for {\em all} $m \geq 0$.
In particular, $V \in \nabla_\eps(\R)$.
\item
If $d:V \rightarrow T^\bullet$ is an $\eps$-tilting coresolution of $V
\in \R$ then $\coim d^m \in \Delta_\eps(\R)$ for {\em all} $m \geq 0$.
In particular, $V \in \Delta_\eps(\R)$.
\end{enumerate}
\end{lemma}

\begin{proof} 
(1) It suffices to show for any exact sequence
$A \stackrel{f}{\rightarrow} B \stackrel{g}{\rightarrow} C$
in a finite or lower finite $\eps$-stratified category
that $B\in\nabla_\eps(\R)$ and $\im f \in \nabla_\eps(\R)$
implies $\im g \in \nabla_\eps(\R)$.
Since $\im f = \ker g$,
there is a short exact sequence $0 \rightarrow \im f \rightarrow
B \rightarrow \im g \rightarrow 0$.
Now apply Corollary~\ref{hanks} (or Corollary~\ref{fli}).

\vspace{1.5mm}\noindent
(2)
An $\eps$-tilting coresolution of $V$ in $\R$
is the same thing as a $(-\eps)$-tilting resolution of $V$ in $\R^\op$.
Hence, this follows as it is the dual statement to (1).
\end{proof}

\begin{theorem}[Tilting resolutions and coresolutions]\label{tr}
Let $\R$ be a finite or lower finite $\eps$-stratified category and
take $V \in \R$.
\begin{enumerate}
\item $V$ has an $\eps$-tilting resolution if and only if
$V \in \nabla_\eps(\R)$.
\item $V$ has an $\eps$-tilting coresolution if and only if
$V \in \Delta_\eps(\R)$.
\end{enumerate}
\end{theorem}

\begin{proof}
We just prove (1), since (2) is the equivalent dual statement.
If $V$ has an $\eps$-tilting resolution, then we must have that $V \in
\nabla_\eps(\R)$ thanks to Lemma~\ref{ben0}(1).
For the converse, 
we claim for $V \in \nabla_\eps(\R)$ that there is a
short exact sequence $0 \rightarrow S_V \rightarrow T_V \rightarrow V
\rightarrow 0$ with $S_V \in \nabla_\eps(\R)$ and $T_V \in
\Tilt_\eps(\R)$.
Given the claim, one can construct an $\eps$-tilting resolution of $V$ by ``Splicing''
(e.g., see \cite[Fig.~2.1]{Wei}), to complete the proof.

To prove the claim,
we argue by induction on the length
 $\sum_{b \in \B} (V:\nabla_\eps(b))$ of a $\nabla_\eps$-flag of $V$.
If this number is one, then $V \cong \nabla_\eps(b)$ for some $b \in \B$, and there is a 
short exact sequence $0 \rightarrow S_V \rightarrow T_V \rightarrow V
\rightarrow 0$ with $S_V \in \nabla_\eps(b)$ and $T_V := T_\eps(b)$ due to
Theorem~\ref{gin}(ii).
If it is greater than one, then there is a short exact sequence 
$0 \rightarrow U \rightarrow V \rightarrow W \rightarrow 0$ such that
$U$ and $W$ have strictly shorter $\nabla_\eps$-flags.
By induction, there are short exact sequences
 $0 \rightarrow S_U \rightarrow T_U \rightarrow U \rightarrow 0$ and
$0 \rightarrow S_W \rightarrow T_W \rightarrow W \rightarrow 0$
with $S_U, S_W \in \nabla_\eps(\R)$ and $T_U, T_W \in \Tilt_\eps(\R)$.
It remains to show that these short exact sequences can 
be assembled to produce the desired short exact sequence for
$V$. The argument is like in
the proof of the Horseshoe Lemma in \cite[Lem.~2.2.8]{Wei}.
\begin{equation}\label{toobig}
\begin{tikzcd}
&0\arrow[d]&0\arrow[d]&0\arrow[d]\\
0\arrow[r]&
S_U\arrow[d]\arrow[r]&
T_U\arrow[r,"i"]\arrow[d]&U \arrow[d,"f"]\arrow[r]& 0\\
0\arrow[r]&
S_V\arrow[d]\arrow[r]&
T_V\arrow[d]\arrow[r,"j"]&V\arrow[d,"g"]\arrow[r]& 0\\
0\arrow[r]&
S_W\arrow[d]\arrow[r]&
T_W\arrow[ur,dashed,"\hat k"]\arrow[d]\arrow[r,"k"]&W\arrow[d]\arrow[r]&0\\
&0&0&0
\end{tikzcd}
\end{equation}
Since $\Ext^1_\R(T_W, U) = 0$, we can lift $k:T_W \rightarrow W$ to
$\hat k:T_W \rightarrow V$ so that $k = g \circ \hat k$.
Let $T_V := T_U \oplus T_W$ and $j:T_V \rightarrow V$ be $\operatorname{diag}(f i,
\hat k)$. This gives us a split short exact sequence 
in the middle column in \eqref{toobig}, such that the right hand
squares commute. Then we let $S_V := \ker
j$, and see  that there are induced
maps making the left hand column and middle row into short exact
sequences 
such that the left hand squares commute too.
\end{proof}

\subsection{Finite Ringel duality}
In this subsection, we review the theory of
Ringel duality for finite $\eps$-stratified categories.
Our exposition is based in part on
\cite[Appendix]{Donkin}, which gives a self-contained treatment in the
highest weight setting, and \cite{AHLU}, where the $+$-highest weight
case is considered assuming $\Lambda
= \{1 < \cdots < n\}$; the 
survey in \cite[Ch.~3]{Rei} is also helpful.
Throughout, we assume that $\R$ is a finite $\eps$-stratified
category with the usual stratification 
$(\B,L,\rho,\Lambda,\leq)$.

\begin{definition}\label{thesetup}
Let $\R$ be a finite $\eps$-stratified category.
By an {\em $\eps$-tilting generator} $T$ for $\R$, we mean an object $T \in
\Tilt_\eps(\R)$
such that $T$ has a summand isomorphic to $T_\eps(b)$ for each $b \in
\B$.
Given such an object, we define the {\em Ringel dual} of $\R$
relative to $T$ to be the finite Abelian category
$\R' := B\fdlmod$
where
$B := \End_\R(T)^\op$.
 We also define the two (covariant) {\em Ringel duality functors}
\begin{align}
\F:=\Hom_{\R}(T,?)&:\R \rightarrow \R',\label{ffunc}\\
  \G := \Cohom_\R(T,?) = \Hom_{\R}(?, T)^*&:\R \rightarrow\R'.\label{gfunc}
\end{align}
Note for the second of these that $\Hom_{\R}(V,T)$ is naturally a
finite-dimensional right $B$-module for $V \in \R$, hence, its dual is
a left $B$-module.
\end{definition}

\begin{theorem}[Finite Ringel duality]\label{Creek}
In the setup of Definition~\ref{thesetup},
the Ringel dual $\R'$ of $\R$ relative to $T$
is a finite $(-\eps)$-stratified category with 
stratification 
$(\B,L',\rho,\Lambda,\geq)$
and distinguished objects
\begin{align*}
P'(b) &= \F T_\eps(b),
&I'(b) &= \G T_\eps(b),&
L'(b) &= \hd P'(b) \cong \soc I'(b),\\
\Delta'_{-\eps}(b) &= \F \nabla_\eps(b),
&\nabla'_{-\eps}(b) &= \G \Delta_\eps(b),
& T'_{-\eps}(b) &= \F I(b)\cong \G P(b).
\end{align*}
The restrictions
$\F:\nabla_\eps(\R)\rightarrow \Delta_{-\eps}(\R')$
and
$\G:\Delta_\eps(\R)\rightarrow \nabla_{-\eps}(\R')$
are equivalences; in fact, they induce isomorphisms
\begin{align}\label{abitbetter}
\Ext^n_\R(V_1,V_2) &\cong
\Ext^n_{\R'}(FV_1,FV_2),
&
\Ext^n_\R(W_1,W_2) &\cong
\Ext^n_{\R'}(GW_1,GW_2),
\end{align}
for all $V_i \in \nabla_\eps(\R)$, $W_i \in \Delta_\eps(\R)$ and $n \geq 0$.
\end{theorem}

Before the proof,
we give
some applications.
  
\begin{corollary}[Double centralizer property]\label{mustard}
Suppose that the finite $\eps$-stratified category $\R$
 in Theorem~\ref{Creek} is the category
$A\fdlmod$ for a finite-dimensional algebra $A$, so that $T$ is
an $(A,B)$-bimodule.
Let $T' := T^*$ be the
dual $(B,A)$-bimodule. Then the following holds.
\begin{enumerate}
\item
$T'$ is a $(-\eps)$-tilting generator for $\R'=B\fdlmod$ and there is an algebra isomorphism
\begin{equation}\label{thisisiso}
\mu:A \stackrel{\sim}{\rightarrow} \End_{\R'}(T')^\op
\end{equation}
sending $x \in A$ to $\mu(x):T'\rightarrow T', v \mapsto vx$.
So the Ringel dual of $\R'$ relative to $T'$
is equivalent to the original category $\R$.
\item
Denote the Ringel duality functors for $\R'$ relative to $T'$ now by
\begin{align}\label{newrd1}
\G_* := \Hom_{\R'}(T',?)&:\R'\rightarrow\R,\\
\F^* := \Cohom_{\R'}(T',?) = \Hom_{\R'}(?,T')^*&:\R' \rightarrow \R.\label{newrd2}
\end{align}
We have that
$\F^*\cong T \otimes_B ?$ and $\G \cong T'
\otimes_A ?$,
hence,  $(\F^*, \F)$ and $(\G, \G_*)$ are adjoint pairs.
\end{enumerate}
\end{corollary}

\begin{proof}
(1) Note that $\G A$ is a $(-\eps)$-tilting generator since $\G P(b)
\cong T'_{-\eps}(b)$ for $b \in \B$.
Actually, $\G A = \Hom_A(A, T)^* \cong T^* = T'$. Thus, $T'$ is
a $(-\eps)$-tilting generator for $\R'$.
Its opposite endomorphism algebra is isomorphic to $A$ as stated since 
$\G$ defines an algebra isomorphism
$$
A \cong \End_{A}(A)^\op \stackrel{\sim}{\rightarrow} 
\End_{B}(\G A)^\op \cong 
\End_{B}(T')^\op.
$$

\vspace{1.5mm}
\noindent
(2)
As $\F^*$ is right exact and commutes with direct sums, 
a standard argument using the Five Lemma
shows that
it is isomorphic to $(\F^* B) \otimes_{B} ? \cong
T \otimes_{B} ?$.
Thus, $\F^*$ is left adjoint to $\F$. 
Similarly, $\G \cong T' \otimes_A ?$ is left adjoint to $\G_*$.
\end{proof}
 
The next corollary describes the strata $\R'_\lambda$ of the Ringel
dual category; see also Lemma~\ref{aftersun} below.
For $\lambda \in \Lambda$,
denote the
quotient functor $\R'_{\geq\lambda}\rightarrow \R'_\lambda$ by
$(j')^\lambda$,
and denote its left and right adjoints by $(j')^\lambda_{!}:\R'_\lambda \rightarrow
\R'_{\geq\lambda}$ and $(j')^\lambda_{*}:\R'_\lambda \rightarrow
\R'_{\geq\lambda}$.
We also have the inclusion $(i')_{\geq \lambda}:\R'_{\geq \lambda} \rightarrow
\R'$ with left and right adjoints $(i')^*_{\geq \lambda}$ and $(i')^!_{\geq \lambda}$.

\begin{corollary}\label{notquite}
For $\lambda \in \Lambda$, the strata $\R_\lambda$ and
$\R'_\lambda$ are equivalent. 
More precisely:
\begin{enumerate}
\item
If $\eps(\lambda)=+$ the functor
$F_\lambda := 
(j')^\lambda \circ (i')_{\geq \lambda}^! \circ \F\circ i_{\leq \lambda} \circ j^\lambda_*:\R_\lambda
\rightarrow \R'_\lambda$ is an equivalence of categories
taking $L_\lambda(b) = j^\lambda L(b)$
to $L_\lambda'(b) = (j')^\lambda L'(b)$.
\item
If $\eps(\lambda)=-$ the functor
$G_\lambda:=(j')^\lambda \circ (i')_{\geq \lambda}^* \circ \G\circ i_{\leq \lambda} \circ  j^\lambda_!:\R_\lambda
\rightarrow \R'_\lambda$ is an equivalence of categories
taking $L_\lambda(b) = j^\lambda L(b)$
to $L_\lambda'(b) = (j')^\lambda L'(b)$.
\end{enumerate}
\end{corollary}

\begin{proof}
We just prove (1), since (2) is similar. 
So assume that $\eps(\lambda)=+$.
We first note that ${\F}_\lambda$ is exact. Indeed, $j^\lambda_*$ is
exact by Theorem~\ref{fund}, so it sends objects of $\R_\lambda$ to
objects of $\R_{\leq \lambda}$ which have filtrations with sections $\nabla_\eps(b)$
for $b \in \B_\lambda$.
Then we apply the exact functor $i_{\leq \lambda}$ followed by
$\F$, which takes short exact sequences in $\nabla_\eps(\R)$ to short exact sequences in $\Delta_\eps(\R)$, to obtain
an object of
$\Delta_{-\eps}(\R'_{\geq\lambda})$. The functor $(i')^!_{\geq \lambda}$ 
is the identity on this subcategory, and finally $(j')^\lambda$ is exact.
Adopting the setup of
Corollary~\ref{mustard}, we can also define
\begin{align*}
\F^*_\lambda &:= j^\lambda \circ i_{\leq \lambda}^*\circ \F^* \circ (i')_{\geq \lambda}\circ (j')^\lambda_{!}:
\R'_\lambda\rightarrow \R_\lambda.
\end{align*}
A similar argument to before gives that this is exact too.
We complete the proof by showing that ${\F}_\lambda$ and
$\F^*_\lambda$ are quasi-inverse equivalences.
Note that $\F^*_\lambda$ is left adjoint to ${\F}_\lambda$.
The counit of adjunction gives us a natural transformation
$\F^*_\lambda \circ {\F}_\lambda \rightarrow
\operatorname{Id}_{\R_\lambda}$. We claim this is an isomorphism.
Since both functors are exact, it suffices to prove this on
irreducible objects: we have $\F^*_\lambda (\F_\lambda L_\lambda(b)) \cong
\F^*_\lambda L'_\lambda(b) \cong L_\lambda(b)$.
Similar argument shows that the unit of adjunction is an isomorphism
in the other direction.
\end{proof}

\begin{corollary}\label{strata2}
Let $\R$ be a finite $\eps$-stratified category.
\begin{enumerate}
\item
All $V \in \nabla_\eps(\R)$ have finite $\eps$-tilting resolutions
if and only if
all positive strata are of finite global dimension.
\item
All $V \in \Delta_\eps(\R)$ have finite $\eps$-tilting coresolutions
if and only if all negative strata
are of finite global dimension.
\end{enumerate}
\end{corollary}

\begin{proof}
We just explain the proof of (1). 
By Theorem~\ref{Creek}, all $V \in \nabla_\eps(\R)$
have finite $\eps$-tilting resolutions if and only if all $V' \in
\Delta_{-\eps}(\R')$ have finite projective resolutions.
 By Lemma~\ref{strata}(1), this is equivalent to 
all negative strata of the $(-\eps)$-stratified category $\R'$
are of finite global dimension.
Equivalently, by Corollary~\ref{notquite}, all positive strata of the
$\eps$-stratified category $\R$ are of finite global dimension.
\end{proof}

\begin{corollary}\label{inparticulartiltings2}
If $\R$ is a finite $+$-stratified (resp., $-$-stratified) category then all $V \in \Delta(\R)$ (resp., $V \in \nabla(\R)$)
have finite $+$-tilting coresolutions (resp., finite $-$-tilting resolutions).
\end{corollary}

The next theorem is a consequence of Happel's tilting theory
for finite-dimensional algebras. To prepare for this, we
explain the 
connection between $\eps$-tilting objects
in our setting and the general notions of tilting and cotilting modules from that theory; e.g., see \cite{H}, \cite{Rei}.
Suppose that 
$\R = A\fdlmod$ is a finite $\eps$-stratified algebra 
for a finite-dimensional algebra $A$, 
and let $T$ be an $\eps$-tilting generator for $\R$.
If all negative strata are of finite global dimension (this
assumption being vacuous in the case $\eps=+$)
then $T$ is a {\em tilting module} in the sense of tilting theory;
if all positive strata are of finite global dimension (this assumption being vacuous in the case $\eps=-$) then $T$ is a {\em cotilting module}.
These assertions follow using Theorem~\ref{gf1} to see that
$\Ext^1_\R(T,T) = 0$,
Lemma~\ref{strata} to see that $\pd T < \infty$ or $\injd T < \infty$,
and
Corollary~\ref{strata2}.
Without assumptions on the global dimensions of strata,
 $T$ need not be tilting or cotilting, but 
Theorem~\ref{tr} implies that it is still an example of a {\em Wakamatsu tilting module}\footnote{With this in mind, the fact that the map 
  (\ref{thisisiso}) is an isomorphism could also be deduced from \cite[Cor.~2]{Wakamatsu}.} 
  as defined in \cite[Ch.~3]{BR}; see also \cite[$\S$4.1]{Rei}. 
 The {\em WT-conjecture} formulated in \cite[Ch.~3]{BR}
 is the assertion that any 
 Wakamatsu tilting module of finite projective (resp., injective) dimension is 
 tilting (resp., cotilting). 
 This motivates the following conjecture in our special situation; we will prove this 
 assuming a mild additional hypothesis on strata in Lemma~\ref{proofofconj} below. 

\begin{conjecture}[$\eps$T-conjecture]\label{wakconj}
Suppose that $\R$ is a 
finite fully stratified category and $\eps$ is a given sign function.
For $b \in \B$,
the $\eps$-tilting module $T_\eps(b)$ is of finite projective (resp., injective) dimension if and only if $T_\eps(b)$ belongs to $\Tilt_+(\R)$ (resp.,  $\Tilt_-(\R)$).
\end{conjecture}

Let $\der{R} \F$ and $\der{L} \G$ be the total 
derived functors of the Ringel duality functors. 
These are triangulated functors between the bounded derived categories $D^b(\R)$ and $D^b(\R')$.

\begin{theorem}[Derived equivalences]\label{happels}
Let $\R'$ be the Ringel dual of a finite $\eps$-stratified category
$\R$.
Assume that all negative strata (resp., all positive strata) of $\R$
are of finite global dimension.
Then 
$\der{R} \F:D^b(\R)\rightarrow D^b(\R')$ (resp., $\der{L}
\G:D^b(\R)\rightarrow D^b(\R')$) is 
an
equivalence of triangulated categories.
Moreover, if $\R$ is of finite
global dimension, then so is $\R'$.
\end{theorem}

\begin{proof}
Assuming $\R$ has finite global dimension, this all follows by
\cite[Lem.~2.9, Th.~2.10]{H};
the hypotheses there hold thanks to
Corollary~\ref{strata2}.
To get the derived equivalence without assuming $\R$ has finite global
dimension, we cite instead Keller's exposition
of Happel's result in
\cite[Th.~4.1]{Keller}, since it assumes slightly less;
the hypotheses (a) and (c) there hold due to Corollary~\ref{strata2}(2)
and Lemma~\ref{strata}(1).
\end{proof}

\begin{corollary}
If $\R$ is $+$-highest weight (resp., $-$-highest weight) and
$\R'$ 
is the Ringel dual relative
to a $+$-tilting generator (resp., $-$tilting generator),
then $\der{R}\F:D^b(\R)\rightarrow D^b(\R')$
(resp., $\der{L}\G:D^b(\R)\rightarrow D^b(\R')$)
is an equivalence.
\end{corollary}

\proof[Proof of Theorem~\ref{Creek}]
This follows the same steps as in \cite[pp.158--160]{Donkin}.
Assume without loss of generality that $\R  = A\fdlmod$ for a
finite-dimensional algebra $A$. 
For each $b \in \B$, let $f_b \in B
= \End_B(T)^\op$ be an idempotent such that $T f_b \cong T_\eps(b)$.
Then $P'(b) := B f_b$ is an indecomposable projective
$B$-module
and
the modules $$
\big\{L'(b) := \hd P'(b)\:\big|\:b \in \B\big\}
$$ 
give a full set of
pairwise inequivalent irreducible left $B$-modules.
Since $\R'$ is a finite Abelian category, it is immediate that
$(\B,L',\rho,\Lambda,\geq)$ is a stratification of
it. 
Let $\Delta'_{-\eps}(b)$ and $\nabla'_{-\eps}(b)$ be the
  $(-\eps)$-standard and $(-\eps)$-costandard objects of $\R'$ defined
  from this stratification.
Set $V(b) := \F \nabla_\eps(b)$.

 \vspace{1.5mm}
\noindent
{Step 1:}
 {\em For $b \in \B$ we have that $P'(b) \cong \F T_\eps(b)$.}
This follows immediately from the equality
$\Hom_A(T, T) f_b = \Hom_A(T,
T f_b)$.

\vspace{1.5mm}
\noindent
{Step 2:}
{\em The functor $\F$ sends short exact sequences of objects in $\nabla_\eps(\R)$ to short exact sequences in $\R'$.}
This follows because $\Ext^1_\R(T, V) = 0$ for $V \in \nabla_\eps(\R)$ by
the usual $\Ext^1$-vanishing between $\Delta_\eps$- and
$\nabla_\eps$-filtered objects.

\vspace{1.5mm}
\noindent
{Step 3:}
{\em
For $a,b \in \B$, we have that $[V(b):L'(a)] =
(T_\eps(a):\Delta_\eps(b))$.}
The left hand side is
$\dim f_a V(b) =\dim f_a \Hom_{A}(T, \nabla_\eps(b))\cong 
\dim \Hom_{A}(T_\eps(a), \nabla_\eps(b))$, which equals the right
hand side.

\vspace{1.5mm}
\noindent
{Step 4:} 
{\em $V(b)$ is a non-zero quotient of $P'(b)$, thus, $\hd V(b)=L'(b)$.} By
Theorem~\ref{gin}(i), there is a short exact sequence
$0 \rightarrow K \rightarrow T_\eps(b) \rightarrow \nabla_\eps(b)\rightarrow 0$
with $K \in \nabla_\eps(\R)$. Hence, Step~2 implies that
$V(b)$ is quotient of $P'(b)$. It is non-zero by
Step~3.

\vspace{1.5mm}
\noindent
{Step 5:}
{\em We have that $V(b)\cong\Delta'_{-\eps}(b)$.}
Let $\lambda := \rho(b)$. We treat
the cases $\eps(\lambda)=+$ and $\eps(\lambda)=-$ separately.
If $\eps(\lambda)=+$ we must show that $V(b)$ is the largest quotient of
$P'(b)$ with the property that $[V(b):L'(a)] \neq 0\Rightarrow \rho(a)
\geq \rho(b)$. We have already observed in Step~4 that $V(b)$ is
a quotient of $P'(b)$. Also $(T_\eps(a):\Delta_\eps(b)) \neq
0\Rightarrow \rho(b) \leq \rho(a)$ by Theorem~\ref{gin}(iii).
Using Step~3, this imples that $V(b)$ has the
property $[V(b):L'(a)] \neq 0\Rightarrow \rho(a)
\geq \rho(b)$.
It remains to show that any strictly larger quotient of $P'(b)$ fails
this condition. To see this, since $\eps(\lambda)=+$, a
$\nabla_\eps$-flag in $T_\eps(b)$
has $\nabla_\eps(b)$ at the top and other sections $\nabla_\eps(c)$ for $c$ with
$\rho(c) < \rho(b)$.
In view of Step~4, any strictly larger quotient of $P'(b)$ than $V(b)$
therefore 
has an additional composition factor
$L'(c)$ arising from the head of $V(c)$ for some
$c$ with $\rho(c) < \rho(b)$.

Instead, if $\eps(\lambda)=-$, 
we use the characterization of $\Delta'_{-\eps}(b)$ from Lemma~\ref{charac}(1): we must show
that $V(b)$ is the largest quotient of $P'(b)$ with the property
that
$[\rad V(b):L'(a)] \neq 0\Rightarrow \rho(a)
> \rho(b)$.
Since $\eps(\lambda)=-$, we have that
$(T_\eps(b):\nabla_\eps(b)) = 1$ and
$(T_\eps(b):\nabla_\eps(a)) \neq 0 \Rightarrow \rho(a) < \rho(b)$
for $a \neq b$.
Hence, using Step~3 again,
the quotient $V(b)$ of $P'(b)$ has the required
properties.
A $\nabla_\eps$-flag in $T_\eps(b)$ has $\nabla_\eps(b)$ at the top
and other sections $\nabla_\eps(c)$ for $c$ with $\rho(c) \leq
\rho(b)$.
So any strictly larger quotient of $P'(b)$ than $V(b)$
has a composition factor $L'(c)$ arising from the head of
$V(c)$ for $c$ with $\rho(c) \leq \rho(b)$.
In case $c=b$, this violates the requirement that the quotient has $L'(b)$ appearing
with multiplicity one; otherwise, it violates the requirement that all other
composition factors of the quotient are of the form $L'(a)$ with
$\rho(a) > \rho(b)$.

\vspace{1.5mm}
\noindent
{Step 6:} {\em $\R'$ is a finite $(-\eps)$-stratified category.}
In view of Step~5, it suffices to show that $P'(b)$ has a filtration with
sectons $V(c)$ for $c$ with $\rho(c)\leq \rho(b)$.
Since
$T_\eps(b)$ has a $\nabla_\eps$-flag with sections $\nabla_\eps(c)$ for $c$ with $\rho(c) \leq
\rho(b)$, 
this follows using Steps~1 and 2.

\vspace{1.5mm}
\noindent
{Step 7:} {\em For any $U \in \Tilt_\eps(\R)$ and $V \in \R$, the map
  $f:\Hom_{A}(U, V) \rightarrow \Hom_{B}(\F U, \F V)$ induced by $\F$ is
  an isomorphism.} It suffices to prove this when $U = T$,
so that the right hand space is $\Hom_{B}(B, \F V)$ and $\F V =
\Hom_A(T, V)$.
This special case follows because $f$ is the inverse of the
isomorphism
$\Hom_{B}(B, \F V) \rightarrow \F V, \theta \mapsto \theta(1)$.

\vspace{1.5mm}
\noindent
{Step 8:}
{\em For any $V, W \in \nabla_\eps(\R)$ and $n \geq 0$, 
  the functor $\F$ induces a linear isomorphism
$\Ext^n_{\R}(V,W) \stackrel{\sim}{\rightarrow}
\Ext^n_{\R'}(\F V, \F W)$.}
Take an $\eps$-tilting resolution $d:T_\bullet \rightarrow V$ in the
sense of Definition~\ref{tres}, which exists thanks to
Theorem~\ref{tr}.
The functor $\F$ takes this resolution
to a complex
$$
\cdots \longrightarrow \F T_1 \longrightarrow \F T_0 \longrightarrow
\F V
\longrightarrow 0.
$$
In fact, this complex is exact. To see this, take $m \geq 0$ and
consider 
the short exact sequence
$
0 \rightarrow \ker d_m \rightarrow T_m \rightarrow \im d_m \rightarrow
0.
$
All of $\ker d_m$, $T_m$ and $\im d_m$ have
$\nabla_\eps$-flags due to Lemma \ref{ben0}(1). 
Hence, thanks to Step~2, we get a short exact sequence
$$
0
\longrightarrow \F(\ker d_m) \stackrel{i}{\longrightarrow} \F T_m
\stackrel{p}{\longrightarrow} \F (\im d_m)
\longrightarrow 0
$$ 
on applying $\F$. 
Since $\F$ is left exact, the canonical map $\F(\im d_m) \rightarrow \F T_{m-1}$
is a monomorphism. Its image is all $\theta:T \rightarrow
T_{m-1}$ with image contained in $\im d_m$.
As $p$ is an epimorphism,
any such $\theta$ can be written as $d_m \circ \phi$ for $\phi:T
\rightarrow T_m$, i.e., $\theta \in \im (\F d_m)$.
Thus, $\F(\im d_m) \cong \im (\F d_m)$, and 
$0 \rightarrow \ker (\F d_m) \longrightarrow \F T_m
\rightarrow \im (\F d_m)
\rightarrow 0$ 
is exact, as required.
In view of Step~1, we have constructed a projective resolution of
$\F V$ in $\R'$:
$$
\cdots \longrightarrow \F T_1 \longrightarrow \F T_0
\longrightarrow \F V \longrightarrow
0.
$$

Next, we use the projective resolution just constructed
to compute $\Ext^n_{\R'}(\F V, \F I)$ for any injective $I \in \R$.
We have a commutative
diagram
$$
\begin{tikzcd}
0\arrow[r]&\Hom_{\R}(V,I)\arrow[r]\arrow[d,"f"]&
\Hom_{\R}(T_0, I)\arrow[r]\arrow[d,"f_0"]
&
\Hom_{\R}(T_1, I)\arrow[r]\arrow[d,"f_1"]
&\cdots\\
0\arrow[r]&\Hom_{\R'}(\F V, \F I)\arrow[r]&\Hom_{\R'}(\F T_0,
\F I)\arrow[r]&\Hom_{\R'}(\F T_1, \F I)\arrow[r]&\cdots
\end{tikzcd}
$$
with vertical maps induced by $\F$.
The maps $f_0,f_1,\dots$ are isomorphisms due to Step~7. 
Also
the top row is exact as $I$ is injective.
We deduce that the bottom row is exact at  the positions
$\Hom_{\R'}(\F T_m, \F I)$ for all $m \geq 1$.
It is exact at positions $\Hom_{\R'}(\F V, \F I)$
and $\Hom_{\R'}(\F T_0, \F I)$ as
$\Hom_{\R'}(?,\F I)$ is left exact.
Thus, the bottom row is exact everywhere.
So the map $f$ is an isomorphism too
and $\Ext^n_{\R'}(\F V,\F I) = 0$ for $n > 0$.

Finally,
take a short exact sequence 
$0 \rightarrow W \rightarrow I \rightarrow Q \rightarrow 0$ in $\R$
with $I$ injective. We have that $Q \in \nabla_\eps(\R)$ by Corollary~\ref{hanks}.
Hence, using Step~2 and the previous paragraph,
there is a commutative diagram
$$
\begin{tikzcd}
\!\Hom_{\R}(V, W)\!\arrow[r,hookrightarrow]\arrow["f_1",d]&\!\Hom_{\R}(V, I) \!
\arrow[r]\arrow[d,"f_2"]&
\!\Hom_{\R}(V, Q)\!\arrow[d,"f_3"]
\arrow[r,twoheadrightarrow]&\!\Ext^1_{\R}(V,W)\arrow[d,"f_4"]\\
\!\!\!\Hom_{\R'}(\F V,
\F W)\!\arrow[r,hookrightarrow]&\!\Hom_{\R'}(\F
V, \F I) \!
\arrow[r]&
\!\Hom_{\R'}(\F V, \F Q)\!
\arrow[r,twoheadrightarrow]&\!\Ext^1_{\R'}(\F V,
\F W)
\end{tikzcd}
$$
with exact rows.
As $f_2$ is an isomorphism, we get that $f_1$ is injective. Since this
is proved for all $W$, this means that $f_3$ is injective too. Then a
diagram chase gives that $f_1$ is surjective, hence, $f_3$ is
surjective and $f_4$ is an isomorphism.
Degree shifting now gives the 
isomorphisms $\Ext^n_{\R}(V,W)\stackrel{\sim}{\rightarrow}
\Ext^n_{\R'}(\F V, \F W)$ for $n \geq 2$ as well.

\vspace{1.5mm}
\noindent
{Step 9:}
{\em We have that $T'_{-\eps}(b) \cong \F I(b)$.}
By Steps~5 and 8, we get that $$
\Ext^1_{\R'}(\Delta'_{-\eps}(a), \F I(b)) \cong
\Ext^1_{\R}(\nabla_\eps(a), I(b)) = 0$$ 
for all $a \in \B$.
Hence, by the homological criterion for $\nabla_{-\eps}$-flags in the 
$(-\eps)$-stratified category $\R'$, 
the $A$-module $\F I(b)$ has a $\nabla_{-\eps}$-flag. It also has a
$\Delta_{-\eps}$-flag with bottom section isomorphism to 
$\Delta'_{-\eps}(b)$ 
due to Steps~2 and 5.
So $\F I(b) \in \Tilt_{-\eps}(\R')$.
It is indecomposable as $\End_{\R'}(\F I(b)) \cong \End_{\R}(I(b))$ by
Step~8, which is local.
Therefore $\F I(b) \cong T'_{-\eps}(b)$ due to Theorem~\ref{gin}(i).

\vspace{1.5mm}
\noindent
{Step 10:}
{\em The restriction $\F:\nabla_\eps(\R) \rightarrow
  \Delta_{-\eps}(\R')$ is an equivalence of categories.}
It is full and faithful by Step~8. It remains to show that it is dense, i.e.,
for any $V' \in \Delta_{-\eps}(\R')$
there exists $V \in \nabla_\eps(\R)$ with $\F V
\cong V'$.
The proof of this goes by induction on the length of a
$\Delta_{-\eps}$-flag of $V'$. If this length is one, we are done by
Step~5.
For the induction step, 
consider $V'$ fitting into a short exact sequence
$0 \rightarrow U' \rightarrow V' \rightarrow W' \rightarrow 0$
for shorter
$U', W' \in \Delta_{-\eps}(\R')$.
By induction there are $U, W \in
\nabla_\eps(\R)$
such that $\F U \cong U'$ and $\F W \cong W'$.
Then we use the isomorphism $\Ext^1_{\R'}(\F W, \F U)
\cong \Ext^1_\R(W, U)$ from Step~8 to see that there is an extension
$V$ of $U$ and $W$ in $\R$ 
such that $\F V \cong V'$.

\vspace{1.5mm}
\noindent
{Step 11:}
{\em 
The dual right $A$-module $T^*$ to $T$ 
is a $(-\eps)$-tilting generator for
$\R^\op = \fdrmod A$
such that $\End_A(T^*)^\op = B^\op$.
Moreover, letting 
$\F^\op:=\Hom_A(T^*,?):\fdrmod A\rightarrow \fdrmod B$ be the corresponding Ringel duality
functor, we have that $\G \cong ?^*\circ \F^\op\circ ?^*$.}
The first statement is clear from Theorem~\ref{opstrat}, observing
that $\End_A(T^*)^\op \cong \End_A(T)$
since $*:A\fdlmod\rightarrow \fdrmod A$ is a contravariant
equivalence.
It remains to observe that
$* \circ \F^\op \circ* \cong \Hom_A(T^*,?^*)^*
\cong \Hom_A(?, T)^* = \G$.

\vspace{1.5mm}
\noindent
{Step 12:}
{\em 
The restriction $\G:\Delta_\eps(\R)\rightarrow \nabla_{-\eps}(\R')$ is
an equivalence of categories
inducing isomorphisms as in (\ref{abitbetter}), such that
$\G T_\eps(b) \cong I'(b)$, $\G \Delta_\eps(b) \cong
  \nabla'_{-\eps}(b)$
and $\G P(b) \cong T'_{-\eps}(b)$.}
This follows using Step~11 and the analogs for $F^\op$ of the statements about $F$
establishd thus far.
\endproof

\subsection{Tilting objects in the upper finite and essentially finite cases}\label{sbags}
Throughout the subsection, $\R$ will be either be an upper finite
or an essentially finite
$\eps$-stratified category with the usual stratification $(\B,L,\rho,\Lambda,\leq)$. 
It is still possible to make sense of $\eps$-tilting objects but now the iterative procedure used to construct the indecomposable ones
in the proof of Theorem~\ref{gin} does not terminate after finitely many steps. Consequently, we must allow for tilting objects which have infinite $\Delta_\eps$- and $\nabla_\eps$-flags;
see (\ref{megan}) below for a baby example of this phenomenon.

Suppose to start with that $\R$ is an upper finite $\eps$-stratified category.
Using the notions of ascending $\Delta_\eps$-flags and
descending $\nabla_\eps$-flags 
introdued in Definition~\ref{tpc}, we set
\begin{equation}
\Tilt_\eps(\R) := \Delta_\eps^\asc(\R) \cap \nabla_\eps^\desc(\R).
\end{equation}
We emphasize that objects of $\Tilt_\eps(\R)$ are in particular
objects of $\R$, so all of their composition multiplicities are
finite.
Like in Lemma~\ref{karo}, $\Tilt_\eps(\R)$ is an additive Karoubian subcategory of
$\R$.

\begin{theorem}[Classification of tilting objects in the upper finite case]\label{moregin}
Assume that $\R$ is an upper finite $\eps$-stratified category.
For $b \in \B_\lambda$,
there is an indecomposable object $T_\eps(b) \in
\Tilt_\eps(\R)$ satisfying the following properties:
\begin{itemize}
\item[(i)] $T_\eps(b)$ has an ascending $\Delta_\eps$-flag
with bottom section\footnote{We mean that there is
an ascending $\Delta_\eps$-flag 
$(V_\omega)_{\omega \in \Omega}$
in which $\Omega$ has a smallest non-zero element
$1$ such that $V_1 \cong \Delta_\eps(b)$.} isomorphic to
$\Delta_\eps(b)$;
\item[(ii)] $T_\eps(b)$ has a descending $\nabla_\eps$-flag with top
  section\footnote{Similarly, we mean that $V / V_1 \cong \nabla_\eps(b)$.}
  isomorphic to $\nabla_\eps(b)$;
\item[(iii)] $T_\eps(b) \in \R_{\leq\lambda}$ and $j^\lambda T_\eps(b) \cong
\left\{
\begin{array}{ll}P_\lambda(b)&\text{if $\eps(\lambda)=+$}\\
I_\lambda(b)&\text{if $\eps(\lambda) = -$}
\end{array}
\right.$.
\end{itemize}
These properties determine $T_\eps(b)$ uniquely
up to isomorphism: if $T$ is any indecomposable object of
$\Tilt_\eps(\R)$
satisfying any one of the properties (i)--(iii) then $T \cong T_\eps(b)$;
hence, it satisfies the other two properties as well.
\end{theorem}

\begin{proof}
{Existence:}
Replacing $\R$ by $\R_{\leq\lambda}$ if necessary and using
Theorem~\ref{gf9},
we reduce to the special case that
$\lambda$ is
the largest element of the poset $\Lambda$.
Assuming this,
the first step in the construction of $T_\eps(b)$
is to define a direct system $(V_\omega)_{\omega \in
  \Omega}$ of objects of $\R$.
This is indexed by
the directed set $\Omega$ of all finite upper sets in
$\Lambda$.
Let $V_\varnothing := 0$.
Then take $\varnothing \neq \omega \in \Omega$ and denote it instead by $\Lambda\u$.
Letting $j:\R\rightarrow\R\u$ be the corresponding finite
$\eps$-stratified quotient of $\R$,
we set $V_\omega := j_! T\u_\eps(b)$. 
By Theorem~\ref{gf6}(6),
this has a $\Delta_\eps$-flag.
Given also $\omega < \upsilon \in \Omega$, i.e., another upper set
$\Lambda\uu$ containing $\Lambda\u$, let $k:\R\rightarrow\R\uu$ be the corresponding quotient.
Then $j$ factors as $j = \bar\jmath \circ k$ for an induced quotient functor
$\bar\jmath:\R\uu \rightarrow\R\u$.
Since $\bar\jmath T\uu_\eps(b) \cong T\u_\eps(b)$ by
Corollary~\ref{jt},
we deduce from Corollary~\ref{iuseitoften}(2) that there is a short exact sequence
$$
0 \longrightarrow \bar\jmath_! T\u_\eps(b) \longrightarrow T\uu_\eps(b)
\longrightarrow Q \longrightarrow 0$$ 
such that $Q$ has a $\Delta_\eps$-flag
with sections $\Delta_\eps\uu(c)$ for $c$ with $\rho(c)
\in \Lambda\uu\setminus\Lambda\u$.
Applying $k_!$ and using the exactness from Theorem~\ref{gf6}(6)
again,
we deduce that there is an embedding $f_\omega^\upsilon:V_\omega
\hookrightarrow V_\upsilon$ with $\operatorname{coker} f_\omega^\upsilon
\in \Delta_\eps(\R)$.
Thus, we have  a direct system
$(V_\omega)_{\omega \in
  \Omega}$.
Now let $T_\eps(b) := \varinjlim V_\omega \in \Ind(\R_c)$.
Using the induced embeddings $f_\omega:V_\omega \hookrightarrow
T_\eps(b)$, 
we identify each $V_\omega$ with a subobject of $T_\eps(b)$.
We have shown for $\omega < \upsilon$ 
that $V_\upsilon / V_\omega \in \Delta_\eps(\R)$ and,
moreover, $j V_\upsilon = j V_\omega$ 
where $j:\R\rightarrow\R\u$ is the quotient
associated to $\omega$.

In this paragraph, we show that $T_\eps(b)$ actually lies in $\R$ rather than $\Ind(\R_c)$, i.e., all of the
composition multiplicities $[T_\eps(b):L(c)]$ are finite.
To see this, take $c \in \B$.
Let $\omega=\Lambda\u \in \Omega$ be some fixed finite
upper set such that $\rho(c)\in\Lambda\u$, and $j:\R\rightarrow\R\u$ be
the quotient functor as usual. Then for any $\upsilon \geq
\omega$
we have that
$$
[V_\upsilon:L(c)] = [j V_\upsilon: L\u(c)] = [j V_\omega:L\u(c)] =
[V_\omega:L(c)].
$$
Hence, $[T_\eps(b):L(c)] = [V_\omega:L(c)] < \infty$.

So now we have defined $T_\eps(b) \in \R$ together with an
ascending $\Delta_\eps$-flag
$(V_\omega)_{\omega \in \Omega}$.
The smallest non-empty element of $\Omega$ is $\omega := \{\lambda\}$,
and $V_\omega = j^\lambda_! P_\lambda(b) =
\Delta_\eps(b)$
if $\eps(\lambda) = +$, or $j^\lambda_! I_\lambda(b)$
if $\eps(\lambda)=-$.
Since $j^\lambda T_\eps(b) = j^\lambda V_\omega$, we deduce that (iii)
holds. Also by construction $T_\eps(b)$ has an ascending
$\Delta_\eps$-flag.
To see that it has a descending $\nabla_\eps$-flag, take any $a \in \B$.
Let $\omega = \Lambda\u \in \Omega$ be such that $\rho(a)\in\Lambda\u$.
Then $\Delta_\eps(a) = j_! \Delta\u_\eps(a)$
and $j T_\eps(b) = j V_\omega = T_\eps\u(b)$, 
so by Theorem~\ref{gf6}(5) we
get that
$$
\Ext^1_{\R}(\Delta_\eps(a), T_\eps(b)) 
\cong \Ext^1_{\R\u}(\Delta\u_\eps(a),
T_\eps\u(b)) = 0.
$$
By Theorem~\ref{gf7}, this shows that
$T_\eps(b) \in \nabla_\eps^\desc(\R)$.

Note finally that $T_\eps(b)$ is indecomposable.
This follows because $j T_\eps(b)$ is indecomposable for every
$j:\R\rightarrow \R\u$ (adopting the usual notation). 
Indeed, by the construction we have that $j T_\eps(b) \cong T\u_\eps(b)$ 
This completes the construction of the indecomposable object
$T_\eps(b) \in \Tilt_\eps(\R)$. We have shown that it satisfies (iii),
and it follows easily that it also satisfies (i) and (ii).

\vspace{1.5mm}
\noindent
{Uniqueness:}
Since (iii) implies (i) and (ii), it suffices to show that any
indecomposable $U \in \Tilt_\eps(\R)$ satisfying either (i) or (ii) is
isomorphic to the object $T := T_\eps(b)$ just constructed. We explain this
just in the case of (i), since the argument for (ii) is similar.
We take a short exact sequence $0 \rightarrow \Delta_\eps(b)
\rightarrow T \rightarrow Q \rightarrow 0$ with $Q \in
\Delta_\eps^\asc(\R)$.
Using the $\Ext$-vanishing from Lemma~\ref{dry},
we deduce like in the proof of Theorem~\ref{gin}
that the inclusion $f:\Delta_\eps(b) \hookrightarrow
T$
extends to $\bar f:U \rightarrow T$.
In fact, $\bar f$ is an isomorphism. To see this, take a finite upper set
$\Lambda\u$ containing $\lambda$
and consider the quotient $j:\R\rightarrow\R\u$ as usual.
Both $j U$ and $j T$ are isomorphic to $T\u_\eps(b)$ by the
uniqueness in Theorem~\ref{gin}.
The proof there implies that any homomorphism $jT \rightarrow
j U$ which restricts to an isomorphism on the subobject $\Delta\u_\eps(b)$ is an
isomorphism.
We deduce that $j \bar f$ is an
isomorphism.
Since holds for all choices of $\Lambda\u$, it follows that $\bar
f$ itself is an isomorphism.
\end{proof}

\begin{corollary}
Any object of $\Tilt_\eps(\R)$
is isomorphic to $\bigoplus_{b \in \B} T_\eps(b)^{\oplus n_b}$
for {unique} multiplicities $n_b \in \N$.
Conversely, any such direct sum belongs to $\Tilt_\eps(\R)$.
\end{corollary}

\begin{proof}
Let us first show that any direct sum
$U := \bigoplus_{b \in \B} T_\eps(b)^{\oplus n_b}$
belongs to $\Tilt_\eps(\R)$. The only issue is to see that $U$ actually
belongs to $\R$ rather than $\Ind(\R_c)$, i.e., it has finite composition multiplicities.
But for a given $c \in \B$, the multiplicity 
$[T_\eps(b):L(c)]$ is zero unless $\rho(c) \leq \rho(b)$.
There are only finitely many such $b \in \B$,
so $[U:L(c)] = \sum_{b \in \B} n_b [T_\eps(b):L(c)] < \infty$.

Now take any $U \in \Tilt_\eps(\R)$.
Let $\Omega$ be the directed set of all finite upper sets in
$\Lambda$.
Take $\omega =\Lambda\u\in \Omega$.
Let $j:\R\rightarrow \R\u$ be the quotient functor as usual.
Then we have that $j U \in \Tilt_\eps(\R\u)$,
so it decomposes as a finite direct sum as 
$j U \cong \bigoplus_{b \in \B\u}
T_\eps\u(b)^{\oplus n_b(\omega)}$ for $n_b(\omega) \in \N$.
There is a corresponding direct summand $T_\omega \cong \bigoplus_{b
  \in \B\u} T_\eps(b)^{\oplus n_b(\omega)}$ of $U$.
Then $T = \varinjlim T_\omega$.
Moreover, for $b \in \B\u$, the multiplicities $n_b(\omega)$ are stable in the sense
that
$n_b(\upsilon) = n_b(\omega)$ for 
all $\upsilon > \omega$.
We deduce that $U \cong \bigoplus_{b \in \B}
T_\eps(b)^{\oplus n_b}$ where $n_b := n_b(\omega)$ for any
sufficiently large $\omega$.
\end{proof}


It remains to discuss tilting objects in 
the essentially finite case.
So now we assume that $\R$ is an essentially finite
$\eps$-stratified category with stratification
$(\B,L,\rho,\Lambda,\leq)$. 
Since $\Lambda$ is interval finite, finite unions of lower sets of the form $(-\infty,\lambda]$ are upper
finite. If $\R\d$ is the Serre subcategory of $\R$ associated to such an upper
finite lower set then its Schurian envelope $\Loc(\R\d)$ in
the sense of Lemma~\ref{coffee}
is a Cartan-bounded upper
finite $\eps$-stratified category which is naturally embedded into $\Loc(\R)$.
This follows from Theorem~\ref{gf}.
For $b \in \B$, we define 
the corresponding
$\eps$-tilting object $T_\eps(b)\in\Loc(\R)$ as follows:
pick any upper finite lower set $\Lambda\d$ such that $\rho(b) \in \Lambda\d$,
let $\R\d$ be the corresponding Serre 
subcategory of $\R$, 
then let $T_\eps(b)$ be the $\eps$-tilting object in $\Loc(\R\d)$
from Theorem~\ref{moregin}. This is well-defined independent of the
choice of $\Lambda\d$ by
the uniqueness part of Theorem~\ref{moregin}.
Thus, we have defined the indecomposable $\eps$-tilting objects
$\{T_\eps(b)\:|\:b \in \B\}$ in the essentially finite case too, although these may be of infinite length, i.e., in
general they belong to $\Loc(\R)$ rather than to $\R$ itself.

\begin{definition}\label{tiltingboundeddef}
Suppose that $\R$ is a lower finite, upper finite or essentially finite $\eps$-stratified category with the usual stratification.
We say that it is {\em tilting-bounded} if
the matrix
\begin{equation}\label{TM}
\left(\dim \Hom_\R(T_\eps(a), T_\eps(b)\right)_{a,b \in \B}
\end{equation}
has finitely many non-zero entries in each row and each column.
\end{definition}

The matrix (\ref{TM}) 
is analogous to the Cartan matrix (\ref{CM}) with projectives/injectives replaced by $\eps$-tilting objects.
In the lower finite case, all entries of this matrix are obviously $< \infty$, but in the upper finite or essentially finite cases it is possible that some of these dimensions are $\infty$. 
However they are all finite in the tilting-bounded case:

\begin{lemma}\label{anotherway}
If $\R$ is tilting-bounded then the spaces
$\Hom_\R(T_\eps(a), T_\eps(b))$ are finite-dimensional 
for all $a, b \in \B$.
\end{lemma}

\begin{proof}
In the lower finite case, the indecomposable tilting objects are of finite length, so these spaces are finite-dimensional even without the assumption that $\R$ is tilting-bounded.
In the remaining 
upper finite or essentially finite cases, 
we have that
\begin{equation}\label{easyproof}
\dim \Hom_\R(T_\eps(a), T_\eps(b)) = 
\sum_{c \in \B} (T_\eps(a):\nabla_\eps(c)) (T_\eps(b):\Delta_\eps(c))
\in \N \cup \{\infty\}.
\end{equation}
All of the multiplicities $(T_\eps(a):\nabla_\eps(c))$
and $(T_\eps(b):\Delta_\eps(c))$ are finite.
Moreover, if
$(T_\eps(a):\Delta_\eps(c)) \neq 0$ then
$\Hom_\R(T_\eps(a), T_\eps(c)) \neq 0$. Hence, assuming the tilting-bounded hypothesis, only finitely many of the terms in the sum on the right hand side are non-zero.
\end{proof}

Assuming $\R$ is an essentially finite $\eps$-stratified category once again, assume that $\R$ is also tilting-bounded.
Then the $\eps$-tilting objects $T_\eps(b)$
actually belong to
\begin{equation}
\Tilt_\eps(\R) := \Delta_\eps(\R) \cap \nabla_\eps(\R),
\end{equation}
i.e., they belong to $\R$ rather than to $\Loc(\R)$
of $\R$. Thus, we are in a similar situation to (\ref{to}). 
Theorem~\ref{gin} carries over easily, to
show that $\{T_\eps(b)\:|\:b \in \B\}$ gives a full set of
the indecomposable objects in the additive Karoubian category $\Tilt_\eps(\R)$.
The construction of Theorem~\ref{tr} also carries over unchanged. So all
objects of $\nabla_\eps(\R)$ have $\eps$-tilting resolutions and all
objects of $\Delta_\eps(\R)$ have $\eps$-cotilting resolutions.

\begin{remark}
Most of the interesting examples of essentially finite
highest weight categories which arise ``in nature'' 
seem to satisfy the tilting-bounded hypothesis, although
there is no reason for this to be the case from the recursive
construction of Theorem~\ref{moregin}.
We refer the reader to Remark~\ref{opendoor} for
an
example which is not tilting-bounded.
\end{remark}

\begin{remark}\label{penultimateremark}
The tilting-bounded hypothesis is also interesting in the lower finite case; see Corollary~\ref{finalcorollary} below.
Using (\ref{easyproof}), it is easy to see
in the lower finite case that
$\R$ is tilting-bounded if and only if
for each $b \in \B$ the multiplicities
$(T_\eps(a):\Delta_\eps(b))$ and $(T_\eps(a):\nabla_\eps(b))$
are zero for all but finitely many $a \in \B$.
Natural examples of lower finite highest weight categories
which are definitely {\em not} tilting-bounded include the 
categories $\Rep(G)$ for reductive groups $G$ (unless this is actually a semisimple category), as follows from the results in \cite[$\S$5]{anotherkevin}. In situations involving quantum groups at roots of unity, tilting-boundedness can be checked combinatorially by considering properties of Kazhdan-Lusztig polynomials; e.g., see \cite{SoergelKac}, \cite{Stroppel}.
\end{remark}

\subsection{Semi-infinite Ringel duality}\label{sird}
Now we
extend Ringel duality to lower finite and upper finite
$\eps$-stratified categories. The situation is not as symmetric as in the finite
case and demands different constructions when going from lower finite
to upper finite or from upper finite to lower finite.
If we start with a lower finite $\eps$-stratified category, the
Ringel dual is an upper finite $(-\eps)$-stratified category:

\begin{definition}\label{rd1}
Let $\R$ be a lower finite $\eps$-stratified category
with the usual stratification $(\B,L,\rho,\Lambda,\leq)$.
An {\em $\eps$-tilting generator} for $\R$ is an
object $T = \bigoplus_{i \in I} T_i \in \Ind(\R)$ 
with a given decomposition as a direct sum of
objects $T_i \in \Tilt_\eps(\R)$ such 
that each $T_\eps(b)$ is isomorphic to a summand of $T$.
Define the {\em Ringel dual} of $\R$ relative to $T = \bigoplus_{i \in I} T_i$
to be the Schurian category 
$\R' := A\lfdlmod$
where $$
A :=
\left(\bigoplus_{i,j \in I} \Hom_\R(T_i,T_j)\right)^\op.
$$
Identifying $\Ind(\R'_c)$ with $A\lmod$ as explained in (\ref{errorhere}), 
we have the {\em Ringel duality functor}
\begin{align}\label{rdf1}
\F:=\bigoplus_{i \in I} \Hom_{\R}(T_i,?) &: \Ind(\R) \rightarrow
\Ind(\R'_c).
\end{align}
This functor takes objects of $\R$ to objects of
$\R'$.
\end{definition}

\begin{theorem}[Lower to upper semi-infinite Ringel duality]\label{rt1}
In the setup of Definition~\ref{rd1},
$\R'$ is an upper finite $(-\eps)$-stratified category with stratification $(\B,L',\rho,\Lambda,\geq)$
and distinguished objects
\begin{align*}
P'(b) &\cong \F T_\eps(b),
  &L'(b) &\cong \hd P'(b),\\
\Delta'_{-\eps}(b) &\cong \F \nabla_\eps(b),
& T'_{-\eps}(b) &\cong \F I(b).
\end{align*}
The restriction
$\F:\nabla_\eps^\asc(\R)\rightarrow\Delta_{-\eps}^\asc(\R')$
is an equivalence of categories.
\end{theorem}

The proof will be explained later in the subsection.

In the other direction, if we start from an upper finite
$\eps$-stratified category, the Ringel dual is a lower finite
$(-\eps)$-stratified category:

\begin{definition}\label{rd2}
Let $\R$ be an upper finite $\eps$-stratified 
category with the usual
stratification $(\B,L,\rho,\Lambda,\leq)$.
An {\em $\eps$-tilting generator} is an object $T \in
\Tilt_\eps(\R)$
such that $T_{\eps}(b)$ is isomorphic to a summand of $T$ for every $b
\in \B$.
Let $C := \Coend_{\R}(T)$ be the coalgebra that is the continuous dual
of the pseudo-compact topological algebra $B:=\End_\R(T)^\op$; see Lemma~\ref{prof}.
Then the {\em Ringel dual} of $\R$ relative to $T$ is the
category
$\R' := \fdrcomod C=B\fdlmod$.
Recalling Lemma~\ref{ha}, 
the {\em Ringel duality functor} is
\begin{align}\label{care1}
  G:= \Cohom_{\R}(T,?) = \Hom_\R(?,T)^\star&:
\Ind(\R_c) \rightarrow \Ind(\R'),
\end{align}
which sends finitely generated
objects of $\R$ to objects of $\R'$.
\end{definition}

\begin{theorem}[Upper to lower semi-infinite Ringel duality]\label{rt2}
In the setup of Definition~\ref{rd2},
$\R'$ is a lower finite $(-\eps)$-stratified
category
with stratification $(\B,L',\rho,\Lambda,\geq)$ and
distinguished objects 
\begin{align*}
I'(b) &= G T_{\eps}(b),
&L'(b) &=\soc I'(b),\\
\nabla'_{-\eps}(b) &= G
\Delta_{\eps}(b),
& T'_{-\eps}(b)&= G P(b).
\end{align*}
The restriction
$G:\Delta_{\eps}^\asc(\R)
\rightarrow \nabla_{-\eps}^\asc(\R')$
is an 
equivalence of categories.
\end{theorem}

Again the proof will be explained later.

We proceed to formulate several 
consequences of Theorems~\ref{rt1} and \ref{rt2}.
The first is concerned with a special case.
Recall the definition of Cartan-bounded from just before Lemma~\ref{coffee}, and the definition of tilting-bounded from Definition~\ref{tiltingboundeddef}.

\begin{corollary}\label{finalcorollary}
The Ringel dual of a tilting-bounded lower finite $\eps$-stratified category is a Cartan-bounded upper finite $(-\eps)$-stratified category. Conversely,
the Ringel dual of a Cartan-bounded upper finite $\eps$-stratified
category is a tilting-bounded lower finite $(-\eps)$-stratified category.
\end{corollary}

\begin{proof}
From either Theorem~\ref{rt1} or Theorem~\ref{rt2}, it follows that the Cartan matrix (\ref{CM}) for the upper finite category
is equal to the matrix (\ref{TM}) for the lower finite category.
\end{proof}

The next two corollaries give the
analogs of the double centralizer property from Corollary~\ref{mustard}
in the semi-infinite setting.

\begin{corollary}[Lower to upper double centralizer property]\label{mustard1}
Let notation be as in Definition~\ref{rd1}. Assume in addition
that $\R = \fdrcomod C$ for a coalgebra $C$. 
Let $B := C^*$ be the dual algebra, so that $T$ is a $(B,A)$-bimodule.
Let $T' := T^\circledast$ be the dual $(A,B)$-bimodule.
\begin{enumerate}
\item
$T'$ is a $(-\eps)$-tilting generator for
$\R'$ and there is an algebra isomorphism
\begin{equation}
\mu:B \stackrel{\sim}{\rightarrow} \End_{\R'}(T')^\op
\end{equation}
sending $y \in B$ to $\mu(y):T' \rightarrow T', v \mapsto vy$.
Equivalently, there is a coalgebra isomorphism
\begin{equation}\label{bubbly}
\mu^\star:\Coend_{\R'}(T')
\stackrel{\sim}{\rightarrow} C,\qquad
c^{(i)}_{r,s} \mapsto \tilde c^{(i)}_{r,s}
\end{equation}
where 
$c^{(i)}_{r,s}$ is the element of $\Coend_{\R'}(T')$
corresponding to  $v^{(i)}_s \otimes u^{(i)}_r \in T_i \otimes T_i^*$ 
according to (\ref{C}) for dual bases
$v^{(i)}_1,\dots,v^{(i)}_{d(i)}$ for $T_i$ and $u^{(i)}_1,\dots,u^{(i)}_{d(i)}$
for $T_i^*$,
and $\tilde c^{(i)}_{r,s}\in C$ is defined so that the structure map of the right $C$-comodule $T_i$ sends $v^{(i)}_s \mapsto \sum_{r=1}^{d(i)} v^{(i)}_r \otimes \tilde c^{(i)}_{r,s}$.
So the Ringel dual of $\R'$ relative to
$T'$ in the sense of Definition~\ref{rd2} 
is equivalent to the original category $\R$.
\item
  Denote the Ringel duality 
  functor for $\R'$ relative to $T'$ now by
  \begin{equation}
      F^*:=\Cohom_{\R'}(T',?) = \Hom_{\R'}(?,T')^\star: \Ind(\R_c') \rightarrow \Ind(\R).
      \end{equation}
We have that $F^* \cong T \otimes_A ?$, hence, $(\F^*, \F)$ is an adjoint pair; cf. Lemma~\ref{ha}.
\end{enumerate}
\end{corollary}

\begin{proof}
By Lemma~\ref{dumb}, we have natural isomorphisms
$\Hom_C(T_i, C) \cong T_i^*$ as right $B$-modules, hence,
$\F C \cong T'$ as an $(A,B)$-bimodule.
Since every $I(b)$ appears as a summand of the regular comodule,
and $\F I(b) \cong T'_{-\eps}(b)$ by Theorem~\ref{rt1},
we deduce that $T'$ is a $(-\eps)$-tilting generator for
$\R'$.
To see that $B \cong \End_A(T')^\op$,
we use the fact that $\F$ is an equivalence on $\nabla$-filtered
objects to deduce that
$$\End_A(T')^\op
\cong \End_A(\F C)^\op
\cong \End_C(C)^\op \cong B,
$$
using Lemma~\ref{dumb} again for the final algebra isomorphism.
This produces the isomorphism $\mu$.
To deduce (\ref{bubbly}), we need to 
show that $\mu^\star(c_{r,s}^{(i)})$ and $\tilde c_{r,s}^{(i)}$
take the same value on $y \in B$.
The left hand side gives $c_{r,s}^{(i)}(\mu(y))
= v_s^{(i)}(u_s^{(i)}b)$.
For the right hand side, we have
 that
$y v_s^{(i)} = \sum_{r=1}^{d(i)} \tilde c_{r,s}^{(i)}(y) v_r^{(i)}$, so $c_{r,s}^{(i)}(y) = (y v_s^{(i)}) u_r^{(i)}$.
These are equal. This establishes (1).
Then (2) follows from Lemma~\ref{ha}.
\end{proof}

\begin{corollary}[Upper to lower double centralizer property]\label{mustard2}
Let notation be as in Definition~\ref{rd2}, and assume in addition that
$\R=
A\lfdlmod$ for a locally finite-dimensional locally unital algebra $A
= \bigoplus_{i,j \in I} e_i A e_j$.
Let
$T_i = e_i T$ and $T_i' := T_i^*$, so that
$T' := \bigoplus_{i \in I} T_i' = T^\circledast$. This is a
$(B,A)$-bimodule.
\begin{enumerate}
\item
$T' = \bigoplus_{i \in I} T'_i$ is a $(-\eps)$-tilting generator for $\R'$ 
and there is an algebra isomorphism
\begin{equation}
\mu:A \stackrel{\sim}{\rightarrow} 
\left(\bigoplus_{i,j \in I} \Hom_{\R'}(T'_i,T'_j)\right)^\op
\end{equation}
sending $a \in e_i A e_j $ to $\mu(a):T_i'\rightarrow T_j', v \mapsto va$.
So the Ringel dual of $\R'$ relative to 
$T'$ in the sense of Definition~\ref{rd1} is equivalent to the original category $\R$.
\item
  Denote the Ringel duality 
  functor for $\R'$ relative to $T'$ now by
  \begin{equation}
      G_*:=\bigoplus_{i \in I} \Hom_{\R'}(T_i',?): \Ind(\R') \rightarrow \Ind(\R_c).
      \end{equation}
We have that $G \cong T' \otimes_A ?$,
hence, $(G, G_*)$ is an adjoint pair.
\end{enumerate}
\end{corollary}

\begin{proof}
 Note that $G(A e_i) = \Hom_A(A e_i, T)^* = (e_i T)^* = T_i$.
So Theorem~\ref{rt2} implies that $T = \bigoplus_{i \in I} T_i$ is a
$(-\eps)$-tilting generator for $\R'$.
Moreover,
$$
\Hom_{\R'}(T_i, T_j) = \Hom_{\R'}(G (A e_i), G (A
e_j))
\cong \Hom_{\R}(A e_i, A e_j) = e_i A e_j.
$$
This proves (1) and then (2) follows from Lemma~\ref{ha}.
\end{proof}

\begin{remark}\label{finalremark}
Combining
Corollary~\ref{finalcorollary} with the double centralizer properties just explained, one obtains a restricted version of semi-infinite Ringel duality giving a correspondence
\begin{eqnarray*}
\left\{
\begin{array}{ccc}\text{Tilting-bounded lower finite} \\ 
\text{highest weight categories}
\end{array}
\right\}
&\xleftrightarrow{\text{Ringel duality}}&
\left\{
\begin{array}{ccc}
\text{Cartan-bounded upper finite}\\
\text{highest weight categories}
\end{array}
\right\}.
\end{eqnarray*}
In the upper finite to lower finite direction,
this appeared already in the work of Marko and Zubkov \cite{MZ}.
In more detail, 
if $\R$ is the category of finite-dimensional modules over a descending quasi-hereditary pseudo-compact algebra in the sense of \cite[Def.~3.19]{MZ} and the indecomposable projectives in $\R$ are of finite length as  assumed in \cite[$\S$4]{MZ},
then $\R$ is an essentially finite highest weight category with upper finite weight poset, hence, $\Loc(\R)$ is a Cartan-bounded upper finite highest weight category. In this case, the indecomposable 
tilting modules $T(\lambda)\in \Loc(\R)$ were constructed already in \cite[$\S$4]{MZ}, 
and the appropriate (lower finite) Ringel dual category appears
in \cite[$\S$6]{MZ}.
Also \cite[Lem.~6.5]{MZ} establishes a double centralizer property which is equivalent to Corollary~\ref{mustard2}(1) for such categories. \end{remark}

In the setup of Definition~\ref{rd1},
one can also define a functor 
\begin{equation}\label{Gagaineq}
G:=\Cohom_\R(T,?) = 
\Hom_\R(?,T)^\circledast:\Delta_\eps(\R) \rightarrow \nabla_{-\eps}(\R').
\end{equation}
Theorem~\ref{rt1} plus an argument with duality like 
in Steps~11--12 of the proof of Theorem~\ref{Creek}
 shows that $G$ is an equivalence of categories 
such that
$G\Delta_\eps(b) \cong 
\nabla'_{-\eps}(b)$ and
$G T_\eps(b) \cong I'(b)$ for all $b \in \B$.
Likewise, in the setup of Definition~\ref{rd2},
one can also define
\begin{equation}\label{Fagaineq}
F:=\Hom_\R(T,?):\Delta_\eps(\R) \rightarrow \nabla_{-\eps}(\R').
\end{equation}
Theorem~\ref{rt2} plus an argument involving duality shows that $F$ is an equivalence of categories such that
$F I(b) \cong T'_{-\eps}(b)$ and
$F \nabla_\eps(b) \cong \Delta'_{-\eps}(b)$ for all $b \in \B$.
These functors are needed to formulate the following, which is the semi-infinite counterpart of Corollary~\ref{notquite}.
The proof is similar to the finite case; see also Lemma~\ref{aftersun} below.

\begin{corollary}\label{notquite2}
If $\R$ is a 
lower finite or an upper finite $\eps$-stratified category and $\R'$ is the Ringel dual category relative to some $\eps$-tilting generator as above,
the strata $\R_\lambda$ and
$\R'_\lambda$ are equivalent for all $\lambda \in \Lambda$. 
More precisely:
\begin{enumerate}
\item
If $\eps(\lambda)=+$ the functor
$F_\lambda := 
(j')^\lambda \circ (i')_{\geq \lambda}^! \circ \F\circ i_{\leq \lambda} \circ j^\lambda_*
:\R_\lambda
\rightarrow \R'_\lambda$ is an equivalence of categories
taking $L_\lambda(b) = j^\lambda L(b)$
to $L_\lambda'(b) = (j')^\lambda L'(b)$.
\item
If $\eps(\lambda)=-$ the functor
$G_\lambda:=
(j')^\lambda \circ (i')_{\geq \lambda}^* \circ \G\circ i_{\leq \lambda} \circ  j^\lambda_!:\R_\lambda
\rightarrow \R'_\lambda$ is an equivalence of categories
taking $L_\lambda(b) = j^\lambda L(b)$
to $L_\lambda'(b) = (j')^\lambda L'(b)$.
\end{enumerate}
\end{corollary}

In view of Corollary~\ref{it}, 
Corollary~\ref{strata2} can be applied also in any lower finite $\eps$-stratified category
(without any need to appeal to semi-infinite Ringel duality).
In particular,
if $\R$ is a lower finite $+$-stratified (resp., $-$-stratified) category then all $V \in \nabla(\R)$ (resp., $V \in \Delta(\R)$)
have finite $-$-tilting resolutions (resp., finite $+$-tilting coresolutions).
Using Theorem~\ref{rt1}, one sees that this assertion is equivalent to Lemma~\ref{inparticulartiltingsupper}.

We 
have not investigated 
derived equivalences or any analog of Theorem~\ref{happels} in the semi-infinite setting.

\proof[Proof of Theorem~\ref{rt1}]
We may assume that $\R = \fdrcomod C$ for a coalgebra $C$.
Let $B := C^*$ be the dual algebra, so that 
$\R$ is identified also with $B\fdlmod$.
We can replace the $\eps$-tilting generator $T = \bigoplus_{i \in I} T_i$ 
with any other. This just has the effect of
transforming $A$ into a Morita equivalent locally unital algebra.
Consequently, without loss of generality,
we may assume that $I = \B$ and $T = \bigoplus_{b \in \B} T_\eps(b)$. Then 
$$
A =
\left(\bigoplus_{a,b \in \B}
\Hom_\R(T_\eps(a), T_\eps(b)\right)^\op
$$ 
is a pointed locally finite-dimensional locally
unital algebra with (primitive) distinguished 
idempotents $\{e_b\:|\:b \in
\B\}$. Let $P'(b) := A e_b$ and $L'(b) := \hd
P'(b)$.
Then $\R' = A\lfdlmod$ is a Schurian category, the
$A$-modules
$\{L'(b)\:|\:b \in \B\}$ give a full set of pairwise
inequivalent irreducible objects, and $P'(b)$ is a
projective cover of $L'(b)$ in $\Ind(\R'_c) = A\lmod$.
It is immediate that $(\B,L',\rho,\Lambda,\geq)$ is a
stratification of $\R'$. Let
$\Delta'_{-\eps}(b)$ and $\nabla'_{-\eps}(b)$ be
its $(-\eps)$-standard and $(-\eps)$-costandard objects. Also let
$V(b) := \F \nabla_\eps(b)$.
Now one checks that Steps~1--6 from the proof of
Theorem~\ref{Creek} carry over to the present situation with very
minor modifications. We will not rewrite these steps here, but cite them
freely below. In particular, Step~6 establishes that $\R'$ is an upper
finite $(-\eps)$-stratified category. Also,
$\F \nabla_\eps(b) \cong \Delta'_{-\eps}(b)$ by Step~5. It
just remains to show:
\begin{itemize}
\item
$\F$ restricts to an equivalence of categories between
$\nabla_\eps^\asc(\R)$
and $\Delta_{-\eps}^\asc(\R')$.
\item
$\F I(b) \cong T'_{-\eps}(b)$, the indecomposable $(-\eps)$-tilting
object of $\R'$ labelled by $b\in \B$.
\end{itemize}
This requires some
different arguments compared to the ones from Steps~7--10 in the
proof of Theorem~\ref{Creek}.

Let $\Omega$ be the directed poset consisting of all finite lower sets
in $\Lambda$. Take $\omega=\Lambda\d\in \Omega$.
Let $\nabla_\eps(\R,\omega)$ be the full
subcategory of $\nabla_\eps(\R)$ consisting of the
$\nabla_\eps$-filtered objects with sections $\nabla_\eps(b)$ 
for $b \in \B\d := \rho^{-1}(\Lambda\d)$.
Similarly, we define the subcategory
$\Delta_{-\eps}(\R',\omega)$
of $\Delta_{-\eps}(\R')$.
By Steps~2 and 5, $\F$ restricts to a well-defined functor
\begin{equation}\label{soccer}
\F:\nabla_\eps(\R,\omega) \rightarrow
\Delta_{-\eps}(\R',\omega).
\end{equation}
We claim that this is an equivalence of categories.
To prove it, let
$i:\R\d\rightarrow \R$ be the finite $\eps$-stratified
subcategory of $\R$ associated to $\Lambda\d$.
Let $e := \sum_{b \in \B\d} e_b \in A$.
Then $T\d := \bigoplus_{b \in \B\d} T_\eps(b)$ is an $\eps$-tilting
generator for $\R\d$. As $\End_{\R\d}(T\d)^\op = e A e$,
the Ringel dual $(\R\d)'$ of $\R\d$ relative to $T\d$ is identified with
the
quotient category
$(\R')\d := eAe\fdlmod$ of $\R'$. Let $\F\d:=\Hom_\R(T\d,?):\R\d
\rightarrow (\R')\d$ be the corresponding
Ringel duality functor.
We also know from Theorem~\ref{gf6} that
$(\R')\d$ is the finite
$(-\eps)$-stratified quotient
of $\R'$ associated to
$\Lambda\d$ (which is a finite upper set in $(\Lambda, \geq)$).
Let $j':\R'\rightarrow(\R')\d$ be the quotient
functor, i.e., the functor defined by multiplication by the idempotent $e$.
For a right $C$-comodule $V$, we have that
$$
\F\d(i^! V) \cong \bigoplus_{b \in \B\d} \Hom_\R(T_\eps(b), i^! V)
\cong e \bigoplus_{b \in \B} \Hom_\R(T_\eps(b), V)
\cong j' (\F V).
$$
This shows that 
\begin{equation}\label{mondaymath}
\F\d \circ i^! \cong j' \circ \F,
\end{equation}
so in particular following diagram commutes up to a natural isomorphism:
$$
\begin{tikzcd}
  \R\arrow[r,"F" above] \arrow[d,"i^!" left] & \R'\arrow[d,"j'" right]\\
  \R\d\arrow[r,"F\d" above] &(\R\d)'\equiv(\R')\d.
\end{tikzcd}
$$
By Theorem~\ref{Creek}, $\F\d$ restricts to an equivalence $\nabla_\eps(\R\d)
\rightarrow
\Delta_{-\eps}((\R')\d)$.
Also the restrictions $i^!:\nabla_\eps(\R,\omega)\rightarrow
\nabla_\eps(\R\d)$
and $j':\Delta_{-\eps}(\R',\omega)\rightarrow
\Delta_{-\eps}((\R')\d)$
are equivalences.
This is clear for $i^!$. To see it for $j'$, one shows using
Theorem~\ref{gf6} that
the left
adjoint $(j')_!$ gives a quasi-inverse equivalence.
Putting these things together, we deduce that (\ref{soccer}) is an
equivalence as claimed.

Now we can show that $\F$ defines an equivalence
$\F:\nabla_\eps^\asc(\R)\rightarrow
\Delta_{-\eps}^\asc(\R')$.
Take $V \in \nabla_\eps^\asc(\R)$.
Then $V$ has a distinguished ascending $\nabla_\eps$-flag
$(V_\omega)_{\omega \in \Omega}$ indexed by the set $\Omega$ of finite
lower sets in $\Lambda$. This is defined by setting $V_\omega := i^!
V$ in the notation of the previous paragraph; see the proof of
Theorem~\ref{thethmnew}.
As each comodule $T_\eps(b)$ is finite-dimensional, hence, compact,
the functor $\F$ commutes with direct limits.
Hence, $\F V \cong \varinjlim (\F V_\omega)$. In fact, 
$(\F V_\omega)_{\omega
  \in \Omega}$ is the data of an ascending
$\Delta_{-\eps}$-flag in $\F V \in \R'$.
To see this, we have that $\F V_\omega \in
\Delta_{-\eps}(\R')$ by the previous paragraph. For $\omega < \upsilon$ the
quotient $V_\upsilon / V_\omega$ has a $\nabla_\eps$-flag thanks to
Corollary~\ref{fli}, so $\F V_\upsilon / \F V_\omega \cong \F(V_\upsilon
/ V_\omega)$ has a $\Delta_{-\eps}$-flag.
We still need to show that $\F V$ is locally finite-dimensional.
For this, we prove that
$\dim \Hom_A(\F V, I'(b)) < \infty$
for each $b \in\B$.
Since $I'(b)$ has a finite
$\nabla_{-\eps}$-flag, this reduces to  checking that
$\dim \Hom_A(\F V, \nabla'_{-\eps}(b)) < \infty$
for each $b$.
To see this, pick a finite lower set $\omega$ containing
$\rho(b)$.
Then for $\upsilon > \omega$, $\F V_\upsilon / \F V_\omega$ has a
$\nabla_{-\eps}$-flag with all sections different from
$\nabla'_{-\eps}(b)$, so $\Hom_A(\F V_\upsilon / \F V_\omega,
\nabla'_{-\eps}(b)) = \Ext^1_A(\F V_\upsilon / F V_\omega,
\nabla'_{-\eps}(b))=0$.
It follows that $\Hom_A(\F V_\upsilon, \nabla'_{-\eps}(b)) \cong
\Hom_A(\F V_\omega, \nabla'_{-\eps}(b))$ and
$$
\Hom_A(\F V, \nabla'_{-\eps}(b)) = \Hom_A(\varinjlim (\F V_\upsilon),
\nabla'_{-\eps}(b)) \cong \Hom_A(\F V_\omega, \nabla'_{-\eps}(b)),
$$
which is finite-dimensional.

At this point, we have proved that $\F$ induces a well-defined functor
$$
\F:\nabla_\eps^\asc(\R)\rightarrow
\Delta_{-\eps}^\asc(\R').
$$
We prove that this is an equivalence by showing that 
the left adjoint $\F^* := T \otimes_A ?$ to $\F$ gives a
quasi-inverse.
The left mate of (\ref{mondaymath}) gives an isomorphism
\begin{equation}\label{wednesdaymath}
i \circ (\F\d)^*\cong \F^* \circ (j')_!.
\end{equation}
Combining this with Corollary~\ref{mustard}, we deduce that $\F^*$ restricts to a
quasi-inverse of the equivalence (\ref{soccer}) for each $\omega \in
\Omega$. Also, $\F^*$ commutes with direct limits, and again
any $V' \in \Delta_{-\eps}^\asc(\R')$ has a
distinguished ascending $\Delta_{-\eps}$-flag
  $(V'_\omega)_{\omega \in \Omega}$ 
as we saw in the proof of Theorem~\ref{gf8}. These facts are enough to
show that $\F^*$ restricts to a well-defined functor
$\F^*:\Delta_{-\eps}^\asc(\R')\rightarrow \nabla_\eps^\asc(\R)$
which is quasi-inverse to $\F$.

Finally, we check that $\F I(b) \cong
T'_{-\eps}(b)$.
Let $V := I(b)$ and $(V_\omega)_{\omega \in \Omega}$ be its
distinguished ascending $\nabla_\eps$-flag indexed by the set $\Omega$
of finite lower sets in $\Lambda$ as above. Using the same notation as above,
for $\omega=\Lambda\d\in\Omega$ such that $\rho(b) \in \Lambda\d$, 
we know that $V_\omega$ is an injective hull of $L(b)$ in $\R\d$.
Hence, by Theorem~\ref{Creek},
$\F\d V_\omega$ is the indecomposable $(-\eps)$-tilting object of
$\R\d$ labelled by $b$.
From this, we see that the ascending $\Delta_{-\eps}$-flag $(\F V_\omega)_{\omega\in\Omega}$
in $\F I(b)$ coincides with the distinguished ascending $\Delta_{-\eps}$-flag
in $T'_{-\eps}(b)$ from the construction 
from the proof of Theorem~\ref{moregin}.
\endproof

\proof[Proof of Theorem~\ref{rt2}]
We may assume that $\R = A\lfdlmod$ for a pointed 
locally finite-dimensional locally unital algebra $A = \bigoplus_{a,b \in \B} e_a A e_b$, so that
$T$ is a locally finite-dimensional left $A$-module.
Let $C:= T^\circledast \otimes_A T$ viewed as a coalgebra
according to (\ref{pens}).
By Lemma~\ref{prof} this coalgebra 
is the continuous dual of $B = \End_A(T)^\op$, and 
we may identify $\R$ with the locally finite Abelian category $\fdrcomod C$.
Applying Lemma~\ref{ha}, the Ringel duality functor $G$
becomes the functor $T^\circledast \otimes_A ?:A\lmod\rightarrow\rcomod C$,
with the comodule structure map of $GV := T^\circledast\otimes_A V$ being defined as in (\ref{puns}).
Let 
\begin{equation}\label{june}
I'(b) := G T_{\eps}(b),\qquad
L'(b) := \soc
I'(b),
\qquad
\nabla'_{-\eps}(b) := 
G \Delta_\eps(b).
\end{equation}
Each $I'(b)$ is an
indecomposable injective right $C$-comodule, and 
$\{L'(b)\:|\:b \in \B\}$
is a full set of pairwise inequivalent irreducible $C$-comodules.
To show that $\R'$ is a lower finite $(-\eps)$-stratified category, we
must show for each finite upper set $\Lambda\u$ in $\Lambda$
that the Serre subcategory $(\R')\u$ of $\R'$
generated by $\{L'(b)\:|\:b \in \B\u := \rho^{-1}(\Lambda\u)\}$ is a
finite $(-\eps)$-stratified category for the induced stratification
$(\B\u,L',\rho,\Lambda,\geq)$.

The functor $G$ sends short exact sequences of objects in
$\Delta_{\eps}^\asc(\R)$ to short exact sequences in $\Ind(\R')$.
This follows because  $\Hom_\R(?,T)$ has this property 
thanks to the $\Ext^1$-vanishing from Lemma~\ref{dry}.
Since $\Delta_{\eps}(b) \hookrightarrow
T_{\eps}(b)$,
we deduce that
that $\nabla'_{-\eps}(b) \hookrightarrow I'(b)$. 
Thus, we have that $L'(b) = \soc \nabla'_{-\eps}(b)$.

Now let $\R\u$ be the Serre quotient of $\R$ associated to some finite
upper set
$\Lambda\u \subseteq \Lambda$ and let
$j:\R\rightarrow \R\u$ be the quotient functor.
This is a finite $\eps$-stratified category thanks to Theorem~\ref{gf6}.
In fact, $\R\u = A\u\fdlmod$ where $A\u := eAe$ for $e := \sum_{b \in
  \B\u} e_b$; the quotient functor $j$ is the idempotent truncation functor defined by multiplying by $e$.
By the upper finite analog of Corollary~\ref{jt},
$T\u := e T$ is an $\eps$-tilting
generator for $\R\u$.
Let $B\u := \End_{A\u}(T\u)^\op$ be its
(finite-dimensional) endomorphism
algebra. Then 
$(\R\u)':= B\u\fdlmod$ is the Ringel dual of $\R\u$ relative
to $T\u$.
By the finite Ringel duality from Theorem~\ref{Creek}, $(\R\u)'$ is a finite
$(-\eps)$-stratified category. Let
$G\u:= \Cohom_{\R}(T\u,?) =
\Hom_{\R}(?,T\u)^*:\R\u\rightarrow (\R\u)'$ be its Ringel duality
functor.
The functor $j$ defines an algebra homomorphism $\pi:B\rightarrow
B\u$, hence, we get a functor $\pi^*:(\R\u)' \rightarrow \R'$.
We claim that this gives an isomorphism identifying $(\R\u)'$
with the subcategory $(\R')\u$ of $\R'$.
This will be proved in the next paragraph. Moreover,
making this identification, we have that
\begin{equation}\label{tuesdaymath}
i' \circ G\u
\cong G \circ j_!.,
\end{equation}
i.e., the following diagram commutes up to natural isomorphism:
$$
\begin{tikzcd}
  \R\u\arrow[r,"G\u" above] \arrow[d,"j_!" left] & (\R\u)' \equiv(\R')\u\arrow[d,"{i'}" right]\\
  \R\arrow[r,"G" above] &\R'.
\end{tikzcd}
$$
This follows because the northeast composition is the functor  $T^\circledast e
\otimes_{eAe} ?$ while the southwest composition is
$T^\circledast \otimes_A Ae \otimes_{eAe} ?$, and
$T^\circledast e \cong T^\circledast \otimes_A Ae$ as bimodules.
Since we already know that $(\R\u)'$ is a finite $(-\eps)$-category,
it follows that $(\R')\u$ one too, with
costandard objects $$
i' (G\u \Delta_\eps(b)) \cong G (j_!
\Delta\u_\eps(b)) \cong G \Delta_\eps(b)
=\nabla'_{-\eps}(b)
$$
thanks again to
Theorem~\ref{Creek} plus Theorem~\ref{gf6}(6).

To prove the claim,
let $C\u := (B\u)^*$ be the (finite-dimensional) dual coalgebra so
that
$(\R\u)' = \fdrcomod C\u$.
Consider
the short exact sequence 
$$
0 \longrightarrow A e\otimes_{eAe} e T
\longrightarrow T \longrightarrow Q \longrightarrow 0
$$
which comes from the upper finite counterpart of
Lemma~\ref{iuseitoften}(2);
thus, $Q \in \Delta_{\eps}^\asc(\R)$ and all of its sections are of
the form $\Delta_{\eps}(b)$ for $b \notin \B\u$, while $Ae
\otimes_{eAe} eT \in \Delta_\eps(\R)$ has sections of the form
$\Delta_\eps(b)$ for $b \in \B\u$.
Applying $G$ and using the exactness noted in the second paragraph of
the proof,
we get a short exact sequence
$$
0 \longrightarrow C\u \longrightarrow C \rightarrow
G Q \rightarrow 0.
$$
The first map $C\u \rightarrow C$ here is dual to the algebra homomorphism
$\pi:B \rightarrow B\u$, so it is a coalgebra homomorphism. It identifies $(\R\u)'$ with the
the Abelian subcategory $\fdrcomod C\u$ of $\R' = \fdrcomod C$. Note also that
the irreducible objects of $\R'$ are $\{L'(b)\:|\:b \in \B'\}$.
To complete the proof of the claim,  it suffices using Lemma~\ref{crack}
to show that the socle of $G Q$ only has constituents of the
for $L'(b)$ for $b \notin \B\u$.
Fix an ascending $\Delta_{\eps}$-flag
$(V_\omega)_{\omega \in \Omega}$ in $Q$.
As $G$ commutes with direct limits,
we deduce that $G Q = \varinjlim (G V_\omega)$. 
The sections in a $\Delta_{\eps}$-flag in $V_\omega$
are $\Delta_{\eps}(b)$ for $b \notin \B\u$,
hence, $G V_\omega$ has a $\nabla_{-\eps}$-flag with
sections $\nabla'_{-\eps}(b)$ for $b \notin \B\u$.
It follows that $\soc (G
V_\omega)$ is of the desired form for each $\omega$, hence, the socle of
$G Q$ is too.

We can now complete the proof of the theorem.
We have shown already that
$\R'$ is a lower finite $(-\eps)$-stratified category.
Theorem~\ref{Creek} plus Corollary~\ref{it} shows for
$\Lambda\u$ chosen to contain $\rho(b)$
that 
$$
T_{-\eps}'(b) \cong
G \u (j P(b))
\cong G(j_!(j P(b)))
\cong G P(b).
$$
Also, for $a, b \in \B\u$, we have that
$$
\Hom_{\R'}(T'_{-\eps}(a), T'_{-\eps}(b))
\cong \Hom_{(\R')\u}(T'_{-\eps}(a), T'_{-\eps}(b))\cong \Hom_{A\u}(
A\u e_a , A\u e_b) \cong e_a A e_b. 
$$ 
These things are true for all choices of
$\Lambda\u$, so we see that the Ringel dual of $\R'$ relative to
$\bigoplus_{b \in \B} T'_{-\eps}(b)$ is the original category $\R=A\lfdlmod$.
This puts us in the situation of Corollary~\ref{mustard1}, and finally
we invoke that corollary (whose proof did not depend on
Theorem~\ref{rt2})
to establish that 
$G:\Delta^\asc_{\eps}(\R) \rightarrow
\nabla^\asc_{-\eps}(\R')$
is an equivalence.
\endproof

\subsection{Essentially finite Ringel duality}\label{ifc}
To complete our account of infinite versions of 
Ringel duality, 
it remains to discuss the essentially finite case.
For this, we impose the tilting-bounded assumption
from Definition~\ref{tiltingboundeddef}.

\begin{definition}\label{bored}
Assume $\R$ is an essentially finite $\eps$-stratified category with
stratification
$(\B,L,\rho,\Lambda,\leq)$. Assume in addition that $\R$ is
tilting-bounded.
An {\em $\eps$-tilting generator} for $\R$ means an
object $T = \bigoplus_{j \in J} T_j \in \Loc(\R)$
with a given decomposition as a direct sum of objects $T_j \in \Tilt_\eps(\R)$ such 
that each $T_\eps(b)$ appears as an indecomposable summand of $T$ with 
multiplicity that is non-zero and finite.
Then we define the {\em Ringel dual} of $\R$ relative to
$T$ to be the category
$\R' := B\fdlmod$ where
$$
B := \left(\bigoplus_{i,j \in J} \Hom_\R(T_i,T_j)\right)^\op.
$$
We denote the system of distinguished idempotents of $B$ arising from
the identity endomorphisms of each $T_j$ by $\{f_j\:|\:j \in J\}$.
Also define the two Ringel duality functors
\begin{align}\label{billie1}
\F := \bigoplus_{j \in J} \Hom_{\R}(T_j,?)&:\R \rightarrow
  \R',\\\label{billie2}
\G := \Cohom_\R(T,?) = \Hom_{\R}(?,T)^*&:\R \rightarrow \R'.
\end{align}
\end{definition}

\begin{theorem}[Essentially finite Ringel duality]\label{intfinrd}
In the setup of Definition~\ref{bored}, the Ringel dual category
$\R'$ is a tilting-bounded essentially finite $(-\eps)$-stratified category
with stratification
$(\B,L',\rho,\Lambda,\geq)$ and
distinguished objects
\begin{align*}
P'(b) &= \F T_\eps(b),
&I'(b) &= \G T_\eps(b),&
L'(b) & =\hd P'(b) \cong \soc I'(b),\\
\Delta'_{-\eps}(b) &= \F \nabla_\eps(b),
&\nabla'_{-\eps}(b) &= \G \Delta_\eps(b),
& T'_{-\eps}(b) &= \F I(b)\cong \G P(b).
\end{align*}
The restrictions
$\F:\nabla_\eps(\R)\rightarrow \Delta_{-\eps}(\R')$
and
$\G:\Delta_\eps(\R)\rightarrow \nabla_{-\eps}(\R')$
are equivalences.
\end{theorem}

\begin{proof}
We may assume that $\R = A\fdlmod$ for an essentially
finite-dimensional pointed
locally unital algebra $A = \bigoplus_{i,j \in I} e_i A e_j$.
Replacing the $\eps$-tilting generator $T = \bigoplus_{j \in J} T_j$ 
by any other changes $B$ to a Morita equivalent
algebra, so we may as well assume simply that $J = \B$ and
$T = \bigoplus_{b \in \B} T_\eps(b)$. Then the algebra $B =
  \bigoplus_{a,b\in \B} f_a B f_b$ is a pointed locally unital algebra.
The assumption that $\R$ is tilting-bounded implies that
$$
\sum_{a \in \B} \dim \Hom_\R(T_\eps(a), T_\eps(b)) < \infty, 
\quad
\sum_{b \in \B} \dim \Hom_\R(T_\eps(a), T_\eps(b)) < \infty
$$
for each $a,b \in \B$.
Thus, $B$ is essentially
finite-dimensional, i.e., $\R'$ is essentially finite Abelian.
The module $P'(b) := A f_b$ is an indecomposable projective
$A$-module,
and 
$$
\big\{L'(b):=\hd P'(b)\:\big|\:b \in \B\big\}
$$ 
is a full set of pairwise inequivalent irreducibles.
Now $(\B,L',\rho,\Lambda,\geq)$ defines a
stratification of $\R'$.
One checks that Steps~1--12 from the proof of
Theorem~\ref{Creek} all go through essentially unchanged in the
present setting. This completes the proof except for one point: we must
observe finally that $\R'$ is
tilting-bounded. This follows because the
relevant matrix
from Definition~\ref{tiltingboundeddef} (with each $T_\eps(b)$ now being replaced by $T_{-\eps}'(b)$)
is the Cartan matrix 
$$
\big(\dim \Hom_A(P(a),
P(b))\big)_{a,b \in \B}
$$ 
of $A$. Its rows and columns have only
finitely many non-zero entries as $A$ is essentially finite-dimensional.
\end{proof}

\begin{corollary}[Essentially finite double centralizer property]\label{mustardier}
Continuing in the general setup of Definition~\ref{bored},
suppose that the 
$\eps$-stratified category $\R$ is 
$A \fdlmod$ for an essentially finite-dimensional locally unital algebra
$A = \bigoplus_{i,j\in I} e_i A e_j$, so that $T =
\bigoplus_{j \in J} T_j$
is an $(A,B)$-bimodule.
  For $i \in I$, 
  let $T_i' := (e_i T)^* \in B\fdlmod$, so that
$T' := \bigoplus_{i \in I} T_i'$ is a $(B,A)$-bimodule.
\begin{enumerate}
\item
The module
$T' = \bigoplus_{i \in I} T'_i$ is a $(-\eps)$-tilting generator for $\R'=B\fdlmod$ and there is an algebra isomorphism
\begin{equation}
\mu: A \stackrel{\sim}{\rightarrow} \left(\bigoplus_{i,j \in I} \Hom_{\R'}(T_i',
 T_j')\right)^\op
\end{equation}
sending $a \in e_i A e_j$ to $\mu(a):T_i'\rightarrow T_j',
t \mapsto ta$.
So the Ringel dual of $\R'$ relative to $T'=\bigoplus_{i \in I} T_i'$
is equivalent to the original category
$\R$.
\item
Denote the Ringel duality functors 
from $\R'$ to $\R$ by
\begin{align}
G_* :=\bigoplus_{i \in I} \Hom_{\R'}(T'_i,?)&:\R'\rightarrow \R,\\
F^* :=\Cohom_\R(T',?) = \Hom_\R(?, T')^*&:\R'\rightarrow
\R.
\end{align}
respectively. We have that $F^*
\cong T \otimes_B ?$ and $G \cong T'
\otimes_A ?$,
hence,  $(\F^*, \F)$ and $(\G, \G_*)$ are adjoint pairs.
\end{enumerate}
\end{corollary}
 
\begin{proof}
For (1), note that $\bigoplus_{i \in I} \G (A e_i)$ 
is a $(-\eps)$-tilting
generator for $\R'$
since $\G P(b)
\cong T'_{-\eps}(b)$ for $b \in \B$.
Actually, $\G (A e_i) = \Hom_A(A e_i, T)^* \cong (e_iT)^* =
T_i'$. Thus, 
$T' = \bigoplus_{i \in I} T_i'$ is
a $(-\eps)$-tilting generator for $\R'$.
To obtain the isomorphism between $A$ and the locally finite endomorphism algebra of
$T'$,
apply the functor $\G$ to
the canonical isomorphism $A 
 \cong \left(\bigoplus_{i,j \in I}
\Hom_A(A e_i, A e_j)\right)^\op$.
To prove (2), we note first that $F^*(B f_j) \cong T
\otimes_B B f_j$. It then follows that $F^*(V) \cong T \otimes_B V$ on
any finite-dimensional $B$-module $V$ by taking a resolution
$P_2 \rightarrow P_1 \rightarrow V \rightarrow 0$ in which $P_1,P_2$
are direct sums of modules of the form $B f_j$, then using the Five
Lemma. The argument for $G$ is similar.
\end{proof}

We leave it to the reader to adapt Corollary \ref{notquite} to
the essentially finite setting.

\subsection{Tilting-rigidity}\label{trsec}
We begin by recalling some well-known definitions:
\begin{itemize}
\item[(QF)] A finite Abelian category $\R$ is {\em quasi-Frobenius} 
if all projective objects are injective.
In that case,
there is a unique 
bijection $\nu:\B\rightarrow\B$,
the {\em Nakayama permutation},
such that $$
P(b) \cong I(\nu(b))
$$
for each $b \in \B$,
where $P(b)$ and $I(b)$ are projective covers and injective hulls of of the irreducible objects $\{L(b)\:|\:b \in \B\}$.
\item[(WS)] A finite Abelian category $\R$ is {\em weakly symmetric} if
it is quasi-Frobenius with Nakayama permutation 
being the identity function.
Equivalently, $P(b) \cong I(b)$ for all $b \in \B$.
\item[(S)] A finite Abelian category $\R$ is
 {\em symmetric}
if there is a natural isomorphism of vector spaces
\begin{equation}\label{rickardsdef}
\Hom_\R(P, V) \cong \Hom_\R(V, P)^*
\end{equation}
for all $P, V \in \R$ with $P$ projective.
\end{itemize}
These are equivalent to saying that
every algebra realization $A$ of $\R$ is
quasi-Frobenius, Frobenius, or symmetric, respectively;
see \cite[$\S$4.4]{Haze} and \cite[Th.~3.1]{Rickard}.
Of course, (QF) $\Rightarrow$ (WS) $\Rightarrow$ (S).
We are going to investigate some 
properties of fully stratified categories 
which involve the properties
(QF), (WS) and (S) at the level of strata.

We assume from now on that $\R$ is a fully stratified category, by which we mean a fully stratified category
of any one of the four types, finite, essentially finite, upper finite or lower finite. We use the usual notation $(\B,L,\rho,\Lambda,\leq)$ for its stratification.

\begin{definition}\label{tiltingrigiddef}
Let $\R$ be a fully stratified category. We say that $\R$ is {\em tilting-rigid}
if $$
\Tilt_+(\R) = \Tilt_-(\R).
$$
 For this to make sense in the essentially finite case,
 it is necessary to assume implicitly that $\R$ is tilting-bounded in the sense of
Definition~\ref{tiltingboundeddef} for some choice (equivalently, all choices) of sign function $\eps$. 
\end{definition}

Highest weight categories are automatically tilting-rigid for trivial reasons, so that Definition~\ref{tiltingrigiddef} is not needed
when working just with highest weight categories.
The importance of tilting-rigidity first became apparent
in the context of fibered highest weight categories in 
\cite{MO}, \cite{FM}, where it is formulated as
the property ``tilting $=$ cotilting".
The following lemma shows in a tilting-rigid category that the subcategories $\Tilt_\eps(\R)$ coincide 
for all choices of $\eps$, so that we can 
denote them all simply by $\Tilt(\R)$.

\begin{theorem}[Tilting-rigid categories have quasi-Frobenius strata]\label{tiltingrigidstrata}
Let $\R$ be a tilting-rigid fully stratified category.
There is a unique 
bijection $\nu:\B\rightarrow \B$
such that $$
T_+(b) \cong T_-(\nu(b)).
$$
For $\lambda \in \Lambda$, this function 
leaves $\B_\lambda \subseteq \B$ invariant,
and the stratum $\R_\lambda$ is quasi-Frobenius with 
Nakayama permutation $\nu|_{\B_\lambda}$.
Moreover, for any sign function $\eps:\Lambda\rightarrow \{\pm\}$, we have that
\begin{equation}\label{walk}
T_\eps(b) \cong \left\{
\begin{array}{ll}
T_+(b) &\text{if $\eps(\lambda)=+$,}\\
T_+(\nu^{-1}(b))&\text{if $\eps(\lambda)=-$.}
\end{array}
\right.
\end{equation}
\end{theorem}

\begin{proof}
There is obviously a unique function
$\nu:\B\rightarrow \B$ such that $T_+(b) \cong T_-(\nu(b))$.
This function is injective and 
leaves each of the finite subsets $\B_\lambda$
invariant, hence, it is actually a bijection.
To see that $\R_\lambda$ is quasi-Frobenius with $\nu|_{\B_\lambda}$
as its Nakayama permutation, 
we must show that $P_\lambda(b) \cong I_\lambda(\nu(b))$ for each $b \in \B_\lambda$.
This follows using $T_+(b) \cong T_-(\nu(b))$ together with Theorem~\ref{gin}(3) or Theorem~\ref{moregin}(3)
(which one depends on the particular setting we are in). 
Finally, take $b \in \B_\lambda$ and 
a sign function $\eps$. 
Then $T_+(b) \cong T_-(\nu(b))$ has both a $\Delta$-flag and a $\nabla$-flag,
hence, it has a $\Delta_\eps$-flag and a $\nabla_\eps$-flag.
It follows that it is isomorphic to $T_\eps(b')$
for a unique $b' \in \B_\lambda$. 
Applying $j^\lambda$ and using Theorems~\ref{gin} or \ref{moregin} again gives that $b'=b$ if $\eps(\lambda)=+$ or $b' = \nu(b)$
if $\eps(\lambda) = -$, and the formula (\ref{walk}) follows.
\end{proof}

The argument used to prove the next lemma is based on 
the proof of \cite[Th.~2.2]{CM}. Note this proves Conjecture~\ref{wakconj}
assuming an additional hypothesis on the strata.

\begin{lemma}\label{proofofconj}
Suppose that $\R$ is a finite fully stratified category 
and $\eps:\Lambda\rightarrow\{\pm\}$ is some given sign function.
\begin{itemize}
\item[(1)] Assume that 
$L_\lambda(a)$ 
is isomorphic to a subobject of a projective object in $\R_\lambda$
for all $a \in \B_\lambda$ and $\lambda \in \Lambda$ with
$\eps(\lambda)=+$.
Then for $b \in \B$, 
$T_\eps(b)$ has finite injective dimension if and only if
$T_\eps(b) \in \Tilt_-(\R)$.
\item[(2)] Assume 
that
$L_\lambda(a)$ 
is isomorphic to a quotient of an injective object in $\R_\lambda$ for all $a \in \B_\lambda$
and $\lambda \in \Lambda$ with
$\eps(\lambda)=-$.
Then 
for $b \in \B$, 
$T_\eps(b)$ has finite projective dimension if and only if
$T_\eps(b) \in \Tilt_+(\R)$.
\end{itemize}
\end{lemma}

\begin{proof}
We just prove (1), (2) being the equivalent dual statement.
If $T_\eps(b) \in \Tilt_-(\R)$ then $T_\eps(b)$ has a $\nabla$-flag, so it has 
finite injective dimension thanks to Corollary~\ref{inparticulartiltings}.
Conversely, suppose that $T_\eps(b)$ has finite injective dimension.
Since $T_\eps(b) \in \Tilt_\eps(b)$, it has both a $\Delta_\eps$-flag and a $\nabla_\eps$-flag. Hence, as $\R$ is fully stratified,
it has both a $\bar\Delta$-flag 
and a $\bar\nabla$-flag. 
To show that $T_\eps(b) \in \Tilt_-(\R)$, it remains to show that $T_\eps(b)$ has a $\nabla$-flag.
This follows from the homological criterion (Theorem~\ref{gf1}) 
if we can show that $\Ext^1_\R(\bar\Delta(c), T_\eps(b)) = 0$ for all $c \in \B$.
By assumption, $T_\eps(b)$ has finite injective dimension, so there is a greatest $d$ 
such that $\Ext^d_\R(\bar\Delta(a), T_\eps(b)) \neq 0$ for some $a \in \B$.
Now the goal is to show that $d=0$.

Suppose for a contradiction that $d \neq 0$.
Since $\Ext^d_\R(\Delta_\eps(a), T_\eps(b)) = 0$,
we must have that $a \in \B_\lambda$ for $\lambda$ with
$\eps(\lambda) = +$. By the assumption on strata,
there exists $a' \in \B_\lambda$ such that $L_\lambda(a) \hookrightarrow P_\lambda(a')$. Let
$0 = V_0 < \cdots < V_n = \Delta(a')$ be the
$\bar\Delta$-flag for $\Delta(a')$ obtained by applying the exact functor 
$j_!^\lambda$ to a composition series for $P_\lambda(a')$ chosen so that its bottom section is isomorphic to 
$L_\lambda(a)$. 
For each $r=1,\dots,n$ 
we have that $V_r / V_{r-1} \cong \bar\Delta(a_r)$ for some $a_r \in \B_\lambda$
with $a_1 = a$.
Applying $\Hom_R(?,T_\eps(b))$ to the short exact sequence
$0 \rightarrow V_{r-1} \rightarrow V_r \rightarrow \bar\Delta(a_r)\rightarrow 0$ and using 
$\Ext^{d+1}_\R(\bar\Delta(a_r), T_\eps(b)) = 0$ gives
a surjection
$\Ext^d_\R(V_r, T_\eps(b)) \twoheadrightarrow \Ext^d_\R(V_{r-1}, T_\eps(b))$.
Since $\Ext^d_\R(V_1, T_\eps(b)) \neq 0$ by the choice of $a$, we deduce that $\Ext^d_\R(V_r,T_\eps(b)) \neq 0$ for all $r=1,\dots,n$.
Taking $r = n$ gives 
$\Ext^d_\R(\Delta(a'), T_\eps(b)) \neq 0$.
This is a contradiction since $T_\eps(b)$ has a $\bar\nabla$-flag. 
\end{proof}

The following extends
\cite[Th.~2.2]{CM} from fibered highest weight categories
to fully stratified categories; cf. Remark~\ref{moremark}.

\begin{theorem}[Homological criterion for tilting-rigidity]
\label{goren}
For a finite fully stratified category $\R$, the following properties are equivalent:
\begin{itemize}
\item[(i)] $\R$ is tilting-rigid;
\item[(ii)] $\R$ is Gorenstein\footnote{All projectives have finite injective dimension and all injectives have finite projective dimension.} and all of its strata are quasi-Frobenius;
\item[(iii)] $\R$ is Gorenstein and for each $\lambda \in \Lambda$ and $b \in \B_\lambda$ the irreducible object $L_\lambda(b)$ appears in the socle of some projective in $\R_\lambda$;
\item[(iii$'$)] $\R$ is Gorenstein and for each $\lambda \in \Lambda$ and $b \in \B_\lambda$ the irreducible object $L_\lambda(b)$ appears in the head of some injective in $\R_\lambda$.
\end{itemize}
\end{theorem}

\begin{proof}
We may assume that $\R = A\fdlmod$ for a finite-dimensional algebra $A$.

\vspace{1.5mm}
\noindent
(i)$\Rightarrow$(ii).
All strata are quasi-Frobenius by Theorem~\ref{tiltingrigidstrata}.
The injective left $A$-module $A^*$ has a finite $-$-tilting resolution $0 \rightarrow T_n \rightarrow\cdots\rightarrow T_1 \rightarrow T_0 \rightarrow A^* \rightarrow 0$ by Corollary~\ref{inparticulartiltings2}. As $\R$ is tilting-rigid, this is also a finite $+$-tilting-resolution, so each $T_i$ has a $\Delta$-flag.
Using Corollary~\ref{inparticulartiltings}, it follows that each $T_i$ has finite projective dimension.
We deduce that $A^*$ has finite projective dimension 
by arguing as in the proof of \cite[Th.~4.3.1]{Wei}; 
cf. the proof of (2)$\Rightarrow$(1) from \cite[Th.~2.2]{CM}. The dual argument
gives that $A$ has finite injective dimension.
Hence, $A$ is Gorenstein.

\vspace{1.5mm}
\noindent
(ii)$\Rightarrow$(iii), (iii$'$). This follows immediately
since $P_\lambda(b) \cong I_\lambda(\nu(b))$ for all $b \in \B_\lambda$, where $\nu$ is the Nakayama permutation.

\vspace{1.5mm}
\noindent
(iii)$\Rightarrow$(i).
It suffices to show that each $T_+(b)$ belongs to $\Tilt_-(\R)$.
As $\bigoplus_{b \in \B} T_+(b)$ is tilting in the general sense of tilting theory (cf. the discussion before Conjecture~\ref{wakconj}),
the assumption that $A$ is Gorenstein together with \cite[Lem.~1.3]{HU} 
implies that $\bigoplus_{b\in \B} T_+(b)$ is cotilting. Hence, it has finite injective dimension, so
each $T_+(b)$ has finite injective dimension.
Then we apply Lemma~\ref{proofofconj}(1) with $\eps = +$.

\vspace{1.5mm}
\noindent
(iii$'$)$\Rightarrow$(i).
This follows by 
the dual argument to the proof of (iii)$\Rightarrow$(i).
\end{proof}

\begin{corollary}\label{gorencor}
If $\R$ is a finite fibered highest weight category, it is 
tilting-rigid if and only if it is Gorenstein.
\end{corollary}

\begin{proof}
In a fibered highest weight category each stratum has a unique irreducible object (up to isomorphism). Therefore the second parts of (iii) and (iii$'$) in Theorem~\ref{goren} hold automatically. 
\end{proof}

Now we are going to consider the Ringel dual $\R'$ of a 
tilting-rigid fully stratified 
category $\R$
as in Definitions~\ref{thesetup}, \ref{rd1}, \ref{rd2} or \ref{bored} (depending on the setting).
These definitions all involve the choice of a sign function $\eps$ and the choice of an $\eps$-tilting generator $T$.
By (\ref{walk}), an $\eps$-tilting generator for some choice of $\eps$ is an $\eps$-tilting generator for all $\eps$,
 so it makes sense to drop the prefix $\eps$, referring 
 to $T$ simply as a tilting generator.
 Fixing such a choice, let $\R'$ be the corresponding Ringel dual category, and
let $F$ and $G$ be the Ringel duality functors from those definitions together with (\ref{Gagaineq}) and (\ref{Fagaineq}) 
in the lower finite and upper finite cases, respectively.
Note these functors only depend on the choice of tilting generator, not on the choice of sign function $\eps$, i.e., they are the same 
functors for all $\eps$.
For each $\lambda \in \Lambda$, there are now
{\em two}
equivalences of categories
\begin{align}\label{labelF}
F_\lambda &= 
(j')^\lambda \circ (i')_{\geq \lambda}^! \circ \F\circ i_{\leq \lambda} \circ j^\lambda_*:\R_\lambda \rightarrow \R_\lambda',\\
G_\lambda &= 
(j')^\lambda \circ (i')_{\geq \lambda}^* \circ \G\circ i_{\leq \lambda} \circ  j^\lambda_!:\R_\lambda \rightarrow \R_\lambda'\label{labelG}
\end{align}
between strata; see
Corollary~\ref{notquite} (which also holds in 
the essentially finite case) and Corollary~\ref{notquite2}.
The following lemma gives a more explicit description of these functors.

\begin{lemma}\label{aftersun}
Let $\R$ be a finite, tilting-bounded essentially finite, upper finite or lower finite  $\eps$-stratified category with the usual stratification $(L,\B,\rho,\Lambda,\leq)$.
Suppose that $\R'$ is the Ringel dual of $\R$ with respect to 
some given tilting generator $T = \bigoplus_{i \in I} T_i$
such that the index set $I$ contains $\B$ and
$T_b\:(b \in \B)$ is a direct sum of $T_\eps(b)$
and copies of $T_\eps(c)$ for $c \in \B$ with $\rho(c) < \rho(b)$.
For $\lambda \in \Lambda$, let 
$T_\lambda := \bigoplus_{b \in \B_\lambda} T_b \in \R_{\leq \lambda}$.
There is an algebra isomorphism
\begin{equation*}
\phi_\lambda:A_\lambda\stackrel{\sim}{\rightarrow}
\End_{\R_\lambda}(j^\lambda T_\lambda)^\op
\end{equation*}
between the
natural algebra realization $A_\lambda$ for the stratum $\R'_\lambda$ 
and the endomorphism algebra of $j^\lambda T_\lambda \in \R_\lambda$.
Moreover:
\begin{itemize}
\item[(1)] If $\eps(\lambda)=+$ then
$F_\lambda \cong
\Hom_{\R_\lambda}(j^\lambda T_\lambda, ?):\R_\lambda\rightarrow A_\lambda\fdlmod$
with the action of $A_\lambda$ defined via $\phi_\lambda$.
\item[(2)] If $\eps(\lambda)=-$ then
$G_\lambda \cong \Hom_{\R_\lambda}(?, j^\lambda T_\lambda)^*:\R_\lambda \rightarrow A_\lambda\fdlmod$
with the action of $A_\lambda$ defined via $\phi_\lambda$.
\end{itemize}
\end{lemma}

\begin{proof}
We just explain the argument in detail 
if $\R$ is a finite $\eps$-stratified category;
the other cases are similar but there are minor notational differences.
We have that $\R'= A\fdlmod$ for $A:=\End_\R(T)^\op$. The functors $F$ and $G$ are $\Hom_\R(T,?)$ and
$\Hom_\R(?, T)^*$, respectively.
Let $e_b \in A$ be the projection of $T$
onto $T_\eps(b)$ and set $e_\lambda := \sum_{b \in \B_\lambda} e_b$.
Let $A_{\geq \lambda}$ be the quotient of $A$ by the two-sided ideal generated by the idempotents 
$\{e_\mu\:|\:\mu \in \Lambda\text{ with }\mu\not\geq\lambda\}$. This 
is the natural realization of the Serre subcategory 
$\R_{\geq \lambda}'$ of $\R'$.
Then the stratum $\R_\lambda'$ is realized by the basic
finite-dimensional algebra $A_\lambda := \bar e_\lambda A_{\geq \lambda} \bar e_\lambda$,
where we write $\bar x$ for the canonical image 
of $x \in A$ under the quotient map $A \twoheadrightarrow A_{\geq \lambda}$.
The idempotents $\{\bar e_b\:|\:b \in \B_\lambda\}$ are representatives for the conjugacy classes of primitive idempotents in $A_\lambda$. 

By Theorem~\ref{gin}(3), $j^\lambda T_\lambda$
is a minimal projective generator for $\R_\lambda$ if $\eps(\lambda)=+$
or a minimal injective cogenerator for $\R_\lambda$ if $\eps(\lambda)=-$.
In either case, $\End_{\R_\lambda}(j^\lambda T_\lambda)^\op$
is the basic algebra realizing the stratum $\R_\lambda$. Since $\R_\lambda$ and $\R_\lambda'$ are equivalent, it follows that $A_\lambda \cong \End_{\R_\lambda}(j^\lambda T_\lambda)^\op$.
However, the argument so far does not produce the desired explicit isomorphism $\phi_\lambda$ 
between these algebras.
To obtain this, 
since we have already seen that the dimensions agree, it suffices to construct a surjective algebra homomorphism
$\phi_\lambda:A_\lambda \twoheadrightarrow \End_{\R_\lambda}(j^\lambda T_\lambda)^\op$.

Let $\R_{\geq \lambda}$ be the Serre quotient of $\R$ 
associated to the upper set $(\lambda,\infty]$, so that
$\R_{\geq \lambda}$ 
has irreducible objects labelled by $\B_{\geq \lambda}$.
Denote the quotient functor by $j^{\geq \lambda}:\R\rightarrow\R_{\geq\lambda}$.
The functor $j^{\geq \lambda}$ defines  
an algebra homomorphism 
\begin{equation}\label{above}
A = \End_{\R}(T)^\op \rightarrow 
\End_{\R_{\geq \lambda}}(j^{\geq \lambda} T)^\op.
\end{equation}
This homomorphism is surjective. To see this, Corollary~\ref{iuseitoften}(2) 
gives a short exact sequence
$0 \rightarrow j^{\geq \lambda}_! j^{\geq \lambda} T \rightarrow T \rightarrow Q \rightarrow 0$ in which $Q$ has a $\Delta_\eps$-flag. Applying
$\Hom_{\R}(?,T)$ to this gives surjectivity of the first map below:
$$
\Hom_\R(T,T) \twoheadrightarrow \Hom_\R(j^{\geq \lambda}_! j^{\geq \lambda} T, T)
\stackrel{\sim}{\rightarrow} \Hom_{\R_{\geq \lambda}}(j^{\geq \lambda}T, j^{\geq \lambda}T).
$$
The second map comes from the adjunction. The composite is the map (\ref{above}), so indeed it is surjective.
Now we note that this map sends each $e_\mu$ for $\mu\not\geq\lambda$ to zero, so it factors through the quotient $A \twoheadrightarrow A_{\geq \lambda}$
to give a surjective homomorphism $A_{\geq \lambda}
\twoheadrightarrow \End_{\R_{\geq \lambda}}(j^{\geq \lambda} T)^\op$.
Then we restrict to 
$\bar e_\lambda A_{\geq \lambda} \bar e_\lambda$
to obtain the homomorphism $\phi_\lambda$.

It just remains to prove (1) and (2).
The universal property of Serre quotients 
produces a unique fully faithful
functor $i_\lambda$ making the following diagram of functors commute:
$$
\begin{tikzcd}
  \R\arrow[r,"j^{\geq \lambda}" above] & \R_{\geq \lambda}\\
  \R_{\leq \lambda}\arrow[u,"i_{\leq \lambda}" left,hookrightarrow]\arrow[r,"j^\lambda" below] & \R_\lambda.\arrow[u, "i_\lambda" right,hookrightarrow]
\end{tikzcd}
$$
Thus, $j^{\geq \lambda} \circ i_{\leq \lambda} \cong i_\lambda \circ j^\lambda$.
Composing on the left with $j^{\geq \lambda}_*$ and on the right with
$j^\lambda_*$, using that $j^\lambda \circ j^\lambda_* \cong \Id$
and $j^{\geq \lambda}_* \circ j^{\geq \lambda} \cong \Id$ on objects in the image of
$i_{\leq \lambda} \circ j^\lambda_*$, we deduce that
\begin{equation}\label{smallcat1}
j^{\geq \lambda}_* \circ i_\lambda \cong i_{\leq \lambda} \circ j^\lambda_*.
\end{equation}
Using this, we have that
$$
F_\lambda \cong 
(j')^\lambda ((i')^!_{\geq \lambda}\Hom_{\R}(T, j^{\geq \lambda}_* (i_\lambda ?)))
\cong
\bar e_\lambda \Hom_{\R_{\geq \lambda}}(j^{\geq \lambda} T, i_\lambda ?)
= \Hom_{\R_\lambda}(j^{\lambda} T_\lambda, ?),
$$
proving (1).
The proof of (2) is similar, using
the isomorphism
$j^{\geq \lambda}_! \circ i_\lambda \cong i_{\leq \lambda} \circ j^\lambda_!$
in place of (\ref{smallcat1}).
\end{proof}

Returning to the setup before the lemma, so $\R$ is a tilting-rigid fully stratified
category and $\R'$ is its Ringel dual relative to some tilting generator $T$,
we next discuss the labelling of irreducible objects in $\R'$. 
In the general tilting-rigid setting, this depends on a 
choice of sign function $\eps$, since one needs to fix a specific labelling 
$\{T_\eps(b)\:|\:b \in \B\}$ 
of the isomorphism classes of indecomposable summands of $T$. 
Put another way, the labelling of irreducible objects in $\R'$ depends on a labelling $\{L'_b(\lambda)\:|\:b \in \B_\lambda\}$
of irreducible objects in each of the strata $\R_\lambda'$,
which we do given a choice of $\eps$ 
by declaring that
\begin{equation}\label{claws}
L'_\lambda(b) := 
\left\{
\begin{array}{ll}
F_\lambda L_\lambda(b)&\text{if $\eps(\lambda) = +$,}\\
G_\lambda L_\lambda(b)&\text{if $\eps(\lambda) = -$.}
\end{array}\right.
\end{equation}
In the next theorem, we 
see for the first time the advantage of assuming that all of the strata of $\R$ are
symmetric, or at least weakly symmetric, since then the labelling of irreducibles in $\R'$ does not depend on the choice of $\eps$ here.

\begin{theorem}[Ringel duality for tilting-rigid fully stratified categories]\label{tiltingrigiddual}
Let $\R$ be a tilting-rigid fully stratified category.
The Ringel dual $\R'$ of $\R$ with respect to some tilting generator is again tilting-rigid.
Moreover, the following hold for $\lambda \in \Lambda$:
\begin{itemize}
\item[(1)]
$\R_\lambda$ 
is weakly symmetric if and only if $F_\lambda L_\lambda(b) \cong G_\lambda L_\lambda(b)$ for all $b \in \B_\lambda$.
\item[(2)] $\R_\lambda$ is symmetric if and only if $F_\lambda \cong G_\lambda$.
\end{itemize}
\end{theorem}

\begin{proof}
Taking $\eps=+$ in the appropriate Ringel duality theorem (one of Theorems~\ref{Creek}, 
\ref{rt1}, \ref{rt2} or \ref{intfinrd})
gives that $\R'$ is $-$-stratified with indecomposable $-$-tilting objects $\{F I(b)\:|\:b \in \B\}$ in the finite, lower finite or essentially finite cases and 
$\{G P(b)\:|\:b \in \B\}$ in the finite, upper finite or essentially finite cases.
Taking $\eps=-$ gives that $\R'$ is $+$-stratified
with indecomposable $+$-tilting objects
$\{F I(b)\:|\:b \in \B\}$ in the finite, lower finite or essentially finite cases
and $\{G P(b)\:|\:b \in \B\}$ in the finite, upper finite or essentially finite cases.
It follows $\R'$ is fully stratified and its indecomposable $-$-tilting objects and
$+$-tilting objects are the same, i.e.,
$\Tilt_+(\R') = \Tilt_-(\R')$
and $\R'$ is tilting-rigid.

To prove (1) and (2), let $\eps$ be any sign function.
We may assume that the tilting generator is
$T = \bigoplus_{b \in \B} T_\eps(b)$. Let 
$T_\lambda := \bigoplus_{b \in \B_\lambda} T_\eps(b)$
and $A_\lambda
\cong \End_{\R_\lambda}(j^\lambda T_\lambda)^\op$ be as in Lemma~\ref{aftersun}.
Using the explicit descriptions of $F_\lambda$ and $G_\lambda$ from Lemma~\ref{aftersun}(1)--(2), we deduce that $F_\lambda L_\lambda(b) \cong
G_\lambda L_\lambda(b)$ if and only if
$$
\Hom_{\R_\lambda}(j^\lambda T_\lambda, L_\lambda(b))
\cong \Hom_{\R_\lambda}(L_\lambda(b), j^\lambda T_\lambda)^*
$$
as left $A_\lambda$-modules (notation as in Lemma~\ref{aftersun}).
The left hand side is the irreducible
$A_\lambda$-module associated
to the primitive 
idempotent that is the projection of $j^\lambda T_\lambda$ onto the summand isomorphic to $P_\lambda(b)$, and the right hand side is the irreducible $A_\lambda$-module associated to the primitive idempotent that is the projection of $j^\lambda T_\lambda$ onto the 
summand isomorphic to $I_\lambda(b)$.
Thus, these modules are isomorphic for all $b \in \B_\lambda$ if and only if 
$P_\lambda(b) \cong I_\lambda(b)$ for all $b \in \B_\lambda$, 
i.e., the Nakayama permutation of $\R_\lambda$ is the identity,
and $\R_\lambda$ is weakly symmetric.
This proves (1).

To prove (2), using Lemma~\ref{aftersun} again, we have that
$F_\lambda \cong G_\lambda$ if and only if there is a natural isomorphism of left $A_\lambda$-modules
$$
\Hom_{\R_\lambda}(j^\lambda T_\lambda, V) \cong \Hom_{\R_\lambda}(V,j^\lambda T_\lambda)^*
$$
for $V \in \R_\lambda$. Since $j^\lambda T_\lambda$ is a projective generator for $\R_\lambda$ and $A_\lambda=\End_{\R_\lambda}(j^\lambda T_\lambda)^\op$, 
there is such an $A_\lambda$-module 
isomorphism if and only if there is a natural vector space 
isomorphism as in (\ref{rickardsdef}) for all $P, V \in \R_\lambda$ with $P$ projective, i.e., $\R_\lambda$ is symmetric according to the definition we gave earlier.
\end{proof}

In the sequel, we will
only consider tilting-rigid fully stratified categories with the additional property that all strata are weakly symmetric.
By Theorem~\ref{tiltingrigidstrata}, 
a tilting-rigid fully stratified category has this property if and only if 
$\nu = \id$. Thus, a fully stratified category is tilting-rigid with weakly symmetric strata if and only if
\begin{equation}\label{neededinthesequel}
T_+(b) \cong T_-(b)
\end{equation}
for all $b \in \B$. In that case, 
$T_+(b) \cong T_\eps(b)$ for all sign functions $\eps$,
so that one can simply write $T(b)$ in place of $T_\eps(b)$.
Moreover, if $\R'$ is the Ringel dual category to $\R$
with respect to some tilting generator, the irreducible objects of $\R'$ are labelled unambiguously by the
set $\B$; the induced labelling of irreducible objects of the stratum $\R_\lambda'$ satisfies
\begin{equation}\label{stratumlabelling}
L'_\lambda(b) \cong F_\lambda L_\lambda(b) \cong G_\lambda L_\lambda(b)
\end{equation}
for all $\lambda \in \Lambda$ and $b \in \B_\lambda$.

\subsection{\boldmath Bases for morphism spaces between $\Delta$- and $\nabla$-filtered objects}
In this subsection, we explain how to extend 
the construction of \cite[Th.~3.1]{AST} first to
$\eps$-stratified and then to fully stratified categories.
These results will be used in the next section to construct triangular bases for endomorphism algebras of tilting generators.

\begin{theorem}\label{astfact}
Let $\R$ be a finite, lower finite or tilting-bounded essentially finite $\eps$-stratified category
with stratification 
$(\B,L,\rho,\Lambda,\leq)$.
Suppose for each $b \in \B$ that
we are given $T_b \in \Tilt_\eps(\R)$
such that $T_b$ is 
a direct sum of $T_\eps(b)$ and copies of $T_\eps(c)$ for $c$ with 
$\rho(c) < \rho(b)$.
Take $M \in\Delta_\eps(\R)$ and $N \in \nabla_\eps(\R)$.
For each $b \in \B$, choose an embedding $\iota_b:\Delta_{\eps}(b) \hookrightarrow T_b$, a
projection
$\pi_b:T_b \twoheadrightarrow \nabla_{\eps}(b)$, and subsets
$$
Y_b\subset 
\Hom_\R(M,T_b),
\qquad
X_b \subset \Hom_\R(T_b, N)
$$
so that 
$\left\{\bar y := \pi_b \circ y\:\big|\: y \in Y_b\right\}$ is a basis for
$\Hom_\R(M,\nabla_{\eps}(b))$
and 
$\left\{\bar x := x \circ \iota_b\:\big|\:x \in X_b\right\}$ is a basis for
$\Hom_\R(\Delta_{\eps}(b), N)$, as illustrated by the diagram:
\begin{equation}\label{helpful}
\begin{gathered}
\begin{xy}
  \xymatrix{
  	      &   \Delta_\eps(b) \ar@{^{(}->}[d]_{\iota_b}\ar[rd]^{\overline{x}} &  \\
      M \ar@{>}[r]^{y}\ar[rd]_{\overline{y}}     &   T_b\ar@{->>}[d]^{\pi_b} \ar@{>}[r]_{x}&  N\\
                   & \nabla_\eps(b)   &
  }
\end{xy}
\end{gathered}.
\end{equation}
Then the morphisms
$x \circ y$ for all $(y,x) \in\bigcup_{b \in \B} Y_b \times X_b$
give a basis for $\Hom_\R(M,N)$.
\end{theorem}

\begin{proof}
We proceed by induction on $\ell_{\Delta_\eps}(M)+\ell_{\nabla_\eps}(N)$
where $\ell_{\Delta_\eps}(M) := \sum_{b \in \B} (M:\Delta_\eps(b))$
and $\ell_{\nabla_\eps}(N) := \sum_{b \in \B}
(N:\nabla_\eps(b))$.
The base case is this number is zero, hence, $M=N=0$ too, 
which is trivial.
For the induction step, we can
replace $\R$ by the Serre subcategory of $\R$
associated to the lower set of $\Lambda$ generated by all
$\{\lambda\:|\:(M:\Delta_\eps(b))+N:\nabla_\eps(b)) \neq 0\text{ for some }b \in \B_\lambda\}$ to assume that there is some
maximal element $\lambda \in \Lambda$ such that
such that $(M:\Delta_\eps(b)) + (N:\nabla_\eps(b)) \neq 0$
for some $b \in \B_\lambda$.
Then we let $\Lambda\d := \Lambda \setminus \{\lambda\}$,
$\B\d := \rho^{-1}(\Lambda\d)$, and $i:\R\d\rightarrow \R$ be the natural
inclusion of the corresponding Serre subcategory of $\R$. 
Let $j:\R \rightarrow \R_\lambda$ be the quotient functor.

In this paragraph, we treat the special case $N \in \R\d$.
Let $M\d := i^* M$. Note by the choice of $\lambda$ 
that $\ell_{\Delta_\eps}(M\d) + \ell_{\nabla_\eps}(N)| 
< \ell_{\Delta_\eps}(M)+\ell_{\nabla_\eps}(N)$.
By (\ref{repeatedly}) and Theorem~\ref{gf}(2), we have that 
$M\d \in \Delta_\eps(\R\d)$, and there is a short exact sequence
$0 \rightarrow K \rightarrow M \rightarrow M\d \rightarrow 0$
where $K$ has a $\Delta_\eps$-flag with sections of the form
$\Delta_\eps(b)$ for $b \in \B_\lambda$. It follows that
the natural inclusion $\Hom_\R(M\d,N) \hookrightarrow \Hom_\R(M,N)$
is an isomorphism. For $b \in \B\d$, all of the morphisms
$\{y:M \rightarrow T_b\:|\:y \in Y_b\}$ factor through $M\d$ too.
Hence, we can apply the induction hypothesis to deduce that
the morphisms $x \circ y$ for all $(y,x) \in\bigcup_{b \in \B\d} Y_b \times X_b$
give a basis for $\Hom_\R(M\d,N)=\Hom_\R(M,N)$.
Since $X_b = \varnothing$ for $b \in \B_\lambda$, 
we have that
$\bigcup_{b \in \B} Y_b \times X_b=
\bigcup_{b \in \B\d} Y_b \times X_b$, so 
this is just what
is needed.

Now suppose that $N \notin\R\d$ and let $N\d := i^! N \in \R\d$.
We again have that
$\ell_{\Delta_\eps}(M\d) + \ell_{\nabla_\eps}(N)| 
< \ell_{\Delta_\eps}(M)+\ell_{\nabla_\eps}(N)$.
By (\ref{repeatedly}) and 
Theorem~\ref{gf}(4), we have that 
$N\d \in \nabla_\eps(\R\d)$, and
there is a short exact
sequence
$0 \rightarrow N\d\rightarrow N \stackrel{\pi}{\rightarrow} Q \rightarrow 0$
where $Q$ has a $\nabla_\eps$-flag with sections of the form
$\nabla_\eps(b)$ for $b \in \B_\lambda$.
Applying $\Hom_\R(M,?)$ to this and using Theorem~\ref{gf2} gives a short exact sequence
$$
0
\rightarrow
\Hom_\R(M,N\d) \rightarrow
\Hom_\R(M,N) \rightarrow \Hom_\R(M,Q) \rightarrow 0.
$$
For $b \in \B\d$, 
the morphisms $\{x:T_b\rightarrow N\:|\:x \in X_b\}$
have image contained in $N\d$
and are lifts of a basis for $\Hom_{\R\d}(\Delta_\eps(b), N\d)$.
By induction, we get that 
$\Hom_\R(M,N\d)$ has basis given by 
the compositions $x \circ y$ for all $(y,x) \in \bigcup_{b \in \B\d} Y_b
\times X_b$.
In view of this and the above short exact sequence, we are therefore
reduced to showing that the morphisms
$\pi \circ x \circ y$ for $(y,x) \in \bigcup_{b \in \B_\lambda} Y_b
\times X_b$ give a basis for $\Hom_\R(M,Q)$.
We have that 
$Q \cong j_* j Q$ by Corollary~\ref{iuseitoften}(1),
hence, the exact quotient 
functor $j$ defines isomorphisms
$\Hom_\R(M,Q) \stackrel{\sim}{\rightarrow} \Hom_{\R_\lambda}(j M, j Q)$.
Similarly,
$\Hom_\R(M, \nabla_\eps(b)) \stackrel{\sim}{\rightarrow}
\Hom_{\R_\lambda}(j M, j \nabla_\eps(b))$
and
$\Hom_\R(\Delta_\eps(b), N) \stackrel{\sim}{\rightarrow}
\Hom_{\R_\lambda}(j \Delta_\eps(b), j N)$
for $b \in \B_\lambda$.
Moreover, $j\pi:jN \rightarrow jQ$ is an isomorphism.
Thus, we are reduced to showing that the
morphisms
$jx \circ jy$ give a basis for
$\Hom_{\R_\lambda}(jM, jN)$ for all $(y,x) \in \bigcup_{b \in
  \B_\lambda} Y_b \times X_b$.
The sets of morphisms $\bar Y_b := 
\{jy:jM \rightarrow j T_b\:|\:y \in Y_b\}$ and $\bar X_b := 
\{jx:j T_b \rightarrow j N\:|\:x \in X_b\}$ 
appearing here are characterized equivalently as lifts of bases for
$\Hom_{\R_\lambda}(j M, j \nabla_\eps(b))$
and 
$\Hom_{\R_\lambda}(j \Delta_\eps(b), jN)$, respectively.
Let $\bar M := jM$ and $\bar N := jN$.

To complete the proof, we consider the two cases
$\eps(\lambda)=+$ and $\eps(\lambda)=-$ separately. 
The arguments are similar, so we just explain the former.
In this case, 
for $b \in \B_\lambda$, we have that 
$j \nabla_\eps(b) \cong L_\lambda(b)$ and $j \Delta_\eps(b) \cong P_\lambda(b)
\cong j T_b$ by Theorem~\ref{gin}(3). 
The module $\bar M$ is projective in $\R_\lambda$.
We are trying to show 
that the morphisms
$\bar x \circ \bar y$ 
for all $(\bar y,\bar x) \in \bigcup_{b \in \B_\lambda} \bar Y_b \times \bar X_b$
give a basis for $\Hom_{\R_\lambda}(\bar M,\bar N)$
where:
\begin{itemize}
\item$
\bar Y_b \subset \Hom_{\R_\lambda}(\bar M, P_\lambda(b))$ 
is a set lifting a basis of
$\Hom_{\R_\lambda}(\bar M,L_\lambda(b))$;
\item
$\bar X_b$ is a basis of $\Hom_{\R_\lambda}(P_\lambda(b), \bar N)$.
\end{itemize}
Since $\bar M$ is projective, the proof reduces to the case that
$\bar M = P_\lambda(b)$, when the assertion is clear.
\end{proof}

The following restatement in the special case of a highest weight categories
recovers \cite[Th.~3.1]{AST}.

\begin{corollary}
\label{astfactexactly}
Let $\R$ be a finite, lower finite or tilting-bounded essentially finite highest weight category
with poset $(\Lambda,\leq)$ and labelling function $L$.
Suppose for each $\lambda \in \Lambda$ that
we are given $T_\lambda \in \Tilt(\R)$
such that $T_\lambda$ is 
a direct sum of $T(\lambda)$ and copies of $T(\mu)$ for $\mu < \lambda$.
Take $M \in\Delta(\R)$ and $N \in \nabla(\R)$.
For each $\lambda \in\Lambda$, 
choose
an embedding $\iota_\lambda:\Delta(\lambda) \hookrightarrow T_\lambda$, a
projection
$\pi_\lambda:T_\lambda \twoheadrightarrow \nabla(\lambda)$, and subsets
$$
Y_\lambda\subset 
\Hom_\R(M,T_\lambda),
\qquad
X_\lambda \subset \Hom_\R(T_\lambda, N)
$$
so that 
$\left\{\bar y := \pi_\lambda \circ y\:\big|\: y \in Y_\lambda\right\}$ is a basis for
$\Hom_\R(M,\nabla(\lambda))$
and 
$\left\{\bar x := x \circ \iota_b\:\big|\:x \in X_b\right\}$ is a basis for
$\Hom_\R(\Delta(\lambda), N)$.
Then the morphisms
$x \circ y$ for all $(y,x) \in\bigcup_{\lambda \in \Lambda} Y_\lambda \times X_\lambda$
give a basis for $\Hom_\R(M,N)$.
\end{corollary}

For tilting-rigid 
fully stratified categories, 
there is a more refined version of Theorem~\ref{astfact}.

\begin{theorem}\label{upgradedastfact}
Let $\R$ be a finite, lower finite or
essentially finite fully stratified category
with stratification 
$(\B,L,\rho,\Lambda,\leq)$ such that $\R$ is tilting-rigid with weakly symmetric strata.
Suppose for each $b \in \B$ that
we are given $T_b \in \Tilt(\R)$
such that $T_b$ is 
a direct sum of $T(b)$ and copies of $T(c)$ for $c$ with 
$\rho(c) < \rho(b)$.
Take $M \in \Delta(\R)$ and $N \in \nabla(\R)$.
For $a,b \in \B$, choose
embeddings
$\iota_a:\Delta(a) \hookrightarrow T_a$,
$\bar\iota_b:\bar\Delta(b)\hookrightarrow T_b$, 
projections
$\bar\pi_a:T_a \twoheadrightarrow \bar\nabla(a), \pi_b:T_b\twoheadrightarrow \nabla(b)$, and subsets
$$
Y_a
\subset 
\Hom_\R(M,T_a),
\qquad
H(a,b) \subset
\Hom_\R(T_a, T_b),
\qquad
X_b \subset \Hom_\R(T_b, N)
$$ 
so that 
$\big\{\bar y := \bar \pi_a \circ y\:\big|\:y \in Y_a\big\}$ is a basis for
$\Hom_\R(M,\bar\nabla(a))$,
$\big\{\bar h := \pi_b \circ h \circ \iota_a\:\big|\:h \in H(a,b)\big\}$ is a basis for
$\Hom_\R(\Delta(a), \nabla(b))$,
and 
$\big\{\bar x := x \circ \bar\iota_b\:\big|\:x \in X_b\big\}$ is a basis for
$\Hom_\R(\bar\Delta(b), N)$, as illustrated by the diagram:
\begin{equation}\label{morehelpful}
\begin{gathered}
\begin{xy}
  \xymatrix{
  	      &   \Delta(a)
              \ar@{^{(}->}[d]_{\iota_a}\ar@{->}[r]^{\bar h}
&\nabla(b)&  \\
      M \ar@{>}[r]^{y}\ar[rd]_{\overline{y}}     &   T_a\ar@{->>}[d]^{\bar\pi_a} \ar@{>}[r]_{h}&T_b  \ar@{->>}[u]_{\pi_b}\ar@{>}[r]^{x}&  N\\
                   & \bar\nabla(a)&\bar\Delta(b)\ar[ru]_{\overline{x}}\ar@{^{(}->}[u]^{\bar\iota_b}   &
  }
\end{xy}
\end{gathered}.
\end{equation}
Then the morphisms
$x\circ h \circ y$ for 
all
$(y,h,x) \in 
\bigcup_{a,b \in \B}
Y_a \times H(a,b) \times X_b$
give a basis for $\Hom_\R(M,N)$.
\end{theorem}

\begin{proof}
This follows by the same strategy as was used in the proof of
Theorem~\ref{astfact}. The only substantial difference is in the final
paragraph of the proof. By that point,
we have reduced to showing for
projective and injective 
objects $\bar M, \bar N \in \R_\lambda$, respectively,
that the morphisms $\bar x \circ \bar h \circ \bar y$ for all 
$(\bar y,\bar h,\bar x) \in \bigcup_{a,b \in \B_\lambda} \bar Y_a \times \bar H(a,b)
\times \bar X_b$ give a basis for $\Hom_{\R_\lambda}(\bar M, \bar N)$ where:
\begin{itemize}
\item
$\bar Y_a \subset \Hom_{\R_\lambda}(\bar M, P_\lambda(a))$ is a set lifting a
basis of
$\Hom_{\R_\lambda}(\bar M, L_\lambda(a))$;
\item$
\bar H(a,b)$ is a basis for
$\Hom_{\R_\lambda}(P_\lambda(a), I_\lambda(b))$;
\item
$\bar X_b \subset \Hom_{\R_\lambda}(I_\lambda(b), \bar N)$ is a set lifting a
basis of $\Hom_{\R_\lambda}(L_\lambda(b), \bar N)$.
\end{itemize}
Using 
that $\bar M$ is projective and $\bar N$ is injective, the proof of this
reduces to the case that $\bar M = P_\lambda(a)$ and $\bar N = I_\lambda(b)$,
when the assertion is clear.
\end{proof}

\subsection{Chevalley dualities}
Finally, in this section we discuss some further aspects of Ringel duality. These results will be used in the next section to construct symmetric triangular bases for endomorphism algebras of tilting generators.
Like in $\S$\ref{trsec}, the phrase ``fully stratified category" means a fully stratified category $\R$ that is either finite, essentially finite, upper finite or lower finite.
 
Given a finite-dimensional algebra $A$ and
an algebra anti-automorphism
$\sigma:A \rightarrow A$,
there is a contravariant autoequivalence
\begin{equation}\label{notchevduality}
?^\sigmadual:A\fdlmod \rightarrow A\fdlmod
\end{equation}
taking $V$ to its linear dual $V^*$ viewed as a left module by restricting the natural right action along $\sigma$.
If $\R$ is a finite Abelian category and $?^\vee:\R\rightarrow\R$ is a contravariant autoequivalence, we call a pair $(A, \sigma)$ consisting of a finite-dimensional algebra $A$ and an anti-automorphism $\sigma$
a {\em realization} of $(\R,?^\vee)$ if 
there is an equivalence of categories $F:\R\rightarrow A\fdlmod$
such that $F \circ ?^\vee \cong ?^\sigmadual \circ F$.
The following lemma shows that any contravariant autoequivalence of $\R$ admits a realization in this sense. 
In fact, we will only ever consider contravariant autoequivalences that preserve isomorphism classes of irreducible objects, in which case we can say a little more about $\sigma$ as explained at the end of the lemma.

\begin{lemma}\label{sunlem1} 
Let $A$ be a finite-dimensional algebra.
Suppose that $?^\vee$ 
 is a contravariant autoequivalence of $A\fdlmod$.
Then there exists an algebra 
anti-automorphism $\sigma:A \rightarrow A$
such that $?^\vee \cong ?^\sigmadual$.
Moreover, 
if $?^\vee$ preserves isomorphism classes of irreducible $A$-modules, 
then $\sigma$ can be chosen so that it fixes each of
a given set 
$\{e_i\:|\:i \in I\}$ 
of mutually orthogonal idempotents in $A$.
\end{lemma}

\begin{proof}
Consider the functor
$F := ?^* \circ ?^\vee:A\fdlmod \rightarrow A^{\op} \fdlmod$.
Since this is right exact and preserves direct sums, we have that 
$F \cong FA \otimes_A ?$ where $FA$ is the $(A^\op,A)$-bimodule obtained
by applying $F$ 
to the regular $(A,A)$-bimodule $A$.
Note that 
the right action of $x \in A$ on $FA$ here is defined by applying $F$ to the left $A$-module homomorphism $r_x:A \rightarrow A, a\mapsto ax$. 

Viewing $A$ as a left $A^\op$-module with action $x\cdot y := yx$,
we claim that $FA \cong A$ as left $A^\op$-modules.
To see this, let $\{L(b)\:|\:b \in \B\}$ be a full set of pairwise inequivalent irreducible left $A$-modules. Then $A \cong \bigoplus_{b \in \B} P(b)^{\oplus \dim L(b)}$ as left $A$-modules,
where $P(b)$ is the projective cover of $L(b)$.
Let $\B \rightarrow \B, b \mapsto b'$ be the bijection defined from
$L(b)^\vee \cong L(b')$. 
Then $P(b)^\vee \cong I(b')$, the injective hull of $L(b')$.
Hence $F P(b) \cong I(b')^*$ as left $A^\op$-modules.
Here, $I(b)^*$ is the projective cover of the left $A^\op$-module $L(b)^*$.
Using that $\dim L(b) = \dim L(b')^*$, we deduce that 
$$
F A \cong \bigoplus_{b \in \B} (I(b')^*)^{\bigoplus \dim L(b)} \cong \bigoplus_{b \in \B} (I(b)^*)^{\bigoplus \dim L(b)^*}
\cong A
$$ 
as left $A^\op$-modules. This proves the claim.
Similarly, under the additional hypothesis that $?^\vee$ preserves isomorphism classes of irreducible objects and we are given mutually orthogonal idempotents $\{e_i\:|\:i \in I\}$,
we get that $F (A e_i) \cong e_i A$ as left $A^\op$-modules for each $i \in I$.

Now we let
$\phi:FA \stackrel{\sim}{\rightarrow} A$ be some choice of a left $A^\op$-module isomorphism. When the additional hypothesis holds, we may pick this so that it restricts to isomorphisms 
$F (Ae_i) \stackrel{\sim}{\rightarrow} e_i A$ for each $i \in I$.
Transporting the right $A$-module structure on $FA$ through 
$\phi$,
we make the left $A^\op$-module $A$
into an $(A^\op,A)$-bimodule, which we will denote by $A_{\sigma^{-1}}$. 
Explicitly, left action of $x \in A^{\op}$ on $y\in A_{\sigma^{-1}}$ is given by $x \cdot y := yx$ as in the previous paragraph, while
the new right action of $x\in A$ is 
by $y \cdot x := (\phi \circ((r_x)^\vee)^*\circ \phi^{-1})(y)$. 
Since $\End_{A^\op}(A) \cong A$, this right action of $x$
can be written as left multiplication by a unique element
$x' \in A$. 
The resulting map $A \rightarrow A, x \mapsto x'$ is an algebra anti-automorphism. Let $\sigma:A \rightarrow A$ be the inverse anti-automorphism.
Note then that the right action of $x \in A$ on $y \in A_{\sigma^{-1}}$
is by $y \cdot x = \sigma^{-1}(x) y$, explaining our earlier choice of notation.
When the additional hypothesis holds, the choice of $\phi$ ensures that $(e_i)' = e_i$ for $i \in I$, hence, $\sigma(e_i) = e_i$ for each $i \in I$.

For a left $A$-module $V$, let ${_\sigma}V$ be $V$ viewed instead as a left $A^\op$-module by restricting along $\sigma$.
Then ${_\sigma} A$ is an $(A^\op,A)$-bimodule which 
is isomorphic via 
$\sigma:A_{\sigma^{-1}} \stackrel{\sim}{\rightarrow} {_\sigma}A$ to the $(A^\op,A)$-bimodule $A_{\sigma^{-1}} \cong FA$ 
from the previous paragraph. Thus, we have shown that 
$F \cong A_{\sigma^{-1}} \otimes_A ? \cong {_\sigma}A \otimes_A ?  \cong {_\sigma}?:A\fdlmod \rightarrow A^\op \fdlmod$.
Applying $?^*$ gives finally that $?^\vee \cong ?^\sigmadual$.
\end{proof}

\begin{remark}
In the setup of Lemma~\ref{sunlem1}, assume that $?^\vee$ preserves isomorphism classes of irreducible $A$-modules. 
Then we can take the set of mutually orthogonal idempotents at the end of the lemma to be a mutually orthognal set $\{e_b\:|\:b \in \B\}$ of representatives for the conjugacy classes of primitive idempotents in $A$. Then the lemma shows that we can choose the anti-automorphism $\sigma$ so that $\sigma(e_b) = e_b$ for all $b \in \B$. Conversely, if $\sigma:A \rightarrow A$ is an anti-automorphism fixing such a set of reprentatives for the conjugacy classes of primitive idempotents on $A$, it is obvious that the contravariant autoequivalence $?^\sigmadual$ preserves isomorphism classes of irreducible $A$-modules.
\end{remark}

To adapt the above from finite Abelian categories 
to essentially finite Abelian
categories, Schurian categories or locally finite Abelian categories, we need the following definitions:
\begin{itemize}
\item
If $A = \bigoplus_{i,j \in I} e_i A e_j$ is an
essentially or locally finite-dimensional locally unital algebra, a locally unital algebra anti-automorphism $\sigma:A \rightarrow A$
gives rise to 
a contravariant autoequivalence 
$?^\sigmadual$ of the categories $A\fdlmod$ or $A\lfdlmod$, respectively. This is defined by first applying 
the usual duality from left modules to right modules, either $?^*:A\fdlmod\rightarrow \fdrmod A$
or $?^\circledast:A\lfdlmod \rightarrow \lfdrmod A$ depending on the case,
and then converting right modules back to left modules by restricting along
$\sigma$.
\item
If $A$ is a pseudo-compact topological algebra, that is, $A \cong C^*$ for a coalgebra $C$, 
 an algebra anti-automorphism $\sigma:A \rightarrow A$ gives rise to a contravariant autoequivalence $?^\sigmadual$ of $A\fdlmod \cong \fdrcomod C$.
Note in this case that $\sigma$ is necessarily continuous so that 
it is the dual of a coalgebra anti-automorphism $\sigma^\star:C\rightarrow C$; the definition of the 
duality $?^\sigmadual$ could also be formulated in terms of comodules using $\sigma^\star$.
\end{itemize}
Then given an essentially finite Abelian category, a Schurian category, or a locally finite Abelian category $\R$
with a contravariant autoequivalence $?^\vee$, a {\em realization} of $(\R,?^\vee)$
means a pair $(A, \sigma)$ consisting of an algebra $A$ and an anti-automorphism $\sigma:A\rightarrow A$
of the appropriate type 
such that $?^\sigmadual \circ F \cong F \circ ?^\vee$
for some equivalence $F$ from $\R$
to $A\fdlmod, A\lfdlmod$ or $A\fdlmod$, respectively.
The following lemmas are analogs of Lemma~\ref{sunlem1}
in each of these new settings.

\begin{lemma}\label{sunlem3}
Suppose that $A = \bigoplus_{i,j \in I} e_i A e_j$ 
is either an essentially or a locally finite-dimensional
locally unital algebra.
Let $?^\vee$ be a contravariant autoequivalence of $A\fdlmod$ or $A\lfdlmod$, respectively, which preserves isomorphism classes of irreducible objects.
There exists a locally unital algebra anti-automorphism
$\sigma:A \rightarrow A$ such that $?^\vee \cong ?^\sigmadual$.
\end{lemma}

\begin{proof}
In the locally finite-dimensional case, let $F := ?^\circledast \circ ?^\vee:A\lfdlmod \rightarrow A^\op\lfdlmod$.
Viewing 
$\bigoplus_{i \in I} F(A e_i)$ as an $(A^\op, A)$-bimodule in the natural way, 
we have tat
$F \cong \left(\bigoplus_{i \in I} F(A e_i) \right) \otimes_A ?$.
Then we observe for each $i \in I$ 
that $F(A e_i) \cong e_i A$ as left $A^\op$-modules as $?^\vee$ 
preserves isomorphism classes of irreducibles. Now argue as 
in proof of Lemma~\ref{sunlem1}. The essentially finite-dimensional case is similar.
\end{proof}

\begin{lemma}
\label{sunlem2}
Suppose that $A$ is a pseudo-compact topological algebra.
Let $?^\vee$ be a contravariant autoequivalence of $A\fdlmod$
which preserves isomorphism classes of irreducible objects. Then
there exists an algebra anti-automorphism $\sigma:A \rightarrow A$
such that $?^\vee \cong ?^\sigmadual$.
Moreover, given a family
$\{e_i\:|\:i \in I\}$ 
of mutually orthogonal idempotents in $A$,
$\sigma$ can be chosen so that $\sigma(e_i) = e_i$ for all $i \in I$.
\end{lemma}

\begin{proof}
The functor $?^\vee:A\fdlmod\rightarrow A\fdlmod$ 
extends to 
$?^\vee:A\pclmod \rightarrow A\dslmod$ with $(\varprojlim V_\omega)^\vee
:= \varinjlim (V_\omega^\vee),$
taking limits over finite-dimensional submodules $V_\omega \leq V$.
Composing  with 
$?^*$
gives an equivalence $F := ?^* \circ ?^\vee:A\pclmod \rightarrow A^\op \pclmod$.
Moreover, for each $i \in I$ we have that 
$F (A e_i) \cong e_i A$ as a $(A^\op,A)$-bimodule as $?^\vee$ 
preserves isomorphism classes of irreducibles.
Then we argue as in Lemma~\ref{sunlem1} to obtain an 
algebra anti-automorphism $\sigma:A\rightarrow A$
 with $\sigma(e_i) = e_i$ for each $i \in I$ such that 
 $F$ is isomorphic to the functor $A\pclmod \rightarrow A^\op \pclmod$ defined by restriction along $\sigma$.
The lemma follows on composing with $?^\star$ then 
restricting to $A\fdlmod$.
 \end{proof}

With these preliminaries in place, we can now prove 
a result which explains how to transfer a
contravariant autoequivalence 
on a fully stratified category to its Ringel dual.

\begin{theorem}[Dualities commute with Ringel duality]
\label{tiltingduality}
Suppose that $\R$ is a fully stratified category
with stratification $(\B, L, \rho, \Lambda,\leq)$
such that $\R$ is tilting-rigid with weakly symmetric strata, i.e., (\ref{neededinthesequel}) holds. 
Assume also that $\R$ possesses a 
contravariant autoequivalence $?^\vee$ which preserves isomorphism classes of irreducible objects. 
Then we have that 
$T(b)^\vee \cong T(b)$ for all $b \in \B$.
Moreover, letting $\R'$ be the Ringel dual category with respect to some choice of tilting generator and $F, G$ be the usual Ringel duality functors, 
there is an induced contravariant autoequivalence  $?^\wedge$ on 
$\R'$ preserving isomorphism classes of irreducible objects such
that 
\begin{equation}\label{asint}
F \circ ?^\vee \cong\: ?^\wedge \circ G,
\qquad\qquad
G \circ ?^\vee \cong\: ?^\wedge \circ F
\end{equation}
whenever these functors make sense (e.g., these isomorphisms always hold on $\Delta_\eps(\R)$ and on $\nabla_\eps(\R)$, respectively, for any choice of $\eps$).
\end{theorem}

\begin{proof}
We just explain the proof in the case that $\R$ is a finite fully stratified category, leaving the minor modifications needed in the other three cases to the reader.
By Lemma~\ref{sunlem1}, we may assume 
that $\R = A\fdlmod$ for a finite-dimensional algebra $A$
 and that $?^\vee:\R\rightarrow \R$ is 
the functor $?^\sigmadual$ taking a left $A$-module $V$
 to the dual right $A$-module viewed as a left module by 
 restricting the natural right action along
 some given anti-automorphism $\sigma:A \stackrel{\sim}{\rightarrow} A$.
(In the other three cases, one needs to use
Lemmas~\ref{sunlem3}--\ref{sunlem2} here in place of Lemma~\ref{sunlem1}.)

Since $T_+(b)$ has a $\Delta$-flag with $\Delta(b)$ at the bottom, and 
also a $\bar\nabla$-flag, we see using Lemma~\ref{dualitylemma} that $T_+(b)^\vee$ has a $\nabla$-flag with $\nabla(b)$ at the top, and also a $\bar\Delta$-flag.
So it is isomorphic to $T_-(b)$.
As $\R$ is tilting-rigid, $T(b) := T_+(b) \cong T_-(b)$, so we have shown that $T(b)^\vee \cong T(b)$ for all $b \in \B$.

We are given some full tilting module $T$ defining the Ringel dual category $\R'$,
i.e., $\R' = B\fdlmod$ for $B = \End_A(T)^\op$.
From the previous paragraph, we get that
 $T \cong T^\vee$. 
Let $\phi:T \stackrel{\sim}{\rightarrow} T^\vee$ 
be an isomorphism of left $A$-modules.
Equivalently, $\phi$ is the data of a non-degenerate pairing
$\langle\cdot,\cdot\rangle:T \times T \rightarrow \k$ 
with $\langle v,w \rangle := \phi(v)(w)$, and we have that
$\langle xv, w \rangle = \langle v, \sigma(x) w\rangle$ for $v, w \in T$, $x \in A$.
Let $\tau:B \rightarrow B$ be 
the anti-automorphism of $B$ defined
so that $\langle v y, w \rangle = \langle v, w \tau(y)\rangle$
for $v,w \in T$, $y \in B$. It follows that $\phi$ is also an isomorphism of right $B$-modules for the right $B$-module structure on $T^\vee$ obtained by restricting its natural left action on $T^*$ along $\tau$.
Now we can define the contravariant autoequivalence $?^\wedge:B\fdlmod \rightarrow B\fdlmod$
to be $?^\taudual$.

In this paragraph, we check (\ref{asint}).
We just prove the first of these isomorphisms; the latter follows from former (with the roles of $A$ and $B$ reversed) on taking adjoints.
Take $V \in \R$.
Then we have natural left $B$-module isomorphisms
$$
(G V)^{\wedge} \cong \Hom_A(V, T)
\cong \Hom_A(T^\vee, V^\vee) 
\cong \Hom_A(T, V^\vee) = F (V^\vee),
$$
as required. (On the space $\Hom_A(V,T)$ here, the left $B$-module structure is defined by restricting the natural right action along $\tau$.)

It remains to check that $?^\wedge$ preserves isomorphism classes of irreducible objects in $\R'$.
Since the strata are weakly symmetric,
we have that
$$
\nabla'(b)^{\wedge} \cong (G \Delta(b))^{\wedge}
\cong F(\Delta(b)^\vee)
\cong F \nabla(b) \cong \Delta'(b).
$$
This implies that $L'(b)^{\wedge} \cong L'(b)$.
\end{proof}

In examples coming from Lie theory, highest weight categories
usually come equipped with dualities arising from 
anti-involutions which restrict to the identity on the Cartan part. 
The material in the rest of the subsection is an attempt to axiomatize the essential features of such dualities in the more general setting of fully stratified categories.
We start with a definition which will be relevant at the level of strata.

 \begin{definition}\label{sigmasymmetric}
 Let $A$ be a finite-dimensional algebra and $\sigma:A\rightarrow A$
be an anti-involution.
 We say that $A$ is {\em $\sigma$-symmetric} if the following hold:
 \begin{itemize}
 \item[($\sigma$S1)]
 There is a set $\{e_b\:|\:b \in \B\}$ of representatives for the conjugacy classes of primitive idempotents in $A$ such that
 $\sigma(e_b) = e_b$ for all $b \in \B$.
\item[($\sigma$S2)]
There is a non-degenerate associative symmetric bilinear form
 $(\cdot,\cdot):A \times A \rightarrow \k$ such that
 $(x,y) = (\sigma(x), \sigma(y))$
 for all $x, y \in A$.
 \end{itemize}
 \end{definition}
 
If $A$ is $\sigma$-symmetric
in the sense of Definition~\ref{sigmasymmetric}
then it is a symmetric algebra in the usual sense. Moreover, every
 finitely generated projective left $A$-module $P$ possesses a non-degenerate
 symmetric bilinear form $\langle\cdot,\cdot\rangle$
 such that $\langle xv, w \rangle = \langle v, \sigma(x) w \rangle$
 for $v,w \in P, x \in A$; in particular, $P \cong P^\sigmadual$. 
 To see this,
 we may assume without loss of generality that $P$ is 
 indecomposable and that
$P = A e$ 
  for a $\sigma$-invariant primitive idempotent $e$. Then the form
 $\langle\cdot,\cdot \rangle:P \times P \rightarrow \k$ defined
in terms of the given $\sigma$-symmetric form $(\cdot,\cdot)$ 
on $A$ by
$\langle v, w \rangle := (\sigma(v), w)$
for $v,w \in P$
 has these properties; it is non-degenerate because 
 by associativity 
 \begin{equation}
 \label{orthogdec}
 A = eAe \oplus [e A (1-e) + (1-e) A e] \oplus (1-e) A (1-e)
 \end{equation}
 is an orthogonal decomposition 
 of $A$ with respect to $(\cdot,\cdot)$ and the subspaces $eA(1-e)$ and $(1-e)Ae$ are isotropic.
 
 The following lemma shows that $\sigma$-symmetry is preserved by 
 Morita equivalence. The basic point underlying this is that if $A$ is $\sigma$-symmetric and $e \in A$ is a $\sigma$-invariant idempotent, then 
 $\sigma$ restricts to an anti-involution of $eAe$. 
 Moreover a $\sigma$-symmetric 
form $(\cdot,\cdot)$ on $A$ restricts to such a form on $eAe$ so that $eAe$ is also $\sigma$-symmetric; the non-degeneracy of this restriction follows
 from the orthogonal decomposition (\ref{orthogdec}).

\begin{lemma}\label{chevinvlem}
Let $A$ be a 
finite-dimensional algebra which is $\sigma$-symmetric
for some anti-involution $\sigma$.
Let $B$ be another finite-dimensional algebra that is Morita equivalent to $A$,
so that there is an equivalence of categories
$F:B\fdlmod \rightarrow A\fdlmod$.
Then $B$ possesses an anti-involution $\tau:B \rightarrow B$
such that $?^\sigmadual \circ F \cong F \circ ?^\taudual$, and 
$B$ is $\tau$-symmetric for any such anti-involution $\tau$.
Moreover, $\tau$ can be chosen in such a way that it fixes each of some given set
 $\{f_i\:|\:i \in I\}$ of mutually orthogonal idempotents in $B$.
\end{lemma}

\begin{proof}
Let
$\{e_b\:|\:b \in \B\}$ be a set of mutually orthogonal
representatives for the conjugacy classes of primitive idempotents in $A$
with $\sigma(e_b) = e_b$ for all $b$.
Let $e := \sum_{b \in \B} e_b$. Then $e A e$ is the basic algebra that is Morita equivalent to $A$, and it is $\sigma$-symmetric too.
The functors $?^\sigmadual$ on $A\fdlmod$ and $eAe\fdlmod$ obviously commute with the idempotent truncation functor giving an equivalence
$A\fdlmod \rightarrow eAe \fdlmod$.
All of this means that we can replace $A$ with $eAe$ if necessary to assume that $A$ itself is {\em basic} with 
$1 = \sum_{b \in \B} e_b$
being a decomposition of its identity element into 
mutually orthogonal $\sigma$-invariant primitive idempotents.

Now suppose that $B$ is Morita equivalent to $A$
via some given $F:B\fdlmod \rightarrow A\fdlmod$. 
Let $P := FB$ be the $(A,B)$-bimodule obtained by applying $F$ to the regular $(B,B)$-bimodule. Note that $P = \bigoplus_{i \in I} P f_i$ where
$\{f_i\:|\:i \in I\}$ is the given set of mutually orthogonal idempotents in $B$;
we are assuming here that $\sum_{i \in I} f_i = 1_B$ which we can clearly do
by adding one more idempotent to this set if necessary.
As an $A$-module, we have for each $i \in I$ that 
$P f_i\cong \bigoplus_{b \in \B}
A e_b^{\oplus d_i(b)}$ for integers $d_i(b) > 0$; the numbers
$d(b) = \sum_{i \in I} d_i(b)$ are the dimensions of the 
irreducible $B$-modules.
Moreover, $e_i B e_j \cong \End_A(P e_i, P e_j)^\op$. 
Fixing such isomorphisms, we may assume
simply that $P = \bigoplus_{i \in I} P f_i$ with $P f_i = \bigoplus_{b \in \B} A e_b^{\oplus d_i(b)}$,
$B = \End_A(P)^\op$ with $f_i$ being the projection
of $P$ onto the $i$-th summand $P f_i$, and $F=P \otimes_B ?$.

Next we observe that
$B = \End_A(P)^\op$ is isomorphic to an algebra of block matrices, with
blocks indexed by the set $I \times \B$,
and the block in the row indexed by $(i,a)$ and column indexed
by $(j,b)$ being a $d_i(a) \times d_j(b)$ matrix with entries in 
$e_a A e_b$. The multiplication is just matrix multiplication combined with
multiplication in $A$. From this description, it is clear 
that $B$ possesses an anti-involution
$\tau$ defined by taking the transpose of a matrix and applying
$\sigma$ to all of the entries of the result.
For $i \in I, b \in \B$ and $1 \leq r \leq d_i(b)$,
let $f_{i,b;r} \in B$ be the matrix with all entries equal to zero except
for the $r$-th entry in its $(i,b)$-th diagonal block, which is equal
to $e_b$.
This is a primitive idempotent in $B$, and it is fixed by $\tau$.
This verifies the axiom ($\sigma$S1) for this particular anti-involution $\tau$ of $B$.
Next we check that the axiom ($\sigma$S2) is satisfied.
Let $\operatorname{tr}:A \rightarrow \k, x \mapsto (1_A,x)$ be the trace function associated to a
$\sigma$-symmetric form on $A$.
Define $\operatorname{tr}':B \rightarrow \k$
by mapping a matrix in $B$ to the sum of the scalars obtained
by applying $\operatorname{tr}$ to each of its diagonal entries.
Then let $(\cdot,\cdot)':B \times B \rightarrow \k$ be the 
bilinear form defined from $(x,y)' := \operatorname{tr}'(xy)$.
This is a non-degenerate
symmetric bilinear form on $B$ with
$(\tau(x), \tau(y))' = (x,y)'$.

It is clear that $F \circ ?^\taudual \cong ?^\sigmadual \circ F$
since $F$ is isomorphic to the idempotent truncation functor defined by $f := \sum f_{i,b;1}$ summing over all $i \in I, b \in \B$ such that $d_i(b) \neq 0$.
We also have that
$f_i = \sum_{b \in \B} \sum_{r=1}^{d_i(b)} f_{i,b;r}$, so $\tau(f_i) = f_i$ for each $i \in I$. 
So we have now proved the existence of an anti-involution $\tau$
with all of the desired properties. It remains to note given another
other anti-involution $\omega:B \rightarrow B$ with 
$F \circ ?^\omegadual \cong ?^\sigmadual \circ F$ 
that $?^\omegadual \circ ?^\taudual \cong \Id$, hence, we have that
$\omega\circ \tau = \gamma$ for some inner automorphism
$\gamma:B \rightarrow B, x \mapsto u x u^{-1}$; equivalently, 
$\omega = \gamma \circ \tau$.
If that is the case, then $B$ is also $\omega$-symmetric since the bilinear form 
$(\cdot,\cdot)'$ constructed in the previous paragraph also satisfies
$$
(\omega(x), \omega(y))' = (u \tau(x) u^{-1}, u \tau(y) u^{-1})
= (\tau(x), \tau(y)) = (x,y)'
$$
for $x,y \in A$.
\end{proof}

\begin{definition}\label{chevinvdef}
Let $\R$ be a fully stratified category with 
 stratification $(\B,L,\rho,\Lambda,\leq)$.
 We say that
a contravariant autoequivalence $?^\vee$ of $\R$ is a {\em Chevalley duality} if
there is a realization $(A, \sigma)$ of $(\R,?^\vee)$
in which $\sigma$ is a {\em Chevalley anti-involution},
meaning that $\sigma^2 = \id$ and the following two properties hold:
\begin{itemize}
\item[(Ch1)]
There exists a set $\{e_a\:|\:a \in \B\}$ of mutually orthogonal
$\sigma$-invariant idempotents in $A$ such that
$\dim e_a L(b) = \delta_{a,b}$ for all $b \in \B$ with
$\rho(b) \not> \rho(a)$; here, $L(b)$ is the irreducible $A$-module labelled by $b \in \B$.
\item[(Ch2)] 
Let $A_{\leq \lambda}$ be the quotient of $A$ by the two-sided ideal generated by the idempotents $\{e_a\:|\:a \in \B\text{ with }\rho(a) \not\leq \lambda\}$ and $A_\lambda := \bigoplus_{a,b\in \B_\lambda} \bar e_a A_{\leq \lambda} \bar e_b$. 
For each $\lambda \in \Lambda$, we require that $A_\lambda$ possesses a non-degenerate
associative symmetric bilinear form $(\cdot,\cdot)_\lambda$
such that $(\sigma_\lambda(x), \sigma_\lambda(y))_\lambda = 
(x,y)_\lambda$ for all $x,y \in A_\lambda$,
where $\sigma_\lambda:A_\lambda \rightarrow A_\lambda$ is the anti-involution induced by $\sigma$.
\end{itemize}\end{definition}

In view of the following lemma, axiom (Ch2) is vacuous in the case that $\R$ is a highest weight category, since then we have that
$A_\lambda = \k$ and $\sigma_\lambda = \id$.

\begin{lemma}\label{chdl}
Suppose that $(A,\sigma)$ is a realization of $(\R, ?^\vee)$ in which $\sigma$ is a Chevalley involution as in Definition~\ref{chevinvdef}.
The algebra $A_\lambda$ from (Ch2) is the basic algebra realizing the stratum $\R_\lambda$, and it is $\sigma_\lambda$-symmetric in the sense of Definition~\ref{sigmasymmetric}.
We also have that $L(b)^\vee \cong L(b)$ for all $b \in \B$, i.e.,
Chevalley dualities preserve 
isomorphism classes of irreducible objects.
\end{lemma}

\begin{proof}
Let $I$ be the two-sided ideal of $A$ generated by
$\{e_a\:|\:a \in \B\text{ with }\rho(a)\not\leq \lambda\}$.
We claim that $\R_{\leq \lambda}$ is the full 
subcategory of $\R$ consisting of all objects $V$ such that $IV = 0$. To see this, 
if $IV = 0$ then $e_a V = 0$ for all $a \in \B$ with 
$\rho(a) \not\leq\lambda$ then $[V:L(a)] = 0$ for all such $a$
thanks to axiom (Ch1).
So we have that $V \in \R_{\leq \lambda}$.
Conversely, if $V \in \R_{\leq \lambda}$ and $\rho(a) \not \leq \lambda$ then the idempotent $e_a$ is zero on every irreducible
subquotient of $V$ by (Ch1), hence, $e_a V = 0$. This implies that $IV = 0$.

By the claim, the algebra $A_{\leq \lambda} = A / I$ 
gives a realization of
$\R_{\leq \lambda}$. Let $\bar e_b$ denote the image of $e_b$ in $A_{\leq \lambda}$. For $b \in \B_\lambda$, we have that $\bar e_b L(c) = \delta_{b,c}$ for all $c \in \B_\lambda$.
This shows that the mutually orthogonal idempotents $\{\bar e_b \:|\:b \in \B_{\lambda}\}$ are primitive in $A_{\leq \lambda}$.
Hence, $A_\lambda = \bigoplus_{a,b \in \B_\lambda} \bar e_a A_{\leq \lambda} \bar e_b$ is the basic algebra realizing the stratum $\R_\lambda$. It is immediate from the axioms (Ch1)--(Ch2) 
and the definition that $A_\lambda$ is $\sigma_\lambda$-symmetric.

Finally to show that $L(b)^\sigmadual \cong L(b)$ for all $b \in \B$,
suppose that $b \in \B_\lambda$. We have that $e_a L(b)^\sigmadual
\cong (e_a L(b))^*= 0$
for all $a$ with $\rho(a) \not\leq \lambda$, so $L(b)^\sigmadual \in \R_{\leq \lambda}$.
Moreover, $e_b L(b)^\sigmadual \cong (e_b L(b))^*$ is one-dimensional. 
Since $\bar e_b$ is primitive in $A_{\leq \lambda}$ this implies
that $L(b)^\sigmadual \cong L(b)$.
\end{proof}

 \begin{theorem}[Chevalley dualities commute with Ringel duality]\label{ugly}
 Let $\R$ be a fully stratified category with 
 stratification $(\B,L,\rho,\Lambda,\leq)$.
 Assume that $\R$ possesses a Chevalley duality $?^\vee$.
 Fix also a realization $(A, \sigma)$ of $(\R,?^\vee)$
 in which $\sigma$ is a Chevalley involution,
 and let $T(b)$ denote the left $A$-module corresponding to $T_+(b) \in \R$.
\begin{enumerate}
\item
If $\R$ is tilting-rigid and $\operatorname{char} \k \neq 2$
then for each $b \in \B$ 
there exists a non-degenerate symmetric bilinear form
$\langle\cdot,\cdot \rangle:T(b) \times T(b) \rightarrow \k$
satisfying the following {\em $\sigma$-adjunction property}:
\begin{equation}\label{sunriver}
\langle xv, w \rangle = \langle v, \sigma(x)w \rangle
\end{equation}
for $v,w \in T(b)$ and $x \in A$.
\item
Suppose that we are given objects of $\R$ corresponding to
$A$-modules $\{T_b\:|\:b \in \B\}$
such that each $T_b$ is a direct sum of $T(b)$ and
copies of $T(c)$ for $c\in \B$ with $\rho(c) < \rho(b)$.
Assume moreover that each $T_b$ is equipped with 
a non-degenerate symmetric bilinear form
$\langle \cdot,\cdot \rangle$
satisfying the $\sigma$-adjunction property.
Then, $\R$ is tilting-rigid with symmetric strata, and
there is an induced Chevalley duality $?^\wedge$ on the Ringel dual $\R'$ of $\R$ satisfying (\ref{asint}).
\end{enumerate}
\end{theorem}

\begin{proof}
(1) Suppose that 
$b \in \B_\lambda$ for some $\lambda \in \Lambda$.
For the purpose of proving (1) for $T(b)$, 
we can replace $\R$ by $\R_{\leq \lambda}$
and the algebra $A$ realizing $\R$ by the corresponding
quotient algebra to assume without loss of generality that 
$\R = \R_{\leq \lambda}$.
So now $\R$ is either finite or upper finite, and the 
chosen algebra $A$ is either
a finite-dimensional algebra or a locally finite-dimensional locally
unital algebra, respectively. Let $\{e_a\:|\:a \in \B\}$ be the 
mutually orthogonal 
$\sigma$-invariant idempotents given by the axiom (Ch1).
Let $e_\lambda := \sum_{b \in \B_\lambda} e_b$
and $A_\lambda := e_\lambda A e_\lambda$. By Lemma~\ref{chdl}, 
this is the basic
finite-dimensional algebra realizing the top stratum $\R_\lambda$, and
$\{e_b\:|\:b \in \B_\lambda\}$ is 
a set of representatives for the 
conjugacy classes of primitive idempotents in $A_\lambda$.
The anti-involution $\sigma$ of $A$ restricts to an anti-involution $\sigma_\lambda$ of $A_\lambda$.
Also $e_\lambda T(b)$ is isomorphic to the indecomposable projective $A_\lambda$-module $A_\lambda e_b$.

\vspace{1.5mm}
\noindent
Claim 1: {\em
Let $\psi:T(b) \rightarrow T(b)$ be an $A$-module homomorphism
and $\bar\psi:e_\lambda T(b) \rightarrow e_\lambda T(b)$
be its restriction, which is an $A_\lambda$-module homomorphism.
Then $\psi$ is an isomorphism if and only if $\bar\psi$ is an isomorphism.}
The forward implication is clear. For the converse,
let $E := \End_A(T(b))^\op$
and $E_\lambda := \End_{A_\lambda}(e_\lambda T(b))^\op$.
As $e_\lambda T(b)$ is an indecomposable $A_\lambda$-module,
the algebra $E_\lambda$ is a finite-dimensional local algebra, so its Jacobson radical is of codimension one and 
any non-unit is nilpotent. The algebra $E$ is also a finite-dimensional local algebra in the finite case,
while in the upper finite case it is a
pseudo-compact topological algebra with Jacobson radical $J(E)$ having codimension one. In either case, any 
element of $E$ is either a unit or it belongs to $J(E)$.
Let $\bar E$ be the image of $E$ under the homomorphism $E \rightarrow E_\lambda$ defined by restriction. The Jacobson radical 
of $\bar E$ is the image of $J(E)$, so it is again of codimension one\footnote{In fact, one can show that $\bar E = E_\lambda$ but we do not need to use this here.}
in $\bar E$.
We are given $\psi \in E$ such that $\bar\psi$ is a unit in $E_\lambda$.
This means that $\bar\psi$ is not nilpotent, hence, it is also a unit in $\bar E$. It follows that $\bar\psi \notin J(\bar E)$
so $\psi \notin J(E)$. This shows that $\psi$ is a unit in $E$, i.e., it is an isomorphism as required.

\vspace{1.5mm}
\noindent
Claim 2: {\em
Let $\langle\cdot,\cdot\rangle$ be a bilinear form on $T(b)$ 
with the $\sigma$-adjunction property.
Then $\langle\cdot,\cdot\rangle$ is non-degenerate
if and only if its restriction $\langle\cdot,\cdot\rangle_\lambda$ to $e_\lambda T(b)$ is non-degenerate.}
To see this, observe that the form $\langle\cdot,\cdot\rangle$ induces an $A$-module homomorphism $\theta:T(b) \rightarrow T(b)^\sigmadual$ with $\theta(v)(w) = \langle v, w \rangle$,
and the form is non-degenerate if and only if this induced homomorphism is an isomorphism.
Similarly, the restriction $\langle\cdot,\cdot\rangle_\lambda$ induces an $A_\lambda$-module homomorphism
$\bar\theta:e_\lambda T(b) \rightarrow 
(e_\lambda T(b))^\sigmadual$,
and the restricted form is non-degenerate if and only if 
$\bar \theta$ is an isomorphism.
If we identify $(e_\lambda T(b))^\sigmadual$ with $e_\lambda (T(b)^\sigmadual)$ in the natural way, we see that $\bar\theta$ is the restriction of $\theta$.
We are given that $\R$ is tilting-rigid, and its strata are $\sigma_\lambda$-symmetric which implies that they are weakly symmetric, so there is an $A$-module isomorphism $\phi:T(b)^\sigmadual \stackrel{\sim}{\rightarrow} T(b)$ according to Theorem~\ref{tiltingduality}. 
This restricts to an $A_\lambda$-module isomorphism
$\bar\phi:e_\lambda 
(T(b)^\sigmadual) \rightarrow e_\lambda T(b)$.
Now Claim~2 is reduced to showing that the $A$-module homomorphism
$\phi\circ\theta:T(b) \rightarrow T(b)$
is an isomorphism if and only if its restriction
$\bar\phi\circ\bar\theta:e_\lambda T(b) \rightarrow e_\lambda T(b)$ is an isomorphism. This follows from Claim 1.

\vspace{1.5mm}
\noindent
Claim 3: {\em The socle of $A_\lambda e_b$ is irreducible, and any non-zero vector $z_b \in \soc (A_\lambda e_b)$ satisfies $\sigma_\lambda(z_b) = z_b$.}
By (Ch2),
there is a non-degenerate associative symmetric bilinear form $(\cdot,\cdot)_\lambda$ on $A_\lambda$ with $(\sigma_\lambda(x), \sigma_\lambda(y))_\lambda = (x,y)_\lambda$ for all $x,y \in A_\lambda$. 
By the discussion before Lemma~\ref{chevinvlem}, 
$A_\lambda e_b$ is self-dual, so it has irreducible socle
isomorphic to its head. Moreover, $(\cdot,\cdot)_\lambda$ restricts to a non-degenerate associative symmetric bilinear form on $e_b A_\lambda e_b$.
This is a local symmetric algebra, so its Jacobson radical $J$ is a two-sided
ideal of codimension one and $J^\perp$ is a two-sided ideal of dimension one.
Let $z_b$ be a non-zero vector in $J^\perp$. We must have
that $(e_b,z_b)_\lambda \neq 0$ by the non-degeneracy of the form.
Moreover, $z_b$ also spans the socle of $A_\lambda e_b$.
It remains to show that $\sigma_\lambda(z_b) = z_b$.
Since $\sigma_\lambda$ leaves $J^\perp$ invariant we certainly have that
$\sigma_\lambda(z_b) = c z_b$ for $c \in \k$. Now we use the $\sigma_\lambda$-symmetry of the form:
$$
(e_b, z_b)_\lambda = 
(\sigma_\lambda(e_b), \sigma_\lambda(z_b))_\lambda 
= 
(e_b, c z_b)_\lambda.
$$
Since $(e_b,z_b)_\lambda \neq 0$, this implies that $c=1$.

\vspace{1.5mm}
\noindent
Claim 4: {\em 
Suppose that $\langle\cdot,\cdot\rangle_\lambda$ is a bilinear form on
 $A_\lambda e_b$ with the $\sigma_\lambda$-adjunction property, i.e.,
 the analog of (\ref{sunriver}) with $\sigma$ replaced by $\sigma_\lambda$
 holds for all $x \in A_\lambda$. This form is non-degenerate
 if and only if $\langle e_b, z_b \rangle_\lambda \neq 0$ for
$z_b$ as in Claim 3.
 }
 Suppose first that $\langle e_b, z_b \rangle_\lambda \neq 0$.
 Take any $0 \neq x \in A_\lambda e_b$. Since the socle of $A_\lambda e_b$ is one-dimensional,
 there exists $y \in A_\lambda$ such that
 $yx = z_b$. Then $\langle e_b, yx \rangle_\lambda \neq 0$ so
 $\langle \sigma(y) e_b, x \rangle_\lambda \neq 0$.
 This shows that the function $A_\lambda e_b \rightarrow (A_\lambda e_b)^*,
 x\mapsto \langle ?,x\rangle_\lambda$ is injective, hence, the form is non-degenerate.
 Conversely, 
suppose that $\langle e_b, z_b \rangle_\lambda = 0$.
Then the $A_\lambda$-submodule $\{x \in A_\lambda e_b\:|\:\langle x, z_b \rangle_\lambda = 0\}$ contains $e_b$, hence, it is all of $A_\lambda e_b$.
So the form is degenerate.

\vspace{1.5mm}
Now we can complete the proof of (1).
As noted in the proof of Claim~2, $T(b) \cong T(b)^\sigmadual$.
Let $[\cdot,\cdot]$ be the bilinear form on $T(b)$ corresponding to such an 
isomorphism. 
This form is non-degenerate
and has the $\sigma$-adjunction property.
However, it is not necessarily symmetric, so we symmetrize by 
letting $\langle\cdot,\cdot\rangle$ be the form on $T(b)$ defined from 
$$
\langle v, w \rangle := [v, w] + [w, v].
$$
Using that $\sigma$ is an involution, it is easy to check that this new form still has the $\sigma$-adjunction property, and now it is symmetric, but we do not yet know that it is non-degenerate. 
To see this, let $\iota:A_\lambda e_b\stackrel{\sim}{\rightarrow} e_\lambda T(b)$ be an $A_\lambda$-module isomorphism.
Let $[\cdot,\cdot]_\lambda$ and $\langle\cdot,\cdot\rangle_\lambda$ be
the bilinear forms on $A_\lambda e_b$
defined from
$[ x,y ]_\lambda := [ \iota(x), \iota(y)]$ and
$\langle x,y \rangle_\lambda := \langle \iota(x), \iota(y)\rangle$.
Applying Claim 2, we see that the form $[ \cdot,\cdot ]_\lambda$ is non-degenerate, and
the goal is to show that $\langle\cdot,\cdot\rangle_\lambda$ is non-degenerate.
Applying Claim 4, we have that $[e_b, z_b]_\lambda \neq 0$ and 
we need to show that $\langle e_b, z_b \rangle_\lambda \neq 0$.
This follows since 
$$
\langle e_b, z_b \rangle_\lambda = 
[ e_b, z_b ]_\lambda+
[ z_b, e_b ]_\lambda
= [ e_b, z_b ]_\lambda + 
[ e_b, \sigma_\lambda(z_b) ]_\lambda
= 
2 [ e_b, z_b ]_\lambda \neq 0,
$$
using that $\sigma_\lambda(z_b) = z_b$ by Claim 3
together with the hypothesis that $\operatorname{char} \k \neq 2$.

\vspace{1.5mm}
\noindent
(2) 
We are given non-degenerate symmetric bilinear forms $\langle\cdot,\cdot\rangle$ on each $T_b$ satisfying the $\sigma$-adjunction property. It follows that 
$T_b \cong T_b^\sigmadual$. 
Since $T_+(b)^\vee \cong T_-(b)$ for each $b \in \B$, 
this is enough to deduce that $\R$ 
is tilting-rigid. Also the assumption that $?^\vee$ is a Chevalley duality implies that the basic algebra
$A_\lambda$ realizing $\R_\lambda$ is $\sigma_\lambda$-symmetric, hence, $\R_\lambda$ is symmetric.

Now the argument proceeds in a similar way to 
the proof of Theorem~\ref{tiltingduality}.
We just explain the details in the finite case; the other three cases are similar but there are slight notational differences.
We may assume that the tilting generator used to define the Ringel dual 
category is $T = \bigoplus_{b \in \B} T_b$. 
Then $\R' = B\fdlmod$ for $B := \End_A(T)^\op$.
The given forms on each $T_b$
give us a non-degenerate symmetric bilinear form $\langle\cdot,\cdot\rangle$ on $T$ satisfying (\ref{sunriver}),
with the summands $T_b$ being mutually orthogonal.
Define an anti-automorphism $\tau$
of $B$ from the equation
$\langle vy, w \rangle = \langle v, w \tau(y) \rangle$
for $v,w \in T$ and $y \in B$. This gives us a contravariant autoequivalence $?^\vee := ?^\taudual$
on $\R'$, and we get the isomorphisms
(\ref{asint}) like in the proof of Theorem~\ref{tiltingduality}.

As $\langle\cdot,\cdot\rangle$ is symmetric and $T$ is a faithful $B$-module,
the following calculation implies that $\tau^2=\id$:
$$
\langle vy, w \rangle = \langle v, w\tau(y) \rangle
= \langle w \tau(y), v \rangle = 
\langle w, v \tau^2(y) \rangle = \langle v \tau^2(y), w \rangle.
$$
For each $b \in \B$, 
let $f_b \in B$ be the idempotent projecting $T$ onto the summand $T_b$.
Using that the restriction of the form $\langle\cdot,\cdot\rangle$ to this summand is non-degenerate, 
it follows that $\tau(f_b) = f_b$. So $\{f_b\:|\:b \in \B\}$
is a set of mutually orthogonal $\tau$-invariant idempotents in $B$.
The idempotent $f_b$ is equal to the primitive idempotent projecting $T_b$ onto its summand $T(b)$ plus other orthogonal primitive idempotents which project onto summands $T(a)$ for $a \in \B_{< \rho(b)}$. Bearing in mind we are using the opposite ordering on $\Lambda$ on the Ringel dual side,
this is just what we need for the property (Ch1).

Finally, to see that property (Ch2) holds, let $B_\lambda$ be the
algebra obtained from $B$ according to the construction of (Ch2) and 
$\tau_\lambda:B_\lambda \rightarrow B_\lambda$ be the anti-involution induced by $\tau$. The pair $(B_\lambda,\tau_\lambda)$
is a realization of $(\R_\lambda', ?^\wedge)$, where $?^\wedge$ here is the contravariant autoequivalence of $\R_\lambda'$ induced by the one on $\R'$. 
We also have the pair $(A_\lambda, \sigma_\lambda)$
realizing $\R_\lambda$ with its contravariant autoequivalence induced by $?^\vee$. We know already by Lemma~\ref{chdl} that $A_\lambda$ is $\sigma_\lambda$-symmetric, and (Ch2) follows if we can show that $B_\lambda$ is $\tau_\lambda$-symmetric. This follows from Lemma~\ref{chevinvlem}
since the functor $F_\lambda:A_\lambda\fdlmod\rightarrow B_\lambda \fdlmod$ is an equivalence 
satisfying $F_\lambda \circ ?^\vee \cong ?^\wedge \circ F_\lambda$. Indeed, 
Theorem~\ref{tiltingrigiddual}(2) gives that
 $F_\lambda \cong G_\lambda$, while (\ref{asint})
 and the definitions (\ref{labelF})--(\ref{labelG})
give that the dualities
 $?^\vee:\R_\lambda \rightarrow \R_\lambda$
 and $?^\wedge:\R_\lambda' \rightarrow \R_\lambda'$
 satisfy
$ G_\lambda \circ ?^\vee\cong
?^\wedge \circ F_\lambda$.
\end{proof}

\section{Generalizations of quasi-hereditary algebras}

Now we give some applications of semi-infinite Ringel duality.
First, we use it to show that
any upper finite 
highest weight category can be realized as $A\lfdlmod$ for an upper finite based
quasi-hereditary algebra $A$. The latter notion, which is
Definition~\ref{doacc}, 
already exists in the literature in some equivalent forms. 
When $A$ is finite-dimensional, it gives 
an alternative algebraic characterization of the usual notion
of quasi-hereditary algebra.
Then, in $\S$\ref{newsec2}, we introduce further
notions of based $\eps$-stratified algebras and based $\eps$-quasi-hereditary
algebras, which correspond to $\eps$-stratified categories and
$\eps$-highest weight categories, respectively.
In $\S$\ref{newsec3}, we introduce
based stratified algebras and based properly stratified
algebras corresponding to fully stratified and fibered highest weight categories, respectively.
Finally, in $\S\S$\ref{stb}--\ref{std}, we discuss the related notions of triangular bases and a triangular decompositions.

\subsection{Based quasi-hereditary algebras}\label{newsec1}
The following definition is a simplified
version of \cite[Def.~2.1]{ELauda} translated 
from the framework of $\k$-linear
categories to that of locally unital algebras.
Also, for finite-dimensional algebras, it is equivalent to
\cite[Def.~2.4]{KM}. These assertions 
will be explained in more detail in Remarks~\ref{oo}--\ref{ooo} below.

\begin{definition}\label{doacc}
By a {\em finite} (resp., {\em upper finite}, resp., {\em essentially finite})
 {\em based quasi-hereditary algebra}, we mean 
 a finite-dimensional (resp., locally finite-dimensional, resp., essentially
finite-dimensional) 
locally unital algebra $A = \bigoplus_{i,j \in I} e_i A e_j$ with
the following additional data:
\begin{itemize}
\item[(QH1)] A 
subset $\Lambda \subseteq I$ 
indexing {\em special idempotents} $\{e_\lambda\:|\:\lambda \in
\Lambda\}$.
\item[(QH2)]
A partial order $\leq$ making the set $\Lambda$ into a poset which is upper finite in the upper finite case and interval finite in the essentially finite case.
\item[(QH3)] Sets 
$Y(i,\lambda) \subset e_i A e_\lambda$, $X(\lambda,j) \subset e_\lambda A e_j$
for $\lambda \in \Lambda, i,j \in I$.
\end{itemize}
Let $Y(\lambda) := \bigcup_{i \in I} Y(i,\lambda)$
and $X(\lambda) := \bigcup_{j \in I} X(\lambda,j)$.
 We impose the following axioms:
\begin{itemize}
\item[(QH4)]
The products $y x$ for
$(y,x) \in \bigcup_{\lambda \in \Lambda} Y(\lambda) \times X(\lambda)$
are a basis for $A$.
\item[(QH5)]
For $\lambda,\mu \in\Lambda$,
the sets $Y(\mu,\lambda)$ and $X(\lambda,\mu)$ are empty unless
$\mu \leq \lambda$.
\item[(QH6)]
We have that $Y(\lambda,\lambda) = X(\lambda,\lambda) = \{e_\lambda\}$
for each $\lambda \in \Lambda$.
\end{itemize}
We say that $A$ is {\em symmetrically based} if there is also some given
algebra anti-involution
$\sigma:A \rightarrow A$ with $\sigma(e_i) = e_i$ and
$Y(i,\lambda) = \sigma(X(\lambda,i))$ for all $i\in I, \lambda \in \Lambda$.
\end{definition}

We refer to the given basis for $A$ from (QH4) as 
the {\em triangular basis}; it is certainly not 
unique since one can replace any
$Y(i,\lambda)$ or $X(\lambda,j)$ by another basis 
that spans the same subspace up to ``higher terms''.
If $A$ is symmetrically based
rather than merely based, this basis is a {\em cellular basis} in the general sense of \cite{GL}, \cite{Westbury}.
However, Definition~\ref{doacc} is 
considerably more restrictive than the general notions of
cellular algebra or category introduced in {\em loc. cit.}. 
In fact, for finite-dimensional algebras, Definition~\ref{doacc} is equivalent to the usual notion of quasi-hereditary algebra, as we will explain more fully below. 

\begin{remark}\label{jabs}
It is clear from (QH4) that
$A = \sum_{\lambda \in \Lambda} A e_\lambda A$.
Hence, $A$ is Morita equivalent to the idempotent truncation
$\bigoplus_{\lambda,\mu \in \Lambda} e_\lambda A e_\mu$.
This means that if one is prepared to pass to a Morita equivalent algebra then one can
assume without loss of generality that
the sets $\Lambda$ and $I$ in Definition~\ref{doacc}
are actually equal, i.e., all distinguished idempotents are special. However, in naturally-occurring examples, one often encounters
situations in which the set $I$ is strictly larger than $\Lambda$.
\end{remark}

\begin{remark}\label{tlex0}
A classical example of a finite symmetrically based
quasi-hereditary algebra is the
Schur algebra $S(n,r)$ with its basis of {\em
  codeterminants} $\xi_{\bi,\ell(\lambda)} \xi_{\ell(\lambda),\bj}$
  as constructed by Green in \cite{Greenco}; one definitely needs $I \supsetneq \Lambda$ in this example.
  \end{remark}
  
  \begin{remark}\label{tlex}
  For a well-known infinite-dimensional example, consider the path
  algebra $A$ of the Temperley-Lieb category
  $\TL(\delta)$ for any value of its parameter $\delta \in \k$.
 The natural diagram basis gives a triangular basis making $A$ into an 
 upper finite symmetrically based quasi-hereditary algebra. 
For this, one takes $I = \Lambda = \N$ ordered by the opposite of the
natural ordering.
The set $Y(\lambda)$ (resp., $X(\lambda)$)
consists of all cap-free Temperley-Lieb diagrams with $\lambda$ strings 
at the bottom
(resp., all cup-free Temperley-Lieb diagrams with $\lambda$ strings at
the top). The anti-automorphism $\sigma$ is defined by reflecting diagrams in a horizontal axis.
\end{remark}

\begin{lemma}\label{fire}
Let $A$ be a finite, essentially finite or upper finite 
based quasi-hereditary algebra.
For $\lambda \in \Lambda$, 
any element $f$ of the two-sided ideal $A e_\lambda A$
can be written as a linear combination of elements of the form 
$y x$ for $y \in Y(\mu)$, $x \in X(\mu)$ and $\mu \geq \lambda$.
\end{lemma}

\begin{proof}
We first consider the upper finite case. By considering the triangular basis, we may assume that
$f =y_1 x_1 y_2 x_2$ for
$y_1 \in Y(\mu_1), x_1 \in X(\mu_1,\lambda)$,
$y_2 \in Y(\lambda,\mu_2), x_2 \in X(\mu_2)$
and $\mu_1,\mu_2 \geq \lambda$.
If $\mu_1=\mu_2=\lambda$ then $x_1 = e_\lambda = y_2$ and $f = y_1
x_2$, as required. This finished the proof for $\lambda$ maximal. If $\mu_r > \lambda$ for some $r \in \{1,2\}$, then we have that
$f \in A e_{\mu_r} A$ for this $r$, and are done by downward induction on the partial order on $\Lambda$.  

The finite and essentially finite cases are similar.
Now, assuming that $f \in e_i A e_j$ for $i,j \in I$, there are only finitely many
$\mu \in \Lambda$ such that $e_i A e_\mu \neq 0$ or 
$e_\mu A e_j \neq 0$. Letting $\Lambda'$ be the finite set of all such $\mu$,
we can then again proceed by downward induction on the partial order on
$\Lambda'$.
\end{proof}

\begin{corollary}\label{climbing}
Let $\Lambda\u$ be an upper set in $\Lambda$.
The two-sided ideal $J_{\Lambda\u}$ of $A$ generated
by $\{e_\lambda\:|\:\lambda \in \Lambda\u\}$ has basis 
$\big\{y x\:\big|\:(y,x) \in \bigcup_{\lambda \in \Lambda\u} Y(\lambda)\times X(\lambda)\big\}$.
\end{corollary}

\begin{proof}
Let $J$ be the subspace of $A$ with basis
given by the products
$y x$ for $y \in Y(\lambda), 
x \in X(\lambda)$ and $\lambda \in \Lambda\u$.
For any such element $y x \in J$, we have that 
$y x = y e_\lambda x$, hence, $y x \in J_{\Lambda\u}$.
This shows that $J \subseteq J_{\Lambda\u}$.
Conversely, any element of $J_{\Lambda\u}$ is a linear
combination of elements of $A e_\lambda A$ for $\lambda \in \Lambda\u$.
In turn, Lemma~\ref{fire} shows that any element of $A e_\lambda A$ for $\lambda
\in \Lambda\u$
is a linear combination of elements $y x$
for $y \in Y(\mu), x \in X(\mu)$ and $\mu \geq \lambda$.
All of these elements $y x$ belong to $J$ because $\Lambda\u$ is an upper set; thus $J_{\Lambda\u} \subseteq J$.
\end{proof}

\begin{remark}\label{oo}
We have formulated Definition~\ref{doacc} 
only for algebras over our usual 
ground field $\k$, but the definition makes sense 
with $\k$ replaced by some more general commutative ground ring $R$ (``finite-dimensional" being interpreted as ``free of finite rank").
Then, 
in the symmetrically based upper finite case,
Definition~\ref{doacc} 
is equivalent to the notion of an
object-adapted cellular category from \cite[Def.~2.1]{ELauda}.
This can be seen from Corollary~\ref{climbing} and \cite[Lemmas~2.6-2.8]{ELauda}.
Elias and Lauda also note in {\em loc. cit.} that the diagrammatic Hecke category
$\mathcal{H}_{BS}(W,S)$ of \cite{EW} associated to a Coxeter system $(W,S)$
is an example of an object-adapted cellular category. In our language, the path algebra $H$
of $\mathcal{H}_{BS}(W,S)$ 
is an upper finite symmetrically based quasi-hereditary algebra
defined over the ground ring $R = \mathbb{Q}[\mathfrak{h}]$, that is, the ring of regular functions arising from a realization $\mathfrak{h}$ of $(W,S)$. 
A cellular basis is given by the double 
light leaves basis. (One needs some assumptions on the realization
as in \cite{EW} for this basis to be defined.)
\end{remark}

\begin{remark}\label{ooo}
In the finite case, Definition~\ref{doacc} is equivalent to the notion of based
quasi-hereditary algebra from \cite[Def.~2.4]{KM}.
To see this, one takes the set $\Lambda$ indexing the special
idempotents in our setup
to be the set $I$ from {\em loc. cit.}
(which indexes mutually orthogonal idempotents $e_i \in A$
according to \cite[Lem.~2.8]{KM}).
Then we take our set $I$ to be the set $\Lambda \sqcup \{0\}$, i.e.,
we add one more element indexing one more idempotent
$e_0 := 1_A - \sum_{\lambda \in \Lambda} e_\lambda$.
Kleshchev and Muth established the equivalence of their notion
of based quasi-hereditary algebra with the original notion of
quasi-hereditary algebra from \cite{CPS} (providing the partial order
on $\Lambda$ is actually a total order);
for ground fields, we will reprove this equivalence in a different way below.
See also \cite{DuR} which established a similar result
using a related notion of standardly based algebra.
\end{remark}

Let $A$ be a based quasi-hereditary algebra as in Definition~\ref{doacc}.
For $\lambda \in \Lambda$, let $A_{\leq\lambda}$ be the quotient of $A$ by the two-sided ideal
generated by the idempotents $e_\mu$ for $\mu \not\leq \lambda$.
For $x \in A$, we often write simply $\bar x$ for the image of $x$ in
$A_{\leq\lambda}$. Corollary~\ref{climbing} implies that 
\begin{equation}\label{climbingcor}
A_{\leq\lambda} = \bigoplus_{i,j \in I} \bar e_i A_{\leq
  \lambda} \bar e_j
\end{equation}
is based quasi-hereditary in
its own right, with special idempotents indexed by elements of the
lower set $(-\infty,\lambda]$ and basis given
by the products $\bar y \bar x$ for $y \in Y(\mu), x \in X(\mu)$ and $\mu \in (-\infty,\lambda]$.
Define the {\em standard} and {\em costandard modules}
associated to $\lambda \in \Lambda$ by
\begin{align}\label{cellmodules}
\Delta(\lambda) &:= A_{\leq\lambda} \bar e_\lambda&
\nabla(\lambda) &= (\bar e_\lambda A_{\leq\lambda})^\circledast.
\end{align}
These are left $A$-modules which are projective and injective as
$A_{\leq\lambda}$-modules, respectively.
In the finite or essentially finite case, $\bar e_\lambda A_{\leq \lambda}$ is finite-dimensional and one could 
just take the full linear dual in (\ref{cellmodules}),
but in general in
the upper finite case $\Delta(\lambda)$ and $\nabla(\lambda)$ are only locally finite-dimensional.
The modules $\Delta(\lambda)$ may also be called
{\em cell modules} and the modules $\nabla(\lambda)$ {\em dual cell modules}.
The vectors $\{ y \bar e_\lambda\:|\:y \in Y(\lambda)\}$ give the {\em
  standard basis}
for $\Delta(\lambda)$.
Similarly, the vectors
$\{\bar e_\lambda x\:|\:x \in X(\lambda)\}$ are a basis for the
right $A$-module $\bar e_\lambda A$;
the dual basis to this is the {\em costandard basis}
for $\nabla(\lambda)$.

\begin{theorem}[Highest weight categories from based quasi-hereditary algebras]\label{oneway}
Let $A$ be a finite (resp., upper finite, resp., essentially finite) 
based quasi-hereditary algebra.
The modules 
$$
\{L(\lambda) := \hd \Delta(\lambda)\cong\soc\nabla(\lambda)\:|\:\lambda \in \Lambda\}$$ 
give a complete
set of pairwise inequivalent irreducible left $A$-modules.
Moreover,
the category
$\R := A\fdlmod$ (resp., $\R := A\lfdlmod$, resp., $\R := 
A\fdlmod$) is a finite (resp., upper finite, resp., essentially finite) highest weight category with the given weight poset 
$(\Lambda,\leq)$.
Its standard and costandard objects
$\Delta(\lambda)$ and $\nabla(\lambda)$ are as defined by
(\ref{cellmodules}).
If $A$ is symmetrically based with anti-involution $\sigma$ then $?^\sigmadual:\R\rightarrow\R$ is Chevalley duality of $\R$ in the sense of Definition~\ref{chevinvdef}.
\end{theorem}

\begin{proof}
For $\lambda \in \Lambda$, let $P_\lambda$ be the left ideal $A
e_\lambda$. We start by establishing the claim that $P_\lambda$ has a $\Delta$-flag with $\Delta(\lambda)$
at the top and other sections of the form $\Delta(\mu)$ for $\mu >
\lambda$.
To prove this, fix some $\lambda$ and set $P := P_\lambda$ for short.
This module has basis $\big\{y x\:\big|\:(y,x) \in \bigcup_{\mu \geq \lambda}
 Y(\mu)\times X(\mu,\lambda)\big\}$.
Let $\{\mu_1,\dots,\mu_n\}$ 
be the finite set $\{\mu \in [\lambda,\infty)\:|\: X(\mu,\lambda) \neq \varnothing\}$
ordered so that $\mu_r \leq \mu_s \Rightarrow r \leq s$;
in particular, $\mu_1=\lambda$.
For $1\leq r\leq n$ let $P_r$ be the subspace of $P$ spanned by 
$\big\{yx\:\big|\:(y,x) \in \bigcup_{s=r+1}^n Y(\mu_s)\times X(\mu_s,\lambda)\big\}$.
They define a 
filtration $P =: P_0 > P_1 > \cdots > P_n = 0,$ since each $P_r$ is a $A$-submodule of $P$.
Moreover, there is, for each $0\leq r\leq n$ an $A$-module isomorphism
\begin{equation}\label{firms}
\theta_r:\bigoplus_{x \in X(\mu_r,\lambda)} 
\Delta(\mu_r)
\stackrel{\sim}{\rightarrow}
 P_{r-1}/P_r
\end{equation}
which in case $r\geq 1$ sends the basis vector $y \bar e_{\mu_r}\:(y \in Y(\mu_r))$ in the $x$th copy of
$\Delta(\mu_r)$
to $y x + P_r \in P_{r-1}/P_r$. 
This defines clearly a linear isomorphism,
so we just need to check that it is an $A$-module homomorphism.
For this take $y \in Y(j,\mu_r)$ and $u \in e_i A e_j$.
Expand $uy$ in terms of the triangular basis as 
$\sum_p c_p y_p + \sum_q c'_q y_q' x'_q$
for scalars $c_p, c_q'$, $y_p \in Y(i,\mu_r)$,
$y_q' \in Y(i,\nu_q), x_q' \in X(\nu_q,\mu_r)$ and $\nu_q > \mu_r$.
Then we have that $u y \bar e_{\mu_r} = \sum_p c_p y_p \bar e_{\mu_r}$
and $u y x + P_r = \sum_p c_p y_p x + P_r$, since the ``higher terms''
$y_q'x_q'$ act as zero on both $\bar e_{\mu_r}$ and $x + P_r$.
This shows that $\theta_r$ intertwines the actions of $u$ and so the claim follows, since $P_0 / P_1 \cong \Delta(\lambda)$ by construction. 

Now we can classify the irreducible $A$-modules.
The first step for this is to 
show that $\Delta(\lambda)$ has a unique irreducible
quotient.
To see this, note that the ``weight space'' 
$e_\lambda \Delta(\lambda)$ is one-dimensional with basis
$\bar e_\lambda$, due to the fact that
$Y(\lambda,\lambda) = \{e_\lambda\}$.
This is a cyclic vector, so any proper submodule of $\Delta(\lambda)$ 
must intersect
$e_\lambda \Delta(\lambda)$ trivially. It follows that the sum of all
proper submodules is proper, so $\Delta(\lambda)$ has a unique
irreducible quotient $L(\lambda)$.
Since $e_\lambda L(\lambda)$ is one-dimensional and all other $\mu$
with $e_\mu L(\lambda) \neq 0$ satisfy $\mu < \lambda$,
the modules $\{L(\lambda)\:|\:\lambda \in \Lambda\}$ are pairwise
inequivalent.
To see that they give a full set of irreducible $A$-modules, let $L$
be any irreducible $A$-module.
In view of Remark~\ref{jabs}, there exists $\lambda \in \Lambda$ such
that $e_\lambda L \neq 0$.
Then $L$ is a quotient of $P_\lambda = A e_\lambda$.
By the claim we proved already, it follows that
$L$ is a quotient of $\Delta(\mu)$ for some
$\mu \geq \lambda$, i.e., $L \cong L(\mu)$.

Thus, we have shown that the modules $\{L(\lambda)\:|\:\lambda \in
\Lambda\}$ give a full set of pairwise inequivalent irreducible left
$A$-modules. Now consider the stratification of $\R$
arising from the given partial order on the index set $\Lambda$. 
In the recollement situation of (\ref{rst}), 
the Serre subcategory $\R_{\leq\lambda}$ (resp., $\R_{<\lambda}$)
may be identified with $A_{\leq\lambda} \lfdlmod$ 
(resp., $A_{\leq\lambda}\fdlmod$), 
and the Serre quotient $\R_\lambda = \R_{\leq
  \lambda} / \R_{<\lambda}$ is $A_\lambda \fdlmod$
where $A_\lambda := \bar e_\lambda A_{\leq\lambda} \bar e_\lambda$.
The algebra $A_\lambda$ has basis $\bar e_\lambda$, i.e., it 
is a copy of the ground field $\k$. This shows that all strata are
simple in the sense of 
Lemma~\ref{allsimple}. Moreover, the standard and costandard objects
in the general 
sense of (\ref{rumble}) are obtained by applying the standardization
functor
$j^\lambda_! := A_{\leq\lambda} \bar e_\lambda \otimes_{A_\lambda}?$
and the costandardization functor
$j^\lambda_* := \bigoplus_{i \in I} \Hom_{A_\lambda}(\bar e_\lambda A_{\leq\lambda}
\bar e_i,?)$ to the irreducible $A_\lambda$-module
$A_\lambda=\k\bar e_\lambda$. Clearly, the resulting modules are isomorphic to
$\Delta(\lambda)$ and $\nabla(\lambda)$ as defined by (\ref{cellmodules}).
The axiom
$(\widehat{P\Delta})$ follows from the claim.

For the final statement about Chevalley duality, the observations made earlier in the proof establish property (Ch1) from Definition~\ref{chevinvdef}, and (Ch2) is vacuous as we are in the highest weight setting. Hence, $\sigma$ is a Chevalley anti-involution.
\end{proof}

Finally in this subsection we are going to prove a converse to Theorem~\ref{oneway}. This will be deduced from the next theorem
together with an application of Ringel duality.
In fact, the next theorem is a reformulation of the main result of \cite{AST}.

\begin{theorem}[Based quasi-hereditary algebras from highest weight categories]\label{TBAnew}
Let $\R$ be a finite (resp., lower finite, resp., tilting-bounded essentially finite) highest weight category with weight poset $(\Lambda,\leq)$ and labelling function $L$. Suppose we are given $\Lambda \subseteq I$ and a tilting generator $T = \bigoplus_{i \in I} T_i$ for $\R$ such that each $T_\lambda\:(\lambda \in \Lambda)$ is a direct sum of $T(\lambda)$ and other $T(\mu)$
for $\mu < \lambda$.
Let
$$
A := \left(\bigoplus_{i,j \in I} \Hom_\R(T_i, T_j)\right)^\op.
$$
\begin{itemize}
\item[(1)]
For $i,j \in I$ and $\lambda \in \Lambda$, pick morphisms
$$
Y(i,\lambda) \subset \Hom_\R(T_i, T_\lambda),
\qquad
X(\lambda,j)
\subset \Hom_\R(T_\lambda,T_j)
$$
lifting bases
for $\Hom_\R(T_i, \nabla(\lambda))$ and $\Hom_\R(\Delta(\lambda),T_j)$
as in Corollary~\ref{astfactexactly},
such that
$Y(\lambda,\lambda) = X(\lambda,\lambda) = \{\id_{T_\lambda}\}$.
Then $\big\{yx\:\big|\:(y,x) \in \bigcup_{i,j \in I} 
\bigcup_{\lambda \in \Lambda} Y(i,\lambda) \times X(\lambda,j)\big\}$ 
is a triangular basis making $A$
into a finite (resp., upper finite, resp., essentially finite) based quasi-hereditary algebra
with respect to the opposite poset $(\Lambda,\geq)$.
\item[(2)]
If in addition $\R$ has a Chevalley duality $?^\vee$ and, in a suitable
realization, the modules corresponding to each $T_i$ possess non-degenerate symmetric bilinear forms 
satisfying the adjunction property as in (\ref{sunriver}), then the triangular basis in (1) can be chosen so that $A$ is symmetrically based.
\end{itemize}
\end{theorem}

\begin{proof}
(1) We have all of the necessary data in place 
to have a based quasi-hereditary algebra, taking $e_i := \id_{T_i}$ in the obvious way. To check the axioms,
Corollary~\ref{astfactexactly} checks (QH4), and we have chosen the lifts so that
$Y(\lambda,\lambda) =
\{e_\lambda\} = X(\lambda,\lambda)$ as in (QH6). For (QH5), note that $Y(\mu,\lambda)$
and $X(\lambda,\mu)$ are empty unless $\mu \geq \lambda$
because $\Hom_{\R}(T_\mu, \nabla(\lambda))$ and 
$\Hom_{\R}(\Delta(\lambda), T_\mu)$ are zero unless $\lambda \leq \mu$.

\vspace{1mm}
\noindent
(2)
Suppose that we are working in a particular algebra realization $(B, \tau)$ of $(\R,?^\vee)$ in which $\tau$ is a Chevalley anti-involution 
and each $T_i$ admits a non-degenerate symmetric bilinear form with the
$\tau$-adjunction property.
Let $T := \bigoplus_{i \in I} T_i$ and $\langle\cdot,\cdot\rangle:T \times T \rightarrow \k$ be the orthogonal sum of the given forms.
Then we obtain an algebra anti-involution $\sigma:A \rightarrow A$ 
such that $\langle v x, w \rangle = \langle c, w \sigma(x) \rangle$
for all $v,w \in T$, $x \in A$; cf. the proof of Theorem~\ref{ugly}(2).
This fixes each of the idempotents $e_i \in A$.
The bilinear form on $T_i$ induces a $B$-module isomorphism
$\phi_i:T_i
\stackrel{\sim}{\rightarrow} T_i^\taudual$.
Also let $\pi_\lambda:T_\lambda \twoheadrightarrow \nabla(\lambda)$ be some choice of epimorphism for each $\lambda \in \Lambda$ as needed for Corollary~\ref{astfactexactly}. 
Then define the embeddings 
$\iota_\lambda:\Delta(\lambda) \hookrightarrow T_\lambda$ 
there so that
there are induced isomorphisms 
$\Delta(\lambda)\stackrel{\sim}{\rightarrow} \nabla(\lambda)^\taudual$
making the following diagrams commute for all $\lambda \in \Lambda$:
$$
\begin{tikzcd}
\Delta(\lambda)\arrow[r,"\sim", left]
\arrow[d,hookrightarrow,"\iota_\lambda", right]
&
\nabla(\lambda)^\taudual\arrow[d,hookrightarrow,"\pi_\lambda^*", right]\\
T_\lambda \arrow[r, "\phi_\lambda", right]&T_\lambda^\taudual.
\end{tikzcd}
$$
Now we pick the sets $X(\lambda,i)$ lifting bases for $\Hom_B(\Delta(\lambda),T_j)$ as in Corollary~\ref{astfactexactly}.
Then define $Y(i,\lambda) := \{\phi_\lambda^{-1} \circ x^* \circ \phi_i\:|\:x \in X(\lambda,i)\}$.
This set lifts a basis for $\Hom_B(T_i, \nabla(\lambda))$
 as stipulated in Corollary~\ref{astfactexactly}. 
 Using these choices, the construction from the previous paragraph makes $A$ into a based quasi-hereditary algebra.
 Moreover, we now have that
 $Y(i,\lambda) = \sigma(X(\lambda,i))$ for all $i, \lambda$, so $A$
 is symmetrically based with the underlying anti-involution $\sigma$.
\end{proof}

\begin{corollary}[Quasi-hereditary algebras are based quasi-hereditary]\label{TBA}
Let $$
A = \bigoplus_{i,j \in I} e_i A e_j
$$ 
be an algebra realization
of a finite (resp., upper finite, resp., tilting-bounded essentially finite) 
highest weight category $\R$, with weight poset $(\Lambda,\leq)$
and labelling function $L$.
\begin{itemize}
\item[(1)]
There is an idempotent  expansion $A = \bigoplus_{i,j \in \hat I}
\hat e_i A \hat e_j$ of $A$
with $\Lambda\subseteq \hat I$, and 
subsets $$
Y(i,\lambda) \subset \hat e_i A \hat e_\lambda, \qquad
X(\lambda,j)
\subset \hat e_\lambda A \hat e_j$$ 
for all $\lambda\in\Lambda$ and $i,j \in
\hat I$
making $A$ into a finite (resp., upper finite, resp., essentially finite) 
based quasi-hereditary algebra with respect to the given
ordering on $\Lambda$. 
\item[(2)] If $\operatorname{char} \k \neq 2$ and 
$\R$ has a Chevalley duality $?^\vee$
then the choices in (1) can be made so that $A$ is symmetrically based
 with anti-involution $\sigma$ realizing $?^\vee$.
 \end{itemize}
\end{corollary}

\begin{proof}
(1) Let $A = \bigoplus_{i,j \in \hat I} \hat e_i A \hat e_j$ be an idempotent
expansion indexed by a set $\hat I$
chosen so that $\Lambda \subseteq \hat I$ and
$\hd (A \hat e_\lambda) \cong L(\lambda)$ for each $\lambda  \in \Lambda$.
We are going to apply the Ringel duality from
Definition~\ref{thesetup} (resp., Definition \ref{rd2}, resp., Definition  \ref{bored}).
In the finite or upper finite cases, we fix a choice of tilting generator $T$ for
$\R$ and let $B := \End_\R(T)^{\op}$.
In the essentially finite case, we fix a tilting generator
$T = \bigoplus_{j \in J} T_j$ for $\R$ then let $B := \left(\bigoplus_{i,j \in J}
  \Hom_\R(T_i, T_j)\right)^\op$.
Then in all cases the category $\R' := B\fdlmod$ is the Ringel dual of the original category.
It is a finite (resp., lower finite, resp., tilting-bounded essentially
finite) highest
weight category with
irreducible objects denoted
$\{L'(\lambda)\:|\:\lambda \in \Lambda\}$
and weight poset $(\Lambda,\geq)$.
Let $T'_i := (\hat e_i T)^* \in \R'$.
By Corollary~\ref{mustard} (resp., Corollary~\ref{mustard2}, resp., Corollary~\ref{mustardier}),
$T' = \bigoplus_{i \in \hat I} T'_i$ is a tilting generator for $\R'$
such that the original algebra
$A = \bigoplus_{i,j \in \hat I} \hat e_i A \hat e_j$ is isomorphic as a locally
unital algebra
to $\left(\bigoplus_{i,j \in \hat I} \Hom_{\R'}(T'_i,
  T'_j)\right)^\op$.
Moreover, $T'_\lambda$ is the indecomposable tilting module
$T'(\lambda)$ for each $\lambda \in \Lambda$.
To make $A$ into a based quasi-hereditary algebra, it remains to apply Theorem~\ref{TBAnew}(1) with $\R$, $(\Lambda,\leq)$ and $T_i$ replaced by $\R'$, $(\Lambda,\geq)$ and $T_i'$ in the present setup.

\vspace{1mm}
\noindent
(2) Assume that $\R$ has a Chevalley duality $?^\vee$. 
Then the category $\R'$
admits a Chevalley duality $?^\wedge$ such that the Ringel duality functors
intertwine $?^\vee$ and $?^\wedge$ as in (\ref{asint}).
This follows by Theorem~\ref{ugly}, using the assumption that $\operatorname{char} \k \neq 2$
and part (1) of the theorem to establish 
the existence of suitable bilinear forms as in part (2).
Hence, $\R'$ has a realization $(B,\tau)$ with $\tau$ being a Chevalley involution
realizing $?^\wedge$.
Then we can appeal to Theorem~\ref{TBAnew}(2), again using
Theorem~\ref{ugly}(1) to obtain suitable bilinear forms on each $T'_i$,
to deduce that 
the triangular basis can be chosen so that $A$ is symmetrically based.
In particular, this gives an anti-involution $\sigma:A\rightarrow A$ fixing
each $\hat e_i$. It remains to note that $?^\sigmadual$ realizes  $?^\vee$. It suffices to check this on finitely generated 
projectives when it follows from (\ref{asint}) (applied twice since we have used Ringel duality twice).
\end{proof}

In the finite case, Corollary~\ref{TBA} 
recovers \cite[Prop.~3.5]{KM} (but note that the result in {\em loc. cit.} is also valid over more general ground rings).

\subsection{\boldmath Based $\eps$-stratified and $\eps$-quasi-hereditary algebras}\label{newsec2}
In this subsection, we upgrade the results of $\S$\ref{newsec1} (excluding any that involve 
Chevalley duality) to
$\eps$-stratified and $\eps$-highest weight categories.
The main new definition is as follows.

\begin{definition}\label{eclairs}
By a {\em finite} (resp., {\em upper finite}, resp., {\em essentially finite}) {\em 
based $\eps$-stratified algebra}, we mean a finite-dimensional (resp., locally finite-dimensional, resp., essentially finite-dimensional) locally
unital algebra $A = \bigoplus_{i,j \in I} e_i A e_j$ with the
following additional data:
\begin{itemize}
\item[($\eps$S1)]
A subset $\B \subseteq I$ indexing the {\em special idempotents}
  $\{e_b\:|\:b\in \B\}$.
\item[($\eps$S2)]
A poset $(\Lambda,\leq)$ 
which is upper finite in the upper finite case and interval finite in the essentially finite case, such that $\Lambda \cap I = \varnothing$.
\item[($\eps$S3)]
A sign function
$\eps:\Lambda \rightarrow \{\pm\}$.
\item[($\eps$S4)]
A function $\rho:\B\rightarrow \Lambda$ with finite fibers $\B_\lambda
:= \rho^{-1}(\lambda)$.
\item[($\eps$S5)]
Sets $Y(i,b) \subset e_i A e_b$ and $X(b,j) \subset e_b A e_j$
for all $b \in \B$ and $i,j \in I$.
\end{itemize}
Let $Y(b) := \bigcup_{i \in I} Y(i,b)$ and $X(b) := \bigcup_{j \in I}
X(b,j)$.
There are then four axioms, the first three of which are as follows:
\begin{itemize}
\item[($\eps$S6)]
The products $y x$ for $(y,x) \in \bigcup_{b \in \B} Y(b)\times
  X(b)$
  are a basis for $A$.
\item[($\eps$S7)]
For $a,b \in \B$, the sets $Y(b,a)$ and $X(a,b)$ are empty
  unless $\rho(b) \leq \rho(a)$.
\item[($\eps$S8)]
The following hold for all $\lambda \in \Lambda$ and $a,b \in \B_\lambda$:
\begin{itemize}
\item
if $\eps(\lambda)=-$ then
$Y(a,a) = \{e_a\}$ and
  $Y(a,b) = \varnothing$ for $a \neq b$;
\item if 
$\eps(\lambda)=+$ then
$X(a,a) = \{e_a\}$ and
  $X(a,b) = \varnothing$ for $a \neq b$.
\end{itemize}
\end{itemize}
To formulate the fourth axiom, 
let $e_{\lambda} := \sum_{b \in \B_\lambda} e_b$ for short\footnote{This notation is unambiguous due to the assumption $\Lambda \cap I = \varnothing$.}
let $A_{\leq
  \lambda}$ be the quotient of $A$ by the two-sided ideal generated by
$\{e_{{\mu}}\:|\:\mu \not\leq\lambda\}$, and set $A_\lambda := \bar e_{{\lambda}}
A_{\leq\lambda} \bar e_{{\lambda}}$ (where $\bar x \in A_{\leq\lambda}$ denotes the image of
$x \in A$ as usual).
Then:
\begin{itemize}
\item[($\eps$S9)]
For each $\lambda \in \Lambda$, the finite-dimensional algebra
$A_\lambda$ is
basic and $\bar e_\lambda = \sum_{b \in \B} \bar e_b$
is a decomposition of its identity element into mutually orthogonal primitive idempotents.
\end{itemize}
\end{definition}

Definition~\ref{eclairs} in the special case that the stratification function $\rho$ is a bijection deserves its own name: 

\begin{definition}\label{eclairs2}
A {\em finite} (resp., {\em upper finite}, resp., {\em essentially finite}) {\em based
  $\eps$-quasi-hereditary algebra} is a 
  finite-dimensional (resp., locally finite-dimensional, resp., essentially finite-dimensional) locally
unital algebra $A = \bigoplus_{i,j \in I} e_i A e_j$ with the
following additional data:
\begin{itemize}
\item[($\eps$QH1)]
A subset $\Lambda \subseteq I$ indexing the {\em special idempotents}
  $\{e_\lambda\:|\:\lambda\in \Lambda\}$.
\item[($\eps$QH2)]
A partial order $\leq$ making the set $\Lambda$ into a poset which is interval finite in the essentially finite case and upper finite in the upper finite case.
\item[$\eps$QH3)]
A sign function
$\eps:\Lambda \rightarrow \{\pm\}$.
\item[($\eps$QH4)]
Sets $Y(i,\lambda) \subset e_i A e_\lambda$, $X(\lambda,j) \subset e_\lambda A e_j$
for $i,j \in I$ and $\lambda \in \Lambda$.
\end{itemize}
Let $Y(\lambda) := \bigcup_{i \in I} Y(i,\lambda)$ and $X(\lambda) := \bigcup_{j \in I}
X(\lambda,j)$.
The axioms are as follows:
\begin{itemize}
\item[($\eps$QH5)]
The products $y x$ for $(y,x) \in \bigcup_{\lambda \in \Lambda} Y(\lambda)\times
  X(\lambda)$
  are a basis for $A$.
\item[($\eps$QH6)]
For $\lambda,\mu\in\Lambda$, the sets $Y(\mu,\lambda)$ and $X(\lambda,\mu)$ are empty
  unless $\mu \leq \lambda$.
\item[($\eps$QH7)]
If $\eps(\lambda)=-$ then
$Y(\lambda,\lambda) = \{e_\lambda\}$, and
if  $\eps(\lambda)=+$ then $X(\lambda,\lambda) = \{e_\lambda\}$.
\end{itemize}
\begin{itemize}
\item[($\eps$QH8)]
For each $\lambda \in \Lambda$, the finite-dimensional algebra
$A_\lambda$ as defined in Definition~\ref{eclairs} 
is basic and local.
\end{itemize}
\end{definition}

From now on, we just formulate the results for based $\eps$-stratified
algebras, since based $\eps$-quasi-hereditary algebras
are a special case.
The development below parallels the treatment in the previous
subsection, but there are some additional subtleties.
 
 Remark~\ref{jabs}
remains true: one can always pass to a Morita equivalent algebra in which 
all of the distinguished idempotents are special.
The analog of Lemma~\ref{fire} is as follows.

\begin{lemma}\label{parasites}
Let $A$ be a finite, essentially finite or upper finite
based $\eps$-stratified
algebra. For $\lambda \in \Lambda$, any
element $f$ of the
two-sided ideal $A e_{{\lambda}} A$ can be written as a linear combination of
elements of the form $y x$ for $y \in Y(a), x \in X(a)$ and $a \in
\B_{\geq\lambda}$.
\end{lemma}

\begin{proof}
This is similar to the proof of Lemma~\ref{fire}. We just explain in
the upper finite case.
We may assume that
$f =y_1 x_1 y_2 x_2$ for
$y_1 \in Y(a_1), x_1 \in X(a_1,b)$,
$y_2 \in Y(b,a_2), x_2 \in X(a_2)$,
$b \in \B_\lambda$ and 
$a_1,a_2 \in \B_{\geq\lambda}$.
If $a_1 \in \B_{>\lambda}$ or $a_2 \in \B_{>\lambda}$, we are done by
induction.
If $a_1,a_2 \in \B_\lambda$, there are two cases according to
whether $\eps(\lambda) = +$ or $\eps(\lambda)=-$.
The arguments for these are similar, so we just go through the former
case when $\eps(\lambda)=+$.
Then we have that $a_1 = b$ and $x_1 = e_b$.
Hence $f = y_1 y_2 x_2$.
Then we use the basis again to expand $y_1 y_2$ as a linear combination of
terms $y_3 x_3$
for $y_3 \in Y(a_3), x_3 \in X(a_3,a_2)$ and $a_3 \in \B_{\geq\lambda}$.
If $a_3 \in \B_ \lambda$ then we get that $a_3=a_2$ and $x_3 = e_{a_2}$, so
$y_3 x_3 x_2 = y_3 x_2$ as required. If $a_3 \in \B_{>\lambda}$, we can then rewrite $y_3 x_3 x_2$ in the desired
form by induction.
\end{proof}

\begin{corollary}\label{climbing2}
Let $\Lambda\u$ be an upper set in $\Lambda$ and $\B\u := \rho^{-1}(\Lambda\u)$.
The two-sided ideal $J_{\Lambda\u}$ of $A$ generated
by $\{e_{{\lambda}}\:|\:\lambda \in \Lambda\u\}$ has basis
$\big\{y x\:\big|\:(y,x) \in \bigcup_{b \in \B\u}
Y(b)\times X(b)\big\}$.
\end{corollary}

Let $A$ be a based $\eps$-stratified algebra as in Definition~\ref{eclairs}.
For $\lambda \in \Lambda$, Corollary~\ref{climbing2} implies that
$A_{\leq\lambda}$ has basis
$\big\{\bar y \bar x\:\big|\:\text{$y \in Y(b), x \in X(b)$ and $b \in \B_{\leq\lambda}$}\big\}$.
Hence, the basic algebra
$A_\lambda = \bar e_{{\lambda}} A_{\leq\lambda} \bar e_{{\lambda}}$ has basis $$\textstyle
\big\{\bar y\:\big|\:y \in \bigcup_{a, b \in \B_\lambda} Y(a,b)\big\}\quad\text{if
$\eps(\lambda)=+$},
\qquad
\big\{\bar x\:\big|\:x \in \bigcup_{a,b \in \B_\lambda} X(a,b)\big\}\quad\text{if $\eps(\lambda)= -$.}
$$
Let $j^\lambda:A_{\leq\lambda}\lfdlmod \rightarrow A_\lambda\fdlmod, V \mapsto \bar e_{{\lambda}} V$ be the 
quotient functor $V \mapsto \bar e_{{\lambda}} V$, then define the
standardization and costandardization 
functors
\begin{align}\label{pinkeye}
j^\lambda_! &:= A_{\leq\lambda} \bar e_{{\lambda}} \otimes_{A_\lambda} ?,
&
j^\lambda_* &:= \bigoplus_{i \in I} \Hom_{A_\lambda}(\bar
              e_{{\lambda}} A_{\leq\lambda} \bar e_i,?),
\end{align}
which are left and right adjoints of $j^\lambda$, respectively.

\begin{lemma}\label{xmas}
If $\lambda \in \Lambda$ has $\eps(\lambda)=-$
then the standardization functor $j^\lambda_!$ is exact.
\end{lemma}

\begin{proof}
There is an isomorphism of right $A_\lambda$-modules
$\bigoplus_{a \in \B_\lambda} \bigoplus_{y \in Y(a)}
\bar e_a A_\lambda
\stackrel{\sim}{\rightarrow}
A_{\leq\lambda} \bar e_{{\lambda}}
$
sending the vector
$\bar e_a$ in the $y$th copy of $\bar e_a A_\lambda$ to $\bar y \in A_{\leq\lambda}
\bar e_{{\lambda}}$.
To see this, note as $\eps(\lambda)=-$ that the projective 
$A_\lambda$-module
$\bar e_a A_\lambda$ has basis $\big\{\bar x\:\big|\:x \in \bigcup_{b \in \B_\lambda}X(a,b)\big\}$,
and $A_{\leq\lambda} \bar e_{{\lambda}}$ has basis
$\big\{\bar y \bar x\:\big|\: (y,x) \in \bigcup_{a,b\in\B_\lambda} Y(a)\times X(a,b)\big\}$.
Hence, $A_{\leq\lambda} \bar e_{{\lambda}}$ 
is a projective right $A_\lambda$-module, and the exactness follows.
\end{proof}

Continuing with $A$ being a 
based $\eps$-stratified algebra, we let
\begin{align}
P_\lambda(b) &:= A_\lambda \bar e_b,
&
I_\lambda(b) &:= (\bar e_b A_\lambda)^\circledast,
&
L_\lambda(b) &:= \hd P_\lambda(b) \cong \soc I_\lambda(b)
\end{align}
for $b \in \B_\lambda$.
These give full sets of indecomposable projective, indecomposable injective, and irreducible $A_\lambda$-modules, respectively.
Then we define standard, proper standard, costandard and proper
costandard modules
\begin{align}\label{cellmodules1}
\Delta(b) &:= A_{\leq\lambda} \bar e_b \cong 
j^\lambda_! P_\lambda(b),
&
\bar\Delta(b) &:= j^\lambda_! L_\lambda(b),\\
\nabla(b) &:= (\bar e_b A_{\leq\lambda})^\circledast \cong 
j^\lambda_* I_\lambda(b),
&\bar\nabla(b) &:= j^\lambda_* L_\lambda(b),\label{cellmodules2}
\end{align}
cf. (\ref{rumble}).
Adopt the shorthands $\Delta_\eps(b)$ and $\nabla_\eps(b)$ from
(\ref{device}) too.
The module $\Delta_\eps(b)$ has a standard basis indexed by the
set $Y(b)$.
In the case that $\eps(\lambda)=+$, when $\Delta_\eps(b) = \Delta(b)$,
this basis is $\{y \bar e_b\:|\:y \in Y(b)\}$.
In the case that $\eps(\lambda)=-$, when $\Delta_\eps(b) =
\bar\Delta(b)$, let $\tilde e_b$ be
the canonical image of $\bar e_b$ under the natural quotient map
$\Delta(b)\twoheadrightarrow \bar\Delta(b)$.
Then the basis is $\{y \tilde e_b\:|\:y \in Y(b)\}$.
(One can also construct a costandard basis for $\nabla_\eps(b)$ 
indexed by $X(b)$ by taking a certain dual basis, but we will not need this
here.)

\begin{theorem}[$\eps$-Highest weight categories from based $\eps$-stratified algebras]\label{onewayplus}
Let $A$ be a finite (resp., upper finite, resp., essentially finite) based
$\eps$-stratified algebra as above.
The modules $$
\{L(b) := \hd \Delta_\eps(b)\cong\soc\nabla_\eps(b)\:|\:b \in \B\}$$ 
give a complete
set of pairwise inequivalent irreducible left $A$-modules.
Moreover, 
$\R := A\fdlmod$ (resp., $\R := A\lfdlmod$, resp., $\R := 
A\fdlmod$) is a finite (resp., upper finite, resp., essentially
finite) $\eps$-stratified category with
stratification $(\B,L,\rho,\Lambda,\leq)$.
Its strata may be identified with the categories $\R_\lambda :=
A_\lambda\fdlmod$ with
standardardization and costandardization functors
as in (\ref{pinkeye}).
\end{theorem}

\begin{proof}
For $b \in \B$, let $P_b$ be the left ideal $A e_b$.
We claim 
that $P_b$ has a $\Delta_\eps$-flag with $\Delta_\eps(b)$
at the top and other sections of the form $\Delta_\eps(a)$ for $a \in \B$
with $\rho(a) \geq \rho(b)$.
To prove this, suppose that $b \in \B_\lambda$ and set $P := P_b$ for
short.
Note $P$ has basis $\big\{yx\:\big|\:(y,x) \in \bigcup_{a \in \B_{\geq\lambda} }Y(a) \times X(a,b)\big\}$.
Let $\{\mu_1,\dots,\mu_n\}$ 
be the finite set $$
\big\{\mu \in [\lambda,\infty)\:\big|\:\textstyle\bigcup_{a \in \B_\mu} X(a,b) \neq
\varnothing\big\}
$$
ordered so that $\mu_r \leq \mu_s \Rightarrow r \leq s$;
in particular, $\mu_1=\lambda$.
Let $P_r$ be the subspace of $P$ spanned by
$\big\{y x\:\big|\: (y,x) \in \bigcup_{s=r+1}^n \bigcup_{a \in \B_{\mu_s}}  Y(a)\times X(a,b)\big\}$.
This defines a 
filtration $P = P_0 > P_1 > \cdots > P_n = 0$ in which the section
 $P_{r-1}/P_r$ has basis $\big\{yx+P_r\:\big|\:(y,x) \in \bigcup_{a \in \B_{\mu_r}} Y(a)\times X(a,b)\big\}$.
Now we show that each $P_{r-1}/P_r$ 
has a $\Delta_\eps$-flag with sections of the form $\Delta_\eps(a)$
for $a \in \B_{\mu_r}$.
There are two cases:

\vspace{1.5mm}
\noindent
Case 1: $\eps(\mu_r) = +$.
In this case, there is an $A$-module isomorphism
$$
\theta:\bigoplus_{a \in \B_{\mu_r}}
\bigoplus_{x \in X(a,b)} 
\Delta(a)
\stackrel{\sim}{\rightarrow}
 P_{r-1}/P_r
$$
sending the basis vector $y \bar e_{a}\:(y \in Y(a))$ in the $x$th copy of
$\Delta(a)$
to $y x + P_r \in P_{r-1}/P_r$. This follows from properties of the
basis and is similar to the proof of (\ref{firms}).

\vspace{1.5mm}
\noindent
Case 2: $\eps(\mu_r) = -$.
Note that $P_{r-1}/P_r$ is naturally an $A_{\leq\mu_r}$-module.
Let $Q_r := \bar e_{{\mu_r}} (P_{r-1} / P_r)$. This is an
$A_{\mu_r}$-module with basis $\{x + P_r\:|\: x \in
X(a,b), a \in \B_{\mu_r}\}$.
We claim that the natural multiplication map
$$
A_{\leq\mu_r} \bar e_{{\mu_r}} \otimes_{A_{\mu_r}} Q_r
\rightarrow P_{r-1} / P_r,
\qquad
y \bar e_{{\mu_r}} \otimes (x+P_r) \mapsto y x + P_r
$$ 
is an isomorphism.
This follows because the module on the left hand side is spanned by the vectors $
\big\{y \bar e_{{\mu_r}} \otimes (x+P_r)\:\big|\:
(y,x) \in \bigcup_{a \in \B_{\mu_r}} Y(a) \times X(a,b)\big\}$, 
and the images of these vectors under multiplication
are a basis for the module on the right.
Hence, $P_{r-1}/ P_r \cong j_!^{\mu_r} Q_r$. We deduce that 
it has a $\Delta_\eps$-flag
with sections of the form $\bar\Delta(a)\:(a \in \B_{\mu_r})$
on applying the standardization functor to a composition series for
$Q_r$, using the exactness from 
Lemma~\ref{xmas}.

\vspace{1.5mm}
\noindent
We can now complete the proof of the claim. The only thing left is to
check that the top section of the $\Delta_\eps$-flag we have
constructed so far is isomorphic to $\Delta_\eps(b)$.
This follows from the constructions just explained:
in the case $\eps(\lambda)=+$ we showed that 
$P_0 / P_1 \cong \Delta(b) = \Delta_\eps(b)$, while
if $\eps(\lambda)=-$ then the 
top section is $j_!^\lambda L_\lambda(b) = \Delta_\eps(b)$.

Using the claim just established, we can now classify the irreducible
$A$-modules.
For $b \in \B_\lambda$, the proper standard module
$\Delta_\eps(b)$ has irreducible head denoted $L(b)$.
This follows by the usual 
properties of adjunctions and the quotient functor $j^\lambda:A_{\leq
  \lambda}\lfdlmod \rightarrow A_\lambda \fdlmod, V \mapsto \bar
e_{{\lambda}} V$.
Moreover, $L(b)$ is the unique (up to isomorphism) 
irreducible $A_{\leq\lambda}$-module
such that $j^\lambda L(b) \cong L_\lambda(b)$,
hence, the modules $\{L(b)\:|\:b\in\B\}$ are pairwise
inequivalent.
To see that they give a full set of irreducible $A$-modules, let $L$
be any irreducible $A$-module.
By the analog of Remark~\ref{jabs}, there exists $b \in \B$ 
such that $e_b L \neq 0$.
Then $L$ is a quotient of $P_b = A e_b$.
Finally, using the claim, we deduce that
$L$ is a quotient of
$\Delta_\eps(a)$ for some $a \in \B$ with $\rho(a) \geq \rho(b)$ and thus $L$ is isomorphic to  $L(a)$.

Having classified the irreducible $A$-modules $\{L(b)\:|\:b \in \B\}$, 
$(\B,L,\rho,\Lambda,\leq)$ defines a stratification of $\R$.
We are in the recollement situation of (\ref{rst}), with
$\R_\lambda$ identified with $A_\lambda\fdlmod$.
Since (\ref{cellmodules1})--(\ref{cellmodules2}) 
agrees with (\ref{rumble}), the standard, proper standard,
costandard and proper costandard modules are the correct objects.
Moreover, the claim established at the start of the proof verifies the
property $(\widehat{P\Delta}_\eps)$.
\end{proof}

The goal in the remainder of the subsection is to prove a converse to
Theorem~\ref{onewayplus}.

\begin{theorem}[Based $\eps$-stratified algebras from $\eps$-highest weight categories]\label{TBA2new}
Let $\R$ be a finite (resp., lower finite, resp., tilting-bounded essentially finite) $\eps$-stratified category with stratification $(\B,L,\rho,\Lambda,\leq)$.
Suppose we are given $\B \subseteq I$ disjoint from $\Lambda$ and an $\eps$-tilting generator $T = \bigoplus_{i \in I} T_i$
such that each $T_b\:(b \in \B)$ is a direct sum of $T_\eps(b)$ and other $T_\eps(c)$
for $c$ with $\rho(c) < \rho(b)$.
Let $$
A := \left(\bigoplus_{i,j \in I} \Hom_\R(T_i, T_j)\right)^\op
$$
For $i,j \in I$ and $b \in \B$, pick morphisms
$$
Y(i,b) \subset \Hom_\R(T_i, T_b),\qquad X(b,j)
\subset \Hom_\R(T_b,T_j)
$$
lifting bases
for $\Hom_\R(T_i, \nabla_\eps(b))$ and $\Hom_\R(\Delta_\eps(b),T_j)$
as in Theorem~\ref{astfact} such that
$Y(b,b) = \{\id_{T_b}\}$ when $\eps(b) = +$ and
$X(b,b) = \{\id_{T_b}\}$ when $\eps(b) = -$.
These choices make $A$
into a finite (resp., upper finite, resp., essentially finite) 
based $(-\eps)$-stratified algebra 
with respect to the poset $(\Lambda,\geq)$ (the opposite ordering on $\Lambda$ compared to $\R$).
\end{theorem}

\begin{proof}
We need to check the axioms ($\eps$S6)--($\eps$S9).
Theorem~\ref{astfact} checks the first one.
The axioms ($\eps$S7)--($\eps$S8) also hold.
For example, if $\eps(\lambda) = +$ and $b \in \B_\lambda$, we have 
that $Y(b,b) = \{e_b\}$ by the choice of lifts, and 
$\Hom_{\R}(T_b, \nabla_{\eps}(a))$ is zero unless $a=b$ or $\rho(a) < \rho(b)$ (remembering we are checking these axioms for $-\eps$ not $\eps$).
It remains to check the final axiom ($\eps$S9). 
The algebra $A_\lambda$ in the statement of the axiom 
(remembering that we are working now with the opposite ordering) is the same as the algebra $A_\lambda$ in Lemma~\ref{aftersun}. By that lemma, there is an algebra isomorphism
\begin{equation}\label{windy}
\phi_\lambda:
A_\lambda \stackrel{\sim}{\rightarrow} \End_{\R_\lambda}(j^\lambda T_\lambda)^\op,
\end{equation}
where $T_\lambda := \bigoplus_{b \in \B_\lambda} T_b$.
If $\eps(\lambda)=+$ then $j^\lambda T_\lambda$ is 
a minimal projective generator for $\R_\lambda$
 thanks to
Theorem~\ref{gin}(3),
so the algebra on the
right hand side of (\ref{windy}) 
is basic and
$\bar e_\lambda = \sum_{b \in \B_\lambda} \bar e_b$
is a decomposition of its identity element as a sum of mutually orthogonal primitive idempotents.
If $\eps(\lambda)=-$, we have instead that $j^\lambda T_\lambda$
is a minimal injective cogenerator for $\R_\lambda$ and the conclusion follows similarly.
\end{proof}

\begin{corollary}\label{TBA2}
Let $\R$ be a finite (resp., upper finite, resp., tilting-bounded essentially finite)
$\eps$-stratified category with
the usual stratification $(\B,L,\rho,\Lambda,\leq)$. Let
$A = \bigoplus_{i, j \in I} e_i A e_j$ be an algebra realization of $\R$.
There is an idempotent 
expansion $A = \bigoplus_{i,j \in \hat I} \hat e_i A \hat e_j$
with $\B \subseteq \hat I$, and finite sets
$Y(i,b) \subset \hat e_i A \hat e_b, X(b,j) \subset \hat e_b A \hat
e_j$ for all $i,j \in \hat I$ and $b\in \B$,
making $A$ into a finite (resp., upper finite, resp., essentially finite) based
$\eps$-stratified algebra
with $\rho$ as its stratification function.
\end{corollary}

\begin{proof}
This follows from Theorem~\ref{TBA2new} in the same way as Corollary~\ref{TBA}
was deduced from Theorem~\ref{TBAnew}.
\end{proof}

\subsection{Based stratified and properly stratified algebras}\label{newsec3}
In this subsection, we consider modified versions of
Definitions~\ref{eclairs} and \ref{eclairs2} which
involve bases which do not depend 
on the sign function $\eps$.
These definitions, which were inspired in part by
\cite[Def.~2.17]{ELauda}, are relevant when studying
fully stratified rather than 
merely $\eps$-stratified categories.

\begin{definition}\label{strawberries}
A {\em finite} (resp., {\em upper finite}, resp., {\em essentially finite}) {\em based stratified algebra} is a 
finite-dimensional (resp.,
locally finite-dimensional, resp., essentially finite-dimensional) locally
unital algebra $A = \bigoplus_{i,j \in I} e_i A e_j$ with the
following additional data:
\begin{itemize}
\item[(S1)] A subset $\B \subseteq I$ indexing {\em 
special idempotents}
  $\{e_b\:|\:b\in \B\}$.
\item[(S2)]
A poset $(\Lambda,\leq)$ 
which is upper finite in the upper finite case and
interval finite in the essentially finite case, such that $\Lambda \cap I = \varnothing$.
\item[(S3)] A function 
$\rho:\B\rightarrow \Lambda$ with finite fibers $\B_\lambda
:= \rho^{-1}(\lambda)$.
\item[(S4)]
Sets $Y(i,a) \subset e_i A e_a$, 
$H(a,b) \subset e_a A e_b$, $X(b,j) \subset e_b A e_j$
for $i,j \in I$ and $a,b \in \B$.
\end{itemize}
Let 
$Y(a) := \bigcup_{i \in I} Y(i,a)$ and
$X(b) 
:= \bigcup_{j \in I} X(b,j)$.
The axioms are as follows:
\begin{itemize}
\item[(S5)] The products $y h x$ for $(y,h,x) \in \bigcup_{a,b\in\B} Y(a) \times H(a,b) \times X(b)$
  are a basis for $A$.
\item[(S6)]  For $a,b \in \B$ with $a \neq b$, the set $H(a,b)$ is empty unless $\rho(a) = \rho(b)$,
the sets $Y(b,a)$ and $X(a,b)$ are
empty unless $\rho(b) <\rho(a)$,
and $Y(a,a) = X(a,a) = \{e_a\}$.
\item[(S7)] 
The finite-dimensional algebra $A_\lambda$ defined as in Definition~\ref{eclairs} is
basic and 
$\bar e_\lambda = \sum_{b \in \B_\lambda} \bar e_b$
is a decomposition of its identity element as a sum of mutually orthogonal primitive idempotents.
\end{itemize}
We say that $A$ is {\em symmetrically based} if there is also some given
algebra anti-involution
$\sigma:A \rightarrow A$ with $\sigma(e_i) = e_i$ and
$Y(i,b) = \sigma(X(b,i))$ for all $i\in I, b \in \B$,
such that each of the algebras $A_\lambda\:(\lambda \in \Lambda)$ is $\sigma_\lambda$-symmetric
in the sense of Definition~\ref{sigmasymmetric}, where $\sigma_\lambda$ here is the anti-involution of $A_\lambda$ induced by $\sigma$.
 \end{definition}

Here is the same definition rewritten in the special case that the
stratification function $\rho$ is a bijection.

\begin{definition}\label{strawberries2}
A {\em finite} (resp., {\em upper finite}, resp., {\em essentially finite}) {\em based properly stratified algebra} is a finite-dimensional (resp., locally finite-dimensional, resp.,
essentially finite-dimensional) 
locally unital algebra $A = \bigoplus_{i,j \in I} e_i A e_j$ with the
following additional data:
\begin{itemize}
\item[(PS1)] A subset $\Lambda \subseteq I$ indexing {\em special idempotents}
  $\{e_\lambda\:|\:\lambda\in \Lambda\}$.
\item[(PS2)]
A poset $(\Lambda,\leq)$ upper finite in the upper finite case and 
interval finite in the essentially finite case.
\item[(PS3)]
Sets $Y(i,\lambda) \subset e_i A e_\lambda$, 
$H(\lambda) \subset e_\lambda A e_\lambda$, $X(\lambda,i) \subset e_\lambda A e_i$
for $\lambda \in \Lambda$, $i \in I$.
\end{itemize}
Let 
$Y(\lambda) := \bigcup_{i \in I} Y(i,\lambda)$ and
$X(\lambda) 
:= \bigcup_{i \in I} X(\lambda,i)$.
The axioms are as follows.
\begin{itemize}
\item[(PS4)] The products $y h x$ for $(y,h,x) \in \bigcup_{\lambda \in
    \Lambda} Y(\lambda) \times H(\lambda) \times X(\lambda)$
  are a basis for $A$.
\item[(PS5)]  For $\lambda,\mu \in \Lambda$, the sets
  $Y(\mu,\lambda)$ and $X(\lambda,\mu)$ are empty unless $\mu \leq
  \lambda$, and $Y(\lambda,\lambda) = X(\lambda,\lambda) =
  \{e_\lambda\}$.
\item[(PS6)] The finite-dimensional algebra $A_\lambda$ defined as in Definition~\ref{eclairs2}
  is basic and local.
\end{itemize}
We say that $A$ is {\em symmetrically based} if there is also some given
algebra anti-involution
$\sigma:A \rightarrow A$ with $\sigma(e_i) = e_i$ and
$Y(i,\lambda) = \sigma(X(\lambda,i))$ for all $i\in I, \lambda \in \Lambda$,
such that each of the algebras $A_\lambda\:(\lambda \in \Lambda)$ is $\sigma_\lambda$-symmetric, where $\sigma_\lambda$ here is the anti-involution of $A_\lambda$ induced by $\sigma$.
\end{definition}

In the remainder of the subsection, we just explain the results for based
stratified algebras, since based properly stratified algebras are a
special case.
For the next lemma,
we adopt the shorthands
\begin{align}
YH(i,b) 
&:= \left\{y h\:\big|\:
(y,h) \in \textstyle\bigcup_{a \in \B} Y(i,a) \times H(a,b)\right\},\\
HX(b,j) &
:= \left\{h x\:\big|\:
(h,x) \in \textstyle\bigcup_{a \in \B} H(b,a) \times X(a,j)\right\}.\label{lovely}
\end{align}
Also set
$YH(b) := \bigcup_{i \in I} YH(i,b)$
and
$HX(b) := \bigcup_{j \in I} HX(b,j)$.

\begin{lemma}\label{threat}
Suppose that $A$ is a based stratified algebra
as in Definition~\ref{strawberries}. Also let
$\eps:\Lambda\rightarrow \{\pm\}$ be any choice of sign function.
Then $A$ is a based $\eps$-stratified algebra
in the sense of Definition~\ref{eclairs}
with the required sets
$Y(i,b)$ and $X(b,j)$ for that
being the sets
$YH(i,b)$ and $X(b,j)$ 
in the present setup
if $\eps(\rho(b))=+$,
or the sets $Y(i,b)$ and $HX(b,j)$ in the present setup if $\eps(\rho(b))=-$.
\end{lemma}

\begin{proof}
This follows on comparing Definitions~\ref{eclairs} and
\ref{strawberries}.
\end{proof}

This means that the results from the previous
subsection apply 
to based stratified algebras too.
In particular, we define the standard, proper standard, costandard and
proper costandard modules as in 
(\ref{cellmodules1})--(\ref{cellmodules2}). The modules $\Delta(b)$ and $\bar\Delta(b)$ have
standard bases
$\{y \bar e_b\:|\:y \in YH(b)\}$ and 
$\{y \tilde e_b\:|\:y \in Y(b)\}$, respectively. Similarly, one can
introduce costandard bases for $\nabla(b)$ and $\bar\nabla(b)$ indexed
by the sets $HX(b)$ and $X(b)$, respectively.
Note also that the basic algebra $$
A_\lambda = \bigoplus_{a,b \in \B_\lambda} \bar e_a A_\lambda \bar e_b
$$ 
has basis $\big\{\bar h\:\big|\:h \in \bigcup_{a,b \in \B_\lambda} H(a,b)\big\}$.

\begin{theorem}[Fully stratified categories from based stratified algebras]\label{twoway}
Let $A$ be a finite (resp., upper finite, resp., essentially finite) 
based stratified algebra as above.
The modules $$
\{L(b) := \hd \Delta(b)\cong \hd \bar\Delta(b)\cong\soc\bar\nabla(b) \cong \soc\nabla(b)\:|\:b \in \B\}$$ 
give a full
set of pairwise inequivalent irreducible left $A$-modules.
Moreover, 
$\R := A\fdlmod$ (resp., $\R:=A\lfdlmod$, resp., $\R := A\fdlmod$) is a finite (resp., upper finite, resp.,
essentially finite) fully stratified category with
stratification $(\B,L,\rho,\Lambda,\leq)$ with strata $\R_\lambda := A_\lambda\fdlmod$.
If $A$ is symmetrically based with anti-involution $\sigma$ then $?^\sigmadual:\R\rightarrow\R$ is a Chevalley duality of $\R$ in the sense of Definition~\ref{chevinvdef}.
\end{theorem}

\begin{proof}
Using Lemma~\ref{threat}, 
the first part follows from Theorem~\ref{onewayplus} applied twice,
once with $\eps=+$ and once with $\eps=-$.
For the final part about Chevalley duality, axiom (Ch1) from Definition~\ref{chevinvdef} is established in the course of the proof of Theorem~\ref{onewayplus}, and (Ch2) follows from the definition of symmetrically
based stratified algebra.
\end{proof}

For the converse recall the definition of tilting-rigid from Definition~\ref{tiltingrigiddef}.

\begin{theorem}[Based stratified algebras from fully stratified categories]\label{TBA3new}
Let $\R$ be a 
finite (resp., lower finite, resp., essentially finite) fully stratified category with stratification $(\B,L,\rho,\Lambda,\leq)$.
Assume that $\R$ is tilting-rigid with weakly symmetric strata.
Suppose we are given $\B \subseteq I$ disjoint from $\Lambda$ and
a tilting generator $T = \bigoplus_{i \in I} T_i$
such that each $T_b\:(b \in \B)$ is a direct sum of $T(b)$ and other $T(c)$
for $c$ with $\rho(c) < \rho(b)$.
Let
$$
A := \left(\bigoplus_{i,j \in I} \Hom_\R(T_i, T_j)\right)^\op
$$
\begin{itemize}
\item[(1)]
For $i,j \in I$ and $a,b \in \B$, pick morphisms
$$
Y(i,a) \subset \Hom_\R(T_i, T_a), \:
H(a,b) \subset \Hom_\R(T_a,T_b), \: 
X(b,j) \subset \Hom_\R(T_b,T_j)
$$
lifting bases
for $\Hom_\R(T_i, \bar\nabla(a)$, $\Hom_\R(\Delta(a),\nabla(b))$  and $\Hom_\R(\bar\Delta(b),T_j)$
as in Theorem~\ref{upgradedastfact}
such that
$Y(b,b) = X(b,b) = \{\id_{T_b}\}$.
These choices give a triangular basis making
into a finite (resp., upper finite, resp., essentially finite) 
based stratified algebra 
with respect to the poset $(\Lambda,\geq)$ (the opposite ordering on $\Lambda$ compared to $\R$).
\item[(2)]
If in addition $\R$ has a Chevalley duality $?^\vee$ and, in a suitable
realization, the modules corresponding to each $T_i$ possess non-degenerate symmetric bilinear forms 
satisfying the adjunction property as in (\ref{sunriver}), then the triangular basis in (1) can be chosen so that $A$ is symmetrically based.
\end{itemize}
\end{theorem}

\begin{proof}
Part (1) is similar to the proof of Theorem~\ref{TBA2new}, using Theorem~\ref{upgradedastfact}
in place of Theorem~\ref{astfact}. Part (2) follows in the same way as in the proof of Theorem~\ref{TBA}(2).
\end{proof}

\begin{corollary}\label{TBA3}
Let $\R$ be a finite (resp., upper finite, resp., essentially finite) fully stratified category with
stratification $(\B,L,\rho,\Lambda,\leq)$.
Let $A = \bigoplus_{i, j \in I} e_i A e_j$ be an algebra realization of $\R$.
\begin{itemize}
\item[(1)]
Assume that $\R$ is tilting-rigid with weakly symmetric strata.
Then there is an idempotent 
expansion $A = \bigoplus_{i,j \in \hat I} \hat e_i A \hat e_j$
with $\B \subseteq \hat I$, and finite sets
$$
Y(i,a) \subset \hat e_i A \hat e_a, 
\:
H(a,b) \subset \hat e_a A \hat e_b, 
\:
X(b,j) \subset \hat e_b A \hat e_j
$$ 
for all $i,j \in \hat I$ and $a,b \in \B$,
making $A$ into an upper finite (resp., essentially finite)
based stratified algebra.
\item[(2)]
Assume that $\R$ is tilting-rigid with a Chevalley duality $?^\vee$
and that $\operatorname{char} \k \neq 2$.
Then the choices in (1) can be made so that $A$ is symmetrically based
 with anti-involution $\sigma$ realizing $?^\vee$.
 \end{itemize}
\end{corollary}

\begin{proof}
This follows from Theorem~\ref{TBA3new} in the same way as Corollary~\ref{TBA}
was deduced from Theorem~\ref{TBAnew}. One also needs to use the fact that the Ringel dual $\R'$ of $\R$ is tilting-rigid by Theorem~\ref{tiltingrigiddual}.
\end{proof}

\subsection{Algebras with a triangular basis}\label{stb}
The final axiom (S7) of Definition~\ref{strawberries}, namely, 
that the algebra $A_\lambda$ is basic, 
is quite restrictive.
However, this assumption is not essential, as we will explain in this subsection. The following simply repeats
Definition~\ref{strawberries} with the final axiom dropped, but at the same time we switch to using the notation $\partial:\BS\rightarrow\Lambda$ where we had $\rho:\B\rightarrow \Lambda$ 
before.

\begin{definition}\label{raspberries}
Let $A = \bigoplus_{i,j \in I} e_i A e_j$ be a 
finite-dimensional (resp., locally finite-dimensional, resp., essentially finite-dimensional) locally unital algebra.
We say that $A$ 
has a {\em triangular basis} if we are given the
following additional data:                          
\begin{itemize}
\item[(TB1)] A subset $\BS \subseteq I$ indexing {\em special idempotents}
  $\{e_s\:|\:s\in \BS\}$.
\item[(TB2)] A poset $(\Lambda,\leq)$
which is upper finite in the locally finite-dimensional
case and interval finite in the essentially finite-dimensional 
case, such that $\Lambda\cap I = \varnothing$.
\item[(TB3)] A function 
$\partial:\BS\rightarrow \Lambda$ with finite fibers $\BS_\lambda := \partial^{-1}(\lambda)$.
\item[(TB4)]
Sets $Y(i,s) \subset e_i A e_s$, 
$H(s,t) \subset e_s A e_t$, $X(t,j) \subset e_t A e_j$
for $i,j\in I$ and $s,t \in \BS$.
\end{itemize}
Let 
$Y(s) := \bigcup_{i \in I} Y(i,s)$ and
$X(t) 
:= \bigcup_{j \in I} X(t,j)$.
The axioms are as follows:
\begin{itemize}
\item[(TB5)] The products $y h x$ for  $(y,h,x) \in \bigcup_{s,t \in \BS} Y(s) \times H(s,t) \times X(t)$
  are a basis for $A$.
\item[(TB6)]  For $s,t \in \BS$ with $s\neq t$, the set $H(s,t)$ is empty unless $\partial(s) = \partial(t)$, the sets $Y(t,s)$ and $X(s,t)$ are
  empty unless $\partial(t) < \partial(s)$,
and $Y(s,s)=X(s,s) = \{e_s\}$.
\end{itemize}
\end{definition}

Suppose that $A$ has a triangular basis as in Definition~\ref{raspberries}. 
We define algebras $A_\lambda = \bar e_{{\lambda}} A_{\leq\lambda} \bar e_{{\lambda}}$ for each $\lambda \in \Lambda$ 
like at the end of Definition~\ref{eclairs}. 
Thus, we let $e_{{\lambda}} := \sum_{s \in \BS_\lambda} e_s$,
then set
$A_\lambda := \bar e_{{\lambda}}
A_{\leq\lambda} \bar e_{{\lambda}}$ 
where $A_{\leq\lambda}$ is the quotient of $A$ by the two-sided ideal generated by
    $\{e_{{\mu}}\:|\:\mu\not\leq\lambda\}$.
Corollary~\ref{climbing2} carries over to show that
$A_{\leq\lambda}$ has basis $\bar y \bar h \bar x$ for all $y \in Y(s), h \in
H(s,t), x \in X(t)$ and $s,t \in \BS$ with $\partial(s), \partial(t)\leq \lambda$.
Hence, $A_\lambda$ has basis $\big\{\bar h\:\big|\:h \in \bigcup_{s,t \in \BS_\lambda}
H(s,t)\big\}$. 
Let $j^\lambda:A_{\leq\lambda}\lmod \rightarrow A_\lambda\lmod, V \mapsto \bar e_{{\lambda}} V$ be the 
quotient functor and define $j^\lambda_!$ and $j^\lambda_*$ analogously to (\ref{pinkeye}).

\begin{lemma}\label{xmas2}
The functors $j^\lambda_!$ and $j^\lambda_*$ are exact.
\end{lemma}

\begin{proof}
By the argument from the proof of Lemma~\ref{xmas}, there is an isomorphism of right $A_\lambda$-modules
$\bigoplus_{s \in \BS_\lambda} \bigoplus_{y \in Y(s)}
\bar e_s A_\lambda
\stackrel{\sim}{\rightarrow}
A_{\leq\lambda} \bar e_{{\lambda}}
$
sending the vector
$\bar e_s$ in the $y$th copy of $\bar e_s A_\lambda$ to $\bar y \in A_{\leq\lambda}
\bar e_{{\lambda}}$.
So the
right $A_\lambda$-module $A_{\leq\lambda} \bar e_{{\lambda}}$ 
is  projective, which implies the exactness of $j^\lambda_!$.
Similarly, 
the left $A_\lambda$-module $\bar e_{{\lambda}} A_{\leq\lambda}$ 
is projective, which implies the exactness of $j^\lambda_*$.
\end{proof}

  The following theorem is essentially \cite[Th.~3.5]{GRS},
  although we give a self-contained proof since our notation is different enough. See Remark~\ref{history} for
  further historical discussion.

\begin{theorem}[Fully stratified categories from triangular bases]\label{grs}
Let $A$ be a finite-dimensional (resp., locally finite-dimensional, resp. essentially finite-dimensional algebra with a triangular basis as above.
Let $\rho:\B \rightarrow \Lambda$ be a function whose fibers $\B_\lambda := \rho^{-1}(\lambda)$ label a full set $\{L_\lambda(b)\:|\:b \in \B_\lambda\}$
of pairwise inequivalent irreducible left $A_\lambda$-modules.
Let $\bar\Delta(b) := j^\lambda_! L_\lambda(b)$ and
$\bar\nabla(b) := j^\lambda_* L_\lambda(b)$
for $b \in \B_\lambda$.
Then the modules 
$$
\{L(b):= 
\hd \bar\Delta(b) \cong \soc \bar\nabla(b) \:|\:b \in \B\}
$$ 
give a full set of pairwise inequivalent irreducible left $A$-modules.
Moreover, $\R := A \fdlmod$ (resp., $\R := A\lfdlmod$, resp., $\R := A\fdlmod$) is a finite (resp., upper finite, resp., essentially finite) 
fully stratified category with stratification $(\B,L,\rho,\Lambda,\leq)$. Its strata are the categories $\R_\lambda := A_{\lambda}\fdlmod$, with
standardization and
costandardization functors as in (\ref{pinkeye}).
\end{theorem}

\begin{proof}
Take $u \in \BS_\lambda$ and any
$b \in \B_\lambda$ such that $\bar e_u L_\lambda(b) \neq 0$.
We claim that 
$A e_u$ has a
$\bar\Delta$-flag with $\bar\Delta(b)$ at the top and other
sections of the form $\bar\Delta(c)$ for $c$ with
$\rho(c) \geq \lambda$.
To see this, let $P := A e_u$ for short.
Note $P$ has basis
$$\big\{yhx\:\big|\:\textstyle(y,h,x) \in \bigcup_{\mu \geq \lambda} \bigcup_{s,t \in \BS_\mu} Y(s) \times H(s,t) \times X(t,u)\big\}.
$$
Let 
$\{\mu_1,\dots,\mu_n\}$ be the finite set
$
\big\{\mu \in [\lambda,\infty)\:\big|\:\textstyle\bigcup_{t \in S_\mu} X(t,u) \neq \varnothing\big\}
$ 
enumerated in some order refining $\leq$.
There is a filtration $P = P_0 > P_1 > \cdots > P_n = 0$
in which the section $P_{r-1}/P_r$ has basis 
$\big\{yhx+P_r\:\big|\:(y,h,x) \in \bigcup_{s,t \in \BS_{\mu_r}} Y(s) \times H(s,t) \times X(t,u)\big\}$.
Moreover, $
P_{r-1} / P_r \cong j^{\mu_r}_! Q_r$
 where 
$Q_r := \bar e_{{\mu_r}} (P_{r-1}/P_r)$.
This follows by a similar argument to the Case 
2 in the proof of Theorem~\ref{onewayplus}.
Since $j^{\mu_r}_!$ is
exact by Lemma~\ref{xmas2}, it follows that $P_{r-1} / P_r$ has a $\bar\Delta$-flag with
sections $\bar\Delta(c)$ for $c \in \B_{\mu_r}$.
So we have proved that $P$ has a $\bar\Delta$-flag with sections $\bar\Delta(c)$ for $c \in \B$ with $\rho(c) \geq \lambda$.
Moreover, $P_0 / P_1 \cong j_!^{\lambda} (A_\lambda \bar e_u)$. Since
$A_\lambda \bar e_u$ has $L_\lambda(b)$ in its head, it follows
that the $\bar\Delta$-flag can be chosen so that it has $\bar\Delta(b)$ at its top.

Now we can classify the irreducible left $A$-modules.
As in the penultimate paragraph of the proof of Theorem~\ref{onewayplus}, the modules
$\{L(b) := \hd \bar\Delta(b)\:|\:b \in \B\}$
are pairwise inequivalent irreducible $A$-modules.
It remains to show that any irreducible left $A$-module $L$ is isomorphic to some such module.
There exists $u \in \BS$ such that $e_u L \neq 0$. Hence, $L$ is a quotient of $A e_u$. By considering the filtration of $A e_u$ from the previous paragraph we deduce that $L$ is a quotient of $\bar \Delta(c)$ for some $c \in \B$, i.e., $L \cong L(c)$.

At this point, we have in hand the data of a stratification of $\R$
with strata 
$\R_\lambda := A_{\lambda}\fdlmod$ and
standardization and
costandardization functors as in (\ref{pinkeye}). 
For each $b \in \B_\lambda$, choose  $u \in \BS_\lambda$ such that $\bar e_u L_\lambda(b) \neq 0$ then set $P_b := A e_u$.
The claim established in the first paragraph of the proof
checks that
these modules satisfy the property
$(\widehat{P\Delta}_-)$, hence, $\R$ is an upper finite (resp., essentially finite) $-$-stratified category. Finally we deduce that it is fully stratified using 
the criterion from Lemma~\ref{rain}(iv) plus Lemma~\ref{xmas2}.
\end{proof}

\begin{corollary}\label{laughter}
Let $A$ be as above.
If each of the finite-dimensional algebras $A_\lambda$ is quasi-hereditary (e.g., they could all be semisimple), then the stratification can be refined to make the category $\R$ from Theorem~\ref{grs} into a
highest weight category. 
\end{corollary}

\begin{proof}
Combine Theorem~\ref{grs} and Corollary~\ref{eve}.
\end{proof}

\begin{remark}\label{history}
We did not fully appeciate the utility of Definition~\ref{raspberries}
before seeing \cite{GRS}, in which Gao, Rui and Song
introduce a notion of an
algebra with a {\em weak triangular decomposition}
and give a (slightly different) proof of Theorem~\ref{grs} for such
algebras. They justify their definition by
constructing several interesting families of examples, namely, cyclotomic
quotients of the affine
oriented Brauer and HOMFLY-PT skein categories and of the affine
Brauer and Kauffman skein categories.
In the special case that $I =
\BS$, i.e., all distinguished idempotents are special, our notion of an
algebra with a triangular basis is
exactly equivalent to the notion
of an algebra with a weak triangular decomposition.
More precisely, a weak triangular decomposition is the data of
subspaces $A^- = \bigoplus_{i,j \in I} e_i A^- e_j, A^\circ =
\bigoplus_{i,j \in I} e_i A^\circ e_j,
A^+ = \bigoplus_{i,j \in I} e_i A^+ e_j$ for $i,j \in I$ subject to certain axioms.
Picking homogeneous bases $Y(i,j), H(i,j)$ and $X(i,j)$ for $e_i A^- e_j,
e_i A^\circ e_j$ and $e_i A^+ e_j$, respectively,
produces a triangular basis in the sense of
Definition~\ref{raspberries}. Conversely given
a triangular basis one obtains a weak triangular decomposition by
replacing the bases by the subspaces that they span.
\end{remark}

\subsection{Algebras with a triangular decomposition}\label{std}
Let $A$ be an algebra with a triangular basis as in
Definition~\ref{raspberries} and assume in addition that
$I = \BS$, i.e., all of the distinguished idempotents are special.
Let
$A^\flat$ and $A^\sharp$ be the subspaces spanned by
$\big\{yh\:\big|\:(y,h) \in \textstyle\bigcup_{i,j\in I} Y(i) \times H(i,j)\big\}$
and
$\big\{hx\:\big|\:(h,x) \in \textstyle\bigcup_{i,j\in I} H(i,j) \times X(j)\big\}$,
respectively.
If it happens that
these subspaces are locally unital subalgebras\footnote{Locally unital
    subalgebra means subspace closed under multiplication and containing
    all of the distinguished idempotents.} of $A$
then $A$ has a triangular decomposition in the following sense.

\begin{definition}\label{prankster}
Let $A = \bigoplus_{i, j \in I} e_i A e_j$ be a finite-dimensional (resp., locally finite-dimensional, resp., essentially finite-dimensional) locally unital algebra.
A {\em triangular decomposition} of $A$
is the following additional data:
\begin{itemize}
\item[(TD1)]
A poset $(\Lambda,\leq)$ which is upper finite in the locally finite-dimensional case or interval finite in the essentially finite-dimensional case.
\item[(TD2)]
A function $\partial:I \rightarrow \Lambda$ with finite
fibers $I_\lambda := \partial^{-1}(\lambda)$.
\item[(TD3)] Locally unital subalgebras $A^\flat$ and $A^\sharp$.
\end{itemize}
We call $A^\flat$ and $A^\sharp$ the {\em negative} and {\em positive Borel subalgebras}.
Let $A^\circ := A^\flat \cap A^\sharp$. 
This is also a locally unital subalgebra called the {\em Cartan subalgebra}.
The following axioms are required to hold:
\begin{itemize}
\item[(TD4)] $A^\flat$ is a projective right $A^\circ$-module
  and $A^\sharp$ is a projective left $A^\circ$-module.
\item[(TD5)] The natural multiplication map 
$A^\flat \otimes_{A^\circ} A^\sharp \rightarrow
A$ is a linear isomorphism.
\item[(TD6)] For $i,j \in I$,
  $e_j A^\flat e_i$ and $e_i A^\sharp e_j$ are zero unless
  $\partial(j) \leq \partial(i)$,
  and $e_i A^\flat e_j = e_i A^\sharp e_j$ when $\partial(i) = \partial(j)$.
\end{itemize}
\end{definition}

\begin{remark}\label{triangularhistory2}
Our formulation of Definition~\ref{prankster} has been influenced by the definition of a {\em triangular category} from a
  recent preprint of Sam and Snowden \cite{SS}; these are finite-dimensional categories
satisfying equivalent axioms to algebras with an
  upper finite triangular decomposition in the above sense in which the Cartan subalgebra is semisimple.
In an earlier draft, we had formulated a
slightly more restrictive notion which we now refer to a {\em split} 
triangular decomposition, as follows. 
Let $A = \bigoplus_{i,j \in I} e_i A e_j$ be a finite-dimensional (resp., locally finite-dimensional, resp., essentially finite-dimensional) locally unital algebra.
We say that $A$ has a {\em split triangular decomposition} if we have the additional data:
\begin{itemize}
\item[(STD1)]
A poset $(\Lambda,\leq)$ which is upper finite in the locally finite-dimensional case and interval finite in the essentially finite-dimensional case.
\item[(STD2)]
A function $\partial:I \rightarrow \Lambda$ with finite
fibers $I_\lambda := \partial^{-1}(\lambda)$.
\item[(STD3)] Locally unital subalgebras $A^-$, $A^\circ$ and $A^+$.
\end{itemize}
Letting $\K := \bigoplus_{i \in I} \k e_i$, the axioms are:
\begin{itemize}
\item[(STD4)] The subspaces $A^\flat := A^- A^0$ and $A^\sharp := A^0
  A^+$
  are subalgebras.
\item[(STD5)] The natural multiplication map 
$A^- \otimes_{\K} A^\circ \otimes_{\K} A^+ \rightarrow A$ is a linear
isomorphism.
\item[(STD6)] For $i,j \in I$ with $i \neq j$,
    $e_i A^\circ e_j$ is zero unless $\partial(i) = \partial(j)$,
  $e_j A^- e_i$ and  $e_i A^+ e_j$ are zero unless $\partial(j) < \partial(i)$, and
  $e_i A^\flat e_i  = e_i A^\sharp e_i = \k e_i$ for all $i \in I$.
\end{itemize}
The axiom (STD5) implies that $A^\flat \cong A^- \otimes_\K A^\circ$
and $A^\sharp \cong A^\circ \otimes_\K A^\sharp$. Hence, by associativity of tensor product we have that
$$
A^\flat \otimes_{A^\circ} A^\sharp
\cong
A^- \otimes_{\K} A^\circ \otimes_{A^\circ} A^\circ \otimes_{\K} A^+
\cong
A^- \otimes_\K A^\circ\otimes_\K A^+ \cong A,
$$
proving (TD5).
Moreover, the isomorphisms 
$A^\flat \cong A^- \otimes_\K A^\circ$
and $A^\sharp \cong A^\circ \otimes_\K A^\sharp$ show
that $A^\flat$ and $A^\sharp$ are $I$-free 
in the sense of Definition~\ref{locfreedef}
as right and left $A^\circ$-modules, respectively, which implies (TD4).
Axiom (TD6) is also easily deduced from (STD6).
When they hold, the axioms (STD4)--(STD6) are easier to check than (TD4)--(TD6),
so this gives a practical way to obtain triangular decompsitions. In fact,
most of the examples of triangular
decompositions arising from diagrammatic
monoidal categories considered in \cite{SS} and elsewhere are split
triangular decompositions, so the split formulation is useful.
 \end{remark}

\begin{remark}\label{triangularhistory1}
In \cite{HN},
Holmes and Nakano introduced a notion of a $\Z$-graded algebra
with a triangular decomposition. To explain the connection to our setup,
suppose we are 
given a unital $\Z$-graded algebra $\tilde A = \bigoplus_{\lambda \in \Z} \tilde A_\lambda$.
There is an associated locally unital algebra $A =
\bigoplus_{\lambda,\mu \in \Z} e_\lambda A e_\mu$ 
with $e_\lambda A e_\mu := \tilde
A_{\lambda-\mu}$ and multiplication induced by multiplication in $\tilde A$ in the
natural way. 
Moreover, any $\Z$-graded left
$\tilde A$-module $V = \bigoplus_{\lambda \in \Z} V_\lambda$ can be viewed as a left $A$-module 
with $e_\lambda V := V_\lambda$; this 
defines an isomorphism from the usual category $\tilde
A\lgmod$ of $\Z$-graded $\tilde A$-modules and degree-preserving
morphisms to the category $A\lmod$ of locally unital
$A$-modules. 
If we start with $\tilde A$ that is a finite-dimensional $\Z$-graded algebra
with a triangular decomposition $(\tilde A^-, \tilde A^\circ, \tilde A^+)$ 
as in \cite{HN} (see also \cite[Def.~3.1]{BT})
then the essentially finite-dimensional locally unital algebra $A$ and the subalgebras $A^\circ, A^-$
and $A^+$
obtained via this construction has a split triangular
decomposition, with $I = \Lambda =
\Z$ ordered in the natural way.
\end{remark}

To make the connection
with Definition~\ref{raspberries}, 
suppose that $A$ has a triangular decomposition. 
For $\lambda \in \Lambda$, let 
$1_\lambda := \sum_{i \in I_\lambda} e_i$. The axioms imply that 
  $e_i A^\circ e_j = 0$ unless $\partial(i) = \partial(j)$, so
  $1_\lambda A^\circ 1_\mu = 0$ for $\lambda \neq \mu$.
It follows that $\{1_\lambda\:|\:\lambda \in \Lambda\}$
are mutually orthogonal central idempotents in $A^\circ$,
and the Cartan subalgebra has the ``block" decomposition
\begin{equation}\label{blockk}
  A^\circ =\bigoplus_{\lambda \in \Lambda} A_\lambda^\circ
  \qquad\text{ where }\qquad A_\lambda^\circ := 1_\lambda A^\circ = A^\circ 1_\lambda.
\end{equation}

\begin{lemma}\label{iflocallyfree}
Let $A$ be as in Definition~\ref{prankster}
with $\Lambda\cap I = \varnothing$.
Suppose we are given $\BS \subseteq I$
such that all
$e_i A^\flat$ and $A^\sharp e_j$ are $\BS$-free as 
right and left $A^\circ$-modules, respectively.
For $i,j \in I, s,t \in \BS$, one can choose subsets
$Y(i,s) \subset e_i A^\flat e_s$, $X(t,j) \subset
e_t A^\sharp e_j$ so that
\begin{itemize}
  \item[(i)]
$e_i A^\flat =\bigoplus_{s \in \BS} \bigoplus_{y \in Y(i,s)} y A^\circ$ with
$y A^\circ \cong e_s A^\circ$ for $y \in Y(i,s)$;
\item[(ii)]
  $A^\sharp e_j =\bigoplus_{t \in \BS} \bigoplus_{x \in X(t,j)} A^\circ x$ with
  $A^\circ x\cong A^\circ e_t$ for $x \in X(t,j)$;
\item[(iii)]
  $Y(t,t) =  X(t,t)=\{e_t \}$ for all $t \in \BS$.
 \end{itemize}
Also let $H(s,t)$ be a basis for $e_s A^\circ e_t$.
This makes $A = \bigoplus_{i,j \in \hat I} e_i A e_j$ into an algebra with a triangular basis in the sense
of Definition~\ref{raspberries} with $\partial:\BS \rightarrow \Lambda$ being the restriction of the given function $\partial:I \rightarrow \Lambda$.
For $\lambda \in \Lambda$ and 
$e_{{\lambda}} := \sum_{s \in \BS_\lambda} e_s$, 
the subquotient $A_\lambda = \bar e_{{\lambda}} A_{\leq \lambda} \bar e_{{\lambda}}$
defined after Definition~\ref{raspberries} is isomorphic to the subalgebra
$e_{{\lambda}} A^\circ_\lambda e_{{\lambda}}$ of $A^\circ_\lambda$. Moreover, we
have that $A^\circ_\lambda = A^\circ_\lambda e_{{\lambda}} A^\circ_\lambda$
so $A_\lambda$
is Morita equivalent to $A^\circ_\lambda$.
\end{lemma}

\begin{proof}
By the definition of $\BS$-free, there are subsets $Y(i,s) 
\subset e_i A^\flat e_s$ as in (i).
Since $e_i A^\flat e_s$ is zero unless $\partial(i) \leq \partial(s)$, we have that $Y(i,s) = \varnothing$ unless $\partial(i) \leq \partial(s)$.
Suppose that $t \in \BS_\lambda := \BS \cap I_\lambda$.
By (TD6), we have that
$$
e_t A^\flat 1_\lambda
= e_t A^\circ_\lambda 
= \bigoplus_{s \in \BS_\lambda} \bigoplus_{y \in Y(t,s)} y A^\circ_\lambda,
$$ 
i.e., the sets $Y(t,s)$ for $s \in \BS_\lambda$ come from an $\BS$-free
decomposition of $e_t A^\circ_\lambda$. 
This means we can choose them so that $Y(t,t) = \{e_t\}$ as in (iii), in which case
$Y(t,s) = \varnothing$ for $s \in \BS_\lambda$ with $s \neq t$. Hence, for $s,t \in \BS$ with $s \neq t$, we have that $Y(t,s) = \varnothing$ unless $\partial(t) < \partial(s)$.
Similarly, we choose subsets $X(t,j) \subset e_t A^\sharp e_j$ according to (ii) and (iii), and 
then for $s, t \in \BS$ with $s \neq t$ we have that $X(s,t) = \varnothing$ 
unless $\partial(t) < \partial(s)$.
Note also  that $H(s,t) = \varnothing$ unless $\partial(s) = \partial(t)$
due to (\ref{blockk}).
Thus we have the required data from (TB1)--(TB4), and the conditions of (TB6) are satisfied.

In this paragraph, we check (TB5).
Let $Y(s) = \bigcup_{i \in I} Y(i,s)$ and $X(t) = \bigcup_{j \in I} X(t,j)$.
We have seen already that 
$A^\flat = \bigoplus_{s \in \BS} \bigoplus_{y \in Y(s)} y A^\circ$
and $A^\sharp = \bigoplus_{t \in \BS} \bigoplus_{x \in X(t)} A^\circ x$.
Tensoring these together, we deduce that 
$$
A^\flat \otimes_{A^\circ} A^\sharp = \bigoplus_{s,t \in \BS} \bigoplus_{y \in Y(s), x \in X(t)}
y A^\circ \otimes_{A^\circ} A^\circ x.
$$
Each summand $y A^\circ \otimes_{A^\circ} A^\circ x$ here is isomorphic to
$e_s A^\circ \otimes_{A^\circ} A^\circ e_t \cong e_s A^\circ e_t$.
We deduce that $A^\flat \otimes_{A^\circ} A^\sharp$ has basis
$\big\{yh \otimes x = y \otimes hx\:\big|\: (y,h,x) \in \bigcup_{s,t \in \BS} Y(s)\times H(s,t)\times X(t)\big\}$.
Then we use (TD5) to see that the axiom (TB5) is satisfied.

Finally we must identify the algebra $A_\lambda$.
The quotient map $A \twoheadrightarrow A_{\leq \lambda}$
restricts to 
a homomorphism $\phi:A^\circ \rightarrow A_{\leq \lambda}$ which further restricts to
\begin{equation}\label{newyear}
\phi_\lambda: e_{{\lambda}} A^\circ_\lambda e_{{\lambda}}
\stackrel{\sim}{\rightarrow} A_\lambda.
\end{equation}
The subalgebra $A_\lambda^\circ$ has basis $\{yhx\:|\:(y,h,x) \in
\bigcup_{i,j \in I_\lambda, s,t \in \BS_\lambda}
Y(i,s) \times H(s,t) \times X(t,j)\}$, hence,
$A_\lambda^\circ = A^\circ_\lambda e_{{\lambda}} A_\lambda^\circ$.
The subalgebra $A_\lambda$ of
$e_{{\lambda}} A_\lambda^\circ e_{{\lambda}}$ has basis $\bigcup_{s,t \in \BS_\lambda} H(s,t)$.
It follows that $\phi_\lambda$ sends a basis to a basis, so it is an isomorphism.
\end{proof}

The freeness assumption in Lemma~\ref{iflocallyfree} may seem restrictive, but
one can always pass to an idempotent expansion so that this is the case.
In fact, we can do this in such a way that the algebras $A_\lambda$ are {\em basic}, thereby giving $A$ the structure of a based stratified algebra rather than merely an algebra with a triangular basis:

\begin{theorem}[Based stratified algebras from triangular decompositions]\label{myconstruction}
Suppose that $A$ has a triangular decomposition as in Definition~\ref{prankster}.
Let $A^\circ = \bigoplus_{i,j \in \hat I} \hat e_i A^\circ \hat e_j$ be an
idempotent expansion of $A^\circ = \bigoplus_{i,j \in I} e_i A^\circ e_j$
such that
\begin{itemize}
\item[(i)]
$\hat I \cap \Lambda = \varnothing$;
\item[(ii)] 
$\hat I$ contains a subset $\B$ indexing a full set
$\{\hat e_b\:|\:b \in \B\}$ of pairwise non-conjugate primitive
idempotents in $A^\circ$;
\item[(iii)] there is a function $q:\hat I \rightarrow I$
with $|q^{-1}(i)| < \infty$ and
$e_i = \sum_{j \in q^{-1}(i)} \hat e_j$ for $i \in I$.
\end{itemize}
Then $A = \bigoplus_{i,j \in \hat I} \hat e_i A \hat e_j$ 
has a triangular decomposition with the given Borel subalgebras,
taking the function from (TD2) now to be
$\rho := \partial \circ q:\hat I \rightarrow \Lambda$.
Moreover, 
$\hat e_i A^\flat$ and $A^\sharp \hat e_j$ are $\B$-free 
as right and left $A^\circ$-modules, respectively.
Hence, we can apply the construction of Lemma~\ref{iflocallyfree} 
to $A = \bigoplus_{i,j \in \hat I} \hat e_i A \hat e_j$ to
make $A$ into a based stratified algebra in the sense of Definition~\ref{raspberries}
with $\rho:\B\rightarrow\Lambda$ defined by restriction.
\end{theorem}

\begin{proof}
The fact that we have in hand a triangular decomposition of
$A = \bigoplus_{i,j \in \hat I} \hat e_i A \hat e_j$ is immediately clear 
from the nature of Definition~\ref{prankster}.
Since $1_\lambda A^\sharp \hat e_j $ is a finite-dimensional projective 
left $A^\circ_\lambda$,
Lemma~\ref{locfreelem} implies that it is $\B$-free 
as a left $A_\lambda$-module. 
Hence $A^\sharp \hat e_j = \bigoplus_{\lambda \in \Lambda} 1_\lambda A^\sharp \hat e_j$
is $\B$-free as a left module.
Similarly, we get that $\hat e_i A^\flat$ is $\B$-free as a right module.
So now Lemma~\ref{iflocallyfree} can be applied and we obtain a triangular basis
such that $A_\lambda \cong \hat e_{{\lambda}} A^\circ_\lambda \hat e_{{\lambda}}$
for $\hat e_{{\lambda}} := \sum_{b \in \B_\lambda} \hat e_b$.
By the choice of the idempotents $\{\hat e_b\:|\:b \in \B\}$, 
$\hat e_{{\lambda}} A^\circ_\lambda \hat e_{{\lambda}}$
is the basic algebra that is Morita equivalent to $A^\circ_\lambda$,
checking the remaining axiom (S7) needed in order to have a based stratified algebra.
\end{proof}

\begin{corollary}\label{myconstructioncor}
  If $A$ has a triangular decomposition in which the Cartan subalgebra
  $A^\circ$ is semisimple, then there is an idempotent refinement
  $A=\bigoplus_{i,j \in \hat I}e_iAe_j$ of $A$ with
  the structure of a
  based quasi-hereditary algebra in the sense of Definition~\ref{doacc}.
\end{corollary}

\begin{proof}
  The construction in the theorem produces an idempotent refinement of $A$ that is a based stratified algebra with stratification function $\rho:\B\rightarrow \Lambda$.
  Let $\Gamma := \B$
  with partial order $\preceq$ on $\Gamma$
  defined by $a \preceq b$ if and only if $a = b$ or
  $\rho(a) < \rho(b)$.
  Since $A_\lambda$ is basic and semisimple, we have for $a, b \in \B_\lambda$
  that $H(a,b)$ is
  empty unless $a=b$ and $H(a,a)$ may be chosen to be $\{\hat e_a\}$.
  It follows that $A$ is actually a based quasi-hereditary algebra
  with weight poset $(\Gamma,\preceq)$
  and the basis
  which we have constructed.
\end{proof}

\begin{remark}\label{cartandecomp}
  The construction used to prove Theorem~\ref{myconstruction}
    suggests yet another variation on all of these
  definitions, which is weaker than having a triangular decomposition but
  stronger than having a triangular basis. For $A$ like in
  Definition~\ref{prankster} we say that it has a {\em Cartan decomposition}  if there is the following additional data:
\begin{itemize}
\item[(CD1)]
A poset $(\Lambda,\leq)$ which is upper finite in the locally finite-dimensional case and interval finite in the essentially finite-dimensional case.
\item[(CD2)]
A function $\partial:I \rightarrow \Lambda$ with finite
fibers $I_\lambda := \partial^{-1}(\lambda)$.
\item[(CD3)] A locally unital subalgebra $A^\circ$ and $(A^\circ,
  A^\circ)$-subbimodules $A^\flat$ and $A^\sharp$ of $A$.
  \end{itemize}
  The axioms are:
  \begin{itemize}
\item[(CD4)] $A^\flat$ is a projective right $A^\circ$-module
  and $A^\sharp$ is a projective left $A^\circ$-module.
\item[(CD5)] The natural multiplication map 
$A^\flat \otimes_{A^\circ} A^\sharp \rightarrow
A$ is a linear isomorphism.
\item[(CD6)] For $i,j \in I$,
$e_i A^\circ e_j$ is zero unless
  $\partial(i) = \partial(j)$,
  $e_i A^\flat e_j$ and $e_j A^\sharp e_i$ are zero unless
  $\partial(i) \leq \partial(j)$, 
  and $e_i A^\flat e_j = e_i A^\circ e_j = e_i A^\sharp e_j$ when $\partial(i) = \partial(j)$.
\end{itemize}
The proof of Theorem~\ref{myconstruction} carry over to such algebras essentially unchanged.
However we do not know of any compelling examples,
whereas as we noted in Remarks~\ref{history}, \ref{triangularhistory1}
and \ref{triangularhistory2} there are plenty of important
examples of algebras with triangular bases and with triangular
decompositions, justifying both of those definitions.
\end{remark}

If $A$ is a finite-dimensional (resp., locally finite-dimensional, resp. essentially finite-dimensional) algebra with a triangular decomposition,
then we can apply Theorems~\ref{myconstruction} and ~\ref{twoway} to
deduce that $A\fdlmod$ (resp., $A\lfdlmod$, resp., $A\fdlmod$)
is a finite (resp., upper finite, resp., essentially finite) fully
stratified category.
We end the section by making this structure more explicit.
We first define some global standardization and costandardization functors.
\begin{itemize}
\item 
The axioms imply that $A$ is a projective right $A^\sharp$-module and
that there is a locally unital projection homomorphism $A^\sharp \twoheadrightarrow A^\circ$.
Let
\begin{equation}\label{standardizationagain}
j_!:A^\circ\fdlmod \rightarrow A\lmod
\end{equation}
be the exact functor
defined by inflating along this projection homorphism
$A^\sharp \twoheadrightarrow A^\circ$ and then applying the exact induction functor
$A\otimes_{A^\sharp}?:A^\sharp\lmod\rightarrow
A\lmod$.
The fact that it takes finite-dimensional modules to finite-dimensional or locally finite-dimensional modules (as appropriate for the case) 
follows because as functors to
$A^\flat\lmod$
we have that
$A \otimes_{A^\sharp} ?
\cong A^\flat \otimes_{A^\circ} ?$ due to (TD5).
\item The axioms imply that $A$ is a projective left $A^\flat$-module and that there is a locally unital projection homomorphism $A^\flat \twoheadrightarrow A^\circ$.
  Let
  \begin{equation}\label{costandardizationagain}
j_*:A^\circ\fdlmod\rightarrow A\lmod
\end{equation}
be the exact functor defined by first inflating along the projection 
$A^\flat \twoheadrightarrow A^\circ$ then applying the exact coinduction functor
$\bigoplus_{i \in I} \Hom_{A^\flat}\left(A e_i,
  -\right):A^\flat\lmod\rightarrow A\lmod$.
It takes finite-dimensional modules to finite-dimensional or locally finite-dimensional
modules (as appropriate for the case) follows because as a functor to $A^\sharp\lmod$
it is isomorphic to
$\bigoplus_{i \in I}\Hom_{A^\circ}\left(A^\sharp e_i,?\right)$.
\end{itemize}
The following theorem can be proved by mimicking
standard arguments from Lie theory; see \cite{CZ}
noting that $(A^\flat, A^\circ)$ and $(A^\circ, A^\sharp)$
are {\em Borelic pairs} in the sense defined there.
We will deduce it instead from the work already done in
Theorems~\ref{myconstruction} and \ref{twoway}.

\begin{theorem}[Fully stratified categories from algebras with a triangular decomposition]\label{detailsdetails}
Suppose that $A$ 
has a triangular decomposition of one of the three types as above.
Let $\{L^\circ(b)\:|\:b \in \B\}$
be a full set of pairwise inequivalent irreducible left $A^\circ$-modules.
Let $\rho:\B \rightarrow \Lambda$ be the function sending $b \in \B$ to the unique $\lambda \in \Lambda$ 
such that $L^\circ(b)$ is an irreducible $A^\circ_\lambda$-module.
Let $\bar\Delta(b) := j_! L^\circ(b)$ and $\bar\nabla(b) := j_*
L^\circ(b)$; cf. (\ref{standardizationagain})--(\ref{costandardizationagain}).
Then
$$
\big\{L(b) := \hd \bar\Delta(b) \cong \soc \bar\nabla(b)\:\big|\:b \in \B\big\}
$$
is a complete set of pairwise inequivalent irreducible left
$A$-modules.
Moreover, the category $\R := A\fdlmod$ (resp., $A\lfdlmod$, resp., $A\fdlmod$) is a
finite (resp., upper
finite, resp., essentially finite) fully stratified category with
stratification $(\B,L,\rho,\Lambda,\leq)$.
Its strata may be identified with the categories 
$A^\circ_\lambda\fdlmod\:(\lambda \in \Lambda)$ with standardization and costandardization
functors defined by the restrictions of $j_!$ and $j_*$, respectively.
\end{theorem}

\begin{proof}
As explained by Theorem~\ref{myconstruction}, 
we can pass to an idempotent refinement if necessary to assume without loss of generality that the set $I$ indexing the distinguished idempotents is disjoint from $\Lambda$ and 
contains $\B$ as a subset in such a way that
$L^\circ(b) \cong \hd (A^\circ e_b)$ for each $b \in \B$.
The function $\rho:\B\rightarrow \Lambda$ is then the restriction of $\partial:I \rightarrow \Lambda$. Now Theorem~\ref{myconstruction} gives bases making $A$ into a based 
stratified algebra. We
we deduce that $\R$ 
is a finite (resp., upper finite, resp., essentially finite) fully
  stratified category with stratification $(\B,L,\rho,\Lambda,\leq)$ by applying
  Theorem~\ref{twoway}. However for this the strata and the labelling function $L$ 
  are produced in a different way to the formulation here, so we need to
  argue a little further to see that the standardization and costandardization functors here
  and the ones from earlier may be identified.
Using the isomorphism (\ref{newyear}),
the quotient functor $j^\lambda:A_{\leq
    \lambda}\lmod\rightarrow A_\lambda\lmod$ 
    in the setup of  (\ref{pinkeye}) 
    may be
  identified with the functor $j:A_{\leq \lambda}\lmod \rightarrow
  e_{{\lambda}} A_\lambda^\circ e_{{\lambda}}\lmod$
  obtained by restriction to $A^\circ$
  then multiplication by the idempotent $e_{{\lambda}}$.
  Since $A^\circ_\lambda$ and $e_{{\lambda}} A^\circ_\lambda e_{{\lambda}}$ are Morita equivalent, we can instead 
  use the algebra $A^\circ_\lambda$ to realize the stratum, and then this quotient functor gets replaced by
  the functor obtained by restriction to $A^\circ$ then multiplication
  by $1_{\lambda}$.
  It remains to observe that the restrictions of $j_!$ and $j_*$ to
  $A_\lambda^\circ\lmod$ are left and right adjoint to this functor, respectively.
\end{proof}

\begin{corollary}\label{usefultoo}
Suppose that $A$ 
has a triangular decomposition of one of the three types
and that its Cartan subalgebra $A^\circ$ is semisimple.
Let $\{L^\circ(\gamma)\:|\:\gamma \in \Gamma\}$
be a full set of pairwise inequivalent irreducible left $A^\circ$-modules.
Let $\rho:\Gamma \rightarrow \Lambda$ be the function sending $\gamma$ to the unique
$\lambda$ such that $L^\circ(\gamma)$ is an irreducible $A^\circ_\lambda$-module.
Then $\R := A\fdlmod$ (resp., $A\lfdlmod$, resp. $A\fdlmod$)
is a finite (resp., upper finite, resp., essentially finite) highest weight category with weight poset $(\Gamma, \preceq)$ for $\preceq$ defined by
$\beta \preceq \gamma$ if either $\beta = \gamma$ or $\rho(\beta) < \rho(\gamma)$.
Its standard and costandard modules are $\Delta(\gamma)
 := j_! L^\circ(\gamma)$ and $\nabla(\gamma) := j_* L^\circ(\gamma)$
for $\gamma \in \Gamma$.
\end{corollary}

\begin{proof}
This follows from the theorem and Corollary~\ref{myconstructioncor}.
\end{proof}

\begin{remark}
We end by mentioning one last variation on the definitions in this subsection.
We say that a triangular decomposition of $A$ as in Definition~\ref{prankster} is 
a {\em symmetric triangular decomposition}
if in addition there is given a locally unital 
algebra anti-involution $\sigma:A \rightarrow A$
which leaves $A^\circ$ invariant and interchanges $A^\sharp$ and $A^\flat$,
such that for each $\lambda \in \Lambda$
the subalgebras $e_\lambda A^\lambda e_\lambda$ are $\sigma_\lambda$-symmetric in the sense of Definition~\ref{sigmasymmetric}, 
where $\sigma_\lambda$ denotes the restriction of $\sigma$.
Then there is an enhanced version of Theorem~\ref{myconstruction} making $A$ into a symmetrically based stratified algebra,
and an enhanced version of Theorem~\ref{detailsdetails} making $\R$ into a fully stratified category with a Chevalley duality $?^\sigmadual$. We omit the details.
\end{remark}

\section{Examples}

In this section, we explain several examples. For the ones in
$\S\S$\ref{eg4}--\ref{eg5} we give very few details but have tried to
indicate the relevant ingredients from the existing literature. 

\subsection{A finite-dimensional example via quiver and relations}\label{eg1}
Let $A$ and $B$ 
be the basic finite-dimensional 
algebras defined by the following quivers:
\begin{align*}
A \:\:(1 > 2)&: \xymatrix{
\ar@(dl,ul)[]^{s}{1}\ar[r]^{y}&
\ar@(dr,ur)[]_{t}{2}}
&&\text{with relations } s^2=0, t^2=0, ty=0,\\
 B \:\:(1<2)&: 
\xymatrix{
\ar@(dl,ul)[]^{z}{1}\ar@/^/[r]^{u}&
\ar@/^/[l]^{v}{2}}
&&\text{with relations } z^2=0, uv=0, vuzv=0.
\end{align*}
The algebra $A$ has basis $e_1,s;e_2,t;y,ys$ and $B$ has basis
$e_1,z,vu,vuz,zvu,zvuz;$ $e_2,uzv;$ $v,zv;$ $u,uz,uzvu,uzvuz$.
The irreducible $A$- and $B$-modules 
are indexed by the set
$\{1,2\}$. We are going to consider $A\fdlmod$ and $B\fdlmod$
with the stratifications defined by the orders $1 > 2$ and $1 <
2$, respectively.

We first look at $A\fdlmod$.
As usual, we denote its irreducibles by $L(i)$, indecomposable projectives by
$P(i)$, standards by $\Delta(i)$, etc..
The indecomposable projectives and
injectives look as follows
(where we abbreviate the irreducible module $L(i)$ just by $i$):
\begin{align*}
P(1) &=
\begin{tikzpicture}[anchorbase]
\node at (0,0) {{\xymatrix@C=0.9em@R=1.1em{
&1\ar@{-}[dl]|{s}\ar@{-}[dr]|{y}\\
1\ar@{-}[d]|{y}&&2\\
2&&
}}};
\end{tikzpicture},
&
P(2)&=
\begin{tikzpicture}[anchorbase]
\node at (0,0) {{\xymatrix@C=0.9em@R=1.1em{
2\ar@{-}[d]|{t}\\
2
}}};\end{tikzpicture},
&
I(1)&=
\begin{tikzpicture}[anchorbase]
\node at (0,0) {{\xymatrix@C=0.9em@R=1.1em{
1\ar@{-}[d]|{s}\\
1
}}};\end{tikzpicture},
&I(2)&=
\begin{tikzpicture}[anchorbase]
\node at (0,0) {{\xymatrix@C=0.9em@R=1.1em{
&&1\\
2&&\ar@{-}[u]|{s}1\\
&\ar@{-}[ul]|{t}\ar@{-}[ur]|{y}2
}}};\end{tikzpicture}.
\end{align*}
It follows easily that $A\fdlmod$ is a fibered highest weight
category in the sense of Definition~\ref{vector} with the structure of the standards and
costandards as follows:
\begin{equation*}
\begin{aligned}
\Delta(1) &= P(1),&
\bar\Delta(1)&=
\begin{tikzpicture}[anchorbase]
\node at (0,0) {{\xymatrix@C=0.9em@R=1.1em{
1\ar@{-}[d]|{y}\\
2
}}};\end{tikzpicture},&
\Delta(2)&=P(2),&
\bar\Delta(2)&=L(2),\\
\nabla(1)&=I(1),&
\bar\nabla(1)&=L(1),&
\nabla(2)&=
\begin{tikzpicture}[anchorbase]
\node at (0,0) {{\xymatrix@C=0.9em@R=1.1em{
2\ar@{-}[d]|{t}\\
2
}}};\end{tikzpicture},&
\bar\nabla(2)&=L(2).
\end{aligned}
\end{equation*}
This can also be seen from Theorem~\ref{twoway} on noting that $A$ is
a based properly stratified 
algebra in the sense of Definition~\ref{strawberries2} with
$Y(2,1) = \{y\}, X(1,2) = \varnothing$
and $H(1) = \{e_1,s\}, H(2) = \{e_2,t\}$. The basic local algebras realizing the strata are $\k[s]/(s^2)$ and $\k[t]/(t^2)$.
Next we look at the tilting modules in $A\fdlmod$. If one takes the sign function
$\eps = (\eps_1, \eps_2)$ to be either $(+,+)$ or $(-,+)$ then one finds that
the indecomposable $\eps$-tilting modules are:
\begin{align*}
T_+(1) = P(1)
&=
\begin{tikzpicture}[anchorbase]
\node at (0,0) 
{{\xymatrix@C=0.9em@R=1.1em{
\bar\Delta(1)\ar@{-}[d]\\
\bar\Delta(1)
}}};\end{tikzpicture}
=\!\!
\begin{tikzpicture}[anchorbase]
\node at (0,0) 
{{\xymatrix@C=0.9em@R=1.1em{
&\ar@{-}[dl]\nabla(1) 
\ar@{-}[dr]\\
\bar\nabla(2)\!\!\!\!\!\!&&\!\!\!\!\!\!\bar\nabla(2)
}}};\end{tikzpicture}
,&
T_+(2) = P(2)&=\Delta(2)=\!\!\!
\begin{tikzpicture}[anchorbase]
\node at (0,0) 
{{\xymatrix@C=0.9em@R=1.1em{
\bar\nabla(2) \ar@{-}[d]\\
\bar\nabla(2)\\
}}};\end{tikzpicture}.
\end{align*}
These cases are not very interesting since
the Ringel dual category is just $A\fdlmod$ again.
Assume henceforth that $\eps = (-,-)$ or $(+,-)$. 
Then the indecomposable $\eps$-tilting modules are:
\begin{align*}
T_-(1) &=
\begin{tikzpicture}[anchorbase]
\node at (0,0) {{\xymatrix@C=0.9em@R=1.1em{
1\ar@{-}[d]|{s}\ar@{-}[rd]|{y}&&\ar@{-}[ld]|{t}2\\
1\ar@{-}[rd]|{y}&2 &\ar@{-}[ld]|{t}2\\
&2
}}};\end{tikzpicture}
=
\!\!\begin{tikzpicture}[anchorbase]
\node at (0,0) 
{{\xymatrix@C=0.9em@R=1.1em{
\ar@{-}[dr]\bar\Delta(2)\!\!\!\!\!\!&&\!\!\!\!\!\!\bar\Delta(2) \ar@{-}[dl]\\
&\Delta(1)
}}};\end{tikzpicture}
\!\!=\!\!
\begin{tikzpicture}[anchorbase]
\node at (0,0) 
{{\xymatrix@C=0.9em@R=1.1em{
&\ar@{-}[dl]\nabla(1) 
\ar@{-}[dr]\\
\nabla(2)\!\!\!\!\!\!&&\!\!\!\!\!\!\nabla(2)
}}};\end{tikzpicture}\!\!\!,
&
T_-(2)&=P(2).
\end{align*}
To see this, one just has to check that these modules are
indecomposable with the appropriate $\Delta_\eps$- and
$\nabla_\eps$-flags.
This analysis reveals that $A\fdlmod$ is {\em not} tilting-rigid.

The minimal projective resolution of $T_-(1)$ takes the form
$$
\cdots
\longrightarrow
 P(2)\oplus P(2)\longrightarrow
P(2)\oplus P(2)\longrightarrow
P(1)\oplus P(2) \oplus P(2) \longrightarrow 
T_-(1) \longrightarrow 0.
$$
In particular, it is not of finite projective dimension, as follows 
also from Lemma~\ref{proofofconj} since $T_-(1) \not\cong T_+(1)$.
Observe also that there is a non-split short exact sequence
$0 \rightarrow X \rightarrow T_-(1) \rightarrow X \rightarrow 0$
where
$$
X = 
\begin{tikzpicture}[anchorbase]
\node at (0,0) {\xymatrix@C=0.6em@R=1em{
1\ar@{-}[rd]|{y}&&\ar@{-}[ld]|{t}2\\
&2}};\end{tikzpicture}.
$$

Now let $T := T_-(1)\oplus T_-(2)$.
We claim that $\End_A(T)^\op$ is the algebra $B$ defined above.
To prove this, one takes $z:T_-(1)\rightarrow T_-(1)$ to be an endomorphism whose image
and kernel is the submodule $X$ of $T_-(1)$,
$u:T_-(2) \rightarrow T_-(1)$ to be a homomorphism which includes $T_-(2)$ as a submodule of
$X\subseteq T_-(1)$,
and $v:T_-(1) \rightarrow T_-(2)$
to be a homomorphism with kernel containing $X$ and image $L(2)\subseteq T_-(2)$.
Hence, $B\fdlmod$ is the Ringel dual of $A\fdlmod$ relative to $T$.
Note also that the algebra $B$ is based $(+,+)$- and $(-,+)$-quasi-hereditary but it is {\em not}
based $(+,-)$- or $(-,-)$-quasi-hereditary (cf. Definition~\ref{eclairs2}).

One can also analyze $B\fdlmod$ directly. Its projective modules have the
following structure: 
\begin{align*}
P'(1) &=
\begin{tikzpicture}[anchorbase]
\node at (0,0) {{\xymatrix@C=0.9em@R=1.1em{
&1\ar@{-}[dl]|{u}\ar@{-}[dr]|{z}&\\
2\ar@{-}[d]|{v}&&1\ar@{-}[d]|{u}\\
1\ar@{-}[d]|{z}&&2\ar@{-}[d]|{v}\\
1\ar@{-}[d]|{u}&&1\ar@{-}[d]|{z}\\
2&&1\ar@{-}[d]|{u}\\
&&2\\
}}};
\end{tikzpicture},
&
P'(2) &= 
\begin{tikzpicture}[anchorbase]
\node at (0,0) 
{{\xymatrix@C=0.9em@R=1.1em{
2 \ar@{-}[d]|{y}\\
1 \ar@{-}[d]|{z}\\
1 \ar@{-}[d]|{u}\\
2\\
}}};\end{tikzpicture},
\end{align*}
Continuing with $\eps = (-,-)$ or $\eps = (+,-)$, it is easy to check
directly from this that $B\fdlmod$ is $(-\eps)$-highest weight,
as we knew already due to Theorem~\ref{Creek}. However, it is {not}
$\eps$-highest weight for either of these choices of $\eps$, so it is {\em not}
fibered highest weight.

We leave it to the reader to compute explicitly the indecomposable $(-\eps)$-tilting
modules $T'_+(1)$ and $T'_+(2)$
in $B\fdlmod$.
Their structure reflects the structure of the injectives
$I(1)$ and $I(2)$ in $A\fdlmod$.
Let 
$T' := T'_+(1)\oplus T'_+(2) \cong T^*$.
By the double centralizer property from Corollary~\ref{mustard},
we have that $A = \End_B(T')^\op$, as may also be checked
directly.
By Theorem~\ref{happels},
the functor
$\der{R}\Hom_B(T',?):D^b(B\fdlmod)\rightarrow
D^b(A\fdlmod)$
is an equivalence.
Note though that $\der{R}\Hom_A(T,?):D^b(A\fdlmod)\rightarrow
D^b(B\fdlmod)$
is {\em not} one; this follows using \cite[Th.~4.1]{Keller}
since $T_-(1)$ does not have
finite projective dimension.

\subsection{An explicit semi-infinite example}\label{eg1a}
In this subsection, we give a baby example involving a lower finite highest
weight category. Let $C$ be the coalgebra with basis 
$$
\big\{c_{i,j}^{(\ell)}\:\big|i,j,\ell \in \Z, 0 \leq i,j \leq \ell\big\},
$$
counit defined by $\epsilon(c_{i,j}^{(\ell)}) := \delta_{i,\ell} \delta_{j,\ell}$, and comultiplication
$\delta:C \rightarrow C \otimes C$ defined by
\begin{align*}
c^{(i)}_{i,j} &\mapsto 
\!\sum_{\substack{k=0\\i \not\equiv
    j (2)}}^{j}\!
c_{i,k}^{(i)} \otimes c_{k,j}^{(j)}+\!\sum_{\substack{k=j\\k \equiv i (2)}}^{i}
\!c_{i,k}^{(i)} \otimes c_{k,j}^{(k)},\qquad
c^{(j)}_{i,j} \mapsto 
\!\sum_{\substack{k=0\\i \not\equiv
    j (2)}}^{i}
\!c_{i,k}^{(i)} \otimes c_{k,j}^{(j)}+\!\sum_{\substack{k=i\\k \equiv
    j (2)}}^{j}
\!c_{i,k}^{(k)} \otimes c_{k,j}^{(j)},\\
c_{i,j}^{(\ell)} &\mapsto
c_{i,\ell}^{(\ell)} \otimes c_{\ell,j}^{(\ell)} 
+\!\sum_{\substack{k=0\\i \not\equiv \ell (2)}}^{i} \!c_{i,k}^{(i)} \otimes
  c_{k,j}^{(\ell)}
+ \!
\sum_{\substack{k=i\\k \equiv \ell (2)}}^{\ell-1} \!c_{i,k}^{(k)} \otimes
  c_{k,j}^{(\ell)}
+\!\sum_{\substack{k=0\\j \not\equiv \ell (2)}}^{j} \!c_{i,k}^{(\ell)} \otimes
  c_{k,j}^{(j)}
+ \!
\sum_{\substack{k=j\\k \equiv \ell (2)}}^{\ell-1} \!c_{i,k}^{(\ell)} \otimes
  c_{k,j}^{(k)}
\end{align*}
for $i,j\geq0$ and $\ell > \max(i,j)$.
We will show that $\R := \fdrcomod C$ is a lower finite highest weight
category with weight poset $\Lambda := \N$ ordered in the natural way.
Then we will determine the costandard, standard and indecomposable
injective and tilting objects explicitly, 
and describe the Ringel dual category $\R'$.
To do this, we mimic some arguments for reductive groups which
we learnt from \cite{J}.

We will need comodule induction functors, which we review briefly.
For any coalgebra $C$ with comultiplication $\delta$, a right $C$-comodule $V$ with structure map
$\eta_R:V \rightarrow V \otimes C$, and a left $C$-comodule $W$ with
structure map $\eta_L:W \rightarrow C \otimes W$, the {\em cotensor product}
$V \cotimes_C W$ is the subspace of the vector space $V \otimes W$
that is the equalizer
of the diagram
$$
V\otimes W \:\:\substack{\eta_R\otimes\id\\\longrightarrow\\\longrightarrow\\\id\otimes\eta_L} \:\:V \otimes C \otimes W.
$$
In particular, $\eta_R:V \rightarrow V \otimes C$ is an isomorphism
from $V$ to the subspace $V \cotimes_C C$, and similarly $\eta_L:W
\stackrel{\sim}{\rightarrow} C \cotimes_C W$.
Now suppose that $\pi:C\rightarrow C'$ is a coalgebra
homomorphism and $V$ is a right $C'$-comodule. 
Viewing $C$ as a left $C'$-comodule
with structure map
$\delta_L := (\pi \otimes \id)
\circ \delta:C \rightarrow C' \otimes C$, we define the induced
comodule to be
$$
\ind_{C'}^C V
:= V \cotimes_{C'} C.
$$
This is a subcomodule of the right $C$-comodule $V \otimes C$
(with structure map $\id \otimes \delta$).
In fact, $\ind_{C'}^C:\rcomod{C'}\rightarrow \rcomod{C}$ defines a functor which is right adjoint to the exact restriction
functor $\res^C_{C'}$, so it is left exact and sends injectives to injectives.

Now let $C$ be the coalgebra defined above, and consider the natural quotient
maps $\pi^\flat:C \twoheadrightarrow C^\flat$ and $\pi^\sharp:C
\twoheadrightarrow C^\sharp$,
where $C^\flat$ and $C^\sharp$ are the quotients of $C$ by the coideals
spanned by $\big\{c_{i,j}^{(\ell)}\:\big|\:\ell
> j\big\}$ or $\big\{c_{i,j}^{(\ell)}\:\big|\:\ell
> i\big\}$, respectively.
These coalgebras have 
bases denoted
$
\big\{c_{i,j} := \pi^\flat(c_{i,j}^{(j)})\:\big|\:0
\leq i \leq j\big\}$ and
$\big\{c_{i,j}:=\pi^\sharp(c_{i,j}^{(i)})\:\big|\:i \geq j \geq
  0\big\}$, 
and comultiplications
$\delta^\flat$ and $\delta^\sharp$
satisfying
\begin{align}\label{safe}
\delta^\flat(c_{i,j}) &= c_{i,i} \otimes c_{i,j} +
\sum_{\substack{k=i+1\\ k \equiv j (2)}}^{j} c_{i,k}\otimes c_{k,j},
&\delta^\sharp(c_{i,j}) =& c_{i,j} \otimes c_{j,j}+
\sum_{\substack{k=j+1\\ k \equiv i (2)}}^{i} c_{i,k}\otimes c_{k,j},
\end{align}
respectively.
Also let 
$C^\circ\cong\bigoplus_{i \geq 0} \k$
be the semisimple coalgebra with basis $\{c_i\:|\:i \geq 0\}$
and comultiplication $\delta^\circ:c_i \mapsto c_i \otimes c_i$.
Note $C^\circ$ is a quotient of both $C^\flat$ and $C^\sharp$
via the obvious maps sending $c_{i,j} \mapsto \delta_{i,j} c_i$; hence, it is
also a quotient of $C$.
It may also be identified with a subcoalgebra of both $C^\flat$ and $C^\sharp$
via the maps sending $c_i \mapsto c_{i,i}$.

Let $L^\circ(i)$ be the one-dimensional irreducible
right $C^\circ$-comodule
spanned by $c_{i,i}$. Since $C^\circ$ is semisimple with these as its
irreducible comodules,
any irreducible right
$C^\circ$-comodule $V$ decomposes as $V = \bigoplus_{i \in I} V_i$
with the ``weight spaces'' $V_i$ being a direct sum of copies of $L^\circ(i)$. Similarly,
any left $C^\circ$-comodule $V$ decomposes as $V  = \bigoplus_{i \in
  I} {_i}V$. This applies in particular to left and right $C^\flat, C^\sharp$ or
$C$-comodules, since these may be viewed as $C^\circ$-comoodules by
restriction.

Since $C^\circ$ is a subcoalgebra of $C^\flat$, the irreducible
comodule $L^\circ(i)$ may also be viewed as an irreducible right
$C^\flat$-comodule. We denote this instead by $L^\flat(i)$; it is
the subcomodule of $C^\flat$ spanned by the vector $c_{i,i}$.
For $i \geq 0$, 
let $I(i) := {_i} C \cong \ind_{C^\circ}^C L^\circ(i)$, let
$\nabla(i)$ be the subcomodule of $I(i)$ spanned by the vectors $\{c_{i,j}^{(i)}\:|\:0 \leq j
  \leq i\}$, and let
 $L(i)$ be the one-dimensional irreducible 
subcomodule of $\nabla(i)$ spanned by
  the vector $c_{i,i}^{(i)}$.
Now we proceed in several steps.

\vspace{1.5mm}
\noindent
{Claim 1:} {\em 
Viewed as a functor to vector spaces, the induction functor
$\ind_{C^\flat}^C$ is isomorphic to the functor $V \mapsto V
\cotimes_{C^\circ} C^\sharp \cong
\bigoplus_{i \geq 0} V_i \otimes {_i}C^\sharp$.
Hence, this functor is exact.}
To prove this, let $\delta_{LR}:=(\pi^\flat \bar\otimes \pi^\sharp) \circ
\delta:C \rightarrow C^\flat 
\cotimes_{C^\circ} C^\sharp$.
As $\delta_{LR}(c_{i,j}^{(\ell)}) = c_{i,\ell} \otimes
c_{\ell,j}$ and these vectors for all $\ell \geq \max(i,j)$ give a
basis for $C^\flat 
\cotimes_{C^\circ} C^\sharp$, this map is a linear isomorphism.
Moreover, the following diagram commutes:
$$
\begin{CD}
C&@>\delta_L>>&C^\flat \otimes C\\
@V\delta_{LR}VV&&@VV\id \otimes \delta_{LR} V\\
C^\flat \cotimes_{C^\circ} C^\sharp&@>>\delta^\flat\otimes \id>&C^\flat
\otimes C^\flat \cotimes_{C^\circ} C^\sharp.
\end{CD}
$$
The vertical maps are isomorphisms. Using the definition of
$\ind_{C^\flat}^C$, it follows for any right
$C^\flat$-comodule $V$ with structure map $\eta$ that the induced module $\ind_{C^\flat}^C V$
is isomorphic as a vector space (indeed, as a right $C^\sharp$-comodule) to the equalizer of the diagram
$$
V\otimes C^\flat \cotimes_{C^\circ} C^\sharp
\:\:\substack{\eta\otimes\id\otimes\id\\\longrightarrow\\\longrightarrow\\\id\otimes\delta^\flat
\otimes \id} \:\:V \otimes C^\flat \otimes C^\flat \cotimes_{C^\circ} C^\sharp.
$$
Since $\ind_{C^\flat}^{C^\flat} V \cong V$, this is naturally isomorphic to $V
\cotimes_{C^\circ} C^\sharp$. As $C^\sharp \cong
\bigoplus_{i \geq 0} {_i}C^\sharp$, 
we get finally that
$V
\cotimes_{C^\circ} C^\sharp \cong \bigoplus_{i \geq 0} V_i \otimes {_i}C^\sharp$.

\vspace{1.5mm}
\noindent
{Claim 2:} {\em For $i \geq 0$, the right $C^\flat$-comodule 
${_i}C^\flat \cong \ind_{C^\circ}^{C^\flat} L^\circ(i)$ has an exhaustive ascending filtration
$0 < V_0 < V_1 <  \cdots$ such that 
$V_0 \cong L^\flat(i)$ and $V_r / V_{r-1} \cong L^\flat(i+2r-1)\oplus
L^\flat (i+2r)$ for $r \geq 1$.
Also, the modules
$\{L^\flat(i)\:|\:i \geq 0\}$ give a full set of pairwise
inequivalent irreducible right $C^\flat$-comodules.}
The first statement follows from (\ref{safe}), defining $V_0$ to be the subspace spanned by $c_{i,i}$, and $V_r$ is
spanned by
$c_{i,i+2r-1}, c_{i,i+2r}$.  To prove the second statement, take any
irreducible $C^\flat$-comodule $L$. Take a non-zero homomorphism 
$\res^{C^\flat}_{C^\circ} L \rightarrow L^\circ(i)$
for some $i$. Then use
adjointness of $\res_{C^\circ}^{C^\flat}$ and
$\ind^{C^\flat}_{C^\circ}$ to obtain an embedding 
$L \hookrightarrow {_i}C^\flat$. Hence, $L \cong
L^\flat(i)$ as a $C^\flat$-comodule.

\vspace{1.5mm}
\noindent
{Claim 3:}
{\em 
We have that $\nabla(i) \cong \ind_{C^\flat}^C L^\flat(i)$ and it
is uniserial with composition factors 
$L(i), L(i-2), L(i-4),\dots, L(a), L(b),\cdots L(i-3), L(i-1)$, where $(a,b)\in\{(0,1),(1,0)\}$ depending on parity of $i$, in order from
bottom to top:}
\begin{equation}
\nabla(i)=
\begin{tikzpicture}[anchorbase]
\node at (0,0){{\xymatrix@C=0.9em@R=1.1em{
i-1\ar@{-}[d]\\
i-3
\ar@{.}[d] \\
i-2\ar@{-}[d] \\
i
}}};\end{tikzpicture}
\end{equation}
The restriction of $\delta_L:C \rightarrow C^\flat \otimes C$ to
$\nabla(i)$ gives an embedding of $\nabla(i)$ into $\ind_{C^\flat}^C
L^\flat(i)$. This embedding is an isomorphism since we know
$\ind_{C^\flat}^C
L^\flat(i)$ has the same dimension $(i+1)$ as $\nabla(i)$ thanks to
Claim 1. 
The determinaton of the
subcomodule structure is straightforward using the definition of $\delta(c_{i,j}^{(i)})$
for $0 \leq j \leq i$.

\vspace{1.5mm}
\noindent
{Claim 4:}
{\em The injective $C$-comodule $I(i)$ has an exhaustive filtration
$0 < I_0 < I_1 < \cdots$ such that $I_0 \cong \nabla(i)$
and $I_r / I_{r-1} \cong \nabla(i+2r-1)\oplus \nabla(i+2r)$
for $r \geq 1$:}

\vspace{-5mm}

\begin{equation}
I(i)=
\begin{tikzpicture}[anchorbase]
\node at (0,0) 
{{\xymatrix@C=0.9em@R=1.1em{
\ar@{.}[d]&
&\ar@{.}[d]\\
\nabla(i+3) \ar@{-}[d]&&\nabla(i+4) \ar@{-}[d] \\
\nabla(i+1) \ar@{-}[dr]&&\nabla(i+2)\ar@{-}[dl] \\
&\nabla(i)
}}};\end{tikzpicture}
\end{equation}
This follows from Claims 1, 2 and 3.

\vspace{1.5mm}
\noindent
{Claim 5:}
{\em The $C$-comodules $\{L(i)\:|\:i \geq 0\}$ give a
  full set of pairwise inequivalent irreducibles. Moreover,  $I(i)$ is the injective hull of
  $L(i)$.
}
By Claim 3, the last part of Claim 2, and an adjunction argument, any
irreducible $C$-comodule embeds into $\nabla(i)$ for some $i$, hence,
it is isomorphic to $L(i)$. The comodule $I(i)$ is injective, 
and it has irreducible socle $L(i)$ by another adjunction
argument. Hence, it is the injective hull of $L(i)$. 

\vspace{1.5mm}
\noindent
{Claim 6:}
{\em The category $\R:=\fdrcomod C$ is a lower finite highest weight category with
  costandard objects $\nabla(i)\:(i \geq 0)$. It also possesses a
 Chevalley duality.}
  We use the criterion from Corollary~\ref{globalcharcor}.
  From Claim 4, it follows that
  the largest submodule of $I(i)$ that belongs to $\R_{\leq i}$ is $\nabla(i)$, which is finite-dimensional. This shows that $\R_{\leq i}$ has enough injectives with the injective hull of $L(i)$ being $\nabla(i)$. We also know already that $[\nabla(i):L(i)] =1$, and 
  the property $(\widehat{I\nabla}^\asc)$ follows from
Claim 4. Hence, $\R$ is a lower finite highest weight category. 
Finally, the Chevalley duality is
defined using the evident coalgebra antiautomorphism of $C$ which 
maps $c_{i,j}^{(\ell)} \mapsto c_{j,i}^{(\ell)}$.

\vspace{1.5mm}
\noindent
{Claim 7:}
{\em The indecomposable tilting comodule $T(i)$ is equal to $L(i)=\Delta(i)=\nabla(i)$ if
  $i=0$, and there are non-split short exact sequences
\begin{align*}
0 &\rightarrow \Delta(i) \rightarrow T(i) \rightarrow
\Delta(i-1)\rightarrow 0,&
0 &\rightarrow \nabla(i-1)\rightarrow
T(i)
\rightarrow \nabla(i) \rightarrow 0
\end{align*}
for $i > 0$.}\\
This is immediate in the case $i=0$. Now for $i > 0$, let $T(i)$ be the
non-split extension of $\nabla(i-1)$ by $\nabla(i)$ that is the
subcomodule of $I(i-1)$ spanned by the vectors
$\{c_{i-1,j}^{(i-1)}, c_{i-1,k}^{(i)} \:|\:0 \leq j \leq i-1,
0 \leq k \leq i\}$. Then one checks that this submodule is self-dual. Since it has a
$\nabla$-flag it therefore also has a $\Delta$-flag, so it must be the
desired tilting object by Theorem~\ref{gin}.

\vspace{1.5mm}
\noindent
{Claim 8:}
{\em 
The Ringel dual category $\R'$ is the category $A\lfdlmod$ of
locally finite-dimensional left modules over the locally unital algebra $A$ defined by the following
quiver:}
\begin{align*}
A \:&:
\xymatrix{
0
\ar@/^/[r]^{y_0}
&
\ar@/^/[l]^{x_0}{1}
\ar@/^/[r]^{y_1}
&
\ar@/^/[l]^{x_1}
2\;\cdots}
&\text{\em{with relations }}y_{i+1}y_i = x_i x_{i+1} = x_i y_i = 0.
\end{align*}
We need to find an isomorphism of algebras $$
A \stackrel{\sim}{\rightarrow} 
\Big(\bigoplus_{i,j \geq 0} \Hom_C(T(i), T(j))\Big)^\op.
$$
For this, we consider $T(i)\:(i=0,1,2,3,\dots)$ with the 
$\nabla$-flags:
\begin{equation}
\begin{tikzpicture}[anchorbase]
\node at (-4,0) {
{\xymatrix@C=0.8em@R=1em{
\save[].[].[].[]*[F.]\frm{}\ar@{.}+UL;+UL\restore
0\\
}}};
\node at (-3,-.25){$\stackrel{x_0}{\longrightarrow}$};
\node at (-3, .25){$\stackrel{y_0}{\longleftarrow}$};
\node at (-2,0) {
{\xymatrix@C=0.8em@R=1em{
\save[].[d].[d].[]*[F.]\frm{}\ar@{.}+UL;+UL\restore
0\ar@{-}[d]&\\
1\ar@{-}[d]&\\
\save[].[].[].[]*[F.]\frm{}\ar@{.}+UL;+UL\restore
0
}}};
\node at (-1,-.25){$\stackrel{x_1}{\longrightarrow}$};
\node at (-1,.25){$\stackrel{y_1}{\longleftarrow}$};
\node at (0,0) {
{\xymatrix@C=0.8em@R=1em{
\save[].[dd].[dd].[]*[F.]\frm{}\ar@{.}+UL;+UL\restore
1\ar@{-}[d]\\
0\ar@{-}[d]\\
2\ar@{-}[d]\\
\save[].[d].[d].[]*[F.]\frm{}\ar@{.}+UL;+UL\restore
0\ar@{-}[d]\\
1
}}};
\node at (1,-.25){$\stackrel{x_2}{\longrightarrow}$};
\node at (1,.25){$\stackrel{y_2}{\longleftarrow}$};
\node at (2,0){
{\xymatrix@C=0.8em@R=1em{
\save[].[ddd].[ddd].[]*[F.]\frm{}\ar@{.}+UL;+UL\restore
2\ar@{-}[d]\\
0\ar@{-}[d]\\
1\ar@{-}[d]\\
3\ar@{-}[d]\\
\save[].[dd].[dd].[]*[F.]\frm{}\ar@{.}+UL;+UL\restore
1\ar@{-}[d]\\
0\ar@{-}[d]\\
2
}}};
\node at (3,.25){$\stackrel{y_3}{\longleftarrow}$};
\node at (3,-.25){$\stackrel{x_3}{\longrightarrow}$};
\node at (4,0){$\cdots$};
\end{tikzpicture}
\end{equation}
We will describe the images, also called $e_i$, $x_i$, $y_i$, of the generators of $A$. 
We send $e_i$ to the identity endomorphism of $T(i)$,
$x_i$ to the morphism $T(i)\rightarrow T(i+1)$ sending the quotient
$\nabla(i)$ of $T(i)$ to the subcomodule $\nabla(i)$ of $T(i+1)$
and $y_i$ to the morphism $T(i+1) \rightarrow T(i)$ sending the
quotient $\Delta(i)$ of $T(i+1)$ to the submodule $\Delta(i)$ of
$T(i)$. The relations are easy to check (remembering the $\op$, e.g., one
must verify that $y_2\circ x_2 = 0 \neq x_2 \circ y_2$). Since the algebra $A$ is very
easy to understand, one also sees that this homomorphism is injective, then it is an isomorphism by
dimension considerations.

\begin{remark}
The above analysis of $\fdrcomod C$
relies ultimately on the observation that the coalgebra $C$ has a
triangular decomposition in a precise sense which is the analog for
coalgebras of Definition~\ref{prankster}. There are also
coalgebra analogs of the other definitions from the previous section, which we intend to develop in more detail in a sequel to this article. 
The coalgebra analog of
Definition~\ref{doacc} is the notion of a {\em based
quasi-hereditary coalgebra}. 
The dual of such a coalgebra whose weight poset is finitely generated and good in the sense of \cite[Def.~3.9]{MZ} is an ascending quasi-hereditary pseudo-compact algebra as defined in \cite[Def.~3.11]{MZ}.
\end{remark}

\vspace{2mm}
One can argue in the opposite direction too, starting from the algebra
$A$ just defined and
computing 
its Ringel dual to recover the coalgebra $C$ (in fact,
this is how we discovered the coalgebra $C$ in the first place). Note
for this that $A$ is an upper finite based quasi-hereditary algebra
with the given basis. In fact, it has an upper finite split triangular
decomposition in the sense of
Remark~\ref{triangularhistory2}
with 
$A^\circ = \bigoplus_{i \in \N} \k e_i$,
$A^+ = \bigoplus_{i \in \N}(\k e_i \oplus \k y_i)$ and
$A^- = \bigoplus_{i \in \N} (\k e_i \oplus \k x_i)$.
Hence, $A\lfdlmod$ is an upper finite highest weight
category. Its standard and costandard modules have the structure
\begin{align}
\Delta'(i)&=
\begin{tikzpicture}[anchorbase]
\node at (0,0){{\xymatrix@C=0.9em@R=1.1em{
i\ar@{-}[d]|{y}\\
i+1
}}};\end{tikzpicture},
&\nabla'(i)&=
\begin{tikzpicture}[anchorbase]
\node at (0,0){{\xymatrix@C=0.9em@R=1.1em{
i+1\ar@{~}[d]|{x}\\
i
}}};\end{tikzpicture}.
\end{align}
Using the characterization from Theorem~\ref{moregin}(i), it follows that
the indecomposable tilting modules for $A$ have
a similar structure to $T'(0)$, which is as follows (to get $T'(i)$ in general
one just has to add $i$ to all of the labels):
\begin{equation}\label{megan}
T'(0)=
\begin{tikzpicture}[anchorbase]
\node (a) at (0,.3) {
{\xymatrix@C=0.9em@R=1.1em{
&& 5\ar@{~}[dr]|{x}\ar@{-}[dl]|{y}&&3\ar@{-}[dl]|{y}\ar@{~}[dr]|{x}&&1 \ar@{-}[dl]|{y}\ar@{~}[dr]|{x}&&\\
\cdots&6&&4&&2&&0\ar@{-}[dr]|{y}&&2\ar@{-}[dr]|{y}\ar@{~}[dl]|{x}&&4\ar@{-}[dr]|{y}\ar@{~}[dl]|{x}&&6\ar@{~}[dl]|{x}&\cdots\\
&&&&&&&&1&&3&&5\\
}}};
\draw (0,.29) circle (.2 cm);
\end{tikzpicture}
\end{equation}
This diagram demonstrates that $T'(0)$ has both an infinite ascending
$\Delta$-flag with $\Delta'(0)$ at the bottom and subquotients as indicated by the straight lines, and an infinite
descending $\nabla$-flag with $\nabla'(0)$ at the top and subquotients
indicated by the wiggly lines; cf. Claim~4 above.
Given the indecomposable tilting modules
$T'(i)$ for $A$, one can now compute the coalgebra $C$ arising
from the tilting generator $T' := \bigoplus_{i \geq 0}
T'(i)$ 
according to the general recipe from
Definition~\ref{rd2}.
We leave this to the reader, but display below the
homomorphisms $f_{i,j}^{(\ell)}:T'(i) \rightarrow
T'(j)$
in the endomorphism algebra $B := \End_A(T')^\op$
which are dual to the basis elements $c_{i,j}^{(\ell)}$ of the
coalgebra $C=B^\star$ as above.

The map $f_{i,i}^{(i)}:T'(i)\rightarrow T'(i)$ is
the identity endomorphism, and 
$f_{i,j}^{(\ell)}:T'(i)\rightarrow T'(j)$ 
for $\ell > \max(i,j)$ has irreducible
image and coimage isomorphic to $L'(\ell)$, i.e., it
sends the (unique) irreducible copy of $L'(\ell)$ in the
head of $T'(i)$ to the irreducible
$L'(\ell)$
in the socle of $T'(j)$.
The remaining maps $f_{i,j}^{(i)}, f_{i,j}^{(j)}:T'(i)\rightarrow T'(j)$ for $i \neq j$
are depicted below:
\begin{align*}
\substack{\phantom{i \not\equiv j(2)}\\{\normalsize  f_{i,j}^{(j)}}\\ i \not\equiv j(2)}:
\begin{tikzpicture}[anchorbase]
\node (a) at (0,.3) {
{\xymatrix@C=0.6em@R=.8em{
&&j\ar@{~}[dr]\ar@{-}[dl]&\\
\cdots\!\!\!&\!\!\! j+1 \!\!\!&&&\\
&&\\
}}};
\node at (1.4,.1){$\cdots$};
\node (a) at (3.3,.25) {
{\xymatrix@C=0.6em@R=.8em{
&\ar@{-}[dl] \!\!\! i+1 \!\!\! \ar@{~}[dr]&\\
&&i\ar@{-}[dr]&&\\
&&&\!\!\! i+1\!\!\!&\!\!\cdots\\
}}};
\draw (3.15,.23) circle (.2 cm);
 \draw [blue,thick] (-1.4,-.1) to (0,-.1) to (1.1,1.2) to (-1.4,1.2);
\end{tikzpicture}
\!\!\!\!
&\mapsto
\begin{tikzpicture}[anchorbase]
\node (a) at (3.45,.25) {
{\xymatrix@C=0.6em@R=.8em{
\cdots\!\!& \!\!\! j+1 \!\!\! \ar@{~}[dr]&\\
&&j\ar@{-}[dr]&&\!\!\!j+2\!\!\!&\!\!\cdots\\
&&&\!\!\! j+1\!\!\!\ar@{~}[ur]&\\
}}};
\draw (3,.23) circle (.2 cm);
 \draw [blue,thick] (5.7,.6) to (2,.6) to (3.3,-.8) to (5.7,-.8);
\end{tikzpicture}\\
\substack{\phantom{i \equiv j(2)}\\{\normalsize  f_{i,j}^{(j)}}\\ i \equiv j(2)}:
\begin{tikzpicture}[anchorbase]
\node (a) at (6.9,-.1) {
{\xymatrix@C=0.6em@R=.8em{
&j\ar@{~}[dl]\ar@{-}[dr]&\\
&&\!\!\! j+1 \!\!\!&\!\!\!\cdots&&\\
}}};
\node (a) at (3.4,.25) {
{\xymatrix@C=0.6em@R=.8em{
&& \ar@{-}[dl]\!\!\! i+1\! \!\!\! \ar@{~}[dr]&\\
\cdots\!\!&\!\!\!i+2\!\!\!&&\!i\ar@{-}[dr]&&\\
&&&&\:\:\:\:&\!\!\!\!\!\!\cdots\\
}}};
\draw (3.87,.18) circle (.2 cm);
 \draw [blue,thick] (7.8,.5) to (5.2,.5) to (6.5,-.9) to (7.8,-.9);
\end{tikzpicture}
\!\!\!\!\!\!\!\!\!\!\!\!\!
&\mapsto
\begin{tikzpicture}[anchorbase]
\node (a) at (3.45,.25) {
{\xymatrix@C=0.6em@R=.8em{
\cdots\!& \!\!\! j+1 \!\!\! \ar@{~}[dr]&\\
&&j\ar@{-}[dr]&&\!\!\!j+2\!\!\!&\!\!\cdots\\
&&&\!\!\! j+1\!\!\!\ar@{~}[ur]&\\
}}};
\draw (3,.23) circle (.2 cm);
 \draw [blue,thick] (5.7,.6) to (2,.6) to (3.3,-.8) to (5.7,-.8);
\end{tikzpicture}
\end{align*}
\begin{align*}
\substack{\phantom{i \equiv j(2)}\\{\normalsize  f_{i,j}^{(i)}}\\ i
  \not\equiv j(2)}:
\begin{tikzpicture}[anchorbase]
\node (a) at (3.45,.25) {
{\xymatrix@C=0.8em@R=1em{
&& \!\!\! \ar@{-}[dl] i+1 \!\!\! \ar@{~}[dr]&\\
\cdots\!\!\!&\!\!\!i+2\!\!\!&&i\ar@{-}[dr]&&\\
&&&&\!\!\! i+1\!\!\!&\!\!\cdots\\
}}};
\draw (3.88,.23) circle (.2 cm);
 \draw [blue,thick] (1.3,-.1) to (4.9,-.1) to (3.5,1.3) to (1.3,1.3);
\end{tikzpicture}
&\!\!\!\!\!\mapsto
\begin{tikzpicture}[anchorbase]
\node (a) at (6.9,-.1) {
{\xymatrix@C=0.6em@R=.8em{
\:\cdots\!\!\!\!\!\!\!\!\!&\:\:\:\:&&\!\!\!i+1\!\!\!\ar@{~}[dl]&\!\!\!\cdots\\
&&\ar@{-}[ul]\! i \!&&&&&\\
}}};
\node (a) at (3.4,.25) {
{\xymatrix@C=0.6em@R=.8em{
\cdots\!\!&\!\!\! j+1\! \!\!\! \ar@{~}[dr]&\\
&&\!j\ar@{-}[dr]&&\\
&&&\!\!\!j+1\!\!\!\ar@{~}[ur]&\\
}}};
\draw (3.44,.23) circle (.2 cm);
 \draw [blue,thick] (7.6,.6) to (6.5,.6) to (5.2,-.8) to (7.6,-.8);
\end{tikzpicture}
\!\!\!\!\!\!\!\!\!\!\!\!\!\\
\substack{\phantom{i \equiv j(2)}\\{\normalsize  f_{i,j}^{(i)}}\\ i \equiv j(2)}:
\begin{tikzpicture}[anchorbase]
\node (a) at (3.45,.25) {
{\xymatrix@C=0.6em@R=.8em{
&& \!\!\! \ar@{-}[dl] i+1 \!\!\! \ar@{~}[dr]&\\
\cdots\!&\!\!\!i+2\!\!\!&&i\ar@{-}[dr]&&\\
&&&&\!\!\! i+1\!\!\!&\!\!\cdots\\
}}};
\draw (3.87,.23) circle (.2 cm);
 \draw [blue,thick] (1.2,-.1) to (4.9,-.1) to (3.5,1.3) to (1.2,1.3);
\end{tikzpicture}
\!\!\!\!
&\!\!\!\!\!\mapsto
\begin{tikzpicture}[anchorbase]
\node (a) at (.2,.3) {
{\xymatrix@C=0.6em@R=.8em{
\cdots\!\!\!\!\!\!&&\!\!\!i+1\ar@{~}[dr]&&\:\:\:\\
&&&\!i \!\ar@{-}[ur]&\\
&&\\
}}};
\node (a) at (3.2,.25) {
{\xymatrix@C=0.6em@R=.8em{
\cdots\!\!& \!\!\! j+1 \!\!\! \ar@{~}[dr]&\\
&&j\ar@{-}[dr]&&\\
&&&\!\!\! j+1\!\!\!&\!\cdots\\
}}};
\draw (3.2,.23) circle (.2 cm);
 \draw [blue,thick] (-1.3,-.1) to (1.5,-.1) to (0.4,1.3) to (-1.3,1.3);
\end{tikzpicture}
\end{align*}

\begin{remark}\label{opendoor}
The above example can be changed slightly to obtain an
essentially finite example with weight poset $\Lambda:=\mathbb{Z}$ ordered by
the opposite of the natural ordering. To do this, let $D$ be the essentially finite-dimensional locally unital
algebra defined by the following quiver:
\begin{align*}
D\:&: \xymatrix{
\cdots
-\!1
\ar@/^/[r]^{y_{-1}}
&
0
\ar@/^/[r]^{y_0}
\ar@/^/[l]^{x_{-1}}
&
\ar@/^/[l]^{x_0}{1}
\ar@/^/[r]^{y_1}
&
\ar@/^/[l]^{x_1}
2\cdots}
&&\text{with relations }y_{i+1}y_i = x_i x_{i+1} = x_i y_i =0.
\end{align*}
Like for $A$, this algebra has a triangular decomposition, so
$D\fdlmod$ is an essentially finite highest weight category.
Recalling that the construction of tilting modules in the essentially finite
case explained in $\S$\ref{ifc} is by
passing to an upper finite truncation, 
the indecomposable tilting module $T(0)$ for $D$ has the same
structure as 
for $A$; see (\ref{megan}).
This module is infinite-dimensional; thus $D\fdlmod$ is {\em not} tilting-bounded.
Note also that 
this algebra $D$ can be obtained from the general construction from
Remark~\ref{triangularhistory1}, starting from the obvious triangular decomposition of
the $\Z$-graded algebra $\bar A =
\k\langle x,y \:|\:x^2=y^2=0, xy=0\rangle$ with $x$ in degree $1$ and
$y$ in degree $-1$; cf. \cite[Ex.~5.12]{BT}.
\end{remark}

\subsection{\boldmath Category $\Ocat$ for affine Lie algebras}\label{eg2}
Perhaps the first naturally-occurring examples of finite highest weight
categories came 
from the
blocks of the BGG category $\Ocat$ for a semisimple Lie
algebra. This context also provides natural examples of {finite}
fibered highest weight categories; see \cite{Mazsurvey} 
for a survey. To get examples of {\it semi-infinite} highest weight
categories, one 
can consider instead blocks of 
the category $\Ocat$ for an affine Kac-Moody Lie algebra. We briefly recall the setup referring to  \cite{Kac}, \cite{Carter} for more details. 

Let $\ocirc{\mg}$ be a finite-dimensional semisimple Lie algebra over
$\mathbb{C}$  and
$$\mg:=
\ocirc{\mg}\otimes_{\mathbb{C}}\mathbb{C}[t,t^{-1}]\oplus\mathbb{C}c\oplus
\mathbb{C}d
$$ 
be the corresponding affine Kac-Moody algebra.
Fix also a Cartan subalgebra $\ocirc{\mh}$ contained in a Borel
subalgebra 
$\ocirc{\mathfrak{b}}$ of $\ocirc{\mg}$.
There are corresponding subalgebras $\mathfrak{h}$ and $\mathfrak{b}$
of $\mathfrak{g}$, namely,
$$
\mh := \ocirc{\mh}\oplus\mathbb{C}c\oplus
\mathbb{C}d,
\qquad
\mathfrak{b} := \left(\ocirc{\mathfrak{b}} \otimes_{\mathbb{C}}\mathbb{C}[t]
+ \ocirc{\mg} \otimes_{\mathbb{C}} t \mathbb{C}[t]\right) \oplus\mathbb{C}c\oplus
\mathbb{C}d. 
$$
Let
$\{\alpha_i\:|\:i \in I\} \subset \mh^*$ and $\{h_i\:|\:i \in I\} \subset
\mh$ be the simple roots and coroots of $\mg$ 
and $(\cdot|\cdot)$ be the normalized invariant form on $\mh^*$,
all as in \cite[Ch.~7--8]{Kac}.
The {\em basic imaginary root}
$\delta\in\mh^*$ is the positive root corresponding to the canonical 
central element $c\in\mh$ under $(\cdot|\cdot)$.
The linear automorphisms of $\mh^*$ defined by 
$s_i:\la\mapsto \la-\la(h_i)\alpha_i$ generate the Weyl group
$W$ of $\mg$.  
Let $\rho\in\mh^*$ be the element satisfying $\rho(h_i)=1$ for all
$i\in I$ and $\rho(d) = 0$.
Then define the shifted action of $W$ on $\mh^*$ by
$w\cdot\la=w(\la+\rho)-\rho$ for $w\in W$, $\la\in\mh^*$.

We define the {\it level} of $\la\in\mh^*$ to be $(\la+\rho)(c) \in \mathbb{C}$. It is
{\it critical} 
if it equals the level of $\la=-\rho$, i.e., it is zero\footnote{Many
  authors define the level to be 
$\la(c)$, in which case the  critical level is $-\check{h}$, where $\check{h}$ is the dual Coxeter
  number.}.
We usually restrict our attention to integral weights $\la$, that is,
weights $\la \in \mh^*$ such that $\lambda(h_i) \in \Z$ for all $i \in
I$.
The level of an integral weight is either {\it positive}, {\it negative} or
{\it critical} ($=$ zero).  For any $\la \in \mh^*$, we define
\begin{equation}\label{tomorrow}
\lambda'
 := -\la-2\rho.
\end{equation}
Since $w\cdot(-\la-2\rho)=-w\cdot \la-2\rho$, weights $\la$ and $\mu$
are in the same orbit 
under the shifted action of $W$ if and only if so are $\la'$ and
$\mu'$. 
Note also that
the level of $\la$ is positive (resp., critical)
if and only if the level of $\la'$ is negative (resp., critical).
A crucial fact is that the orbit $W\cdot\lambda$ of an integral
weight $\la$ of positive level
contains a unique weight $\la_{\operatorname{max}}$ such that $\la_{\operatorname{max}}+\rho$ is dominant; e.g., see \cite[Ex.~13.1.E8a, Prop.~1.4.2]{Kumarbook}. 
By \cite[Cor.~1.3.22]{Kumarbook}, 
this weight is maximal in its orbit with respect to the usual
dominance ordering $\leq$ on weights, i.e., $\mu\leq\la$ if $\la-\mu \in
\bigoplus_{i \in I} \N \alpha_i$.
If $\la$ is integral of negative level, we deduce from this discussion
that its orbit contains a unique minimal weight $\la_{\operatorname{min}}$.

For $\la\in \mh^*$, let $\Delta(\la)$ be the Verma module with highest
weight $\la$ and $L(\la)$ be its unique irreducible quotient. Although
Verma modules need not be of finite length, the composition multiplicities
$[\Delta(\la)\::\:L(\mu)]$  are always finite.
There is also the {\em dual Verma
  module}
$\nabla(\la)$ which is the restricted dual $\Delta(\la)^\#$ 
of $\Delta(\la)$, i.e., the sum of the duals of the
weight spaces of $\Delta(\la)$ with the $\mathfrak{g}$-action twisted
by  the Chevalley antiautomorphism. 
All of the modules just introduced are objects in the category 
$\Ocat$ consisting of all $\mg$-modules $M$ which
are semisimple over $\mh$ with finite-dimensional weight spaces and such
that the set of weights of $M$ is contained in the lower set generated
by a finite subset of $\mh^*$; see \cite[$\S$2.1]{Kumarbook}. 
There is also a
larger category $\widehat{\Ocat}$ consisting of
the $\mg$-modules $M$ which are semisimple over $\mh$ and locally
finite-dimensional over $\mathfrak{b}$.

Let $\sim$ be the equivalence relation on $\mathfrak{h}^*$ generated
by $\lambda \sim \mu$ if there exists a positive root $\gamma$
and $n \in \Z$ such that $2(\lambda+\rho|\gamma) =n(\gamma|\gamma)$
and $\lambda-\mu = n\gamma$. 
For a $\sim$-equivalence class $\Lambda$, let $\Ocat_\Lambda$
(resp., $\widehat{\Ocat}_\Lambda$) be the full subcategory of
$\Ocat$ (resp., $\widehat{\Ocat}$) consisting of all
$M \in \Ocat$ (resp., $M \in \widehat{\Ocat}$) 
such that $[M:L(\la)] \neq 0\Rightarrow \la \in \Lambda$.
In view of the linkage
principle from \cite[Th.~2]{KK},
these subcategories may be called the {\em blocks} of $\Ocat$ and of
$\widehat{\Ocat}$, respectively.
In particular, 
by \cite[Th.~4.2]{DGK}, 
any $M \in \Ocat$
decomposes uniquely as a direct sum $M=\bigoplus_{\Lambda \in \mh^* /
  \sim} M_\Lambda$ with $M_\Lambda \in \Ocat_\Lambda$. Note though that
$\Ocat$ is not the coproduct of its blocks in the strict sense
since it is possible to find $M\in\Ocat$ such that 
$M_\Lambda$ is non-zero for infinitely many different $\Lambda$.
The situation is more satisfactory for $\widehat{\Ocat}$:
$\widehat{\Ocat}$ is the product of its blocks since
by \cite[Th.~6.1]{SoergelKac} 
the functor
\begin{eqnarray}
\label{blocks}
\prod_{\Lambda \in \mh^* / \sim}\widehat{\Ocat}_{\Lambda} \rightarrow
\widehat{\Ocat},
\qquad
(M_\Lambda)_{\Lambda \in \mh^* / \sim} \mapsto \bigoplus_{\Lambda \in \mh^* / \sim}
M_\Lambda
\end{eqnarray}
is an equivalence of categories.
Note also that $[\Delta(\la):L(\mu)] \neq 0$ implies that the level of 
$\la$ equals that of $\mu$, since the scalars by which $c$ acts on
$L(\la)$ and $L(\mu)$ must agree. Consequently,
we can talk simply about the level of a block.

A general combinatorial description of the $\sim$-equivalence classses $\Lambda$
 can be found for instance in \cite[Lem.~3.9]{Fiebigsubgeneric}.
For simplicity, we restrict ourselves from now on to integral blocks.
In non-critical levels, one gets exactly the $W$-orbits $W\cdot\lambda$ 
of the integral weights of non-critical level.
In critical level, one needs to incorporate also the 
translates 
by $\mathbb{Z}\delta$. 
From this description, it follows that the poset $(\Lambda, {\leq})$
underlying an integral block $\Ocat_\Lambda$ is upper finite
with unique maximal element $\lambda_{\operatorname{max}}$
if $\Ocat_\Lambda$ is of positive level, and lower finite with
unique minimal element $\lambda_{\operatorname{min}}$ if $\Ocat_\Lambda$ is of negative level. 
In case of the critical level, the poset is 
neither upper finite nor
lower finite, but it is always interval
finite.

\begin{example}\label{exsKM}
Here we give some explicit examples of posets which can occur for  $\mg=\hat{\mathfrak{sl}}_2$, the Kac-Moody algebra for the Cartan matrix 
$\begin{psmallmatrix}
\;\;2&-2\\
-2&\;\;2
\end{psmallmatrix}$.
The labelling set for the principal block is $W \cdot 0=\{\lambda_k, \mu_k\mid k\geq 0\}$
where $\lambda_k:=-\frac{1}{2}k(k+1)\alpha_0 -
\frac{1}{2}k(k-1)\alpha_1$
and $\mu_k :=
-\frac{1}{2}k(k-1)\alpha_0-\frac{1}{2}k(k+1)\alpha_1$. This is a block
of positive level with maximal element $\lambda_0=\mu_0=0$.
Applying the map (\ref{tomorrow}), we deduce that
$W \cdot (-2\rho) =\{\lambda'_k, \mu'_k\mid k\geq
0\}$. This is the labelling set for a block of negative level with minimal
element $\lambda'_0=\mu'_0=-2\rho$.
Finally, we have that
 $W \cdot(\alpha_0-\rho)\sqcup W \cdot
(\alpha_1-\rho) = \{\bar\lambda_k,\bar\mu_k\:|\:k \in \Z\}$ where
$\bar\lambda_k := (k+1)\alpha_0+k\alpha_1 -\rho$ and $\bar\mu_k :=
k\alpha_0 + (k+1)\alpha_1 -\rho$.
This
is the
labelling set for a block of critical level,
and it is neither upper nor lower finite.

\begin{align*}
\xymatrix@C=1em@R=1em{
&&&&&\ar@{.}[d]\ar@{.}[dr]&\ar @{.} [dl]\ar@{.}[d]&&\ar@{.}[d]&&\ar@{.}[d]\\
&0\ar @{-} [dl]_{s_1} \ar @{-} [dr]^{s_0}&&
&&\bar\la_{2}\ar@{-}[d]_{\delta}\ar @{-} [rd]_{\!\!s_1\:\:\:}&\bar\mu_2\ar@{-}[d]^{\delta}\ar@{-}[ld]_{\!\!\!s_0\,\:\:}&&\mu'_3&&\la'_3\\
\mu_1 \ar @{-} [d]_{s_0}&&\lambda_1 \ar @{-} [d]^{s_1}
&&&\bar\lambda_1\ar@{-}[d]_{\delta} \ar@{-}[rd]_{\!\!\!s_1\:\:\:}&\bar\mu_1\ar @{-}[dl]_{\!\!\!s_0\:\:\:}\ar@{-}[d]^{\delta}&&
\lambda'_2 \ar @{-} [u]^{s_1}&&\mu'_2  \ar @{-}[u]_{s_0}\\
\lambda_2 \ar @{-} [d] _{s_1}&&\mu_2\ar @{-} [d]^{s_0}
&,&&\bar\la_{0}\ar @{-} [rd]_{\!\!\!\!\!s_1\!\:\:}\ar@{-}[d]_{\delta}&\bar\mu_0\ar@{-}[d]^{\delta}\ar@{-}[dl]_{\!\!\!s_0\:\:\:}&,\qquad&\mu'_1\ar @{-} [u]^{s_0}&&\lambda'_1\ar@{-} [u]_{s_1}&.\\
\mu_3\ar@{.}[d]&&\la_3\ar@{.}[d] 
&&&\bar\lambda_{-1}\ar@{.}[d]\ar@{.}[dr]&\bar\mu_{-1}\ar@{.}[d]\ar @{.} [dl] &&& \!\!\!-2\rho\!\!\!\ar @{-} [ul]^{s_1}\ar @{-} [ur]_{s_0}\\
&&&&\:\:\:&\:\:\:&
}\\
{\text{Positive level}}
\hspace{20.1mm}
{\text{Critical level}}
\hspace{19mm}
{\text{Negative level}}\hspace{6.7mm}
\end{align*}
\end{example}

Recall the definitions of upper finite and lower finite highest
weight categories from Definitions~\ref{ufc} and \ref{newlfd}, respectively.

\begin{theorem}\label{dennies}
Let $\Ocat_\Lambda$ be an integral block of $\Ocat$ of 
non-critical level.
Then it is an upper finite or lower finite highest weight category
according to whether the level is positive or negative, respectively.
In both cases,
the standard and costandard objects 
are the Verma modules $\Delta(\la)$ and the dual Verma
modules $\nabla(\la)$, respectively, for $\la \in \Lambda$.
The partial order $\leq$ on $\Lambda$ is the dominance order.
\end{theorem}

\begin{proof}
First, we prove the result for
an integral block $\Ocat_\Lambda$ of positive
level. As explained above, the poset $\Lambda$ is upper finite in this case.
Let $\lambda_{\operatorname{max}}$ be its unique maximal weight.
\vspace{1.5mm}

\noindent{Claim 1:}
{\em In the positive level case,
$\Ocat_\Lambda$ is the full
subcategory of $\widehat{\Ocat}_\Lambda$ consisting of
all modules $M$ such that $[M:L(\lambda)] < \infty$ for all $\lambda \in
\Lambda$.}
To prove this, 
given $M \in \Ocat_\Lambda$, it is obvious that all of its
composition multiplicities are finite since $M$ has finite-dimensional
weight spaces. Conversely, suppose that all of the composition
multiplicities of $M \in \widehat{\Ocat}_\Lambda$
are finite. 
All weights of $M$ lie in the lower set generated by
$\lambda_{\operatorname{max}}$. Moreover, for 
$\lambda \leq \lambda_{\operatorname{max}}$, 
the dimension of the $\la$-weight space of $M$ is
$$
\dim M_\la=
\sum_{\mu \in \Lambda} [M:L(\mu)] \dim L(\mu)_\lambda.
$$ 
Since the poset is upper finite, there are only
finitely many
$\mu \in \Lambda$ such that the $\lambda$-weight space 
$L(\mu)_\lambda$ is non-zero, and these weight spaces are
finite-dimensional, so we deduce that $\dim M_\lambda < \infty$.
This proves the claim.

\vspace{1.5mm}

\noindent
Now we observe that the Verma module $M(\la_{\operatorname{max}})$ with  maximal possible highest weight is projective in $\widehat{\Ocat}_\Lambda$. From this and a standard argument involving translation functors through walls (see e.g. \cite{Neidhardt}) and the combinatorics
from \cite[$\S$4]{Fiebigtransl} (see also the introduction of \cite{Fie3}), it follows that there are projective modules $P_\lambda \in \widehat{\Ocat}_\Lambda$ with (finite)
$\Delta$-flags 
as in the axiom ($\widehat{P\Delta}$). Since each $\Delta(\lambda)$ belongs to $\Ocat_\Lambda$, we actually have that $P_\lambda \in \Ocat_\Lambda$.
All that is left to complete the proof of the
theorem in the positive level case
is to show that $\Ocat_\Lambda$ is a Schurian category. 
Let $A := \left(\bigoplus_{\lambda,\mu \in \Lambda}
\Hom_{\mg}(P_\lambda,P_\mu)\right)^\op$. 
Since the multiplicities
$[P_\mu:L(\la)]$ are finite, $A$ is a locally
finite-dimensional locally unital algebra.
Using Lemma~\ref{mor}, we deduce that
$\widehat{\Ocat}_\Lambda$ is equivalent to the category $A\lmod$
of all left $A$-modules. As explained in the discussion after
(\ref{compmults}), $A\lfdlmod$ is the full 
subcategory of $A\lmod$ consisting of all modules with finite
composition multiplicities.
Combining this with Claim 1, we deduce that the equivalence between
$\widehat{\Ocat}_\Lambda$ and $A\lmod$ restricts to an
equivalence between $\Ocat_\Lambda$ and $A\lfdlmod$.
Hence, $\Ocat_\Lambda$ is a Schurian category.

We turn our attention to an integral block $\Ocat_\Lambda$
of negative level. In this case, we know already that the poset
$\Lambda$ is lower finite with a unique minimal element $\lambda_{\operatorname{min}}$.

\vspace{1.5mm}
\noindent{Claim 2:} {\em In the negative level case, the category $\Ocat_\Lambda$ is the full
subcategory of $\widehat{\Ocat}_\Lambda$ consisting of all
modules of finite length.}
For this, it is obvious that any module in $\widehat{\Ocat}_\Lambda$ of
finite length belongs to $\Ocat_\Lambda$.
Conversely,
any object in $\Ocat_\Lambda$ is of finite length thanks to the
formula \cite[2.1.11(1)]{Kumarbook}, taking $\la$ therein to be 
$\lambda_{\operatorname{min}}$.

\vspace{1.5mm}
\noindent
From Claim 2 and Lemma~\ref{char}, it follows that
$\R := \Ocat_\Lambda$ is a locally finite Abelian category. By \cite[Th.~2.7]{Fiebigtransl} the Serre subcategory $\R\d$ of $\R$ associated to 
$\Lambda\d$ is a finite highest weight category for each finite lower
set $\Lambda\d$ of $\Lambda$. We deduce that $\R$ is a lower finite highest weight
category according to Definition~\ref{newlfd}.
\end{proof}

Let $\Ocat_\Lambda$ be an integral block of non-critical level.
The following assertions about projective and injective modules follow
from Theorem~\ref{dennies} and the general theory from
$\S\S$\ref{pct}--\ref{errorhere};
see also \cite[Rem.~6.5]{SoergelKac}.
\begin{itemize}
\item
In the positive level case, when $\Ocat_\Lambda$ is a Schurian
category, $\widehat{\Ocat}_\Lambda$ has enough projectives and injectives.
Moreover, the projective covers of the irreducible modules are
the modules $\{P(\la)\:|\:\la \in \Lambda\}$
constructed in the proof of Theorem~\ref{dennies}, and these belong to $\Ocat_\Lambda$.
Their restricted duals $I(\la) := P(\la)^\#$ are the indecomposable
injective modules in $\widehat{\Ocat}_\Lambda$, 
and also belong to $\Ocat_\Lambda$.
\item
The situation is completely different 
in the negative level case, as 
we need to pass to 
$\widehat{\Ocat}_\Lambda$, which is the ind-completion of
the finite Abelian category 
$\Ocat_\Lambda$, before we can talk about injective modules.
In $\widehat{\Ocat}_\Lambda$, the irreducible module
$L(\la)\:(\la\in\Lambda)$ 
has an injective hull $I(\la)$ in
$\widehat{\Ocat}_\Lambda$, which possesses a (possibly infinite) ascending
$\nabla$-flag in the sense of Definition~\ref{good}.
However,
$\widehat{\Ocat}_\Lambda$ usually does not have any
projectives at all (although one could construct such modules in the
pro-completion of $\Ocat_\Lambda$ as done e.g. in \cite{Fie3}).
\end{itemize}
The following results about tilting modules 
are consequences of the general
theory developed in $\S$\ref{sbugs} and $\S$\ref{sbags}. They
already appeared in an equivalent form 
in \cite{SoergelKac}.
\begin{itemize}
\item In the negative level case, tilting modules are objects in
  $\Ocat_\Lambda$ admitting
  both a (finite) $\Delta$-flag and a (finite) $\nabla$-flag. The isomorphism classes
  of 
  indecomposable tilting modules in $\Ocat_\Lambda$ are
  parametrized by their highest weights.
They
may also be
constructed by applying translation functors to the Verma module $\Delta(\la_{\operatorname{min}})$. 
\item
In the positive level case, tilting 
modules are objects in $\Ocat_\Lambda$ which admit both a (possibly
infinite) ascending $\Delta$-flag 
and 
a (possibly infinite) descending $\nabla$-flag in the sense of
Definition~\ref{tpc}. Again, the isomorphism classes of indecomposable tilting modules are
parametrized by their highest weights.
\end{itemize}
In both cases, our characterization of the
indecomposable 
tilting module 
$T(\lambda)$ of highest weight $\lambda$
is slightly different from the one given in \cite[Def.~6.3]{SoergelKac}. 
From our definition, one sees immediately that $T(\lambda)^\# \cong T(\lambda)$.

\begin{remark}
In the literature dealing with positive level,
it is common to pass to a different category of
modules,  e.g., to the Whittaker category in \cite{Bez} or to
truncated versions of $\Ocat$ in  \cite[$\S$3]{SVV}, before
contemplating tilting modules.
\end{remark}

Our next result is concerned with the Ringel duality between integral
blocks of positive and negative level.
This depends crucially on a special case of 
the {Arkhipov-Soergel equivalence} from 
\cite{Ark}, \cite{SoergelKac}.
Let $S$ be Arkhipov's 
semi-regular bimodule, which is the bimodule $S_\gamma$ 
of \cite{SoergelKac}
with $\gamma := 2\rho$ as in 
\cite[Lem.~7.1]{SoergelKac}.
For $\la \in \mathfrak{h}^*$,
let $T(\la)$ be the indecomposable tilting module from
\cite[Def.~6.3]{SoergelKac} (which is the same as in the previous paragraph for integral
$\lambda$ of positive or negative level).
Also let $P(\la)$ be a projective cover of $L(\la)$ in
$\widehat{\Ocat}$ whenever such an object exists;
cf. \cite[Rem.~6.5(2)]{SoergelKac}.

\begin{theorem}[Arkhipov-Soergel equivalence]\label{as}
Tensoring with the semi-regular bimodule defines an equivalence 
$S \otimes_{U(\mathfrak{g})} ?:
\Delta(\Ocat) \rightarrow
\nabla(\Ocat)$
between the exact 
subcategories of $\Ocat$ consisting of objects
with (finite) $\Delta$- and $\nabla$-flags, respectively. Moreover
(assuming for the second isomorphism that $P(\la)$ exists):
\begin{enumerate}
\item
$S \otimes_{U(\mathfrak{g})}\Delta(\la)\cong\nabla(\lambda')$.
\item
$S \otimes_{U(\mathfrak{g})}P(\la)\cong T(\lambda')$ (assuming
$P(\la)$ exists).
\end{enumerate}
\end{theorem}

\begin{corollary}\label{KacMoody}
Assume that ${\Ocat}_\Lambda$ is an integral block of negative level.
Let
$\Ocat_\Lambda'$ be the Ringel dual of
$\Ocat_\Lambda$
relative to some choice of $T = \bigoplus_{i \in I} T_i$
as in 
Definition~\ref{rd1},
and let $F$
be the
Ringel duality functor from (\ref{rdf1}).
Also let $\Lambda' := \{\la'\:|\:\la\in\Lambda\}$.
Then there is an equivalence of categories
$E:
\Ocat_\Lambda'
\rightarrow
\Ocat_{\Lambda'}
$  
such that $E \circ F
: \nabla(\Ocat_\Lambda)\rightarrow \Delta(\Ocat_{\Lambda'})$ is a quasi-inverse to
the Arkhipov-Soergel equivalence 
$S\otimes_{U(\mathfrak{g})} ?:
\Delta(\Ocat_{\Lambda'})\rightarrow \nabla(\Ocat_\Lambda)$.
\end{corollary}

\begin{proof}
Note to start with
that
$\Ocat_{\Lambda'}$ is an integral block of positive
level.
Moreover,
the map $(\Lambda,\geq) \rightarrow (\Lambda', \leq), \lambda \mapsto \lambda'$ 
is an order isomorphism.

Choose a quasi-inverse $D$ to 
$S\otimes_{U(\mathfrak{g})} ?:
\Delta(\Ocat_{\Lambda'})\rightarrow \nabla(\Ocat_\Lambda)$, 
and set $P_i :=  D T_i$.
By Theorem~\ref{as}(2), $(P_i)_{i \in I}$ is a projective generating
family for 
$\Ocat_{\Lambda'}$.
Moreover, recalling that $\Ocat_\Lambda'$ is the category $A\lfdlmod$ where $A
:=\left(\bigoplus_{i,j \in I}
 \Hom_{\Ocat_{\Lambda}}(T_i, T_j)\right)^\op$, the functor
$D$ induces an isomorphism via which we can identify $A$ with
$\left(\bigoplus_{i,j \in I}
 \Hom_{\Ocat_{\Lambda'}}(P_i, P_j)\right)^\op$.

As explained in the proof of Theorem~\ref{dennies},
the functor $$
H:=\bigoplus_{i \in I} \Hom_{\Ocat_{\Lambda'}}(P_i,?):\Ocat_{\Lambda'}
\rightarrow A\lfdlmod
$$ 
is an equivalence of categories.
Moreover, we have that
$$
H \circ D = \bigoplus_{i \in I} \Hom_{\Ocat_{\Lambda'}}(P_i,
D ?)
\cong
\bigoplus_{i \in I} \Hom_{\Ocat_{\Lambda}}(S
\otimes_{U(\mathfrak{g})} P_i,?)
\cong
\bigoplus_{i \in I} \Hom_{\Ocat_{\Lambda}}(T_i,
?) = F.
$$
Letting $E$ be a quasi-inverse equivalence to $H$, 
it follows that $E \circ F \cong D$.
\end{proof}

\begin{remark}
In the setup of Corollary~\ref{KacMoody}, the
Arkhipov-Soergel equivalence extends to 
an equivalence
$S\otimes_{U(\mathfrak{g})} ?:
\Delta^{\asc}(\Ocat_{\Lambda'})\rightarrow \nabla^{\asc}(\Ocat_\Lambda)$,
which 
is a quasi-inverse to
$E\circ F:
 \nabla^\asc(\Ocat_\Lambda)\rightarrow \Delta^\asc(\Ocat_{\Lambda'})$.
These functors interchange the indecomposable injectives in
$\widehat{\Ocat}_\Lambda$ with the indecomposable tiltings in
$\Ocat_{\Lambda'}$.
\end{remark}

Finally we discuss the situation for an
integral critical block $\Ocat_\Lambda$. 
As we have already explained,
in this case the poset $\Lambda$ is neither upper nor lower finite.
In fact, these blocks do not fit 
into the framework of this article at all,
since the Verma modules
have infinite length and there are no projectives.
One sees this already for the Verma module $\Delta(-\rho)$ for $\mg =
\widehat{\mathfrak{sl}}_2$, which has composition factors
  $L(-\rho-m\delta)$ for $m \geq 0$, each appearing with multiplicity
  equal to the number of partitions of $m$; see
  e.g. 
\cite[Th.~4.9(1)]{AF}.
However, there is an autoequivalence
$\Sigma:=L(\delta)\otimes ?:
\widehat{\Ocat}_\Lambda \rightarrow \widehat{\Ocat}_\Lambda$, which makes it
possible to pass to 
the {\em restricted category}
$\widehat{\Ocat}_\Lambda^{\res}$, which we define next.

Let $A_n$ be the vector space of natural transformations 
$\Sigma^n \rightarrow \operatorname{Id}$.
This gives rise to a graded algebra
$A:=\bigoplus_{n\in\mathbb{Z}} A_n$. 
Then the restricted category
$\widehat{\Ocat}_\Lambda^{\res}$ is the full
subcategory of $\widehat{\Ocat}_\Lambda$ consisting of all modules 
which are
annihilated by the induced action of $A_n$ for $n\neq 0$;
cf. \cite[$\S$4.3]{AF}. 
The irreducible modules in the restricted category are the same as in
$\widehat{\Ocat}_\Lambda$ itself. There are also the {\em restricted Verma modules} 
\begin{equation}
\Delta(\la)^\res:=\Delta(\la)\bigg/ 
\sum_{\eta \in A_{\neq 0}}
\im \left(\eta_{\Delta(\la)}:\Sigma^n\Delta(\la)\rightarrow \Delta(\la)\right)
\end{equation}
from \cite[$\S$4.4]{AF}.
In other words, $\Delta(\la)^\res$ 
is the largest quotient of $\Delta(\la)$ that belongs to the
restricted category.
Similarly, the {\it restricted dual Verma module}
${\nabla}(\la)^\res$
is the largest submodule of $\nabla(\la)$ that belongs to the
restricted category.

The restricted category $\widehat{\Ocat}_\Lambda^{\res}$
is no longer indecomposable: by \cite[Th.~5.1]{AFlinkage} it
decomposes further as
\begin{equation}
\widehat{\Ocat}_\Lambda^{\res}
= \prod_{\overline{\Lambda} \in \Lambda / W} \widehat{\Ocat}_{\overline{\Lambda}}^{\res}
\end{equation}
where $\Lambda / W$ denotes the orbits of $W$ under the dot action.
For instance, the poset $\Lambda$ for the critical level displayed in
Example~\ref{exsKM} splits into two orbits
$W\cdot(\alpha_0-\rho)$
and $W\cdot (\alpha_1-\rho)$ (i.e., one removes the edges 
labelled by $\delta$). In the most singular case,
$\widehat{\Ocat}_{-\rho}^{\res}$ is
a product of simple blocks; in particular,
$\Delta^\res(-\rho)=L(-\rho)= {\nabla}^\res(-\rho)$.

\begin{conjecture}[Critical block conjecture]\label{cong} Let $\widehat{\Ocat}_{\overline\Lambda}^{\res}$ be a regular integral critical block in the sense of \cite{AFlinkage}.
Let $\Ocat_{\overline\Lambda}^{\res} := \operatorname{Fin}\left(\widehat{\Ocat}_{\overline\Lambda}^{\res}\right)$ 
be the full subcategory consisting of all modules of finite length.
Then 
$\Ocat_{\overline\Lambda}^{\res}$ is 
an essentially finite highest weight category with 
standard and costandard objects 
$\Delta(\la)^{\res}$ and ${\nabla}(\la)^{\res}$ for $\la\in {\overline\Lambda}$. 
Moreover,
the indecomposable projective modules in 
$\Ocat_{\overline\Lambda}^{\res}$
 are also its indecomposable tilting modules, and therefore $\Ocat_{\overline\Lambda}^{\res}$ is tilting-bounded and Ringel self-dual.
\end{conjecture}

This conjecture is true 
for the basic example of a critical
block from Example~\ref{exsKM}
thanks to \cite[Th.~6.6]{Fiebigsubgeneric}; the same category
arises as the principal block of category $\Ocat$ 
for $\mathfrak{gl}_{1|1}(\mathbb{C})$ discussed in $\S$\ref{eg5}
below. 
The conjecture is also
consistent with the so-called {\it Feigin-Frenkel conjecture}
\cite[Conj.~4.7]{AF}, which says 
that composition multiplicities of restricted Verma modules 
are related to the periodic Kazhdan-Lusztig polynomials from
\cite{Lusztigper2} (and Jantzen's
generic decomposition patterns from \cite{Jandec}).
These polynomials depend on the relative position of the given
pair of weights and, 
when not too close to walls, they
vanish for weights that are far apart.
This is
consistent with the conjectured
existence of indecomposable projectives of finite
length in regular blocks of the restricted category.

\begin{remark}
It seems to us 
that the Feigin-Frenkel conjecture might have an explanation in
terms of a sequence of equivalences of categories 
similar to  \cite[(7)]{FG}. Ultimately this should 
connect
$\Ocat_{\overline\Lambda}^{\res}$ with 
representations of the quantum group
analog  
of Jantzen's thickened Frobenius kernel
$G_1 T$. Assuming that $\ell$ (the order of the root of unity) is odd and bigger than or equal to the Coxeter number, the latter are
known by \cite[\S17]{AJS} to be essentially finite highest weight
categories controlled by the periodic
Kazhdan-Lusztig polynomials. Also, in these categories, 
tilting modules are projective, hence, the
Ringel self-duality would be an obvious consequence.
\end{remark}

\subsection{Rational representations}\label{eg3}
As we noted in Remark~\ref{cpsref}, the definition of
lower finite highest weight category originated in the work of Cline,
Parshall and Scott \cite{CPS}. As well as the BGG category $\Ocat$ already mentioned, their work was motivated by
the representation theory of a reductive algebraic group $G$ in
positive characteristic, as developed for example in \cite{J}:
the symmetric tensor\footnote{Locally finite Abelian, monoidal, rigid, $\End(\unit)=\k$.}
category $\Rep(G)$ of finite-dimensional
rational representations of $G$ is a lower finite highest weight category. 
Tilting modules for $G$ were studied in \cite{Dtilt}, although
our formulation of semi-infinite Ringel duality
from $\S$\ref{sird} is not mentioned explicitly there:
Donkin instead took the approach pioneered in \cite{Dschur} of 
truncating to a finite lower set before taking Ringel duals.
In fact, now, there is monoidal structure in play and the story is even richer.

To give more details, we fix a maximal torus $T$ contained in an opposite pair of Borel
subgroups $B^+$ and $B^-$ of $G$.
Then the weight poset $\Lambda$ is the set $X^+(T)$ of dominant
characters of $T$ with respect to $B^+$.
We denote the natural duality on $\Rep(G)$ by 
$V \mapsto V^*$ (with action defined via
$g \mapsto g^{-1}$).
The costandard objects are the induced modules 
$H^0(\lambda) := 
H^0(G/B^-, \mathcal{L}_\lambda)$ and the standard objects are the Weyl modules
$V(\lambda) := H^0(G/B^+, \mathcal{L}_\lambda^*)^*$.
For the partial order $\leq$, one can use the usual dominance
ordering on $X^+(T)$, or the more refined Bruhat order of \cite[$\S$II.6.4]{J}.
This makes $\Rep(G)$ into 
a lower finite highest weight category by \cite[Prop.~II.4.18]{J} and \cite[Prop.~II.6.13]{J}.
In fact, in the case of $\Rep(G)$, all of the general results about ascending
$\nabla$-flags found in $\S\ref{lfsc}$ were known already before
the time of \cite{CPS}, e.g., they are discussed 
in Donkin's book
\cite{Dbook} (and called there {\em good filtrations}). 

Let $\Tilt(G)$ be the full subcategory of $\Rep(G)$ consisting of
all tilting modules.
A key theorem in this setting is that tensor products of tilting modules are
tilting; this is the Donkin-Mathieu-Wang theorem \cite{Dbook}, \cite{Mat}, \cite{Wang}.
Thus, $\Tilt(G)$ is a symmetric pseudo-tensor\footnote{Additive Karoubian, monoidal, rigid, 
$\End(\unit)=\k$.} category.
Let $(T_i)_{i \in I}$ be a monoidal generator for $\Tilt(G)$, 
i.e., each $T_i$ is a tilting module and every
indecomposable tilting module is isomorphic to a summand of a tensor
product $T_{\bi} := T_{i_1}\otimes\cdots \otimes T_{i_n}$
for some $n \geq 0$ and $\bi = (i_1,\dots,i_n) \in I^n$.
Then define $\Acat$ to be the category with objects $\bI := \bigsqcup_{n \geq 0} I^n$ and
morphisms defined from $\Hom_{\Acat}(\bj, \bi) :=\Hom_G(T_{\bi},
T_{\bj})$, composition being induced by the opposite of
composition in $\Rep(G)$.
The category $\Acat$ is naturally a 
strict symmetric monoidal category, with
the tensor product of objects being
by concatenation of sequences. 
The evident monoidal functor $\Acat \rightarrow \Tilt(G)^\op$ 
extends to the Karoubi envelope of
$\Acat$, and the resulting functor
$\operatorname{Kar}(\Acat) \rightarrow \Tilt(G)^\op$ is a 
symmetric monoidal equivalence.

Forgetting the monoidal structure, one can think instead
in terms of the locally finite-dimensional locally unital algebra 
$A = \bigoplus_{\bi,\bj \in \bI} e_\bi A e_\bj$ 
that is the path algebra
of $\Acat$ in the sense of Remark~\ref{dataof}.
It becomes convenient to identify $T = \bigoplus_{\bi \in \bI} T_\bi$ and $T^\circledast = \bigoplus_{\bi \in \bI} T_\bi^*$ 
with the tensor algebras
\begin{equation}\label{rainier}
T = T(V), \quad T^\circledast= T(V^*)\qquad\text{where}\qquad V := \bigoplus_{i \in I} T_i.
\end{equation}
Note that $T$ is naturally a right $A$-module
and $T^\circledast$ is a left $A$-module.
Since $T$ is a tilting generator for $\Rep(G)$
in the sense of Definition~\ref{rd1},
$A\lfdlmod$ is the Ringel dual of $\Rep(G)$ with respect to 
$T$. Theorem~\ref{rt1} implies that $A\lfdlmod$
is an upper finite highest weight category with 
 poset $(X^+(T), \geq)$.
 Moreover, by Corollary~\ref{mustard1},
$T^\circledast$ is a tilting generator for $A\lfdlmod$ with 
$\Coend_A(T^\circledast)\cong \k[G]$
as coalgebras.

At this point, the monoidal structure on the category $\Acat$ 
comes back into the picture
since the $A$-module $T$ comes from a 
faithful symmetric monoidal functor (``fiber functor")
$T:\Acat \rightarrow (\fVec)^\op$. 
Consequently, by classical arguments of Tannaka duality (e.g., see \cite[$\S$2]{DM} and \cite[$\S$5.4]{EGNO}), $\Coend_A(T^\circledast)$ can be endowed with the structure of a commutative 
Hopf algebra which reconstructs the coordinate algebra of $G$.
To explain this in more detail, we use the setup of (\ref{C}), so now
we are identifying the coalgebra $\Coend_A(T^\circledast)$ with 
\begin{equation}\label{thec}
C := T \otimes_A T^\circledast = T(V) \otimes_A T(V^*).
\end{equation}
Then the algebra structure on $C$ is induced by the natural multiplication on 
the tensor product of algebras $T(V) \otimes T(V^*)$, that is,
\begin{equation}
(v \otimes u)\cdot(v' \otimes u') := (v \otimes v') \otimes (u \otimes u')
\end{equation}
for $v,v' \in T(V)$ and $u,u' \in T(V^*)$.
If we pick a basis $v^{(i)}_1,\dots,v^{(i)}_{d(i)}$ for each $T_i$ and let $u^{(i)}_1,\dots,u^{(i)}_{d(i)}$ be the dual basis for $T_i^*$,
then the elements 
\begin{equation}\label{alggens}
\big\{c_{r,s}^{(i)} := v_s^{(i)} \otimes u_r^{(i)}
\:\big|\:i \in I, 1 \leq r, s \leq d(i)\big\}
\end{equation}
generate $C$ as an algebra. The coalgebra structure satisfies
\begin{align}
\delta(c_{r,s}^{(i)}) &= \sum_{t = 1}^{d(i)}
c_{r,t}^{(i)} \otimes c_{t,s}^{(i)},
&\eps(c_{r,s}^{(i)}) &= \delta_{r,s}.
\end{align}
Now the reconstruction theorem can be formulated as follows.

\begin{theorem}[Tannakian reconstruction]\label{rdt}
The above construction makes the coalgebra 
$C=\Coend_A(T^\circledast)$ into a commutative Hopf algebra which 
is isomorphic (as a Hopf algebra) to the coordinate algebra $\k[G]$
via the unique algebra homomorphism
sending
$c_{r,s}^{(i)}\in C$ to the matrix coefficient function
$\tilde c_{r,s}^{(i)} \in \k[G]$ defined by
$
g v^{(i)}_s = \sum_{r=1}^{d(i)}
\tilde c_{r,s}^{(i)}(g) v^{(i)}_r$
for $g \in G$.
\end{theorem}

\begin{proof}
For $\bi = (i_1,\dots,i_n) \in \bI^n$
and $\br = (r_1,\dots,r_n), \bs = (s_1,\dots,s_n) \in \Z^n$
with $1 \leq r_k,s_k \leq d(i_k)$ for each $k$, 
let $c_{\br,\bs}^{(\bi)} := 
(v_{r_1}^{(i_1)} \otimes\cdots\otimes v_{r_n}^{(i_n)})
\otimes (u_{r_1}^{(i_1)} \otimes \cdots \otimes u_{r_n}^{(i_n)})
\in C$.
These are the elements in the formula (\ref{bubbly}), and they span $C$.
The coalgebra isomorphism $C \stackrel{\sim}{\rightarrow} \k[G]$ from Corollary~\ref{mustard1}(i)
sends $c_{\br,\bs}^{(\bi)}\in C$
to $\tilde c_{r_1,s_1}^{(i_1)} \cdots \tilde c_{r_n,s_n}^{(i_n)} \in \k[G]$.
So to be an algebra isomorphism, we must have that
$c_{\br,\bs}^{(\bi)} = c_{r_1,s_1}^{(i_1)}\cdots c_{r_n,s_n}^{(i_n)}$,
which is exactly the definition of multiplication given above.
\end{proof}

Theorem~\ref{rdt} recovers a classical result: it is
a special case of \cite[Th.~2.11]{DM}, which implies that $\k[G]$ is isomorphic to $\Coend(F)$ where $F:\Rep(G) \rightarrow \fVec$ is the forgetful functor. To deduce Theorem~\ref{rdt} from this statement, one also needs to observe that $\Coend(F) \cong \Coend_{A}(T)$; this holds because the algebraic group $G$ is isomorphic to its image in its representation on
$V=\bigoplus_{i \in I} T_i$ by weight considerations. 

\begin{remark}\label{lostphone}
To get a full set of relations between the generators (\ref{alggens}) of $C$,
one just needs to take the equations
$v x \otimes u = v \otimes x u$ 
for $x:\bi\rightarrow \bj$ running over a system of monoidal generators 
for $\Acat$ and all $v \in T_\bi, u \in T_\bj^*$.
\end{remark}

\begin{remark}
Theorem~\ref{TBAnew} can often be applied in this context to 
give $A$ (or some idempotent expansion of $A$) the structure
of an upper finite (perhaps symmetrically) based quasi-hereditary algebra.
\end{remark}

\def\CUP{\:
\begin{tikzpicture}[baseline=.7mm]
	\draw[-,thick] (0.25,0.35) to[out=-90, in=0] (0.05,0);
	\draw[-,thick] (0.05,0) to[out = 180, in = -90] (-0.15,0.35);
\end{tikzpicture}\:}

\def\CAP{\:
\begin{tikzpicture}[baseline=.85mm]
	\draw[-,thick] (0.25,0) to[out=90, in=0] (0.05,0.35);
	\draw[-,thick] (0.05,0.35) to[out = 180, in = 90] (-0.15,0);
\end{tikzpicture}\:}

\def\BUBBLE{\begin{tikzpicture}[baseline=1mm]
  \draw[-,thick] (0.2,0.2) to[out=90,in=0] (0,.4);
  \draw[-,thick] (0,0.4) to[out=180,in=90] (-.2,0.2);
\draw[-,thick] (-.2,0.2) to[out=-90,in=180] (0,0);
  \draw[-,thick] (0,0) to[out=0,in=-90] (0.2,0.2);
\end{tikzpicture}}

\def\CUPCAP{
\begin{tikzpicture}[baseline=3mm]
	\draw[-,thick] (0.2,0) to[out=90, in=0] (0,0.35);
	\draw[-,thick] (0,0.35) to[out = 180, in = 90] (-0.2,0);
		\draw[-,thick] (0.2,0.85) to[out=-90, in=0] (0,0.5);
	\draw[-,thick] (0,0.5) to[out = 180, in = -90] (-0.2,0.85);
\end{tikzpicture}}

\def\CROSSING{
\begin{tikzpicture}[baseline=3mm]
	\draw[-,thick] (0.25,0) to (-0.25,.85);
	\draw[thick,-] (-0.25,0) to (0.25,.85);
\end{tikzpicture}}

\def\POSCROSSING{
\begin{tikzpicture}[baseline=3mm]
	\draw[-,thick] (0.25,0) to (-0.25,.85);
	\draw[line width=5pt,-,white] (-0.25,0) to (0.25,.85);
	\draw[thick,-] (-0.25,0) to (0.25,.85);
\end{tikzpicture}}

\def\NEGCROSSING{
\begin{tikzpicture}[baseline=3mm]
	\draw[thick,-] (-0.25,0) to (0.25,.85);
	\draw[-,line width=5pt,white] (0.25,0) to (-0.25,.85);
	\draw[-,thick] (0.25,0) to (-0.25,.85);
\end{tikzpicture}}

\def\IDID{
\begin{tikzpicture}[baseline=3mm]
\draw[-,thick] (0,0) to (0,.85);
\draw[-,thick] (0.3,0) to (0.3,.85);
\end{tikzpicture}}

The first example comes from $G = SL_2$.
For this, we may take
$I := \{|\}$ and $T_|$ to be the 
natural two-dimensional
representation $V$ of $G$ with its standard basis $v_1,v_2$;
we also use $u_1,u_2$ to denote the dual basis of $V^*$.
The module $V$ is a monoidal generator for $\Tilt(G)$ by weight considerations.
Note also that $V$ possesses an invariant symplectic form 
such that $(v_1,v_2) = 1$, hence, $V \cong V^*$.
The object 
set $\bI = \{|^{\otimes n}\:|\:n \in \N\}$ in the above setup may be identified with $\N$.
Hence, $T = \bigoplus_{n \geq 0} T_n$ 
is the tensor algebra $T(V) = \bigoplus_{n \geq 0} T^n(V)$
and $T^\circledast$ is $T(V^*)$.
As is well known, 
the monoidal category $\Acat$ in this case is the {\em Temperley-Lieb category} $\TL(-2)$; see e.g. \cite{GW}. 
It is easy to verify that 
$$
C= T(V) \otimes_A T(V^*)
\cong \k[c_{1,1},c_{1,2},c_{2,1},c_{2,2}] / (\det - 1)
$$
where $c_{r,s} = v_s \otimes u_r$ as above and 
$\det = c_{1,1}c_{2,2}-c_{2,1}c_{1,2}$.
Of course this is $\k[SL_2]$.

This example becomes more interesting if we 
replace the Temperley-Lieb category $\TL(-2)$ with its $q$-analog
 $\TL(-q-q^{-1})$ for $q \in \k^\times$. Recall that this is generated as a strictly pivotal monoidal 
 category by one object $|$ and two morphisms
$\CUP:0 \rightarrow 2$ and $\CAP:2 \rightarrow 0$
subject to  
$\BUBBLE = -q-q^{-1}$.
Assuming $q$ has a square root $q^{1/2} \in \k$, it is braided with braiding defined by
\begin{align}
\POSCROSSING&:=q^{1/2}\:\,\IDID\;+q^{-1/2}\CUPCAP,
&
\NEGCROSSING &= q^{-1/2}\:\IDID\;+q^{1/2}\CUPCAP.
\end{align}
As mentioned in Remark~\ref{tlex}, the natural diagram basis makes the 
path algebra $A$ of $\Acat := \TL(-q-q^{-1})$ into an upper finite based quasi-hereditary algebra with weight poset $(\N,\geq)$.
Hence, $A\lfdlmod$ is an upper finite highest weight category.

Next let $V$ be a two-dimensional vector space with basis $v_1,v_2$
and $(\cdot,\cdot):V \times V \rightarrow \k$ be the bilinear form with
$(v_1,v_2) = 1, (v_2,v_1) = -q^{-1}$ and $(v_1,v_1) = (v_2,v_2) = 0$.
A relation check shows that
there is a monoidal functor $T:\Acat \rightarrow (\fVec)^\op$ such that
$T(|)=V$ and
\begin{align}
T\left(\textstyle\CUP\right) &: V \otimes V \rightarrow \k, &v_i \otimes v_j &\mapsto (v_i,v_j),\\
T\left(\textstyle\CAP\right) &: 
\k \rightarrow V \otimes V, &1 &\mapsto v_2 \otimes v_1 - q v_1 \otimes v_2.
\end{align}
Equivalently, the tensor algebra 
$T = T(V)$ is a right $A$-module, and its dual
$T^\circledast = T(V^*)$ 
is a left $A$-module. 
Then we define $C$ as in (\ref{thec}).
The coend construction makes $C$ into a cobraided Hopf 
algebra, hence,
$\fdrcomod C$ is a braided tensor category.
Now one can check directly using the homological criterion for 
$\nabla$-flags from Theorem~\ref{gf7}
that $T^\circledast$ is a tilting generator for $A\lfdlmod$. Hence,
$\fdrcomod C$
is the Ringel dual of the upper finite highest weight category $A\lfdlmod$,
so it is a lower finite highest weight category thanks to Theorem~\ref{rt2}.

To obtain explicit generators and relations for $C$ in our setup, let $u_1,u_2$ be the basis for $V^*$ dual to $v_1,v_2$.
Then $C$ is generated by 
$\{c_{r,s} := v_s \otimes u_r\:|\:r,s=1,2\}$,
and the comultiplication and counit are defined by 
$\delta(c_{r,s}) = c_{r,1} \otimes c_{1,s} + c_{r,2} \otimes c_{2,s}$,
$\eps(c_{r,s}) = \delta_{r,s}$. 
By Remark~\ref{lostphone}, the following equations give a full set of relations for the algebra $C$:
\begin{align*}
(v_i \otimes v_j) \otimes \left(\CUP 1\right)&=
\left(v_i \otimes v_j \CUP\right) \otimes 1,\\
\left(1 \CAP\right)  \otimes (u_i \otimes u_j)&=
1 \otimes \left(\CAP u_i \otimes u_j \right).
\end{align*}
To expand these, note that the left 
$A$-module $T^\circledast = T(V^*)$ comes from the monoidal functor
$T^\circledast:\Acat \rightarrow \fVec$ defined by
$T^\circledast(|)=V^*$ and
\begin{align}
T^\circledast\left(\textstyle\CAP\right) &: V^* \otimes V^* \rightarrow \k, &u_i \otimes u_j &\mapsto (v_j,v_i)^{-1},\\
T^\circledast\left(\textstyle\CUP\right) &: 
\k \rightarrow V^* \otimes V^*, &1 &\mapsto u_1 \otimes u_2 - q^{-1} u_2 \otimes u_1.
\end{align}
Using this, the relations become
$c_{1,i} c_{2,j} - q^{-1} c_{2,i} c_{1,j} = (v_i,v_j)$
and $c_{i,2} c_{j,1} - q c_{i,1} c_{j,2} = (v_j,v_i)^{-1}$,
hence, we get
$$
\left(\begin{array}{cc}
c_{2,2} & -q c_{1,2}\\
-q^{-1} c_{2,1} & c_{1,1}
\end{array}\right)
\left(\begin{array}{cc}
c_{1,1}&c_{1,2}\\
c_{2,1}&c_{2,2}
\end{array}\right)
=\left(\begin{array}{cc}
c_{1,1}&c_{1,2}\\
c_{2,1}&c_{2,2}
\end{array}\right)
\left(\begin{array}{cc}
c_{2,2} & -q c_{1,2}\\
-q^{-1} c_{2,1} & c_{1,1}
\end{array}\right)
=I_2.
$$
So $C$ is generated by $c_{1,1},c_{1,2},c_{2,1},c_{2,2}$ subject to the relations needed to ensure that
\begin{equation}\label{neatest}
\left(\begin{array}{cc}
c_{1,1}&c_{1,2}\\
c_{2,1}&c_{2,2}
\end{array}\right)^{-1} =
\left(\begin{array}{cc}
c_{2,2} & -q c_{1,2}\\
-q^{-1} c_{2,1} & c_{1,1}
\end{array}\right).
\end{equation}
Equivalently, $C$
is generated by $c_{1,1},c_{1,2},c_{2,1},c_{2,2}$
subject to the relations
\begin{align*}
c_{i,2} c_{i,1} &= q c_{i,1} c_{i,2},&
c_{2,j} c_{1,j} &= q c_{1,j} c_{2,j},\\
c_{1,2} c_{2,1} &= c_{2,1} c_{1,2},&
c_{2,2} c_{1,1} &= c_{1,1} c_{2,2} + (q-q^{-1}) c_{1,2} c_{2,1},
\end{align*}
and 
$\det_q := c_{1,1}c_{2,2} - q^{-1} c_{1,2} c_{2,1} =1$.
Thus, we have reconstructed the well-known quantized coordinate algebra $\k_q[SL_2]$,
and $\fdrcomod C$ is the category of {\em rational representations of quantum $SL_2$}.

When at a root of unity over 
the ground field is $\mathbb{C}$, the indecomposable projectives and injectives 
in the category of rational representations of quantum $SL_2$ (or indeed 
the quantum group corresponding to a reductive group)
are all finite-dimensional, i.e., 
the category is essentially finite Abelian.
Tiltings are also finite-dimensional, indeed, the category is tilting-bounded
in the sense of Definition~\ref{tiltingboundeddef}.
The structure of the principal block can be worked out explicitly (e.g., see
\cite[Th.~3.12, Def.~3.3]{AT}): it is Morita equivalent
to the locally unital
algebra that is the path algebra of the quiver
$$
\xymatrix{
0
\ar@/^/[r]^{x_0}
&
\ar@/^/[l]^{y_0}{1}
\ar@/^/[r]^{x_1}
&
\ar@/^/[r]^{x_2}
\ar@/^/[l]^{y_1}{2}
&
\ar@/^/[l]^{y_2}
3\cdots}
\quad
\text{with relations }x_{i+1}x_i = y_i y_{i+1} = x_i y_i - y_{i+1}
x_{i+1}=0.
$$
The appropriate partial order on the weight poset $\N$ is the natural order
$0 < 1 < \cdots$.
The indecomposable projectives have the following structure:
\begin{align*}
P(0)&=
\begin{tikzpicture}[anchorbase]
\node at (0,0){{\xymatrix@C=0.9em@R=1.1em{
0\ar@{~}[d]|{x}\\
1\ar@{-}[d]|{y}\\
0
}}};\end{tikzpicture}\!,
&
\!P(1) &=
\begin{tikzpicture}[anchorbase]
\node at (0,0){{\xymatrix@C=0.9em@R=1.1em{
&1\ar@{-}[dl]|{y}\ar@{~}[dr]|{x}\\
0\ar@{~}[dr]|{x}&&2\ar@{-}[dl]|{y}\\
&1
}}};\end{tikzpicture}\!,
&
\!P(2) &=
\begin{tikzpicture}[anchorbase]
\node at (0,0){{\xymatrix@C=0.9em@R=1.1em{
&2\ar@{-}[dl]|{y}\ar@{~}[dr]|{x}\\
1\ar@{~}[dr]|{x}&&3\ar@{-}[dl]|{y}\\
&2
}}};\end{tikzpicture}\!,
&
\!P(3) &=
\begin{tikzpicture}[anchorbase]
\node at (0,0){{\xymatrix@C=0.9em@R=1.1em{
&3\ar@{-}[dl]|{y}\ar@{~}[dr]|{x}\\
2\ar@{~}[dr]|{x}&&4\ar@{-}[dl]|{y}\\
&3
}}};\end{tikzpicture}\!,\,
\dots
\end{align*}
The tilting objects are $T(0) := L(0)$ and $T(n) := P(n-1)$ for $n
\geq 1$.
From this, it is easy to see that the Ringel dual is described by the
same quiver with one additional relation,
namely, that $y_0 x_0 = 0$ (and of
course the partial order is reversed).

\subsection{Tensor product categorifications}\label{eg4}
Until quite recently, most of the
naturally-occurring examples were
highest weight categories (like the ones described in the previous two
subsections). But the work of Webster \cite{W1}, \cite{W2}
and Losev and Webster \cite{LW} has brought to prominence a very
general source of examples that are fully stratified 
but seldom highest weight.

Fundamental amongst these new examples are the categorifications of
tensor products of irreducible highest weight modules of symmetrizable
Kac-Moody Lie algebras.
Rather than attempting to repeat the definition of these here, we refer
the reader to \cite{LW}.
All of these examples
are finite fully stratified categories possessing a Chevalley duality.
They are also tilting-rigid; the proof of this depends on an argument involving
translation/projective functors.
Consequently, the Ringel dual is again a finite fully stratified category
that is tilting-rigid.
In fact, the Ringel dual category is always another tensor product
categorification\footnote{This was
noted in Remark 3.10 of the {\tt arxiv} version of \cite{LW} but the
authors removed this
remark in the published version.} 
(reverse the order of the tensor product). 
In the earlier article \cite{W2}, Webster also wrote down explicit
finite-dimensional algebras which give realization of these
categories. In view of Theorem~\ref{TBA3}, all of Webster's algebras 
admit bases making them into symmetrically based stratified algebras, although these bases are usually hard to construct explicitly.

In \cite{W1}, Webster also introduced some more general tensor product
categorifications, including ones which categorify the tensor product of an integrable lowest
weight module
tensored with an integrable highest weight module. 
The latter are particularly important since they may be
realized as {\em generalized cyclotomic quotients} of the Kac-Moody
2-category.
They are upper finite fully stratified categories.
In type A, they can also be realized 
as generalized cyclotomic quotients of the (degenerate or quantum) Heisenberg category; see \cite[Th.~B]{BSW}.
In the latter realization, they should possess explicit triangular bases, generalizing the ones for the cyclotomic quotients of central charge zero discussed in \cite{GRS}.

\subsection{Deligne categories}\label{jonjustadded}
Another source of upper finite highest weight categories comes from
various
Deligne categories.
The definition of these categories is diagrammatic in nature. For
example,
in characteristic zero,
the Deligne category
${\underline{\operatorname{Re}}\!\operatorname{p}}(GL_\delta)$
is the Karoubi envelope of the oriented Brauer category
$\OB(\delta)$.
This case was studied in the PhD thesis of Reynolds \cite{Reynolds} based
on the observation that it admits a symmetric split triangular decomposition; see also \cite{Bskein}
which 
treats the HOMFLY-PT skein category at the same time.
Rui and Song \cite{RS} have analysed the Brauer category and the Kauffman skein category by similar techniques.
Similar ideas have been developed independently by Sam and Snowden \cite{SS}, 
who also consider other types of Deligne category.

The category of locally finite-dimensional representations of the 
Deligne category
${\underline{\operatorname{Re}}\!\operatorname{p}}(GL_\delta)$
can also be interpreted as 
a special case of the lowest weight tensored highest weight
tensor product categorifications discussed in the previous subsection; 
see the introduction of \cite{Bskein}.
The Ringel dual in this example is equivalent to the Abelian envelope 
${\underline{\operatorname{Re}}\!\operatorname{p}^{ab}(GL_\delta)}$
of Deligne's category
constructed by Entova, Hinich and Serganova \cite{EHS}, which is a
monoidal 
lower finite highest weight category.
In \cite{E}, it is shown that
${\underline{\operatorname{Re}}\!\operatorname{p}^{ab}(GL_\delta)}$
categorifies a highest weight tensored lowest weight representation,
which is the dual result to the one from \cite{Bskein}.
This example will be discussed further in the sequel to this article, where we
give an explicit description of the blocks of 
${\underline{\operatorname{Re}}\!\operatorname{p}^{ab}}(GL_\delta)$ 
via Khovanov's arc coalgebra (an interesting explicit 
example of a based quasi-hereditary coalgebra), thereby proving a conjecture
formulated in the introduction of \cite{BS5}.

These and the other classical families of Deligne categories
${\underline{\operatorname{Re}}\!\operatorname{p}}(O_\delta)$,
${\underline{\operatorname{Re}}\!\operatorname{p}}(P)$
and 
${\underline{\operatorname{Re}}\!\operatorname{p}}(Q)$
are being investigated actively along similar lines by several
groups of authors
and there has been considerable recent progress; e.g., see \cite{Cenv}, \cite{SS}.
There are also many interesting connections here with rational
representations of the corresponding families of classical supergroups.

\subsection{Representations of Lie superalgebras}\label{eg5}
Finally, we mention briefly an interesting source of essentially finite
highest weight categories: the analogs of the BGG category $\Ocat$ for classical Lie superalgebras. 
A detailed account in the case of the Lie superalgebra
$\mathfrak{gl}_{m|n}(\mathbb{C})$ can be found in \cite{BLW}. Its category
$\Ocat$ gives an essentially finite highest weight category which is
neither lower finite nor upper finite.
Moreover, it is tilting-bounded as in Definition~\ref{tiltingboundeddef}, so that the Ringel
dual category is also an essentially finite highest weight category.

There is one very easy special case: the principal block of
category $\Ocat$ for $\mathfrak{gl}_{1|1}(\mathbb{C})$ is
equivalent to the category of finite-dimensional modules over the 
essentially finite-dimensional locally unital algebra
which is
the 
path
algebra of the following quiver:
$$
\xymatrix{
\cdots
-\!1
\ar@/^/[r]^{x_{-1}}
&
0
\ar@/^/[r]^{x_0}
\ar@/^/[l]^{y_{-1}}
&
\ar@/^/[l]^{y_0}{1}
\ar@/^/[r]^{x_1}
&
\ar@/^/[l]^{y_1}
2\cdots}
\:
\text{with relations }x_{i+1}x_i = y_i y_{i+1} = x_i y_i - y_{i+1}
x_{i+1}=0,
$$
see e.g. \cite[p. 380]{BS4}.
This is very similar to the $U_q(\mathfrak{sl}_2)$-example from
$\S$\ref{eg3}, but now the poset $\Z$ (ordered naturally)
is neither lower nor upper
finite. From the category $\Ocat$ perspective, this example is
rather misleading since its projective, injective 
and tilting objects coincide, 
hence, it is Ringel self-dual.

One gets similar examples from $\mathfrak{osp}_{m|2n}(\mathbb{C})$,  as
discussed for example in \cite{BW} and \cite{ES}.
The simplest non-trivial case of $\mathfrak{osp}_{3|2}(\mathbb{C})$
produces the path algebra of a
$D_\infty$ quiver (replacing than the $A_\infty$ quiver above); see \cite[$\S$II]{ES}.
The ``strange'' families
$\mathfrak{p}_n(\mathbb{C})$ and $\mathfrak{q}_n(\mathbb{C})$ also exhibit similar
structures.
The former has not yet been investigated systematically (although basic aspects of the finite-dimensional 
finite-dimensional representations and category $\Ocat$ were
recently studied in \cite{many} and \cite{CC}, respectively). It is an
interesting example of a naturally-occurring highest weight category
which does not admit a Chevalley duality.
For $\mathfrak{q}_n(\mathbb{C})$,
we refer to \cite{BD2} and the references therein. 
In fact, the integral blocks for
$\mathfrak{q}_n(\mathbb{C})$ are 
fibered highest weight 
categories; this observation is due to Frisk \cite{Fr}.

\end{document}